\documentclass[a4paper]{amsart}
\usepackage{amssymb, amsmath, amsthm, chngcntr, enumitem, mathrsfs, mathtools, dsfont, esint}
\usepackage[dvipsnames]{xcolor}
\usepackage{graphicx}
\usepackage{pgfplots}
\usepackage{soul}
\usepackage[multiple]{footmisc}
\usetikzlibrary{arrows.meta,calc,patterns,arrows}
\usepackage[outline]{contour}
\contourlength{1.2pt}
\usepackage[english]{babel}
\usepackage[T1]{fontenc}
\usepackage[hyperfootnotes=false]{hyperref}
\hypersetup{
    colorlinks=true,
    linkcolor=blue,
    filecolor=magenta,
    urlcolor=cyan,
}

\numberwithin{equation}{section}
\newtheorem{theorem}{Theorem}[section]
\newtheorem*{theorem*}{Theorem}

\newtheorem{corollary}[theorem]{Corollary}
\newtheorem{observation}[theorem]{Observation}
\newtheorem{lemma}[theorem]{Lemma}
\newtheorem{proposition}[theorem]{Proposition}

\theoremstyle{definition}

\theoremstyle{definition}
\newtheorem{remark}[theorem]{Remark}

\newtheorem{definition}[theorem]{Definition}
\newtheorem{definition lemma}[theorem]{Definition/Lemma}

\usepackage{mathtools}
\mathtoolsset{showonlyrefs,showmanualtags}

\newcommand{\C}{\mathcal{C}}
\newcommand{\R}{\mathbb{R}}
\newcommand{\Z}{\mathbb{Z}}
\newcommand{\Q}{\mathbb{Q}}
\newcommand{\N}{\mathbb{N}}
\renewcommand{\P}{\mathbb{P}}

\newcommand{\size}{\mathrm{Size\, }}
\newcommand{\energy}{\mathrm{Energy \,}}
\newcommand{\ds}{\displaystyle}
\newcommand{\one}{\mathbf{1}}
\newcommand{\ci}{\tilde{\chi}}
\newcommand{\dist}{\text{dist}}
\newcommand{\osc}{\text{osc}}

\newcommand{\ii}{\mathscr}
\newcommand{\mb}{\mathbf}

\renewcommand{\S}{\mathbb{S}}
\newcommand{\BHT}{{{\mathbb P}_{\textrm{BHT}}}}
\newcommand\sub[2]{\genfrac{}{}{0pt}{}{#1}{#2}}

\def\beq{\begin{equation}}
\def\eeq{\end{equation}}

\def\beq{\begin{equation}}
\def\eeq{\end{equation}}

\def\l{\lambda}

\def\L{\mathcal{L}}

\def\t{\tau}

\def\g{\gamma}

\def\a{\alpha}

\def\d{\delta}
\def\b{\beta}

\def\p{\Phi}
\def\ep{\epsilon}

\def\l{\Lambda}

\def\o{\omega}

\def\R{\mathbb{R}}
\def\Q{\mathcal{Q}}

\def\C{\mathbb{C}}
\def\Z{\mathbb{Z}}
\def\K{\mathcal{K}}

\def\P{\mathbb{P}}
\def\p{\mathcal{P}}
\def\rc{\mathcal{R}}

\def\F{\mathcal{F}}
\def\J{\mathcal{J}}
\def\I{\mathcal{I}}

\def\TT{\mathbb{T}}
\def\N{\mathbb{N}}

\def\n{\mathcal{N}}

\def\f{\mathcal{F}}

\def\Z{\mathbb{Z}}
\def\beq{\begin{equation}}
\def\eeq{\end{equation}}

\def\beq{\begin{equation}}
\def\eeq{\end{equation}}

\def\t{\tau}

\def\g{\gamma}

\def\a{\alpha}

\def\d{\delta}
\def\b{\beta}

\def\p{\Phi}

\def\ep{\epsilon}

\def\l{\lambda}
\def\o{\omega}

\def\n{\nu}
\def\R{\mathbb{R}}
\def\C{\mathbb{C}}

\def\Z{\mathbb{Z}}
\def\K{\mathcal{K}}

\def\P{\mathbb{P}}
\def\PI{\mathcal{I}_{\bar{\P}(I_0)}}
\def\p{\mathcal{P}}

\def\L{\mathcal{L}}

\def\J{\mathcal{J}}

\def\TT{\mathbb{T}}
\def\N{\mathbb{N}}
\def\n{\mathcal{N}}
\def\f{\mathcal{F}}

\def\Z{\mathbb{Z}}
\def\I{\mathcal{I}}
\def\beq{\begin{equation}}
\def\eeq{\end{equation}}

\def\beq{\begin{equation}}
\def\eeq{\end{equation}}

 \usepackage{todonotes}
\usepackage{xcolor}

\usepackage[framemethod=tikz]{mdframed}
\mdfsetup{linecolor=blue,backgroundcolor=gray!10,roundcorner=10pt,leftmargin=2pt,innerleftmargin=2pt,innerrightmargin=2pt}

\newcommand{\rr}{\mathbb}
\newcommand{\ic}{\mathcal}

\def\Xint#1{\mathchoice
   {\XXint\displaystyle\textstyle{#1}}%
   {\XXint\textstyle\scriptstyle{#1}}%
   {\XXint\scriptstyle\scriptscriptstyle{#1}}%
   {\XXint\scriptscriptstyle\scriptscriptstyle{#1}}%
   \!\int}
\def\XXint#1#2#3{{\setbox0=\hbox{$#1{#2#3}{\int}$}
     \vcenter{\hbox{$#2#3$}}\kern-.5\wd0}}

\def\aver#1{\Xint-_{#1}}

\counterwithout{equation}{section}
\counterwithout{section}{section}

\begin{document}
\title[The non-resonant bilinear Hilbert--Carleson operator]{\Large{The non-resonant bilinear Hilbert--Carleson operator}}

\author[C. Benea]{Cristina Benea}
\address[C. Benea]{Universit\'e de Nantes, Laboratoire de Math\'ematiques Jean Leray, Nantes 44322, France}
\email{cristina.benea@univ-nantes.fr}
\author[F. Bernicot]{Fr\'ed\'eric Bernicot}
\address[F. Bernicot]{CNRS - Universit\'e de Nantes, Laboratoire de Math\'ematiques Jean Leray, Nantes 44322, France}
\email{frederic.bernicot@univ-nantes.fr}
\author[V. Lie]{Victor Lie}
\address[V. Lie]{Department of Mathematics, Purdue University, IN 47907 USA \& Institute of Mathematics of the Romanian Academy, Bucharest, RO 70700, P.O. Box 1-764, Romania.}
\email{vlie@purdue.edu}
\author[M. Vitturi]{Marco Vitturi}
\address[M. Vitturi]{School of Mathematical Sciences, University College Cork, WesternGateway Building, Western Road, Cork, Ireland}
\email{marco.vitturi@ucc.ie}


\date{\today}

\begin{abstract}
In this paper we introduce the class of bilinear Hilbert--Carleson operators $\{BC^a\}_{a>0}$ defined by
$$BC^{a}(f,g)(x):= \sup_{\lambda\in \R}  \Big|\int f(x-t)\, g(x+t)\, e^{i\lambda t^a} \, \frac{dt}{t} \Big|$$ and show that in the non-resonant case $a\in (0,\infty)\setminus\{1,2\}$ the operator $BC^a$ extends continuously from $L^p(\R)\times L^q(\R)$ into $L^r(\R)$ whenever $\frac{1}{p}+\frac{1}{q}=\frac{1}{r}$ with $1<p,\,q\leq\infty$ and $\frac{2}{3}<r<\infty$.

A key novel feature of these operators is that -- in the non-resonant case -- $BC^{a}$ has a \emph{hybrid} nature enjoying  both
\begin{enumerate}[label=(\Roman*)]
\item \emph{zero curvature} features inherited from the modulation invariance property of the classical bilinear Hilbert transform (BHT), and
\item \emph{non-zero curvature} features arising from the Carleson-type operator with nonlinear phase $\l t^a$.
\end{enumerate}
In order to simultaneously control these two competing facets of our operator we develop a \emph{two-resolution approach}:
\begin{itemize}
\item A \emph{low resolution, multi-scale analysis} addressing (I) and relying on the time-frequency discretization of $BC^a$ into suitable versions of ``dilated'' phase-space BHT-like portraits. The resulting decomposition will produce rank-one families of tri-tiles $\{\P_m\}_m$ such that the components of any such tri-tile will no longer have area one Heisenberg localization. The control over these families will be obtained via a refinement of the time-frequency methods introduced in \cite{lt1} and \cite{lt2}.

 \item A \emph{high resolution, single scale analysis} addressing (II) and relying on a further discretization of each of the tri-tiles $P\in\P_m$ into a four-parameter family of tri-tiles $\S(P)$ with each of the resulting tri-tiles $s\in\S(P)$ now obeying the area one Heisenberg localization. The design of these latter families as well as the extraction of the cancellation encoded in the non-zero curvature of the multiplier's phase within each given $P$ relies on the LGC-methodology introduced in \cite{lvUnif}.
\end{itemize}
A further interesting aspect of our work is that the high resolution analysis itself involves two types of decompositions capturing the local (single scale) behavior of our operator:
\begin{itemize}
\item A \emph{continuous phase-linearized spatial model} that serves as the vehicle for extracting the cancellation from the  multiplier's phase. The latter is achieved via $TT^{*}$ arguments, number-theoretic tools (Weyl sums) and phase level set analysis exploiting time-frequency correlations.

\item A \emph{discrete phase-linearized wave-packet model} that takes the just-captured phase cancellation and feeds it into the low resolution analysis in order to achieve the global control over $BC^a$.
\end{itemize}

As a consequence of the above, our proof offers a unifying perspective on the distinct methods for treating the zero/non-zero curvature paradigms.
\end{abstract}

\maketitle

\setcounter{tocdepth}{1}
\tableofcontents

\section{Introduction}\label{Intr}

The present problem belongs to an extensive research program\footnote{For an overview of some of the directions that are part of this extensive program please see the historical presentation in Section \ref{Hist}.} that aims to advance our understanding of the subtle interplay between modulation invariance and curvature in harmonic analysis.

Concretely, in this paper we study the boundedness behavior of a (sub)bilinear operator $T: S(\R) \times S(\R) \to  S'(\R)$ that has the following key properties\footnote{In what follows, we let $\tau_b f(x):=f(x-b)$ the translation symmetry, $D_{\a}f(x):=f(\a x)$ the dilation symmetry, and $M_c f(x):=e^{i c x} f(x)$ the modulation symmetry where $x, \a, b, c \in \R$.}:
\begin{enumerate}[label=(P\arabic*)]
\item \label{key:td} $T$ commutes with translations and dilations, \textit{i.e.}
$$T(\tau_b f, \tau_b g)=\tau_b T(f,g)\:\:\:\textrm{and}\:\:\: T(D_{\a} f, D_{\a} g)=D_{\a} T(f,g);$$

\item \label{key:m} $T$ is invariant under the action of modulation, \textit{i.e.}
$$|T(M_c f, M_c g)|= |T(f,g)|;$$

\item \label{key:max} $T$ is a  maximal singular integral operator whose kernel representation involves a maximal (generalized) modulation in the presence of curvature.\footnote{The concept of curvature is of course multifaceted and covers many aspects within and beyond the area of harmonic analysis. For this reason, and also to allow for some flexibility in covering a diverse array of problems that will be discussed in the Introduction of our paper, we will not attempt to provide a precise definition of it. We will thus only observe that the study of an operator in a non-zero curvature regime will canonically rely on an application of the principle of (non)stationary phase, usually targeting the multiplier side and having as an effect the presence of a decaying-type factor quantifying the non-zero curvature. This latter property strongly correlates with the fact that the operator under discussion lacks a modulation invariance structure.}
\end{enumerate}

The above features, expectedly, will play a determinative r\^{o}le in our analysis and will be responsible for the methodology used in approaching our problem and its novelties.

Before entering into a more detailed discussion of the above considerations, we present the main result of our paper.

\subsection{Main result}\label{MRe}  Let $a>0$ be a fixed parameter and $f, g\in S(\R)$ two smooth functions. Define the \emph{bilinear Hilbert--Carleson operator} with monomial phase of exponent $a$ as\footnote{\label{footnote:2pi} Throughout this paper, for notational simplicity we will drop the principal value symbol from singular integral formulas. Also, by convention, when $a>0$ is not an integer, the expression $t^a$ stands here for either $|t|^a$ or for $\textrm{sgn} {t} |t|^a$. Finally, for simplicity we will appeal to a notational abuse and write $e^{i A}$ as a shortcut for the expression $e^{2\pi i A}$, where here $A$ is always assumed to be real.}
\beq\label{Top}
BC^{a}(f,g)(x):= \sup_{\lambda\in \R}  \Big|\int f(x-t) g(x+t) e^{i\lambda t^a}  \frac{dt}{t} \Big|.
\eeq

\begin{observation}\label{nres}
When $a\notin\{1, 2\}$, $BC^a$ has no modulation symmetries beyond the expected ones in \ref{key:m} inherited from the bilinear Hilbert transform. In such a situation we say that $BC^a$ belongs to the family of \emph{non-resonant} bilinear Hilbert--Carleson operators. For more on this subject one is invited to consult Sections \ref{znzpar} and \ref{openquest}.
\end{observation}

With the above notations, we have

\begin{theorem} \label{main} Let\footnote{For a discussion on the necessity of imposing the restriction $a\notin\{1,2\}$ please see the items \ref{OQ:3} and \ref{OQ:4} in the open question Section \ref{openquest} below.} $a\in(0, \infty)\setminus\{1, 2\}$ and assume $p, q, r$ are H\"older indices, \textit{i.e.} $\frac{1}{p}+\frac{1}{q}=\frac{1}{r}$, with $1<p, q \leq \infty$ and $\frac{2}{3}<r<\infty$. Then the non-resonant bilinear Hilbert--Carleson operator extends continuously from $L^p(\R)\times L^q(\R)$ into $L^r(\R)$ with
\beq\label{mainr}
\|BC^{a}(f,g)\|_{L^r}\lesssim_{a,p,q} \|f\|_{L^p} \|g\|_{L^q}.
\eeq
\end{theorem}

\begin{remark}\label{cor} The above theorem extends to more general types of phases that obey suitable conditions meant to preserve the hybrid character of the newly defined operator. A natural such extension that involves extra-technicalities but requires no fundamental modification relative to the proof strategy for Theorem \ref{main} is the following\footnote{Due to space limitation concerns we make no attempt to provide a proof of this result in our present paper.}:
$\newline$
$\newline$
\emph{Let $d\in\N$, $d\geq 3$ and $P \in \mathbb{R}[t]$ be a real polynomial of degree $d$ with $P'(0)=P''(0)=0$. Then the operator
\beq\label{BP}
BC^{P}(f,g)(x):= \sup_{\lambda\in {\mathbb R}} \Big|\int f(x-t)g(x+t) e^{i\lambda P(t)}\frac{dt}{t}\Big|
\eeq
satisfies the bound
\beq\label{BP1}
 \|BC^{P}(f,g)\|_{L^r} \lesssim_{d,p,q} \|f\|_{L^p} \|g\|_{L^q}
\eeq
for the same range of $p, q$ and $r$ as that for $BC^{a}$ above.}
\end{remark}

\begin{observation}\label{coment} As the name suggests, $BC^{a}$ is a ``concatenation'' of two classical objects in harmonic analysis:
\begin{itemize}
\item the bilinear Hilbert transform defined by
\beq\label{bhtd}
B(f,g)(x):= \int f(x-t) g(x+t)  \frac{dt}{t} ;
\eeq
\item the Carleson-type\footnote{Such operators are explicitly defined in \cite{lvUnif} under the name $\gamma$-\emph{Carleson operators} with the intention of creating a continuum family of operators that covers both the classical Carleson operator and the purely curved (no linear term) polynomial Carleson operator considered by Stein \cite{Stein} and Stein and Wainger \cite{SteinWainger}. In this more general setting $\g$ stands for a suitably parametrized planar curve.} operator given by
\beq\label{sw}
C^a(f)(x):= \sup_{\lambda\in \R}  \Big|\int f(x-t) e^{i\lambda t^a}  \frac{dt}{t}  \Big|.
\eeq
\end{itemize}
As a consequence of the above and in direct correspondence with the properties \ref{key:td}-\ref{key:max}, our approach to the \emph{non-resonant} case relies on and unifies
\begin{itemize}
\item the time-frequency methods developed in \cite{lt1} and \cite{lt2} for the study of \eqref{bhtd}, and

\item the non-zero curvature methodology employed in \cite{lvUnif} for treating \eqref{sw}.
\end{itemize}
\end{observation}

For an overview of the key ideas and novelties in our approach the interested reader is invited to consult Section \ref{Genov}. In the reminder of this introductory section we will focus on the motivation/relevance for considering the class of bilinear Hilbert--Carleson operators $\{BC^a\}_{a>0}$ as well as on a discussion of the historical background that reviews several related topics.

\subsection{Motivation}\label{Motiv}  Since in this paper we introduce a new class of operators, \textit{i.e.} $\{BC^a\}_a$, it seems natural to dedicate an entire section to thoroughly presenting the multiple reasons for our interest in its study. Indeed, the introduction of the class of bilinear Hilbert--Carleson operators has a multifaceted motivation:
\begin{itemize}
\item The first and most relevant for us is the novel complexity hierarchy that unfolds within the zero/non-zero curvature paradigm.
\item A second motivation is provided by the general, unifying character of this class of operators, which brings under the same umbrella (and extends) a few classical objects of interest in harmonic analysis.
\item Thirdly, our class naturally relates to some interesting current directions of research.
\end{itemize}

In what follows we discuss in more detail each of the above items.

\subsubsection{A novel manifestation within the zero/non-zero curvature paradigm}\label{znzpar}

As mentioned at the beginning of our introduction, the formulation of our present problem was mainly motivated by the desire to shed more light into the ambitious program of developing a unified theory that simultaneously treats zero and non-zero curvature versions of several important themes within harmonic analysis.

Heuristically speaking, it is well known that the non-zero curvature problems are expected to have a lesser degree of complexity than their zero curvature counterparts, partly because of the (much) richer classes of symmetries obeyed by operators in the latter category.  Thus, in the quest to obtain an integrated theory for treating these two regimes, it becomes natural to focus on the nature of changes in one's approach as one transitions from non-zero to zero curvature problems. In particular, it is of high interest to design relevant and historically motivated hybrid models that contain both ends of the spectrum.  In order to provide better context to our discussion, we continue with the following philosophical discussion:
$\newline$

\noindent \underline{\textsf{Generic proof strategy outline for approaching the zero/non-zero curvature paradigm.}}
$\newline$

In the context of the literature to date -- see also the situations discussed in Section \ref{Hist} below -- one deals with an operator, referred to generically as $T$, that lands in precisely one of the following two scenarios:

\begin{enumerate}[label=(\Roman*), leftmargin=16pt]
\item \label{I:zero:curvature} \underline{the \emph{zero curvature} case}: This is the situation in which $T$ has a (generalized) modulation invariance symmetry, thus requiring an approach that puts equal weight on all frequency locations (\textit{i.e.}, the structure of the operator is invariant under frequency translations). This type of problem requires a time-frequency localization of the input functions on which $T$ acts, thus naturally bringing into the picture wave-packet analysis and, in particular, a time-frequency tile discretization of the phase-space portrait of $T$. As expected, the fundamental structures in this context are represented by geometric configurations of tiles that share the same time-frequency location, called trees. The crux of the proof in these problems is to properly quantify 1) the almost-orthogonality relations among various trees, and 2) the ``amount of information'' relative to the input functions that each of these trees carries. The former aspect is achieved via a careful combinatorial argument, while the latter relies on a greedy algorithm involving concepts such as \emph{mass}, in the spirit of Fefferman's proof of the $L^2$-boundedness of the Carleson operator, and/or \emph{size}, reminiscent of the Lacey-Thiele approach to the boundedness of the bilinear Hilbert transform.
\medskip

\item \label{II:non:zero:curvature} \underline{the \emph{non-zero curvature} case}: in this situation $T$ has no modulation symmetry, and thus the zero frequency is  expected to play a favoured r\^{o}le in its analysis. As a consequence, one applies the decomposition
\beq\label{stdecc}
T=:T_0+T_{\osc},
\eeq
such that
\begin{itemize}
\item $T_0$ is a (maximal truncated) singular operator whose Fourier multiplier has essentially no phase oscillation and whose behavior is well understood.

\item $T_{\osc}$ is a (maximal) oscillatory integral operator whose Fourier multiplier has a highly oscillatory phase. In this situation one performs a further discretization
\beq\label{stdecce}
T_{\osc}=\sum_{m\in\N} T_{m},
\eeq
with each $T_m$ representing the piece of $T$ whose multiplier phase has the height $\approx 2^{\mu m}$ for some properly chosen $\mu>0$ depending on the specific nature of $T$. Via (non)stationary phase analysis, $TT^{*}$-method, and interpolation techniques, the goal is to show that
\beq\label{stdeccee}
\|T_m\|_{L^p\to L^p}\lesssim_p 2^{-\d(\mu,p)m}
\eeq
for some $\d(\mu,p)>0$ and thus conclude that the $L^p$ norm of $T_{\osc}$ is under control via a telescoping sum argument.
\end{itemize}
\end{enumerate}

It is in this realm that the thought of developing a unified theory that covers the entire non-zero/zero curvature spectrum naturally comes to life. However, with the notable exception of the polynomial Carleson theme discussed below in Section \ref{pco},  there is to date no other fundamental class of operators for which such an integrated theory exists. Within this context, our paper is intended as a step forward in providing a better understanding of the subtle interplay between non-zero and zero curvature. Indeed, the present paper provides a first instance in which the analysis of the operator $T$ \emph{simultaneously}
\begin{itemize}
\item has both zero and non-zero curvature features,

\item incorporates both approaches \ref{I:zero:curvature} and \ref{II:non:zero:curvature} above, and

\item requires modulation invariance techniques for both components in decomposition \eqref{stdecc}.
\end{itemize}

These features reflect the \emph{hybrid} nature of our operator $T:=BC^a$ which is designed as a mixture between the bilinear Hilbert transform and a Carleson-type operator:
\begin{enumerate}[label=(\roman*)]
\item regarded from the input side via \ref{key:m}, $BC^{a}$ has zero-curvature features;

\item regarded from the complex exponential kernel side, in the case $a\in (0,\infty)\setminus\{1\}$, our operator $BC^{a}$ has non-zero curvature features.
\end{enumerate}

\begin{observation}\label{reson}
In this newly created hybrid category, it is important to say that any discussion of the complexity hierarchy centered around the concepts of zero/non-zero curvature must take into account the \emph{joint} behavior of the symmetries determined by the interaction between the structure of the input function(s) \emph{and} that of the operator's kernel. Indeed, taking as an example the operator $BC^a$ defined in \eqref{Top}, it is this interaction effect that truly determines the difficulty of the problem, since this joint interference may produce ``\emph{resonances}'' not captured by either of the single items (i) and (ii) above if treated independently -- see the discussion for the $a=2$ case in the third item of Section \ref{openquest} with special emphasis on \eqref{quadm}. For another illuminating example one is invited to read the ``open problems'' part of Section \ref{msrt} that discusses the behavior of suitable maximally modulated singular Radon transforms.
\end{observation}

\subsubsection{Unifying and extending some classical themes}\label{unif2th}

As noted above, our (non-resonant) bilinear Hilbert--Carleson operator $BC^a$ \emph{combines} and -- as we will see immediately below -- \emph{controls} the behavior of the bilinear Hilbert transform $B$ and that of the Carleson-type operator $C^a$. Indeed, assume in what follows that the conclusion of our Theorem \ref{main}  holds. Then we deduce the following:

\begin{itemize}

\item \emph{Implications about the behavior of the  bilinear maximal/maximally truncated/bilinear Hilbert transform}:
We first notice that after linearizing the supremum in \eqref{Top}, if we take $\l(x)=0$ then the resulting object is precisely the classical bilinear Hilbert transform (BHT), and hence any bounds for $BC^a$ trivially imply the same bounds for the BHT. Applying now the strategy described in \eqref{stdecc}-\eqref{stdeccee} for $T=BC^a$ and following the proof\footnote{Note that the proof of Theorem \ref{main_theorem_local_L2} involves only BHT-type behavior and does not rely on any \emph{a priori} understanding of the bilinear maximal Hilbert transform; accordingly, deducing the local-$L^2$ boundedness of the maximally truncated bilinear Hilbert transform in this way does not involve circular reasonings.} of our key underlying result stated as Theorem \ref{main_theorem_local_L2} we deduce that the highly oscillatory component $T_{\osc}$ is bounded within the local-$L^2$ H\"older range. Once at this point, one observes that the low oscillatory component $T_0=BC^a-T_{\osc}$ is equivalent to a maximally truncated bilinear Hilbert transform (up to well behaved error terms) and thus the boundedness of the latter operator -- at least in the  local-$L^2$  H\"older range -- is implied by the corresponding bounds on $BC^a$.

\item \emph{Implications about the behavior of the Carleson-type operator $C^a$}: For this theme one can simply take in \eqref{mainr} the endpoint bound $1<p=r<\infty$ and $q=\infty$ and set $g\equiv 1$ to deduce immediately the corresponding $L^p$ bounds for the Carleson-type operator $C^a$.\footnote{There is an alternate, indirect approach for showing that $L^p$ bounds on $BC^a$ imply corresponding $L^p$ bounds for $C^a$ for $1<p\leq 2$ without relying on the endpoint bounds $(p,\infty,p)$ in our Theorem \ref{main} . Indeed, by first restricting the support of $f$ to $[-1,1]$ and then choosing $g=\one_{[-1,1]}$ one can deduce from the open boundedness range for $BC^a$ that $C^{a}_{\TT}$ -- the restriction of $C^{a}$ to the torus -- maps $L^q(\TT)$ into $L^r(\TT)$ continuously for any $1<r<q<\infty$. Using Stein's maximal principle \cite{s1} and real interpolation one concludes that $C^{a}_{\TT}$ is of strong type $(p,p)$ for any $1<p<2$ and of weak type $(2,2)$. The corresponding conclusions on the real line, \textit{i.e.} the $L^p \to L^p$ boundedness of $C^a$ for $1<p<2$, can now be obtained via a transference principle.}
\end{itemize}

\subsubsection{Connections with other research directions}\label{Otherdirections}
$\newline$

\noindent\textsf{I. Power decay property for multilinear oscillatory integrals and polynomially modulated bilinear Hilbert transforms.}
$\newline$

We start our discussion here with the following:
$\newline$

\noindent\textsf{Problem A. \cite{cltt}} \emph{Fix $m,n\in\N$, $m\geq 2$, $\l$ a real parameter and let $P:\R^m\to \R$ be a suitably smooth function. Define the multilinear functional
\beq\label{mF}
\Lambda_{\l}(f_1,\ldots,f_n):=\int_{\R^m} e^{i \l P(x)}\prod_{j=1}^n f_j(\pi_j(x)) \eta(x)dx,
\eeq
where here $\eta\in C_0^1(\R^m)$, $\pi_j:\R^m \to V_j\subseteq \R^m$ are orthogonal projections with $V_j$ linear subspaces of dimension $s<m$, and $f_j:V_j \to \C$ are locally integrable functions (relative to the induced Lebesgue measure on $V_j$).
$\newline$
Characterize those data $(P, \{V_j\}_j)$ for which the following power decay property holds: there exist $\ep>0$ and $C=C(\|\eta\|_{C^1})>0$ such that
\beq\label{PD}
|\Lambda_{\l}(f_1,\ldots,f_n)|\leq C(1+|\l|)^{-\ep} \prod_{j=1}^n\|f_j\|_{L^{\infty}(V_j)},
\eeq
for all $f_j\in L^{\infty}(V_j)$ and $\l\in\R$.}
$\newline$

This theme of research may be seen as a natural development of some old and fundamental questions investigating the asymptotic behavior of expressions involving highly oscillatory phases. A prime and basic example of the latter is offered by the classical Van der Corput lemma, which states that
\beq\label{vC}
\int_{a}^b e^{i \l \varphi(t)} dt = O((1+|\l|)^{-\ep}),
\eeq
where here $\l>0$, $a<b$, $\varphi$ is a real smooth function with some derivative $\varphi^{(k)}$ bounded away from zero\footnote{If $k=1$ one also assumes $\varphi'$ monotone.} and $\ep=\ep(k)>0$. Another representative example,\footnote{This may serve as a linear prototype for the formulation of Problem A above.} proved in \cite{PS94}, is the relation
\beq\label{PS}
\Big|\int_{\R^2} e^{i \l P(x,y)}f(x) g(y)\eta(x,y) dxdy\Big|\lesssim (1+|\l|)^{-\ep}\|f\|_{L^2}\|g\|_{L^2},
\eeq
whenever $P$ is a real-valued polynomial that obeys suitable nondegeneracy conditions.\footnote{The presence of the $L^2$ norm in the right-hand side of \eqref{PS} as opposed to the $L^{\infty}$ norm appearing in \eqref{PD} is not of key relevance since, via standard interpolation techniques, one can always establish a correspondence between these two situations.}

Returning to the initial Problem A, with the exception of the cases $n=2$ or $s=1$ and $m=n$ which were extensively studied and are thus better understood, few results have been established for more general situations, with these almost exclusively limited to the case of $P$ real analytic or merely polynomial -- for more on this entire subject please see \cite{cltt} and more recently \cite{C20} and the bibliography therein.

Focusing now on the case $P$ polynomial -- in analogy with the restrictions imposed for validating \eqref{PS} -- it turns out that some suitable nondegeneracy condition is in fact necessary: Indeed, one is naturally\footnote{It is immediate that if $P$ is degenerate in the sense of \eqref{Poldec}, then the terms in this decomposition can be distributed to the input functions, thus removing any possibility of decay in \eqref{PD}.} led to the requirement that $P$ does not admit a decomposition of the form
\beq\label{Poldec}
 P=\sum_{j=1}^n p_j \circ \pi_j\qquad \textrm{with}\qquad p_j:V_j \to {\mathbb R}\qquad \textrm{polynomial}.
\eeq
Thus, a natural next step is to consider
$\newline$

\noindent\textsf{Problem B.} \emph{Is the power decay property in \eqref{PD}  equivalent to nondegeneracy, \emph{i.e.}, nonexistence of a decomposition of the form \eqref{Poldec}?} $\newline$
Now as revealed in \cite{cltt}, Problem B (and A in this generality) was motivated and inspired by the desire to answer the following:
$\newline$

\noindent\textsf{Problem C. (Bilinear Hilbert transform with polynomial phase)}  \emph{Let $P=P(x,t)$ be a real polynomial of degree at most $d\in\N$ and, for $f,g\in C_0^1(\R)$, define
\beq\label{BiP}
B^P(f,g)(x):=\int_{\R} e^{iP(x,t)} f(x-t) g(x+t)\frac{dt}{t}.
\eeq
Then, show that for any $p, q, r$ H\"older indices, \textit{i.e.} $\frac{1}{p}+\frac{1}{q}=\frac{1}{r}$, with $1<p, q\leq \infty$ and $\frac{2}{3}<r<\infty$, there exists $C=C(d)>0$ such that
\beq\label{BiPbound}
\|B^P(f,g)\|_{L^r}\leq C\|f\|_{L^p}\|g\|_{L^q},
\eeq
uniformly for all real valued polynomials of degree at most $d$.}
$\newline$

A simple inspection of \eqref{BiP} shows that if $d\leq 2$ then \eqref{BiPbound} can be reduced -- via some simple generalized modulation transformations -- to the classical known bounds for the bilinear Hilbert transform (\cite{lt1,lt2}). Thus the only interesting case is $d\geq 3$, for which the proof of \eqref{BiPbound} can be reduced to the following key statement -- which indeed clarifies the connections between Problems A, B and C: Let $P$ be a real nondegenerate  polynomial of degree $d$ and $\eta\in C_0^1(\R)$ supported away from zero. Then there exists $\ep>0$ such that for any $u\in\N$
\beq\label{BiPK}
\Big\|\int_{\R} e^{i P(2^u x,2^u t)} f(x-t) g(x+t) \eta(t) dt\Big\|_{L^r_x}\leq C 2^{- \ep u} \|f\|_{L^p}\|g\|_{L^q}.
\eeq
Now Problem C together with \eqref{BiPK} and also some particular instances of Problem B were addressed in \cite{cltt} using an approach relying on the concept of  $\sigma-$uniformity introduced by Gowers in \cite{gowers}.

In this context, the connection between the research theme of our present paper and Problem C together with its ``relatives'' Problems A and B above becomes transparent:
\begin{itemize}[leftmargin=24pt]
\item  The bilinear Hilbert--Carleson operator $BC^{a}$ -- together with its polynomial bilinear Hilbert--Carleson generalization
$BC_{d}$ in \eqref{BPHd} -- extends the behavior of $B^P$ in \eqref{BiPbound} to the much wilder classes of phases $P(x,t)=\sum_{j=1}^d a_j(x)t^j$ with $d\in\N$ and $\{a_j\}_j$ arbitrary real measurable functions.

\item In the non-resonant case, Theorem \ref{main}  together with Remark \ref{cor} immediately imply the bounds in \eqref{BiPbound} for operators $B^P$ in \eqref{BiP} having phase functions of the form $P(x,t):=Q(x) R(t)$ with $R$ any real polynomial obeying $R'(0)=R''(0)=0$ and $Q$ any real \emph{measurable} function.

\item Moreover, a fundamental component of the proof of Theorem \ref{main} is the so-called \emph{single-scale estimate} -- stated as Proposition \ref{sgscale} -- whose proof provides a novel (and seemingly first such) treatment of a particular instance of Problem A (and thus \eqref{BiPK}) in the setting $n=3$, $m=2$ and $s=1$ without assuming any $x$-smoothness for the phase $P$ (specifically, we treat phases $P$ obeying the same properties as in the previous item). We note further that the proof of our Proposition \ref{sgscale} relies on the methodology developed in \cite{lvUnif} paired with a number-theoretic approach and thus circumvents the $\sigma-$uniformity technique employed in \cite{cltt} to approach Problems A, B and C.\footnote{It is worth noticing here that -- under the assumption that $P$ is non-degenerate (recall \eqref{Poldec}) and that is properly normalized -- the proof of Problem C reduces to a single scale estimate (in the sense described by Proposition \ref{sgscale}) and does not require any multi-scale/time-frequency analysis.}
\end{itemize}

$\newline$
\noindent\textsf{II. Convergence of (bilinear) Fourier Series}\label{BiFS}
$\newline$

\noindent\textsf{II.1. Historical context: The linear, one dimensional case.} As is well known, the celebrated problem of the pointwise convergence of Fourier Series is the central pillar of the structure on which the edifice of harmonic analysis area was built.\footnote{For more on this one is invited to consult Section \ref{pco} and also the significantly more detailed historical perspective offered in the Introduction of \cite{lv3}.} Indeed, emerging from the foundational study of Fourier on heat propagation, \cite{Fou}, the fundamental question was to convey a proper meaning to the equality sign in
\beq\label{FS}
f(x)=\sum_{n\in \Z} \hat{f}(n) e^{2\pi i n x},
\eeq
where here $f:\R\to \C$ is a suitable $1$-periodic function and $\hat{f}(n)$ is the Fourier coefficient of order $n$ associated to $f$.

The initial inquiry into the subtle pointwise convergence properties of \eqref{FS} motivated throughout the $19\textsuperscript{th}$ century the rigorous reshaping of mathematical analysis through the contributions of personalities such as Dirichlet, Cauchy, Weierstrass, Riemann and many others. Along the way, the pointwise equality in \eqref{FS} was positively settled for differentiable functions and refuted for the class of continuous functions. This latter surprising finding -- manifested in the last decades of the $19\textsuperscript{th}$ century -- shifted the interest towards understanding the ``pathologies'' around the failure of the pointwise convergence behavior in \eqref{FS}. This was the context that generated two of deeply consequential mathematical developments: the set theory of Cantor and the modern theory of integration due to Lebesgue. With these, the theoretical framework around \eqref{FS} was finally properly established.

Having reached at this point, we enumerate -- moving fast-forward in time -- the most relevant contributions to the topic discussed here:
\begin{itemize}
\item The resolution of \eqref{FS} in the $L^p$ sense, $1<p<\infty$, was achieved via complex analysis by Riesz, \cite{Rie},  thus establishing a deep connection with the theory of harmonic functions and their conjugates and consequently with what was soon to become the theory of Hardy spaces. Moreover, Riesz's result, which can be stated equivalently as the $L^p$ boundedness of the Hilbert transform, served as a main prototype for the Calder\'on-Zygmund theory of singular integral operators developed few decades later in \cite{CZ1}, \cite{CZ2}.

\item In the realm of (almost everywhere) pointwise convergence, the first major contribution belongs to Kolmogorov, \cite{Kol1,Kol2}, who proved the almost-everywhere divergence of the Fourier Series within the class of $L^1$ functions.

\item In the positive direction, Carleson, \cite{c1}, settled affirmatively Lusin's conjecture on the almost everywhere convergence of Fourier series for square integrable functions. This was later extended to the $L^p$ case, $1<p<\infty$, by Hunt, \cite{hu}. The proof of both results relied on providing $L^p$-weak bounds for the associated maximal operator
\beq\label{CrT}
    \sup_{N\in\N} \Big|\sum_{|n|\leq N} \hat{f}(n) e^{2\pi i n x}\Big|,
\eeq
    which is equivalent with the so-called (classical) Carleson operator, which in the real line framework may be expressed as
\beq\label{Cr}
Cf(x):=\sup_{N\in\R} \Big|\int_{\xi<N}\hat{f}(\xi) e^{i x \xi} d\xi \Big|.
\eeq
  We add that the approach developed by Carleson in \cite{c1} together with the different perspective offered by Fefferman in \cite{f} constitutes the keystone of today's time-frequency analysis.
\end{itemize}
We end the brief historical overview on \eqref{FS} by mentioning another fascinating aspect, namely the question regarding the optimal spaces for the almost-everywhere convergence of the Fourier series near $L^1$. Far from being just an unanswered curiosity, this question is deeply connected with special Cantor-set-type pathologies and additive-combinatoric time-frequency structures -- for more on these one is invited to consult \cite{lv9} and \cite{lv19}, respectively.
$\newline$

\noindent\textsf{II.2. Historical context: The linear, two dimensional case.} Once the question of norm and pointwise convergence for \eqref{FS} was settled in the $L^p$, $1<p<\infty$, context, the next natural step was the investigation of its higher dimensional analogue. More precisely, one would like to understand whether or not one can associate a meaning to
\beq\label{FSH}
f(x)=\sum_{n\in \Z^2} \hat{f}(n) e^{2\pi i \langle n, x\rangle}=\sum_{n_1,n_2 \in \Z} \hat{f}(n_1,n_2) e^{2\pi i (n_1 x_1+n_2 x_2)},
\eeq
where here $f:\R^2 \to \C$ is a suitable $1-$periodic function in both coordinates $x=(x_1,x_2)$, $\hat{f}(n)$ is the Fourier coefficient of order $n$ associated to $f$ and $\langle\cdot,\cdot \rangle$ stands for the standard inner product in $\R^2$.

As in the case of its one dimensional analogue, this question is deeply interconnected with several fundamental mathematical phenomena among which a central role is played by the increased complexity of the geometric structure of higher dimensional sets. This brings to light new ``pathological'' structures -- specific to the higher dimensional setting -- called Besicovitch sets that turn out to have a profound impact in the study of \eqref{FSH}.\footnote{The concept of Besicovitch set plays a key role in harmonic analysis: a number of fundamental open questions are intimately related to it, among which we mention the Zygmund conjecture on differentiation along vector fields (see Section \ref{htc}), the optimal behavior of Bochner-Riesz means with the related topic of summability of the Fourier Series in two dimensions, and of course the well-known Kakeya and Restriction conjectures.} To better understand this, we discuss antithetically both the discrete setting of \eqref{FSH} and
its continuous Fourier integral analogue:

In the continuous case, we have:
\begin{itemize}
\item Consider first the two-dimensional version of \eqref{Cr}, which on the physical-space side corresponds to
\beq\label{Crn}
Sf(x):=\sup_{N\in\R^2} \Big|\int_{\R^2} f(t) K(x-t) e^{ i \langle N, t \rangle} dt \Big|,
\eeq
where here $K$ stands for a standard Calder\'on-Zygmund kernel. Then, as in the one-dimensional case, $S$ is $L^p$-bounded for $1<p<\infty$ -- see \cite{sj2}.

\item Next, consider the rough convolution operator defined on the Fourier side by
\beq\label{Ball}
T_{B}f(x):=\int_{\R^2} \hat{f}(\xi) \one_{B}(\xi) e^{ i \langle x, \xi \rangle} d\xi,
\eeq
where here $\one_{B}$ stands for the characteristic function of the unit disc (ball) in $\R^2$. Then the celebrated result of
Fefferman \cite{Fef71} states that
\beq\label{Ball1}
T_{B}:L^p(\R^2) \to L^p(\R^2)\:\:\textrm{boundedly}\:\:\Longleftrightarrow\:\:p=2.
\eeq
The proof is based on constructing a certain counterexample based precisely on the existence of Besicovitch sets in any $\R^n$, $n\geq 2$. It follows that the maximal operator $Tf:=\sup_{n>0} |T_{nB} f|$ is unbounded on any $L^p(\R^2)$ except, possibly, for $p=2$.
\end{itemize}

In the discrete case,\footnote{In this context $f:\R^2 \to \C$ is a $1$-periodic function in both variables.} we have the following correspondences:
\begin{itemize}
\item Let $\p$ be an open polygonal region in $\R^2$ containing the origin. Then the relation
\beq\label{Polyg}
f(x_1,x_2)=\lim_{N\to\infty} \sum_{n\in N\p} \hat{f}(n) e^{2\pi i \langle n, x \rangle}
\eeq
holds both almost everywhere and in norm for any $f\in L^p(\TT^2)$, $1<p<\infty$. This result was proved in \cite{Fefferman-ConvergencePolygonalFourierSeries} and is based on the observation that \eqref{Polyg} can be reduced to the one-dimensional Carleson-Hunt result.

\item In the opposite direction, if one considers the problem of summability along rectangles of arbitrary eccentricity, then there are continuous functions $f\in L^2(\TT^2)$ for which the relation
\beq\label{rectangle}
f(x_1,x_2)=\lim_{N_1,N_2\to\infty} \sum_{|n_1|\leq N_1, |n_2|\leq N_2} \hat{f}(n_1,n_2) e^{2\pi i (n_1 x_1+n_2 x_2)},
\eeq
holds nowhere. The proof of this result is provided in \cite{Fefferman-DivergenceMultipleFourierSeries} and is based on constructing the function $f$ as a suitable superposition of smoothened variants of imaginary polynomial exponentials of the form $f_{\l}(x_1,x_2)=e^{i \l x_1 x_2}$.

\item Finally, if one considers the problem of summability along discs, that is
\beq\label{circle}
f(x_1,x_2)=\lim_{N\to\infty} \sum_{n_1^2+n_2^2\leq N^2} \hat{f}(n_1,n_2) e^{2\pi i (n_1 x_1+n_2 x_2)},
\eeq
then due to \eqref{Ball1}, we know that \eqref{circle} is ill posed both in norm and pointwise for $f\in L^p(\TT^2)$ and $p\not=2$, while the pointwise behavior for the $p=2$ case is a well known open problem in harmonic analysis.
\end{itemize}

\bigskip

\noindent\textsf{II.3. The bilinear, one dimensional case.} With the above historical context settled, and following the natural
path of relating linear multidimensional behavior to multilinear one-dimensional behavior,\footnote{For another instance of such a connection one is invited to read the saga of the Hilbert transform along curves (Section \ref{htc}) and that of the bilinear Hilbert transform along curves (Section \ref{bht}).} one is invited to consider the bilinear one dimensional analogue of \eqref{FSH}, that is, the question of what meaning one can assign to the equality
\beq\label{bilinfseries}
f(x) g(x)=\sum_{n_1, n_2 \in \Z} \hat{f}(n_1) \hat{g}(n_2) e^{i n_1 x} e^{i n_2 x} ,
\eeq
where now $f,g:\R \to \C$ are both suitable $1$-periodic functions with $(f,g)\in L^p(\TT)\times L^{q}(\TT)$ for some $1<p,q\leq \infty$.

With the perspective gained from II.2, we now expect that the answer to \eqref{bilinfseries} will fundamentally depend on the geometry of the sets along which we perform the summation. However, in contrast with II.2, the level of difficulty assigned to summability along the earlier considered shapes -- polygons, rectangles and circles -- is partly reversed. Indeed, in order to maintain a clear parallelism with II.2, we now discuss the analogues of \eqref{Polyg}-\eqref{circle} in a reversed order:

\begin{itemize}
\item The problem of almost-everywhere pointwise summability along discs
\beq\label{circleb}
f(x)g(x)=\lim_{N\to\infty} \sum_{n_1^2+n_2^2\leq N^2} \hat{f}(n_1) \hat{g}(n_2) e^{2\pi i (n_1 x+n_2 x)}
\eeq
seems to be open, and we are not aware of any nontrivial\footnote{Several cases are treated in \cite{DiestelGrafakos} but they are more or less straightforward adaptations of the ideas in \cite{Fef71} to the bilinear setting.} positive or negative results regardless of the value of $p$. We leave this as an enticing open problem and mention that convergence in the $L^r$-norm sense for $f\in L^p$ and $g\in L^q$ has been proved in \cite{GrafHonzik:MaxTransference, GrafakosLi} for the local-$L^2$ case ($\frac{1}{r}=\frac{1}{p}+\frac{1}{q}$ with $2\leq p,q,r'<\infty$).

\item Moving to the problem of summability along rectangles of arbitrary eccentricity, one notices that
\beq\label{rectangleb}
f(x)g(x)=\lim_{N_1,N_2\to\infty} \sum_{|n_1|\leq N_1,|n_2|\leq N_2} \hat{f}(n_1) \hat{g}(n_2) e^{2\pi i (n_1 x+n_2 x)}
\eeq
holds both pointwise and in norm. Indeed, these statements are immediate consequences of the classical Carleson-Hunt result discussed at II.1.

\item Finally, we are turning our attention towards summability along polygonal regions $\p$, by studying the behavior of the expression\footnote{For the clarity of our exposition below -- favoring the simplicity of the geometric structure of $\p$ -- we drop the specific value of the limit in \eqref{Polygb}. It is worth mentioning, though, that this value is precisely the expected pointwise product $f(x) g(x)$ when we require that the origin belongs to the interior of $\p$. However, if the latter does not hold, then the value of the limit might be different.}
\beq\label{Polygb}
\lim_{N\to\infty} \sum_{n\in N \p} \hat{f}(n_1) \hat{g}(n_2) e^{2\pi i (n_1 x+n_2 x)} .
\eeq
This is the setting that realizes the natural connection with the main topic of our present study; in our discussion we will consider two fundamental cases:
\end{itemize}

%\todo{here M and C were debating about the convention $2\pi=1$}
%%%%%%%%%%%%%%%%%%%%%%%%%%%%%%%%%%%%%%%%%%%
%
% FIGURE 0 - picture illustrating the P_1, P_2 shapes
%
%%%%%%%%%%%%%%%%%%%%%%%%%%%%%%%%%%%%%%%%%%%
\begin{figure}[ht]
\centering
\begin{tikzpicture}[line cap=round,line join=round,>=Stealth,x=1cm,y=1cm]
\clip(-1.5,-4.5) rectangle (5,5);
\fill[line width=0.5pt,color=black,fill=black,fill opacity=0.1] (0,0) -- (3,0) -- (0,-3) -- cycle;
\fill[line width=0.5pt,color=black,fill=black,fill opacity=0.1] (0,0) -- (3,3) -- (0,3) -- cycle;
\draw [line width=0.5pt,color=black] (0,0)-- (3,0);
\draw [line width=0.5pt,color=black] (3,0)-- (0,-3);
\draw [line width=0.5pt,color=black] (0,-3)-- (0,0);
\draw [line width=0.5pt,color=black] (0,0)-- (3,3);
\draw [line width=0.5pt,color=black] (3,3)-- (0,3);
\draw [line width=0.5pt,color=black] (0,3)-- (0,0);
\draw [->,line width=0.5pt] (0,-4.5) -- (0,4.7);
\draw [->,line width=0.5pt] (-1.5,0) -- (4.7,0);

\draw [fill=black] (3,3) circle (2pt);
\draw (3,3) node[anchor=south west] {\small $(1,1)$};

\draw [fill=black] (0,3) circle (2pt);
\draw (0,3) node[anchor=south east] {\small $(0,1)$};

\draw [fill=black] (0,0) circle (2pt);
\draw (0,0) node[anchor=south east] {\small $(0,0)$};

\draw [fill=black] (3,0) circle (2pt);
\draw (3,0) node[anchor=south west] {\small $(1,0)$};

\draw [fill=black] (0,-3) circle (2pt);
\draw (0,-3) node[anchor=north east] {\small $(0,-1)$};

\draw (1,-1.5) node[anchor=south east] {\small $\mathcal{P}_2$};

\draw (1,2) node[anchor=north east] {\small $\mathcal{P}_1$};
\end{tikzpicture}
\caption{\footnotesize Frequency representation of triangles $\p_1$ and $\p_2$. Summation over dilates of $\p_1$ are known to correspond to bounded operators, but nothing is known about summation over dilates of $\p_2$.} \label{figure:shapes}
\end{figure}
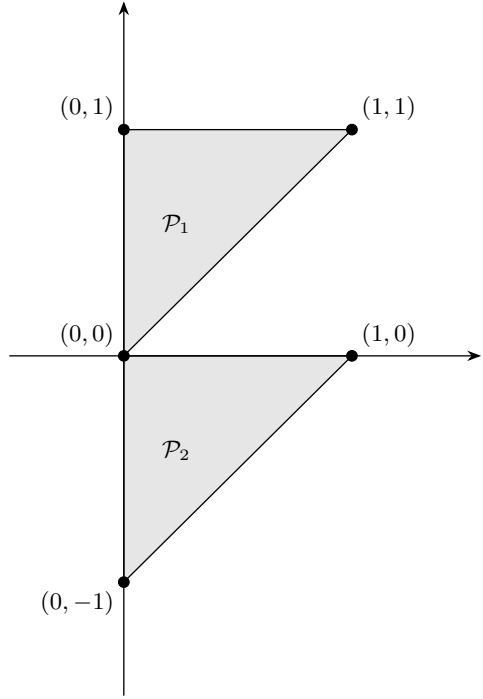

%%%%%%%%%%%%%%%%%%%%%%%%%%%
%

\begin{enumerate}[label=(\alph*), leftmargin=10pt]
\item \label{polygonP1} The triangle $\p=\p_1:= \big\{ (\xi, \eta) \,:\, 0 \leq \xi \leq \eta \leq 1 \big\}.$
\medskip

We start by noticing that at the formal level
$$\lim_{N\to\infty} \sum_{n\in N \p_1} \hat{f}(n_1) \hat{g}(n_2) e^{2\pi i (n_1 x+n_2 x)} = \sum_{0\leq n_1\leq n_2} \hat{f}(n_1) \hat{g}(n_2) e^{2\pi i (n_1 x+n_2 x)} .$$

Now, as in the previous cases II.1 and II.2, one is invited to study the associated (discrete) maximal operator, which, thanks to the transference principle in \cite{GrafHonzik:MaxTransference}, may be reduced to the study of its continuous version defined by
\beq\label{Polygp1}
P_1(f,g)(x):=\sup_{N>0} \Big| \iint_{0 < \xi < \eta < N} \hat{f}(\xi) \hat{g}(\eta) e^{ix(\xi+ \eta)} d \xi d \eta\Big|.
\eeq
As it turns out, the behavior of $P_1$ is equivalent with that of the Bi-Carleson operator
\begin{equation}
\label{def:BiCarleson:series}
Bi\text{-}C(f,g)(x):= \sup_{N} \Big|\iint_{\xi< \eta< N} \hat{f}(\xi) \hat{g}(\eta) e^{ix(\xi+ \eta)} d \xi d \eta \Big|.
\end{equation}
Indeed, this follows from the simple observation that
\begin{align*}
 P_1(f,g) & = Bi\text{-}C(\pi_+f,\pi_+g),
\end{align*}
where here $\pi_+$ stands for the projection on the positive frequencies.

We thus conclude that the boundedness of $P_1$ (and hence the almost everywhere convergence of the expression in \eqref{Polygb}) follows from the work in \cite{bi-Carleson} asserting that $Bi\text{-}C: L^p(\R)\times L^q(\R) \to L^r(\R)$ continuously, whenever $\frac{1}{p}+\frac{1}{q} =\frac{1}{r}<\frac{3}{2}$ and $p, q>1$.

$\newline$
\item \label{polygon:P2} The triangle $\p=\p_2:= \big\{ (\xi, \eta) \, : \, 0 \leq \xi  \leq \eta +1\leq 1 \big\}.$
$\newline$

Following similar steps to those in (a), we have that the almost everywhere convergence in \eqref{Polygb} for $\p=\p_2$ may be formally phrased as
$$
\lim_{N\to\infty} \sum_{n\in N \p_2} \hat{f}(n_1) \hat{g}(n_2) e^{2\pi i (n_1 x+n_2 x)} = \sum_{n_2\leq 0 \leq n_1} \hat{f}(n_1) \hat{g}(n_2) e^{2\pi i (n_1 x+n_2 x)}
$$
and hence we are naturally led to the study of the boundedness properties of
\begin{equation*}
P_2(f,g)(x):=\sup_{N >0} \Big|\iint_{N\p_2} \hat{f}(\xi) \hat{g}(\eta) e^{ix(\xi+ \eta)} d \xi d \eta \Big|=\sup_{N >0} \Big|\iint_{\substack{\eta \leq 0 \leq \xi \\ \xi-\eta \leq N}} \hat{f}(\xi) \hat{g}(\eta) e^{ix(\xi+ \eta)} d \xi d \eta \Big|.
\end{equation*}
Observe from the above that\footnote{As expected, here $\pi_-:=Id-\pi_+$.}
\begin{equation*}
P_2(f,g)(x)= \widetilde{BC}^1(\pi_+f,\pi_-g)(x) ,
\end{equation*}
where here
\begin{equation} \label{eq:ser:simBC^1}
\widetilde{BC}^1(f,g)(x):= \sup_{N >0} \Big| \iint_{\xi-\eta \leq N} \hat{f}(\xi) \hat{g}(\eta) e^{ix(\xi+ \eta)} d \xi d \eta \Big|.
\end{equation}
It remains now to notice that $\widetilde{BC}^1(f,g)$ is equivalent\footnote{Indeed, $\widetilde{BC}^1(f,g)$ may be expressed as a linear combination between $BC^1(f,g)$ and the pointwise product $f \cdot g$.} -- up to smooth error terms -- to $BC^1$.

We thus deduce that the problem of the almost everywhere convergence in \eqref{Polygb} for $\p=\p_2$ is essentially equivalent with the study of the boundedness properties for the resonant bilinear Hilbert--Carleson operator $BC^1$ defined in \eqref{Top}.
\end{enumerate}

Moreover, contrasting the cases \ref{polygonP1} and \ref{polygon:P2} discussed above, it is worth noticing how a simple translation from $\p=\p_1$ to $\p=\p_2$ makes for a startling difference in the nature and complexity of the problem on the almost everywhere convergence of \eqref{Polygb}. Indeed, given that the symmetry relations satisfied by $BC^1$ are closer in spirit to those satisfied by the trilinear Hilbert transform -- see \eqref{quadm0} and \eqref{linm0} in Section \ref{openquest} below, we expect the final resolution of $BC^1$ to be significantly more difficult than that already provided for $Bi\text{-}C$ (or for other currently understood time-frequency operators).

As a consequence of this last remark, except for a very few and specific situations\footnote{As in the case of $\p=\p_1$ that can be dealt with Bi-Carleson and linear Carleson type operators.} the study of \eqref{Polygb} for generic polygons $\p$ requires at the very minimum the understanding of $BC^1$-type operators. We conclude, thus, that the generic formulation of the almost everywhere convergence of bilinear Fourier series presents itself as a deep, intriguing open problem that further motivates our introduction of the class of bilinear Hilbert--Carleson operators $\{BC^a\}_{a>0}$.

\bigskip

\noindent\textsf{II.4. Maximally modulated operators via symbol translation.} As is transparent from the preceding discussion, see \emph{e.g.} \eqref{Cr}, \eqref{def:BiCarleson:series} and \eqref{eq:ser:simBC^1}, Carleson-type behavior may be rephrased on the Fourier side in terms of maximally shifted multipliers. In view of this, one can naturally consider -- in the bilinear setting -- the generic problem of understanding the boundedness properties of the operator
\begin{equation}
\label{def:Carleson:multilin:freq}
\ic C_m(f, g)(x)= \sup_{N= (N_1, N_2) \in \R^2}  \Big|\iint_{\R^2} \hat f(\xi)  \hat g(\eta)e^{ix (\xi +\eta) } m(\xi-N_1, \eta - N_2) d\xi d \eta \Big|,
\end{equation}
where here $m$ is some suitable singular symbol.

Notice now that indeed, as mentioned in the introduction, taking the supremum over all translated copies of the symbol $m$ by  vectors $N \in \R^2$ makes the operator $\ic C_m$ invariant under the action of any modulation of the input functions, that is
\begin{equation}
\label{modinvm}
\ic C_m( M_b f, M_c g)(x)=\ic C_m( f,g)(x) , \qquad \text{for any   } b, c \in \R.
\end{equation}

In the case where $m(\xi, \eta)$ is a Mikhlin multiplier, and thus in particular smooth away from the origin, the corresponding $\ic C_m$ operator was studied in \cite{LiMuscalu}. Due to \eqref{modinvm}, there is no favored frequency location for the usual Whitney-type analysis, and thus a significant effort in \cite{LiMuscalu} is put into producing a ``translation-invariant'' Littlewood-Paley decomposition of the symbol $m$.

Another example that fits the general description \eqref{def:Carleson:multilin:freq} is provided by the Bi-Carleson operator $Bi\text{-}C$ defined earlier in \eqref{def:BiCarleson:series} and studied in \cite{bi-Carleson}. Notice that in this situation the original BHT-type symbol $\ds m(\xi, \eta)= \one_{\{\eta< \xi<0 \}}$ is translated along the line $\{\xi=\eta\}$. Thus, the resulting object may be thought of as a composite operator that integrates the behaviors of both the BHT and the Carleson operator. After a careful analysis of the symbol of $Bi\text{-}C$  that relies on a special case of the result in \cite{LiMuscalu} (that of Mikhlin multipliers translated along the line $\{\xi=\eta\}$), the maximal modulation is seen to act only on the $\eta$ variable side -- and hence involving only the function $g$ -- thus allowing a decoupling of the two behaviors corresponding to BHT and the classical Carleson operator.

Coming now to the main topic of our present paper, we note that the operator $BC^1$ in \eqref{Top} (or equivalently \eqref{eq:ser:simBC^1}), which represents the ``purely zero-curvature'' form of our bilinear Hilbert--Carleson class $\{BC^a\}_{a>0}$, also fits into the description provided by \eqref{def:Carleson:multilin:freq}. Indeed, $BC^1$ can be realized either by translating the two-variable symbol $\one_{\{\eta<\xi \}}$ by any vector $(N_1, N_2) \in \R^2$, or by translating the one-variable (here $\xi-\eta$) symbol $\one_{\{\eta -\xi<0 \}}$ by any real number $N \in \R$.

Finally, at the heuristic level, we remark that $BC^1$ may be seen -- loosely speaking -- as the ``orthogonal frequency shift counterpart'' to $Bi\text{-}C$, since the symbol of the former is invariant under shifts in the direction orthogonal to the line $\{\xi=\eta\}$ while the symbol of the latter is invariant under shifts along the line $\{\xi=\eta\}$. However, as simple as this geometric parallel heuristic is, these two operators turn out to belong to different classes of difficulty.

$\newline$
\noindent\textsf{III. Variational estimates and the bilinear Rubio de Francia operator}
$\newline$

We start this section by quickly revisiting the main theme of Section II.3 centered around the pointwise convergence of (multiple) Fourier Series. We have learned by now that the standard classical approach to this theme proceeds by providing bounds for suitable associated maximal operators.

Indeed, as an example for our discussion, one can see that the one dimensional case exposed in \eqref{FS} was decided via the boundedness properties of the Carleson operator as revealed by \eqref{CrT} and \eqref{Cr}.

However, there exists another approach, pioneered by Bourgain in the context of ergodic averages \cite{Bou87-erg}, which relies on the study of the properties of so-called variational means. Although this route presents more technical challenges -- since such an $s$-variational operator controls the behavior of the corresponding maximal operator -- it has the advantage that it does not appeal to density results\footnote{This is especially useful in the ergodic theory framework where in most of the interesting situations one does not have available a natural dense subclass of functions within which almost everywhere convergence results can be directly established.} and also that it supplies information about the rate of convergence of the quantities under study. Thus, in relation to \eqref{Cr}, one can define for some $s>2$ the so-called variational Carleson operator $\ic C^{s-var}$ given by
\begin{equation}
\label{def:Carleson:var:lin:freq}
\ic C^{s-var}f(x)= \sup_{\substack{N \geq 1 \\ a_1< \ldots< a_N } }\Big ( \sum_{k=1}^{N-1} \Big|\int_{\R} \hat f(\xi) e^{ix \xi } \one_{\{  a_k < \xi <a_{k+1}  \}} d\xi \Big|^s \Big)^\frac{1}{s}.
\end{equation}
In \cite{variational_Carleson} it was shown that $\ic C^{s-var}$ is bounded on $L^p(\R)$ for any $s'<p<\infty$, thus, in particular, providing an alternative means of establishing the pointwise convergence of (one-dimensional) Fourier Series for functions in $L^p$.

Moving now towards the pointwise convergence of the bilinear Fourier series in \eqref{bilinfseries}, we focus our attention on the two central cases of convergence along polygonal regions $\p$ discussed in II.3, that is the convergence along the polygonal regions $\p$. Maintaining the parallelism with that earlier discussion, we have the following:
\begin{itemize}
\item If $\p=\p_1$, then \eqref{Polygb} may be approached by studying the variational analogue of \eqref{def:BiCarleson:series} defined for $s>1$ by
\begin{equation}
\label{def:var:BiCarleson:trans}
T_s(f,g)(x):=  \sup_{\substack{N\geq 1 \\ a_1<\cdots <a_N}} \Big( \sum_{k=1}^{N-1} \Big|\iint_{a_k < \xi <\eta < a_{k+1}} \hat{f}(\xi) \hat{g}(\eta)  e^{ix(\xi+\eta)} d\xi d\eta \Big|^s \Big)^{1/s}.
\end{equation}
This operator was studied in \cite{DoMuscaluThiele:VarIteratedFourier}, although the optimality of the range of $L^p$ spaces for which convergence holds remains open. In a related context, we also mention \cite{var_paraproducts}, which studied variational estimates for paraproducts.

\item If $\p=\p_2$, then \eqref{Polygb} may be approached by studying the variational analogue of \eqref{eq:ser:simBC^1}, which for $s>1$ is defined by
\begin{equation}
\begin{aligned}
&BC^{1,s-var}(f, g)(x) \\
&:=  \sup_{\substack{N\geq 1 \\ a_1<\cdots <a_N}} \Big( \sum_{k=1}^{N-1} \Big|\iint_{a_k < \xi-\eta < a_{k+1}} \hat{f}(\xi) \hat{g}(\eta)  e^{ix(\xi+\eta)} d\xi d\eta \Big|^s \Big)^{1/s}.
\end{aligned}\label{def:BCvar,s,ser}
\end{equation}
Of course the problem of providing bounds for \eqref{def:BCvar,s,ser}, which is more difficult than the corresponding problem for $BC^{1}$, is wide open.
\end{itemize}

At this point, it is worth noticing that if we drop the supremum over the parameters $N, a_1, \ldots, a_N$ in \eqref{def:BCvar,s,ser} above and we consider the sequence $\{ a_k \}_{1 \leq k \leq N}$ to be fixed, we recover a bilinear version of the Rubio de Francia operator associated to arbitrary strips:
\begin{equation}
\label{def:BRF,s:ser}
BRF^{s}(f, g)(x)= \Big( \sum_{k=1}^{N-1} \Big|\iint_{a_k < \xi-\eta < a_{k+1}} \hat{f}(\xi) \hat{g}(\eta)  e^{ix(\xi+\eta)} d\xi d\eta \Big|^s \Big)^{1/s}.
\end{equation}

The boundedness of $BRF^s$ remains a largely open problem, except for the situation of equidistant points $a_1, \ldots, a_N$. This latter situation was covered in \cite{Bernicot} in the local-$L^2$ range for $s=2$.\footnote{Hence the boundedness in the local-$L^2$ range remains valid for any $s \geq 2$.} The proof of \cite{Bernicot} relies heavily on the regular distribution of the strips and combines the BHT time-frequency analysis with concepts such as vectorial trees, vectorial sizes, and energies. In hindsight, the result of \cite{Bernicot} can also be obtained using more recent vector-valued extensions for BHT \cite{vv_BHT}; namely, $B: L^p(\ell^2(\ell^\infty)) \times L^q(\ell^\infty(\ell^2)) \to L^r(\ell^2(\ell^2))$ estimates are combined with the classical result of Rubio de Francia \cite{RF} to obtain boundedness in the local-$L^2$ range. The case of equidistant points $a_1, \ldots, a_N$ makes the $\ell^\infty$ spaces redundant, and also indicates the limitation of the method in treating the general case.

\subsection{Historical context}\label{Hist}

In what follows, we will approach the discussion on the historical background from the perspective provided by the zero/non-zero curvature paradigm, which represents the keystone for our entire analysis in this paper.
\medskip

Indeed, we will review the history of the mathematical literature most closely related to our paper through the lenses offered by the generic strategy discussed at the beginning of Section \ref{znzpar}. Our story will focus, in chronological order, on the two special classes of operators that contribute to the design of our operator, namely the (polynomial) Carleson operator and the bilinear Hilbert transform. Aside from these two, we will briefly discuss a few other related themes.

\subsubsection{The Carleson operator}\label{pco}

In this setting of maximally modulated singular integral operators, the generic formulation of the main question may be phrased as follows:
$\newline$

\noindent \underline{\textsf{The Carleson operator along variable curves.}} \textit{Let $\mathcal{F}$ be an (infinite) family of plane curves $\Gamma:=(t, -\g(t))$ with $\g(t)$ a suitable real piecewise smooth function on $\R$. Study the boundedness properties of the $\F$-Carleson operator, defined by
\beq\label{carlg}
C_{\F}(f)(x):= \sup_{\Gamma\in \F}\left|C_{\Gamma} f (x) \right|:=\sup_{\Gamma\in \F}\Big|\int_{\R} f(x-t) e^{i \g(t)} \frac{dt}{t}\Big|.
\eeq}
$\newline$
\noindent\textsf{The (purely) zero-curvature case}: $\F:=\F_1[1]:=\{(t, at) | a \in \R \}$.
$\newline$

This subject is a central part of the two-century-old saga -- mentioned in Section \ref{Otherdirections}, II.1 -- investigating the behavior of Fourier Series and originating in the work of Fourier on heat propagation (\cite{Fou}). In 1910's, Luzin (\cite{Luz}) conjectured that the Fourier series of a square integrable function converges almost everywhere. Kolmogorov (\cite{Kol1}, \cite{Kol2}) showed that the above fails in the $L^1$ setting. After decades of misapprehension following Kolmogorov's result, Carleson (\cite{c1}) delivered the positive answer to Luzin's conjecture by providing $L^2$-weak bounds for the so-called (classical) Carleson operator $C:=C_{\F}$ with $\F:=\{(t, at) | a \in \R \}$ in \eqref{carlg}.\footnote{Notice that the operator $C$ is the same as $C^{a}$, $a=1$, defined earlier in \eqref{sw}.}
The key underlying difficulty -- the first of its kind in the literature -- resides in the special modulation symmetry relation
\beq\label{carlsym}
C M_{a}f= Cf\quad\quad\forall\:a\in\R \:,
\eeq
which was overcome through a wave-packet discretization of the operator. In this way the boundedness of $C$  was reduced to understanding simultaneously the space and frequency interactions between the (properly grouped) wave-packets.

Thus Carleson's approach, together with Fefferman's alternative proof in \cite{f}, set the foundation for what today we call time-frequency analysis. A third proof combining elements of the first two approaches was provided by Lacey and Thiele in \cite{lt3}.

We end this brief description by mentioning two of the earlier extensions of Carleson's theorem: the $L^p$, $1<p<\infty$, analogue of Hunt \cite{hu}, and the higher dimensional version proved by Sj\"olin in \cite{sj2}.

$\newline$
\noindent\textsf{The non-zero-curvature case}: $\F:=\F_d[\vec{\a}]:=\{(t, \sum_{j=1}^{d} a_j t^{\a_j}) | \{a_j\}_{j=1}^d\subset\R\}$ where here $\vec{\a}:=(\a_1,\ldots, \a_d)\in \prod_{1}^{d}((0,\infty)\setminus\{1\})$ and $d\in\N$ fixed.
$\newline$

This subject originated from two distinct directions: 1) the study of the Hilbert transform along curves -- see a brief description of this in Section \ref{htc} below, and 2) the study of singular integrals on the Heisenberg group in connection with the behavior of  subelliptic partial differential operators.

As a point of convergence for these two directions of inquiry, Stein introduced and initiated the study of the so-called polynomial Carleson operator, defined via \eqref{carlg} by the relation
\beq\label{polyncarl}
C_d:=C_{\F_d},
\eeq
where here, for fixed $d\in\N$,  $\F_d:=\{(t, \sum_{j=1}^{d} a_j t^{j}) | \{a_j\}_{j=1}^d\subset\R\}$.

The first step toward a better understanding of $C_d$ was made, naturally, by analyzing its curved analogue $C^{curv}_{d}:=C_{\F_{d-1}[(2,\ldots,d)]}$ where here $d\geq 2$. As we have already learned by now, a key simplification afforded by considering $C^{curv}_d$ is the non-zero curvature of the phase, which manifests in the fact that, unlike $C_d$, the curved model $C^{curv}_d$ has no (generalized) modulation invariance symmetry.

The $L^p$-boundedness of $C^{curv}_d$, $1<p<\infty$, was first shown for $d=2$ by Stein in \cite{Stein} and then for general $d\geq 2$ (and also extended to arbitrary dimension\footnote{In arbitrary dimensions the (polynomial) Carleson operator is the analogue of \eqref{carlg} obtained by replacing the kernel $\frac{1}{t}$ with a suitable Calder\'on-Zygmund kernel and $\g(t)$ by a suitable multi-variable real polynomial with the obvious modification for the class analogous to $\F_d$.}) by Stein and Wainger in \cite{SteinWainger}. Both approaches follow the philosophy exposed at \ref{II:non:zero:curvature} above and rely on Van der Corput estimates and the $TT^{*}$-method.

In \cite{lvUnif}, the third author extended the one-dimensional\footnote{Though involving non-trivial technicalities, it seems likely that the methods developed in \cite{lvUnif} can be adapted to the higher dimensional case.} version of \cite{SteinWainger} by providing $L^p$ bounds, $1<p<\infty$, for a large class of curved $\F$-Carleson operators that in particular includes $C_{\F_d[\vec{\a}]}$ for any $\vec{\a}:=(\a_1,\ldots, \a_d)\in ((0,\infty)\setminus\{1\})^d$ and $d\in\N$. The methodology developed in \cite{lvUnif} provides  a new approach to Stein-Wainger's polynomial Carleson operator $C^{curv}_d$ and has the advantage -- confirmed also via the present paper -- of being compatible with time-frequency analysis techniques.

$\newline$
\noindent\textsf{Unifying the zero/non-zero-curvature cases}: $\F:=\F_d[\vec{\a}]$ with $\vec{\a}=(\a_1,\ldots, \a_d)\in \prod_{1}^{d}((0,\infty))$, $1\in\{\a_j\}_{j=1}^d$ and $d\in\N$, $d\geq 2$ fixed.
$\newline$

Returning to the original inquiry of Stein regarding the boundedness of the polynomial Carleson operator defined in \eqref{polyncarl}, the first result unifying the zero-curvature case of the classical Carleson operator $C$ with the non-zero curvature case of $C^{curv}_2$ treated in \cite{Stein} by Stein was proved in \cite{lv1}, which established the $L^2$-weak boundedness of $C_2$. The full one-dimensional version of the conjecture of Stein asserting the boundedness of the polynomial Carleson operator for any $L^p$, $1<p<\infty$, was proved affirmatively in \cite{lv3}, thus unifying the results in
\cite{c1} and the one-dimensional version of \cite{SteinWainger}. In analogy with the purely zero-curvature case discussed earlier, the main difficulty is represented by the higher-order\footnote{Here $M_{j, a} f(x):=e^{i a x^j} f(x)$ represents a generalized modulation of order $j\in\N$ and parameter $a\in\R$.} modulation invariance relation
\beq\label{genmodcarl}
C_{d}M_{j,a}f=C_{d}f\quad\quad\forall\:a\in\R\:\:\textrm{and}\:\:\forall\:1\leq j\leq d .
\eeq

Accordingly, the proof of the $L^p$-boundedness of $C_d$ is based on some key new ideas: 1) developing a higher-order wave-packet analysis consistent with \eqref{genmodcarl} that is compatible with the time-frequency approach to the classical Carleson operator $C$, and 2) introducing a local analysis adapted to the concepts of mass and counting function.\footnote{This latter aspect provides a novel tile discretization of the time-frequency plane with the consequence of eliminating certain exceptional sets from the analysis of $C_d$. (These exceptional sets were present in all of the previously known approaches to the boundedness of the classical Carleson operator $C=C_1$ -- see \cite{c1}, \cite{f} and \cite{lt3}). As a consequence, the approach in \cite{lv3} provides the full $L^p$ boundedness range, $1<p<\infty$, with no appeal to interpolation techniques.}

The $n$-dimensional version of Stein's conjecture on the polynomial Carleson operator was proved more recently -- see \cite{zk1} and \cite{lv3n} -- by essentially combining the methodology in \cite{lv3} with the $n$-dimensional Van der Corput estimates proved in \cite{SteinWainger}.

%Finally, it is worth mentioning that by making use of the results from Section 10 in \cite{lvUnif} (see in particular Lemma 51 therein) within the tile construction/discretization from \cite{lv3}  one should be able to follow the same ideas as in \cite{lv3} in order to infer the $L^p$ boundedness, $1<p<\infty$, of $C_{\F_d[\vec{\a}]}$ for any $\vec{\a}\in \prod_{1}^{d}((0,\infty))$ and $d\in\N$.

\begin{observation}\label{Carlvarcurve} For future reference, it might be worth noticing that the preceding unified problem can be equivalently rephrased -- via a standard linearization procedure -- in the following terms:
\medskip

\textit{Study the boundedness of the $\g$-Carleson operator defined by
\beq\label{carlgg}
C_{\g}(f)(x)= \int_{\R} f(x-t)  e^{i \g(x,t)} \frac{dt}{t},
\eeq
where here $\g: \R^2 \to \R$ is (fiberwise) measurable and $\g(x,\cdot)$ is a piecewise smooth function a.e. $x\in\R$.}
\end{observation}

\subsubsection{The bilinear Hilbert transform}\label{bht}

The generic formulation\footnote{Our choice for this formulation (as well as for that in the previous section) is tailored to make more transparent the zero/non-zero curvature paradigm that is at the heart of both the introduction and the treatment of the family of bilinear Hilbert--Carleson operators. Of course, the zero and non-zero curvature sides of the problem can and have been treated independently, with the former being the catalyst for the modern approach to time-frequency analysis.} of the main question within this direction can be stated as follows:
$\newline$

\noindent \underline{\textsf{The bilinear Hilbert transform along curves.}} \textit{Given $\Gamma:=(t, -\g(t))$ a plane curve with $\g$ a suitable piecewise smooth real function, study the boundedness properties of the bilinear Hilbert transform along the curve $\Gamma$, defined as
\beq\label{nhilb}
B_{\Gamma}(f,g)(x):= \int_{\R} f(x-t) g(x+\g(t))\frac{dt}{t}.
\eeq}
$\newline$
\noindent\textsf{The zero-curvature case}: $\g(t)=a t$ with $a\in\R\setminus\{-1,0\}$.
$\newline$

While also presenting connections with ergodic theory via the study of the $L^p$-norm convergence of non-conventional bilinear averages, this theme arose in the realm of harmonic analysis through Calder\'on's program, \cite{Cal}, dedicated to the study of the Cauchy transform along Lipschitz curves (see also \cite{CMM}). His initial plan was to reduce this study -- via a Taylor series argument -- to the problem of providing uniform (in $a$) $L^p$-bounds for the bilinear Hilbert transform $B_{a}\equiv B_{\Gamma_a}$ with $\g(t)=a t$ and $a\in\R\setminus\{-1,0\}$.  While first Calder\'on, \cite{Cal}, (for small Lipschitz constants) and later Coifman, McIntosh and Meyer, \cite{CMM}, were eventually successful -- via a different route -- in providing the $L^2$ bounds for the Cauchy transform along Lipschitz curves, the original problem about the behavior of the  bilinear Hilbert transform established itself as a deep conjecture within classical harmonic analysis.

Two decades later, in \cite{la1}, a key perspective was added based on the observation that, beyond the standard dilation and translation symmetries, the bilinear Hilbert transform $B_a$ also obeys a key modulation symmetry given by
\beq\label{bihilbcsym}
B_{a}(M_{a}f, M g)=M_{1+a} B_{a}(f,g).
\eeq
This suggested that in order to approach this problem one must appeal to a wave-packet analysis in the spirit of that developed by Carleson \cite{c1} and later Fefferman \cite{f} for proving the almost everywhere convergence of Fourier Series for square-integrable functions. Indeed, a few years later, Lacey and Thiele \cite{lt1,lt2} made the main breakthrough in the affirmative resolution of Calder\'on's conjecture by proving that $B_a$ maps $L^p\times L^q$ into $L^r$ continuously for $\frac{1}{p}+\frac{1}{q}=\frac{1}{r}$, $1<p, q\leq \infty$ and $\frac{2}{3}<r<\infty$.\footnote{There is still an open problem regarding the maximal range for $p, q, r$ -- specifically the interval $\frac{1}{2}<r\leq \frac{2}{3}$ -- that guarantees the boundedness of $B_{a}$.} The Hardy-Littlewood maximal analogue of \eqref{nhilb} in the same zero-curvature setting was proved by Lacey in \cite{Lacey}. Interestingly enough, even though a positive operator, the strategy employed for analyzing this bilinear maximal operator was still inspired by the wave-packet analysis performed on the bilinear Hilbert transform, though the presence of the maximal truncation makes the orthogonality argument significantly more involved.\footnote{In this context it is worth recalling the discussion in Section \ref{unif2th} that reveals the key r\^{o}le played by the bilinear maximal operator in bounding the low oscillatory component $T_0$ of our bilinear Hilbert--Carleson operator $BC^a$.}

$\newline$
\noindent\textsf{The non-zero-curvature case}: $\g(t)=\sum_{j=2}^{d} a_j t^j$, with $d>1$.
$\newline$

This problem can be motivated from three different directions: a) from the ergodic theory side in connection with the $L^p$-norm convergence of non-conventional bilinear averages of the type $\frac{1}{N}\sum_{n=1}^N f_1(T^n) f_2(T^{n^2})$
for $T$ an invertible measure-preserving transformation of a finite measure space (see \textit{e.g.} \cite{Fu}, \cite{HKr}); b)
from the number theory side via the so-called Ergodic Roth type theorem(s) (see \textit{e.g.} \cite{Bo86}, \cite{Bo88}, \cite{DGR}, \cite{K19}); and c) from the harmonic analysis side as a natural curved analogue of the standard bilinear Hilbert transform considered above.

Within this latter realm, a first result was obtained in \cite{li}, in which Li proved that $B_{\Gamma}: L^{2}(\R)\times L^{2}(\R) \to L^{1}(\R)$  continuously for the case $\Gamma(t) = (t, t^d)$, $2 \leq d \in \mathbb N$. The proof of this relies on the concept of $\sigma$-uniformity introduced in \cite{cltt} and inspired by Gowers's work in \cite{gowers} and only partially reflects the generic strategy described earlier here at point \ref{II:non:zero:curvature}, lacking the natural scale-type decay in \eqref{stdeccee}.

The general curved case -- this time faithfully reflecting \ref{II:non:zero:curvature} by also including \eqref{stdeccee} -- was completed with different methods by the third author here in \cite{L1}, \cite{lv10}. Indeed, the main combined result states that $B_{\Gamma}$ satisfies $L^p\times L^q \to L^r$ bounds for $1/p+1/q=1/r$ with $1<p<\infty$, $1<q\leq\infty$  and $1\leq r<\infty$ for any $\Gamma=(t,\g(t))$ with $\g\in \n\f$. Here $\n\f$ is a suitable class of curves that includes in particular any Laurent polynomial with no linear term, and within this class the above boundedness range is sharp up to the end-points. The proof of this result encompasses Gabor frame analysis, quadratic phase discretization, orthogonality/$TT^{*}$ methods and shifted square function arguments and serves as a precursor for the methodology developed in \cite{lvUnif}.

Further work regarding the maximal curved analogue was performed in \cite{LX} and \cite{GL}.
\medskip

\begin{observation}\label{BHTvarcurve} One can create a parallelism between the topic treated in the previous section in the form
given by \eqref{carlgg} and the current topic by considering the study of the more general expression
\beq\label{nhilbg}
B_{\g}(f,g)(x):= \int_{\R} f(x-t) g(x+\g(x,t))\frac{dt}{t} ,
\eeq
where here $\g$ has the same properties as the one stated in Observation \ref{Carlvarcurve}. When $\g$ obeys suitable nondegeneracy and non-zero curvature conditions, one can provide $L^p$ bounds for \eqref{nhilbg} via the methodology developed in \cite{lvUnif}; this is also the subject of upcoming work by the third author.\footnote{This includes in particular the ``curved'' polynomial case $\g(x,t)=\sum_{j=2}^d a_j(x) t^j$ with $\{a_j\}_j$ measurable and $d\in\N$, $d\geq 2$.} If one considers the analogue of the polynomial Carleson question, that is, the case when $\g(x,t)=\sum_{j=1}^d a_j(x) t^j$ with $\{a_j\}_j$ measurable functions and $d\in\N$, then this remains a difficult open question that in particular includes the problem of providing uniform bounds for the classical bilinear Hilbert transform.
\end{observation}

\subsubsection{The Hilbert transform along curves}\label{htc}

Mirroring the formulation of the previous questions as displayed in Observations \ref{Carlvarcurve} and \ref{BHTvarcurve}, the loose generic problem addressing the current topic may be stated as follows:
$\newline$

\noindent \underline{\textsf{The Hilbert Transform along variable curves.}} \textit{Let $(x,y)\in\R^2$ and assume that $\Gamma:=(t, -\g(x,y,t))$ is a planar parametrized curve with $\g$ a suitable real function that is piecewise smooth in $t$ and measurable in $x, y$. Investigate the boundedness properties of the Hilbert transform along the curve $\Gamma$, defined as
\beq\label{hilbc}
H_{\Gamma}(f)(x,y):= \int_{\R} f(x-t, y+\g(x,y,t))\frac{dt}{t} .
\eeq}
$\newline$
\noindent\textsf{The zero-curvature case}: $\g(x,y,t)=a(x,y) t$ with $a : \R^2 \to \R$ a suitable function.
$\newline$

We start by noticing that the answer to the simplest possible case $a(x,y)=a$ is essentially equivalent to the celebrated result of Riesz, \cite{Rie}, on the $L^p$-boundedness for $1<p<\infty$ of the classical Hilbert transform, whose history goes back at least to the study of the (conjugate) harmonic functions in the $19^\textsuperscript{th}$ century.

Moving forward, the study of \eqref{hilbc} for more general functions $a$ is intimately connected with the problem of differentiability along vector fields, which revolves around one of the deep open problems in harmonic analysis:
$\newline$

\noindent\textsf{Conjecture (Zygmund):} \emph{Let $\g(x,y,t)=a(x,y) t$ with $a: \R^2 \to \R$ a Lipschitz vector field. Then, taking $\ep_0$ small enough depending on $\|a\|_{Lip}$ and defining the maximal analogue of  \eqref{hilbc} as
\beq\label{maxZ}
M_{\Gamma} f(x,y)=M_{a,\ep_0} f(x,y):=\sup_{0<\ep<\ep_0} \frac{1}{2\ep} \int_{-\ep}^{\ep}|f(x-t,y+a(x,y) t)| dt ,
\eeq
we have that $M_{\Gamma}$ is bounded on $L^p(\R^2)$ for any $1<p<\infty$.}
\medskip

The singular integral analogue is precisely the key motivation for the generic question mentioned at the beginning of this section:

\medskip

\noindent\textsf{Conjecture (Stein):} \emph{With $\g$, $u$ and $\ep_0$ as before, define the Hilbert transform along $\Gamma$ as
\beq\label{HilZ}
H_{\Gamma} f(x,y)=H_{a,\ep_0} f(x,y):=\int_{-\ep_0}^{\ep_0} f(x-t,y+a(x,y) t) \frac{dt}{t}.
\eeq
Then $H_{\Gamma}$ is bounded on $L^p(\R^2)$ for any $1<p<\infty$.}
\medskip

It is worth mentioning that the Lipschitz condition imposed on the vector field $a$ is necessary (see e.g. \cite{Lali}). This is a consequence of the existence of Besicovitch-Kakeya sets in $\R^2$, which prevents any $L^p$ bound on either $M_{\Gamma}$ or $H_{\Gamma}$ even if $a$ is H\"older continuous of class $C^{\a}(\R^2)$ for $\a<1$.

The first major contribution was due to Bourgain in \cite{Bolip}, where he proved $L^2$-bounds for $M_{\Gamma}$ under the assumption that $a$ is analytic. The analogous result for $H_{\Gamma}$ was proved by Stein and Street in \cite{SS1}. In between, some particular cases of vector fields $a$ were treated in \cite{CSWW} and \cite{CNSW}.

The next deep insight into Stein's question was made by Lacey and Li in \cite{LL1}, \cite{Lali} based on the key observation that $H_{\Gamma}$ enjoys a modulation invariance symmetry and hence, as already referred to several times by now, time-frequency techniques in the spirit of Carleson's theorem form the natural framework for its study. The result in \cite{Lali} was of conditional nature: if a suitable Kakeya-type maximal operator is assumed to be $L^2$-bounded, then, under the supplementary assumption that $a$ is of class $C^{1+\ep}$, one has that $H_{\Gamma}$ is $L^2$-bounded. This constitutes the last advancement recorded within the realm of genuine two variable dependence of $a$.

In the past decade, the effort has focused on understanding some model problems related to \eqref{HilZ}, all revolving around a one-variable type behavior of $a$. In such a situation, in contrast with the previous Lipschitz requirement, the mere measurability of $a$ is enough for proving non-trivial bounds. Indeed, Bateman -- in the single annulus case,\cite{Bat}, and then Bateman and Thiele -- for the general, non-localized frequency case, \cite{BT}, proved that $H_{\Gamma}$ is $L^p$ bounded for $p>\frac{3}{2}$. A few years later, in \cite{Gu1}, \cite{Gu2}, the same boundedness range was established by Guo for the situation when $a$ is constant along suitable families of Lipschitz curves.

$\newline$
\noindent\textsf{The non-zero curvature case}: a.e. $(x,y)\in\R^2$ the function $\g(x,y,t)$ ``does not resemble'' a line near $t=0$ and $\pm\infty$.
$\newline$

This direction of research originates in the work of Jones \cite{Jon}, Fabes \cite{Fab}, and Fabes and Rivi\`{e}re \cite{fr}, on  constant-coefficient parabolic differential operators. As a consequence of these, one obtains the $L^2$-boundedness of $H_{\Gamma}$ for $\g(x,y,t)=t^{\a}, \a \in (0,\infty)\setminus\{1\}$. From this point on the study of \eqref{HilZ} in the curved case has been adopted as a theme of independent interest within harmonic analysis area. The first such results were focused on the situation $\g(x,y,t)=\g(t)$ and included works such as \cite{sw70}, \cite{NRS74} and \cite{NRS76}.

The next step was the study  of \eqref{HilZ} for smooth dependence of $\g$ on $x,y$. Thus in \cite{CWW93} and \cite{B02} the authors prove the expected $L^p$ range for the situation $\g(x,y,t)=P(x) \tilde{\g}(t),$ where $P$ is a polynomial and $\tilde{\g}$ is smooth and obeys some suitable non-vanishing curvature condition. In a different direction, this time analyzing the behavior of $H_{\Gamma}$ under the assumption that $\g(x,y,t)$ obeys $(x,y,t)$-smoothness and non-zero curvature in $t$  hypotheses, we have: in the nilpotent setting the work in \cite{Chhilb}, and in the context of singular Radon transforms (and their maximal analogues) i) along differentiable submanifolds the work in \cite{CNSW}, or  ii) along variable curves in a diffeomorphism invariant setting, the work in \cite{SeWa}.

Taking one more step, one considers now the situation $\g(x,y,t)=a(x,y) t^{\a}$ with $\a\in(0,\infty)\setminus\{1\}$ and minimal smoothness in $(x,y)$. In this setting, it was shown in \cite{MRi} that with the mere assumption that $a$ is measurable one can prove $L^p$ boundedness of $M_{\Gamma}$ only for $p>2$. In \cite{GHLR}, requiring the minimal Lipschitz smoothness in $(x,y)$, the authors proved the full $L^p$ range for $M_{\Gamma}$ as well as the single annulus $L^p$-estimate for the singular counterpart $H_{\Gamma}$. Based on the latter and on a further square function estimate, the full treatment of $H_{\Gamma}$ was provided in \cite{DGTZ}.

Finally, focusing on the situation $\g(x,y,t)=\g(x,t)$ assuming only measurability in $x$, the first such study was performed in \cite{GHLR} for $\g(x,t)=a(x) t^{\a}$ with $\a\in(0,\infty)\setminus\{1\}$ and resulted in the proof of the $L^p$ boundedness of both $H_{\Gamma}$ and $M_{\Gamma}$ within the full expected range. Some related results appeared more recently -- see for example \cite{LY2} and \cite{GRSY}. The most recent result, \cite{lvUnif}, establishes via the LGC-methodology the full $L^p$ range for both $H_{\Gamma}$ and $M_{\Gamma}$ for a general class of curves $\g(x,t)$ that includes in particular the case $\g(x,t)=\sum_{j=1}^d a_j(x) t^{\a_j}$ with $\{a_j\}_{j=1}^{d}$ arbitrary real measurable functions, $\{\a_j\}_{j=1}^d\subset (0,\infty)\setminus\{1\}$ and $d\in\N$.

\subsubsection{Maximally modulated singular Radon transforms}\label{msrt}

This theme may be seen as a partial symbiosis between two of the topics discussed earlier: the polynomial Carleson operator (Section \ref{pco}) and the Hilbert transform along curves (Section \ref{htc}). The generic formulation of the problem may be phrased as follows:
$\newline$

\noindent \underline{\textsf{Maximally modulated singular Radon transforms.}} \textit{Fix $d, m \in \N$, $\vec{\a}\in\N^d$ and $\vec{\b}\in \N^m$. Then, given $(t,-P(t))\in \F_m[\vec{\b}]$, we define the $\F_d[\vec{\a}]$--Radon transform along $P$ as
\beq\label{mmsr}
\rc_{\F_d[\vec{\a}]}^{P}(f)(x,y):= \sup_{\Gamma=(\cdot,-\g(\cdot))\in \F_d[\vec{\a}]}\Big|\int_{\R} f(x-t, y+P(t)) e^{i \g(t)} \frac{dt}{t}\Big|.
\eeq
Study the $L^p$ boundedness properties of $\rc_{\F_d[\vec{\a}]}^{P}$.}
$\newline$

As hinted at in the beginning, this subject has its roots in the work of Stein and his collaborators on the Hilbert transform along curves and on the polynomial Carleson operator. Given the synthesizing character of this subject and consequently its more complex nature, far fewer things are known about it. Indeed, very little is understood about the behavior of $\rc_{\F_d[\vec{\a}]}^{P}$ in either zero curvature or, easier, non-zero curvature cases.

The first explicit reference to \eqref{mmsr} appears in the work of Pierce and Yung, \cite{PP}, which addresses a higher dimensional version given by
\beq\label{mmsrhdim}
\Q(f)(x,y):= \sup_{Q\in \tilde{\p}_d(\R^n)}\Big|\int_{\R^n} f(x-t, y-|t|^2) e^{i Q(t)} K(t) dt\Big| ,
\eeq
where here $d, n\geq 2$, $K$ is a standard (translation invariant) Calder\'on-Zygmund kernel on $\R^n$ and $\tilde{\p}_d(\R^n)$ is a suitable class of real polynomials on $\R^n$ of degree at most $d$ that omits linear terms and certain types of quadratic terms. The specific restrictions imposed on \eqref{mmsrhdim} are fundamental for the success of the approach in \cite{PP}: 1) the requirement $n\geq 2$ makes the expression in \eqref{mmsrhdim} less singular since the point singularities $\{0,\infty\}$ of the kernel $K$ are now embedded in a larger ambient space,\footnote{Loosely speaking, this creates more room for the cancellation offered by the phase and thus the ``bad set'' where $T$ is ``large'' has a smaller relative size.} and, 2) the restriction over the class $\tilde{\p}_d(\R^n)$ is essentially equivalent to imposing that $\Q$ has no (generalized) modulation invariance symmetries, thereby placing this analysis in a non-zero curvature context.

These aspects allow the authors in \cite{{PP}} to adapt and refine -- via involved technicalities -- the methods employed by Stein and Wainger for the treatment of the curved polynomial Carleson operator, \cite{SteinWainger}, and prove the fully expected $L^p$ boundedness range for $\Q$.

In a related direction, this time focusing on a zero curvature situation but assuming a higher degree of smoothness on the multiplier side, there is the work in \cite{JR}, providing $L^2$-weak bounds for a suitable anisotropic higher-dimensional version of the Carleson operator.

Returning now to our original problem in \eqref{mmsr}, very little is known. The only result specifically addressing \eqref{mmsr} that we are aware of regards the following model operator considered in \cite{GPRY}:
\begin{equation}
\begin{aligned}
\rc[m,d](f)(x,y)&:=\sup_{a\in\R}\Big|\int_{\R} f(x-t, y-t^m) e^{i a t^d} \frac{dt}{t}\Big| \\
&\approx \Big|\int_{\R} f(x-t, y-t^m) e^{i a(x,y) t^d} \frac{dt}{t}\Big|,
\end{aligned}\label{mmsre}
\end{equation}
Here one works under the assumptions that the linearizing function $a(x,y)$ depends only on one variable and that, essentially,\footnote{Up to, possibly, ``subtracting'' a Carleson type operator from the low oscillatory component.} $m\not=d\geq 2$, thus discarding from the discussion the challenging resonant cases.\footnote{For the latter, see the discussion of the open problem proposed at the end of this section.} This work reveals some of the difficulties that one encounters even when attempting to treat just the model case of single-variable dependence of $a$ in a resonant context. Such an example is provided by taking in \eqref{mmsr} the case $P(t)=t^2$ and assuming that after the linearization of the supremum the phase of the complex exponential has the form $a_1(x)  t + a_2(x) t^2$. It is now immediate to see that this example covers the behavior of the Quadratic Carleson operator $C_2$ treated in \cite{lv1}.

Given the very limited understanding of the main topic discussed in this section, we extend the coverage of our open questions beyond subjects directly connected to the behavior of the bilinear Hilbert-Carleson operator\footnote{For this latter subject, see Section \ref{openquest}.} and end the current section by formulating the following:
$\newline$

\noindent\textsf{Open problem.} Investigate the $L^p$ boundedness of the operator $\rc[m,d]$ defined by \eqref{mmsre} in the following specific situations (in increasing order of difficulty):
\begin{itemize}
\item $m=2$, $d=3$: this corresponds to the non-zero curvature situation, and thus relies on the fact that $\rc[2,3]$ has no modulation invariance symmetry.

\item $m=2$, $d=2$: in this situation we have a linear resonance in the second variable, and thus we are in the hybrid zero/non-zero curvature case with $\rc[2,2]$ obeying the modulation invariance relation
\beq\label{modrel}
\rc[2,2](f)= \rc[2,2](M_{1,c}^2 f)\quad\quad\forall\:c\in\R ,
\eeq
 where here $M_{1,c}^j f(x_1,x_2):=e^{i c x_j} f(x_1,x_2)$ with $j\in\{1,2\}$. As a model case one may want first to investigate the (much) simpler situation of $a(x,y)=a(x)$ measurable.

\item $m=2$, $d=1$: this represents a zero curvature situation within the context represented by \eqref{mmsr} with $\rc[2,1]$ obeying the modulation invariance relations
\beq\label{modrel1}
\rc[2,1](f)= \rc[2,1](M_{1,c}^1 f)\quad\quad\forall \: c \in \R ,
\eeq
and
\beq\label{modrel2}
\rc[2,1](f)= \rc[2,1](M_{2, b}^1 M_{1, b}^2 f)\quad\quad\forall\:b, c \in \R .
\eeq
Again, as a toy model, one may want to consider first the situation $a(x,y)$ measurable depending on only one variable.
\end{itemize}
\medskip

\subsubsection{The triangular Hilbert transform}\label{trht}   This last topic is of more recent extraction and can be viewed as a multivariable analogue of the subject treated in Section \ref{bht}. Mirroring the bilinear Hilbert transform setting, the generic formulation of the present problem may be shaped as follows:
$\newline$

\noindent \underline{\textsf{The triangular Hilbert transform along curves.}} \textit{Given $\Gamma:=(t, -\g(t))$ a plane curve with $\g$ a suitable (piecewise) smooth real function, study the boundedness properties of the triangular Hilbert transform along the curve $\Gamma$ defined by
\beq\label{thilb}
T_{\Gamma}(f,g)(x,y):= \int_{\R} f(x-t,y) g(x,y+\g(t)) \frac{dt}{t} .
\eeq}
$\newline$
\noindent\textsf{The zero curvature case}: $\g(t)=a t$ with $a\in\R\setminus\{0\}$.
$\newline$

This case of \eqref{thilb} was explicitly formulated in \cite{KTZ} and can be traced back to the study of the two dimensional bilinear Hilbert transform introduced in \cite{DT}. The motivation for considering the triangular Hilbert transform $T_{\Gamma}$ with $\Gamma$ a non-degenerate line is multifaceted: 1) its study explores the largely virgin territory of highly singular higher dimensional wave-packet analysis, thus revealing some subtle phenomena not present in the one-dimensional universe; 2) it is a highly singular object that covers both the behavior of the classical Carleson operator and that of the classical bilinear Hilbert transform (and its corresponding behavior with respect to uniform estimates); and 3) can be connected with relevant topics in ergodic theory such as the pointwise convergence of bilinear averages arising from $\Z^2$ actions, \cite{DT}, and the problem of convergence for double ergodic averages with respect to two commuting transformations, \cite{DKST}.

While the problem of providing $L^p$ bounds for the triangular Hilbert transform remains wide open, the investigation of this subject has brought to light some new tools designed to better understand higher dimensional problems with (hidden) modulation invariant symmetries or presenting significant degeneracies. Among these, most notable is the surprising approach of $\textrm{Kova}\check{c}$, \cite{KV}, for establishing the boundedness of the so-called twisted paraproduct, which resembles energy methods in PDE and relies on subtle integration by parts and iterative applications of Cauchy-Schwarz inequalities. Also of relevance here is the work in \cite{Be12}, a suitably degenerate situation appearing in the earlier study of the two-dimensional Hilbert transform \cite{DT}. The methodology developed in \cite{KV} remains at the core of the most recent advances in the study of the triangular Hilbert transform, with \cite{KTZ} discussing a particular model case and \cite{DKST} investigating some interesting connections with ergodic theory.

$\newline$
\noindent\textsf{The non-zero curvature case}: $\g(t)=t^2$.
$\newline$

The motivation for this problem -- considered recently in \cite{CDR} -- is threefold: firstly, it serves as a toy model for the difficult zero curvature case represented by the triangular Hilbert transform discussed above; secondly, it encapsulates the behavior of the particular operators $C_{\F}$ for $\F=\F_1[2]$ and $B_{\Gamma}$ for $\Gamma=(t,-t^2)$ as in the general discussions of Sections \ref{pco} and \ref{bht}, respectively; and thirdly, it provides an application in the number theory realm (specifically to the study of patterns in the Euclidean plane) in the form of a quantitative nonlinear Roth-type theorem.

\subsection{Some related open questions}\label{openquest}
The present study on the boundedness properties of
\beq\label{Topa}
BC^{a}(f,g)(x)= \sup_{\lambda\in \R}  \Big|\int f(x-t) g(x+t)  e^{i\lambda t^a} \frac{dt}{t} \Big|
\eeq
opens up an entire avenue of related open problems that are meant to advance our above described program. In what follows, we list them in increasing level of complexity, as reflected by the richness of their symmetry groups:

$\newline$
\emph{Study the boundedness properties of the following operators:}
\medskip
\begin{enumerate}[leftmargin=8pt, label=\arabic*)]

%\item \label{OQ:2} \emph{$BC_{d}^{nr},\: d\geq 3$: The non-resonant Polynomial Bilinear Hilbert--Carleson operator}
%\medskip

%This operator is defined by
%\beq\label{BPd}
%BC_{d}^{nr}(f,g)(x):= \sup_{P\in\p_d^{0}(\R)} \Big|\int f(x-t)g(x+t) e^{i P(t)}\frac{dt}{t}\Big|\,,
%\eeq
%where here $\p_d^{0}(\R)$ is the class of all real polynomials of degree $d$ with no linear and second order term.

%Leaving aside the discussion about the third order term, the approach of this  problem should be a non-trivial combination of the already existent techniques employed in the present paper and of those in \cite{lvUnif}.
%$\newline$

\item \label{OQ:3} \emph{$BC^{a}$ for $a=2$: quadratic resonance}
\medskip

This question requires key new ideas. Indeed, in addition to the properties displayed in \ref{key:td}-\ref{key:max} (for $T:=BC^{2}$) one has the \emph{quadratic} modulation invariance
\beq\label{quadm}
|BC^{2}(M_{2,c} f, M_{2,c} g)|= |BC^{2}(f,g)|\quad\quad\forall\: c \in \R\,.
\eeq
As an immediate effect of this extra symmetry, the $m$-decaying estimate \eqref{mainrm} in Section \ref{Heur} no longer holds, and thus the strategy exposed in this paper for treating the oscillatory term must be fundamentally changed.\footnote{Indeed, if we assume by contradiction that the strategy in our present paper remains true for $a=2$, we see that by choosing the linearizing function $\l(x)=2^{2m+1}$ for some fixed large $m\in\N$, relation \eqref{mainrm} would imply in particular that $\big|\int_{|t|\gtrsim 1} f(x-t) g(x+t)e^{i 2^{2m+1} t^2 } \frac{dt}{t} \big|\lesssim 2^{-m \d} \|f\|_{L^2} \|g\|_{L^2}$ for some absolute constant $\d>0$ and any $f, g\in L^2(\R)$. However, this latter inequality is trivially false, as can be seen by taking $f(x)=M_{2, -2^m} f_1(x)$ and $g(x)=M_{2, -2^m} g_1(x)$ with $f_1, g_1$ arbitrary functions in $L^2(\R)$.}
$\newline$

\item \label{OQ:4} \emph{$BC^{a}$ for $a=1$: The (linear resonant) bilinear Hilbert--Carleson operator}
\medskip

This situation is even more complex than the previous one due to the extra degree of freedom within the linear modulation symmetry group\footnote{Note the distinction between \eqref{quadm} and \eqref{quadm0}. Similarly, notice the extra degree of freedom in \eqref{linm0} as compared to \ref{key:m}.}:
\begin{itemize}
\item the quadratic modulation symmetry
\beq\label{quadm0}
|BC^{1}(M_{2,c} f, M_{2,-c} g)|= |BC^{1}(f,g)|\quad\quad\forall\:c\in\R ,
\eeq
\item the linear modulation symmetry
\beq\label{linm0}
|BC^{1}(M_{b} f, M_{c} g)|= |BC^{1}(f,g)|\quad\quad\forall\:b, c \in \R .
\eeq
\end{itemize}
As with the previous item, this problem requires key new ideas. In particular, it is worth noticing that $BC^{1}$ encapsulates both the behavior of the bilinear Hilbert transform and that of the classical Carleson operator.
$\newline$

\item \label{OQ:5} \emph{$BC_{d}, d\geq 1$: The bilinear Hilbert -- polynomial Carleson operator}
$\newline$

Define the \emph{bilinear Hilbert -- polynomial Carleson operator} by
\beq\label{BPHd}
BC_{d}(f,g)(x):= \sup_{P\in\p_d(\R)} \Big|\int f(x-t)g(x+t) e^{i P(t)}\frac{dt}{t}\Big|,
\eeq
where here $\p_d(\R)$ is the class of all real polynomials of degree $d\geq 1$.
$\newline$

As expected, because the operator $BC_{d}$ has the richest group of symmetries, this is the most difficult problem considered here. In particular, we have
\begin{equation}\label{genmod}
|BC_{d}(M_{j,b} f, M_{k,c} g)|= |BC_{d}(f,g)|\qquad \forall\: b, c\in\R \text{ and } \forall\: 1\leq j,k\leq d.
\end{equation}

The operator $BC_{d}$ generalizes both the bilinear Hilbert transform and the polynomial Carleson operator.
\end{enumerate}

\subsection{Structure of the paper}\label{struct}

In what follows we briefly outline the organization of our paper:
\begin{itemize}

\item In Section \ref{Genov} we present an overview of most of the main ideas in our proof.

\item In Section \ref{sec:initial:reductions} we perform some standard reductions: we make explicit the decomposition of $T=BC^a$ into its low oscillatory component $T_0$ and its high oscillatory component $T_{\osc}$, provide the treatment of $T_0$ and isolate the main part of each $T_m$ within the decomposition $T_{\osc}=\sum_{m\in\N} T_m$.

\item Section \ref{sec:discretization} presents the first key ingredient: the discretization of the operator. Here we explain the concept of the two-resolution analysis by introducing 1) the high resolution model with its two variants: the discrete phase-linearized wave-packet model (Section \ref{HR1}) and the continuous phase-linearized spatial model (Section \ref{HR2}), and 2) the low resolution model (Section \ref{LR}).

\item Within the high resolution analysis, in Section \ref{sec:onescale} we provide the single scale decay estimate for each\footnote{Recall here that we expressed $T_m=\sum_{k\in\Z} T_{m,k}$.} $T_{m,k}$ with $m\in\N$ and $k\in\Z$ fixed. This is the point where we make use in a key fashion of the phase curvature characterizing the modulated kernel of our operator. Our single scale estimate is achieved via the $TT^{*}$ method and a discrete Van der Corput argument involving exponential sums -- see Section \ref{ttstar}.

\item In Section \ref{sec:HR2} we refine the previously obtained single scale decay estimate by upgrading it to a decay estimate for the associated tri-linear form that achieves a desired time-frequency localization in all of the three input functions. This can be regarded as a single-tile estimate and prepares the ground for the multi-scale analysis performed in the next section.

\item In Section \ref{sec:bilinear:analysis}, via suitable concepts of size and mass, we adapt the bilinear Hilbert transform methods to the context of our problem in order to obtain a global control over $T_m$ and thus conclude our analysis of $T$. In addition, the helicoidal method will be used to obtain the full range of boundedness, including the case where $p$ or $q$ is infinite.

\item Finally, in Appendix A we discuss the lack of $m$-decaying absolute summability for our discrete (linearized) wave-packet model, while in Appendix B we provide the standard stationary phase analysis associated to the multiplier of our operator $T_m$.

\end{itemize}

\subsection*{Acknowledgments}
The authors were all partially supported by ERC project FAnFArE no. 637510. The third author was also supported by NSF grants DMS-1500958 and DMS-1900801 and by a Simons fellowship. The fourth author was also supported by IRC grant GOIPD/2019/434. The authors would like to thank Zubin Gautam for his help in improving the English presentation in the introductory section of this paper.

\section{General overview of our proof: main ideas and picture from above}\label{Genov}

The novel hybrid character of our operator within the zero/non-zero curvature framework as revealed by the specific structure imposed by \ref{key:td}-\ref{key:max} in Section \ref{Intr} will place us in some unprecedented situations whose treatment will require several new ideas that are discussed briefly below.

\subsection{Some heuristics}\label{Heur}

We gradually immerse into our story by noticing that an inspection of the key properties in \ref{key:td}-\ref{key:max} reveals the following:
\begin{itemize}
\item The dilation commutation relation expressed by \ref{key:td} is a statement about the \emph{multiscale} nature of our approach while the translation commutation in the same \ref{key:td} makes the point that our analysis places equal weight on any spatial location.

\item Next, the modulation invariance \ref{key:m} implies that our methodology must rely on \emph{wave-packet decompositions}/\emph{time-frequency analysis} originating in the celebrated work of Carleson (\cite{c1}) on the pointwise convergence of Fourier Series and further developed by the influential contributions of Fefferman (\cite{f}) -- in his reproof of Carleson's Theorem -- and of Lacey and Thiele (\cite{lt1}, \cite{lt2}) -- in their study of the boundedness of the bilinear Hilbert transform.

\item Finally, the form of the integral kernel that includes a maximally modulated kernel in the presence of curvature requires an approach that is sensitive to the level set decomposition of the highly oscillatory phase of the multiplier -- the central element used to extract the cancellation provided by the curvature. Moreover, given the structure of our operator $BC^a$ and the discussion from Sections \ref{znzpar} and \ref{unif2th}, our approach  needs to subsume the Stein-Wainger treatment (\cite{SteinWainger}) of the polynomial Carleson operator with no linear phase. On top of the above, the techniques involved for treating the curvature must be compatible with the wave-packet analysis employed in the second item above. All these aspects make the ``LGC-methodology''\footnote{For an overview of this methodology please see Section \ref{sec:discretization}.} developed in \cite{lvUnif} an essential ingredient in our approach.
\end{itemize}

In order to detail the above description, we follow the (part II) strategy outlined in Section \ref{znzpar} and perform the linearization of the supremum in \eqref{Top} followed by the decomposition of our operator $T$ as
\beq\label{stdec}
T=:T_0+ T_{\osc},
\eeq
where (for $\chi(\cdot,\cdot)$ a smooth function such that $\chi(\lambda,t) \neq 0$ iff $|\lambda t^a| \lesssim 1$) we set
\begin{itemize}
\item the low-oscillatory component $T_0$ as
\beq\label{tlo}
T_0(f,g)(x):=\int f(x-t)g(x+t)e^{i \lambda(x) t^a} \chi(\lambda(x), t) \frac{dt}{t};
\eeq

\item the high-oscillatory component $T_{\osc}$ as
\beq\label{tosc}
T_{\osc}(f,g)(x):=\int f(x-t)g(x+t)e^{i \lambda(x) t^a} (1-\chi(\lambda(x),t)) \frac{dt}{t}.
\eeq
\end{itemize}

The low-oscillatory component $T_0$ behaves much like the maximally truncated bilinear Hilbert transform (see \eqref{Lmax}), and hence, based on Lacey's result \cite{Lacey}, is under control in the same range specified by Theorem \ref{main}.

For the second term $T_{\osc}$, immersing the philosophy from the third item at the beginning of this subsection into the scheme summarized by \eqref{stdecc}-\eqref{stdeccee}, we discretize the high-oscillatory component according to the level sets of the phase function as follows:
\begin{equation}
\label{thi}
\begin{aligned}
T_{\osc}(f,g)(x) & =:\sum_{m\in\N} T_{m}(f,g)(x) \\
 & :=\sum_{m\in\N}\int f(x-t)g(x+t)e^{i \lambda(x) t^a} \chi_m(\lambda(x),t) \frac{dt}{t},
\end{aligned}
\end{equation}
where here $\chi_{m}(\cdot,\cdot)$ is a smooth function such that $\chi_m(\lambda,t) \neq 0$ iff $|\lambda t^a| \sim 2^{am}$.

Thus, our goal -- formulated below as Theorem \ref{main_theorem_pxq_to_r} -- will be to show that each such piece $T_m$ obeys an exponentially decaying bound of the form
\beq\label{mainrm}
\|T_m(f,g)\|_{L^r}\lesssim_{a,p,q} 2^{-\d a m}\|f\|_{L^p}\|g\|_{L^q},
\eeq
where here $\d=\d(p,q)>0$ and $p,q,r$ lie within the range specified in the statement of our main result, Theorem \ref{main}.

Returning now to the more philosophical tone of our discussion from the beginning of this section, we notice that the strategy set forth above is common\footnote{For a more detailed discussion on this, please see Section \ref{Hist}.} to harmonic analysis problems involving a purely non-zero curvature, which canonically rely on two distinct treatments:
\begin{itemize}
\item the \emph{low oscillatory component} encapsulates the purely singular behavior of the kernel (in the absence of the oscillation) and is always controlled by a variant of a maximal/singular operator that is well understood;\footnote{This is the situation where the existence of a mean zero condition -- as for our case $\int_{\R}\frac{dt}{t}=0$ -- plays a determinative role.}

\item the \emph{high oscillatory component} is where the phase cancellation becomes relevant and is captured via the (non)stationary phase principle and orthogonality techniques relying on the $TT^{*}$ method.
\end{itemize}
Consequently, in the non-zero curvature case, the treatments of the low and high oscillatory components are expected to be mutually exclusive, with the latter involving no genuine modulation invariance methods.

Going back to our specific problem, a thorough inspection of the elements displayed in \eqref{tosc}-\eqref{mainrm}
reveal the key philosophical difficulty that we are about to face: on the one hand the exponential decay in $m$ in \eqref{mainrm} is a manifestation of the cancellation encoded in the non-zero curvature behavior of the Carleson-type operator $C^a$; on the other hand, the very symbol of the operator $T_m$ -- see \eqref{multm} below -- is translation invariant under the action $(\xi,\eta) \mapsto (\xi+N, \eta+N)$ for arbitrary $N\in\R$.

This reveals a second novelty of this paper in accordance with the truly hybrid nature of our problem within the non-zero/zero paradigm: the approach of both the low and the high oscillatory components rely in a key fashion on modulation invariance and hence time-frequency analysis methods.

\subsection{Discretization strategy--preamble}
With this context provided to our story, we resume our analysis of \eqref{thi} by presenting an overview of the key steps in proving \eqref{mainrm}.

We start by decomposing each $T_{m}$ as
\begin{align*}
T_{m}(f,g)(x) & =\sum_{k \in \mathbb{Z}}T_{m,k}(f,g)(x) \\
 & =:\sum_{k\in\Z}\rho_a(2^{-a(m-k)}\lambda(x)) \int f(x-t)g(x+t) e^{i\lambda(x) t^a} \rho(2^{-k}t) \frac{dt}{t},
\end{align*}
where here we have conveniently decoupled the variables $\l$ and $t$ by writing $\chi_m(\lambda,t) =: \sum_{k \in \mathbb{Z}} \rho_a(2^{-a (m-k)}\lambda)\rho(2^{-k}t)$ with $\rho,\rho_a\in C_0^{\infty}(\R)$ suitable smooth functions supported away from the origin.

The first manifestation of the non-zero curvature of the phase $\lambda(x)t^a$ -- observe that for each $T_{m,k}$ one has $|\lambda(x)t^a| \sim 2^{am}$ and $|t| \sim 2^{k}$ -- is captured through a stationary phase analysis that facilitates the approximation
\begin{equation}
\begin{aligned}\label{multm}
& T_{m,k}(f,g)(x) \approx 2^{-\frac{a m}{2}}\rho_{a m- a k}(\l(x))  \\
& \qquad \cdot \iint_{\R^2} \hat{f}(\xi)\hat{g}(\eta) e^{ i c (\xi-\eta)^{\frac{a}{a-1}}\l(x)^{-\frac{1}{a-1}}}e^{ix(\xi+\eta)}\psi\big(\frac{\xi-\eta}{2^{a m-k}}\big)d\xi\,d\eta,
\end{aligned}
\end{equation}
where here $\psi\in C_0^{\infty}(\R)$ with $\textrm{supp}\,\psi\subset\{s : 1/2 < |s| < 2\}$.

Analyzing \eqref{multm} on the frequency side, we notice that the combined information for $f$ and $g$ is concentrated within the strip\footnote{For simplicity, we focus in what follows only on the case $\xi>\eta$.} $\ds 2^{am-k} \leq  \xi- \eta \leq 2^{am-k+1}$. Guided by this geometry, we first split the strip into ``big'' squares of side-length $2^{am-k}$, obtaining a collection$\big \{   \big[ 2^{am-k} (\ell +1),   2^{am-k} (\ell +2)  \big]  \times \big[ 2^{am-k} (\ell -1),   2^{am-k} \ell  \big]\big \}_{\ell \in \Z}$ as in Figure \ref{figure:BigCubes}.

%%%%%%%%%%%%%%%%%%%%%%%%%%%%%%%%%%%%%%%%%%%
%
% FIGURE 1 - picture corresponding to the BHT portrait with a stress
%            on a specific big frequency cube
%
%%%%%%%%%%%%%%%%%%%%%%%%%%%%%%%%%%%%%%%%%%%
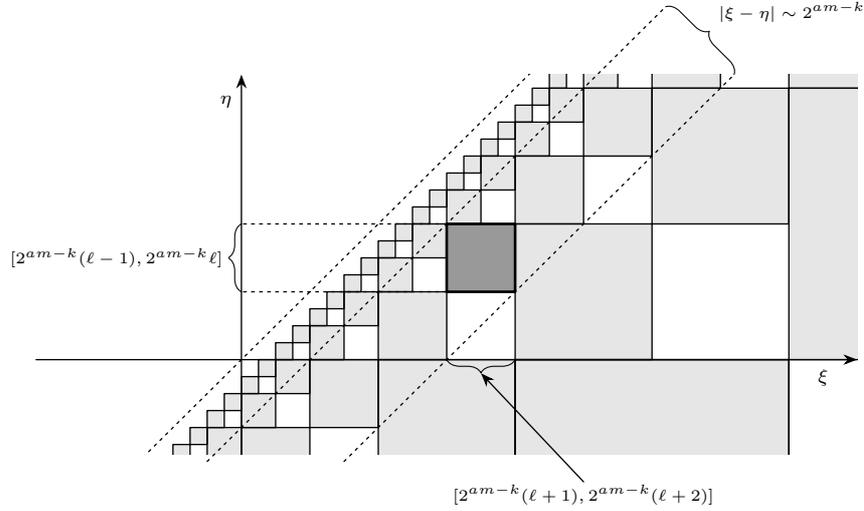
\begin{figure}[ht]
\centering
\begin{tikzpicture}[line cap=round,line join=round,>=Stealth,x=1cm,y=1cm,decoration={brace,amplitude=5pt},scale=0.9]
\clip(-4.5,-2.3) rectangle (10,5.5);

\begin{scope}
\clip(-3,-1.4) rectangle (9,4.2);
\filldraw[line width=0.5pt,color=black,fill=black,fill opacity=0.1] (4,-4) -- (8,-4) -- (8,0) -- (4,0) -- cycle;
\filldraw[line width=0.5pt,color=black,fill=black,fill opacity=0.1] (8,0) -- (12,0) -- (12,4) -- (8,4) -- cycle;
\filldraw[line width=0.5pt,color=black,fill=black,fill opacity=0.1] (0,-4) -- (2,-4) -- (2,-2) -- (0,-2) -- cycle;
\filldraw[line width=0.5pt,color=black,fill=black,fill opacity=0.1] (2,-2) -- (4,-2) -- (4,0) -- (2,0) -- cycle;
\filldraw[line width=0.5pt,color=black,fill=black,fill opacity=0.1] (4,0) -- (6,0) -- (6,2) -- (4,2) -- cycle;
\filldraw[line width=0.5pt,color=black,fill=black,fill opacity=0.1] (6,2) -- (8,2) -- (8,4) -- (6,4) -- cycle;
\filldraw[line width=0.5pt,color=black,fill=black,fill opacity=0.1] (8,4) -- (10,4) -- (10,6) -- (8,6) -- cycle;
\filldraw[line width=0.5pt,color=black,fill=black,fill opacity=0.1] (-1,-3) -- (0,-3) -- (0,-2) -- (-1,-2) -- cycle;
\filldraw[line width=0.5pt,color=black,fill=black,fill opacity=0.1] (0,-2) -- (1,-2) -- (1,-1) -- (0,-1) -- cycle;
\filldraw[line width=0.5pt,color=black,fill=black,fill opacity=0.1] (1,-1) -- (2,-1) -- (2,0) -- (1,0) -- cycle;
\filldraw[line width=0.5pt,color=black,fill=black,fill opacity=0.1] (2,0) -- (3,0) -- (3,1) -- (2,1) -- cycle;
%
% HIGHLIGHTED SQUARE
%
\filldraw[line width=1pt,color=black,fill=black,fill opacity=0.4] (3,1) -- (4,1) -- (4,2) -- (3,2) -- cycle;
\filldraw[line width=0.5pt,color=black,fill=black,fill opacity=0.1] (4,2) -- (5,2) -- (5,3) -- (4,3) -- cycle;
\filldraw[line width=0.5pt,color=black,fill=black,fill opacity=0.1] (5,3) -- (6,3) -- (6,4) -- (5,4) -- cycle;
\filldraw[line width=0.5pt,color=black,fill=black,fill opacity=0.1] (6,4) -- (7,4) -- (7,5) -- (6,5) -- cycle;
\filldraw[line width=0.5pt,color=black,fill=black,fill opacity=0.1] (-1.5,-2.5) -- (-1,-2.5) -- (-1,-2) -- (-1.5,-2) -- cycle;
\filldraw[line width=0.5pt,color=black,fill=black,fill opacity=0.1] (-1,-2) -- (-0.5,-2) -- (-0.5,-1.5) -- (-1,-1.5) -- cycle;
\filldraw[line width=0.5pt,color=black,fill=black,fill opacity=0.1] (-0.5,-1.5) -- (0,-1.5) -- (0,-1) -- (-0.5,-1) -- cycle;
\filldraw[line width=0.5pt,color=black,fill=black,fill opacity=0.1] (0,-1) -- (0.5,-1) -- (0.5,-0.5) -- (0,-0.5) -- cycle;
\filldraw[line width=0.5pt,color=black,fill=black,fill opacity=0.1] (0.5,-0.5) -- (1,-0.5) -- (1,0) -- (0.5,0) -- cycle;
\filldraw[line width=0.5pt,color=black,fill=black,fill opacity=0.1] (1,0) -- (1.5,0) -- (1.5,0.5) -- (1,0.5) -- cycle;
\filldraw[line width=0.5pt,color=black,fill=black,fill opacity=0.1] (1.5,0.5) -- (2,0.5) -- (2,1) -- (1.5,1) -- cycle;
\filldraw[line width=0.5pt,color=black,fill=black,fill opacity=0.1] (2,1) -- (2.5,1) -- (2.5,1.5) -- (2,1.5) -- cycle;
\filldraw[line width=0.5pt,color=black,fill=black,fill opacity=0.1] (2.5,1.5) -- (3,1.5) -- (3,2) -- (2.5,2) -- cycle;
\filldraw[line width=0.5pt,color=black,fill=black,fill opacity=0.1] (3,2) -- (3.5,2) -- (3.5,2.5) -- (3,2.5) -- cycle;
\filldraw[line width=0.5pt,color=black,fill=black,fill opacity=0.1] (3.5,2.5) -- (4,2.5) -- (4,3) -- (3.5,3) -- cycle;
\filldraw[line width=0.5pt,color=black,fill=black,fill opacity=0.1] (4,3) -- (4.5,3) -- (4.5,3.5) -- (4,3.5) -- cycle;
\filldraw[line width=0.5pt,color=black,fill=black,fill opacity=0.1] (4.5,3.5) -- (5,3.5) -- (5,4) -- (4.5,4) -- cycle;
\filldraw[line width=0.5pt,color=black,fill=black,fill opacity=0.1] (5,4) -- (5.5,4) -- (5.5,4.5) -- (5,4.5) -- cycle;
\filldraw[line width=0.5pt,color=black,fill=black,fill opacity=0.1] (-1.75,-2.25) -- (-1.5,-2.25) -- (-1.5,-2) -- (-1.75,-2) -- cycle;
\filldraw[line width=0.5pt,color=black,fill=black,fill opacity=0.1] (-1.5,-2) -- (-1.25,-2) -- (-1.25,-1.75) -- (-1.5,-1.75) -- cycle;
\filldraw[line width=0.5pt,color=black,fill=black,fill opacity=0.1] (-1.25,-1.75) -- (-1,-1.75) -- (-1,-1.5) -- (-1.25,-1.5) -- cycle;
\filldraw[line width=0.5pt,color=black,fill=black,fill opacity=0.1] (-1,-1.5) -- (-0.75,-1.5) -- (-0.75,-1.25) -- (-1,-1.25) -- cycle;
\filldraw[line width=0.5pt,color=black,fill=black,fill opacity=0.1] (-0.75,-1.25) -- (-0.5,-1.25) -- (-0.5,-1) -- (-0.75,-1) -- cycle;
\filldraw[line width=0.5pt,color=black,fill=black,fill opacity=0.1] (-0.5,-1) -- (-0.25,-1) -- (-0.25,-0.75) -- (-0.5,-0.75) -- cycle;
\filldraw[line width=0.5pt,color=black,fill=black,fill opacity=0.1] (-0.25,-0.75) -- (0,-0.75) -- (0,-0.5) -- (-0.25,-0.5) -- cycle;
\filldraw[line width=0.5pt,color=black,fill=black,fill opacity=0.1] (0,-0.5) -- (0.25,-0.5) -- (0.25,-0.25) -- (0,-0.25) -- cycle;
\filldraw[line width=0.5pt,color=black,fill=black,fill opacity=0.1] (0.25,-0.25) -- (0.5,-0.25) -- (0.5,0) -- (0.25,0) -- cycle;
\filldraw[line width=0.5pt,color=black,fill=black,fill opacity=0.1] (0.5,0) -- (0.75,0) -- (0.75,0.25) -- (0.5,0.25) -- cycle;
\filldraw[line width=0.5pt,color=black,fill=black,fill opacity=0.1] (0.75,0.25) -- (1,0.25) -- (1,0.5) -- (0.75,0.5) -- cycle;
\filldraw[line width=0.5pt,color=black,fill=black,fill opacity=0.1] (1,0.5) -- (1.25,0.5) -- (1.25,0.75) -- (1,0.75) -- cycle;
\filldraw[line width=0.5pt,color=black,fill=black,fill opacity=0.1] (1.25,0.75) -- (1.5,0.75) -- (1.5,1) -- (1.25,1) -- cycle;
\filldraw[line width=0.5pt,color=black,fill=black,fill opacity=0.1] (1.5,1) -- (1.75,1) -- (1.75,1.25) -- (1.5,1.25) -- cycle;
\filldraw[line width=0.5pt,color=black,fill=black,fill opacity=0.1] (1.75,1.25) -- (2,1.25) -- (2,1.5) -- (1.75,1.5) -- cycle;
\filldraw[line width=0.5pt,color=black,fill=black,fill opacity=0.1] (2,1.5) -- (2.25,1.5) -- (2.25,1.75) -- (2,1.75) -- cycle;
\filldraw[line width=0.5pt,color=black,fill=black,fill opacity=0.1] (2.25,1.75) -- (2.5,1.75) -- (2.5,2) -- (2.25,2) -- cycle;
\filldraw[line width=0.5pt,color=black,fill=black,fill opacity=0.1] (2.5,2) -- (2.75,2) -- (2.75,2.25) -- (2.5,2.25) -- cycle;
\filldraw[line width=0.5pt,color=black,fill=black,fill opacity=0.1] (2.75,2.25) -- (3,2.25) -- (3,2.5) -- (2.75,2.5) -- cycle;
\filldraw[line width=0.5pt,color=black,fill=black,fill opacity=0.1] (3,2.5) -- (3.25,2.5) -- (3.25,2.75) -- (3,2.75) -- cycle;
\filldraw[line width=0.5pt,color=black,fill=black,fill opacity=0.1] (3.25,2.75) -- (3.5,2.75) -- (3.5,3) -- (3.25,3) -- cycle;
\filldraw[line width=0.5pt,color=black,fill=black,fill opacity=0.1] (3.5,3) -- (3.75,3) -- (3.75,3.25) -- (3.5,3.25) -- cycle;
\filldraw[line width=0.5pt,color=black,fill=black,fill opacity=0.1] (3.75,3.25) -- (4,3.25) -- (4,3.5) -- (3.75,3.5) -- cycle;
\filldraw[line width=0.5pt,color=black,fill=black,fill opacity=0.1] (4,3.5) -- (4.25,3.5) -- (4.25,3.75) -- (4,3.75) -- cycle;
\filldraw[line width=0.5pt,color=black,fill=black,fill opacity=0.1] (4.25,3.75) -- (4.5,3.75) -- (4.5,4) -- (4.25,4) -- cycle;
\filldraw[line width=0.5pt,color=black,fill=black,fill opacity=0.1] (4.5,4) -- (4.75,4) -- (4.75,4.25) -- (4.5,4.25) -- cycle;
\filldraw[line width=0.5pt,color=black,fill=black,fill opacity=0.1] (4.75,4.25) -- (5,4.25) -- (5,4.5) -- (4.75,4.5) -- cycle;
\draw [line width=0.5pt,color=black] (4,-4)-- (8,-4);
\draw [line width=0.5pt,color=black] (8,-4)-- (8,0);
\draw [line width=0.5pt,color=black] (8,0)-- (4,0);
\draw [line width=0.5pt,color=black] (4,0)-- (4,-4);
\draw [line width=0.5pt,dash pattern=on 1pt off 2pt,domain=-5:12] plot(\x,\x);
\draw [line width=0.5pt,dash pattern=on 1pt off 2pt] (3,1) -- (0,1);
\draw [line width=0.5pt,dash pattern=on 1pt off 2pt] (3,2) -- (0,2);
\draw [->,line width=0.5pt] (-3,0) -- (9,0);
\draw [->,line width=0.5pt] (0,-1.4) -- (0,4.2);
\draw [anchor=north] (8.5,0) node {\scriptsize $\xi$};
\draw [anchor= east] (0,3.8) node {\scriptsize $\eta$};
\end{scope}

\draw [line width=0.5pt,dash pattern=on 1pt off 2pt,domain=-0.5:6.2] plot(\x,\x-1);
\draw [line width=0.5pt,dash pattern=on 1pt off 2pt,domain=1.5:7.2] plot(\x,\x-3);
\draw [decorate, color=black] (0,1) -- (0,2)
	node [midway, anchor=east, fill=white, inner sep=2pt, outer sep=5pt]{\tiny $[ 2^{am-k} (\ell -1),   2^{am-k} \ell]$};
\draw [decorate, color=black] (6.2,5.2) -- (7.2,4.2)
	node [midway, anchor=south west, fill=white, inner sep=2pt, outer sep=5pt]{\tiny $|\xi - \eta| \sim 2^{am-k}$};
\draw [decorate, color=black] (4,0) -- (3,0);
\draw [->, line width=0.5pt] (5,-1.8) -- (3.5,-0.2);
\draw [anchor=north] (5,-1.7) node {\tiny $[ 2^{am-k} (\ell +1),   2^{am-k} (\ell+2)]$};
	
\end{tikzpicture}
\caption{\footnotesize Geometry of the frequency plane associated to $T_m$: Whitney ``big'' square decomposition with respect to the singularity  $\xi=\eta$ which may be seen as a $2^{am}$-dilation of the bilinear Hilbert transform's frequency portrait. Here $k$ stands for the scale of the spatial kernel variable $t$ while $m$ relates to the height of the oscillatory phase.} \label{figure:BigCubes}
\end{figure}
%
%%%%%%%%%%%%%%%%%%%%%%%%%%%
%

Once at this point we remark that Figure \ref{figure:BigCubes} captures the very essence of the hybrid character of our operator:

\begin{enumerate}[label=(\roman*)]
\item the dilation commutation and modulation invariance are encoded in the Littlewood-Paley strip decomposition of the frequency plane and the translation invariance of the family of big squares that contribute to the discretization of each of the strips, respectively;
\item the nonzero curvature of the phase manifests (in a first instance) through the $2^{am}$ thickening factor relative to the standard bilinear Hilbert transform's frequency portrait.
\end{enumerate}

Regarded from the angle provided by the above interpretation it becomes natural to split our analysis on two levels:
\begin{itemize}
\item a \emph{low resolution analysis} corresponding to (i) and thus addressing the properties \ref{key:td} and \ref{key:m}: the focus here is to exploit the almost orthogonality \emph{among} (families of) big squares;

\item a \emph{high resolution analysis}  corresponding to (ii) and thus treating \ref{key:max}: this focuses on capturing the cancellation \emph{within each} given big square.
\end{itemize}

In this way we introduce the third novelty of our paper: in treating our  high-oscillatory component $T_{\osc}=\sum_{m\in\N} T_{m}$ we design a ``two-resolution analysis''.

\subsection{Two-resolution analysis: an introduction}\label{TRAI}

In this section we very briefly present some key characteristics of the two-resolution analysis implemented in the present paper.

\begin{enumerate}[label=(\Roman*), leftmargin=16pt]
\item \label{step:high}\emph{the high resolution, single-scale analysis}: as already alluded, this is the regime in which we exploit the non-zero curvature of the phase of the multiplier.

For this, we fix the scale $k\in\Z$, the frequency parameter $\ell\in\Z$ and its dual space parameter $r\in \Z$ and focus on the wave-packet interaction inside a fixed (given) big frequency square
\begin{equation}
\label{freqsq}
\ds \big[ 2^{am-k} (\ell +1),   2^{am-k} (\ell +2)  \big]  \times \big[ 2^{am-k} (\ell -1),   2^{am-k} \ell  \big]\:.
\end{equation}
 Once at this point we employ the LGC-methodology from \cite{lvUnif} as follows:
\begin{itemize}[leftmargin=*]
\item Stage one is represented by a discretization of the multiplier that achieves the \emph{phase linearization}:

The effect of this process will be the division of every ``big'' square $\ds \big[ 2^{am-k} (\ell +1),   2^{am-k} (\ell +2)  \big]  \times \big[ 2^{am-k} (\ell -1),   2^{am-k} \ell  \big]$  into $2^{am}$ smaller squares, as illustrated in Figure \ref{figure:SmallCubes}.\footnote{This amounts to partitioning every side into $2^{am \over 2}$ intervals of equal length.}

%%%%%%%%%%%%%%%%%%%%%%%%%%%%%%%%%%%%%%%%%%%
%
% FIGURE 2 - picture corresponding to the partitioning
%            of the big cube into small cubes
%
%%%%%%%%%%%%%%%%%%%%%%%%%%%%%%%%%%%%%%%%%%%
\begin{figure}[ht]
\centering
\input{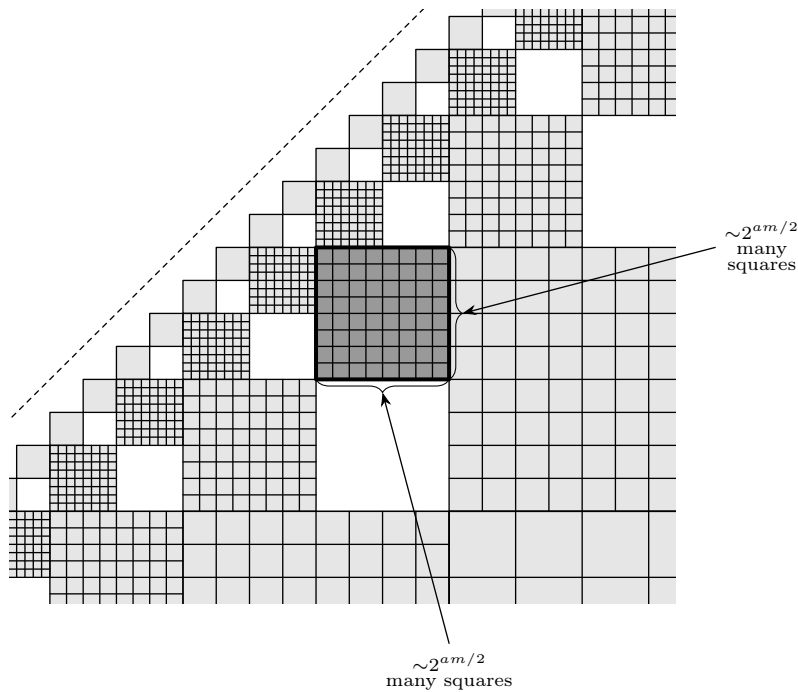}
\caption{\footnotesize The multiplier's phase behavior is essentially linear within each small square obtained by splitting every ``big'' square into $2^{am}$ identical pieces.} \label{figure:SmallCubes}
\end{figure}
%
%Splitting of every ``big'' squares into $2^{am}$ smaller squares.
%%%%%%%%%%%%%%%%%%%%%%%%%%%
%

\item Stage two consists of an adapted \emph{Gabor frame} discretization of the input functions.

\item Stage three employs almost orthogonality methods relying on the \emph{$TT^{*}$-method} and time-frequency \emph{correlations}.
\end{itemize}

An additional, novel ingredient within this analysis is the simultaneous two types of models, that, though equivalent, organize differently the information of the small squares living within the given big square. For an overview of this approach please see Section \ref{HR10}. The main conclusion of this analysis is provided in Proposition \ref{sgscale}, Section \ref{sec:onescale}.

\item \label{step:trans}  \emph{A transitional step -- acquiring compatibility between \ref{step:high} and \ref{step:low}}: this step is required in order to harmonize the non-zero curvature methods from \ref{step:high} with the wave-packet modulation invariant methods from \ref{step:low} and at the heuristic level should be regarded as a decoupling of the action of the two components blended in $BC^a$, namely the bilinear Hilbert transform $B$ and the Carleson-type operator $C^a$.

More concretely, at this step we refine the work performed in \ref{step:high} by obtaining a time-frequency localized version of the single scale estimate in which both the spatial and the frequency localization of all the three inputs $f,g$ and $h$ as well as the information carried by the linearized function $\lambda$ manifest. This refined analysis is required in order to properly implement the multi-scale methodology referred to in the next item. Our treatment of this second stage will involve the mass distribution of an weight encoding the behavior of $\lambda$ (see \eqref{def:eq:weight-oscill0}) as well as the size distribution of the functions $f$ and $g$. At the end of this process we will have obtained a symmetric trilinear single scale estimate that properly quantifies the time-frequency information of each of the input functions as well as the mass carried by $\l$, and all these in a manner that is compatible with the size-energy approach for the bilinear Hilbert transform. The desired localized single-scale estimate is briefly outlined in Section \ref{HR20} and proved in Section \ref{sec:HR2} -- see Proposition \ref{refinedsgscale} therein.

\item \label{step:low} \emph{The low resolution, multi-scale analysis}: this is the regime that addresses the modulation-invariance symmetry \ref{key:m}.

In this situation we combine all the scales $k$ together, allowing full liberty to $\ell, r$ parameters in order to understand the time-frequency interactions among the big frequency squares $$\left\{\ds \big[ 2^{am-k} (\ell +1),   2^{am-k} (\ell +2)  \big]  \times \big[ 2^{am-k} (\ell -1),   2^{am-k} \ell  \big]\right\}_{k,\ell\in\Z}$$ considered together with their dual space squares. A key insight here is the observation that the time-frequency representation of the entire collection of these squares corresponds to a suitable anisotropic dilation of the time-frequency portrait of the bilinear Hilbert transform. The above considerations trigger the appeal to the methods developed by Lacey and Thiele in \cite{lt1}, \cite{lt2}, and refined in the form offered by \cite{MTTBiest2}, based on which -- together with the localized version of the single-scale estimate discussed in \ref{step:trans} -- we will be able to conclude that the dualized version of \eqref{mainrm} holds. The procedure described here is previewed in Section \ref{LRMS} and covered in detail in Section \ref{sec:bilinear:analysis}.
\end{enumerate}

\subsection{Two-resolution analysis: time-frequency portraits}\label{tfp}

In this section we rephrase the above discussion in a geometric language.\footnote{Throughout the remainder of the present section the parameter $m\in\N$ is fixed.} According to the two regimes -- for expository reasons displayed here in reverse order -- we will consider two types of families of (tri-)tiles, illustrated below in Figures \ref{figure:tiles_a} and \ref{figure:tiles_b}:

$\newline$
\noindent A) \emph{The low resolution, multi-scale analysis}. In this regime, guided by the previously mentioned analogy with the bilinear Hilbert transform, we are invited to consider the family $\BHT={{{\mathbb P}^m_{\textrm{BHT}}}}$ of all tri-tiles $P=(P_1,P_2,P_3)$ obeying

\begin{equation}
\label{eq:tile:1}
P_1:= I_P \times  \pmb{\omega}_{P_1} =[r 2^k, (r+1) 2^k] \times  \big[ 2^{am-k} (\ell +1),   2^{am-k}( \ell+2)  \big],
\end{equation}
\begin{equation}
\label{eq:tile:2}
P_2:=I_P \times  \pmb{\omega}_{P_2} = [r 2^k, (r+1) 2^k] \times  \big[ 2^{am-k} (\ell -1),   2^{am-k} \ell  \big],
\end{equation}
and
\begin{equation}
\label{eq:tile:3}
P_3:=I_P \times  \pmb{\omega}_{P_3} =[r 2^k, (r+1) 2^k] \times  \big[ 2^{am-k} (2\ell ),   2^{am-k}(2 \ell+1)  \big],
\end{equation}
with $k,r, \ell$ varying in $\Z$. Each of the tiles above has area $2^{am}$ as opposed to the area one canonically encountered in the time-frequency analysis literature. This is a direct manifestation of the presence of the maximal modulation, which via the curvature of the phase spreads out the localization of the wave-packet. With $P$ as above, recalling \eqref{multm} and defining the dual form $\underline{\L}_{m,k}(f, g, h):=\int T_{m, k}(f, g)(x) h(x) dx$ (with the obvious analogue for $\underline{\L}_{m}$) one is able to decompose\footnote{For a precise definition please see \eqref{lmk}, \eqref{eq:modeltilee} and \eqref{Ptile}.}
\beq\label{lmP}
\underline{\L}_{m}=\sum_{k\in\Z}\underline{\L}_{m,k}=\sum_{P\in \BHT} \underline{\L}_{m,P}\:.
\eeq
As claimed earlier, it now becomes apparent that $\underline{\L}_{m}$ has a similar (but dilated) time-frequency portrait to that of BHT, with the family of tri-tiles $\BHT$ preserving the key rank-one property.\footnote{\textit{I.e.}, any tri-tile $P=(P_1,P_2,P_3)$ is completely determined by the spatial and frequency position of any one of $P_1, P_2$ or $P_3$.}

$\newline$
\noindent B) \emph{The high resolution, single-scale analysis.}\label{highr} In this setting, as already explained, we are dealing with a single-scale analysis that will take place within each (fixed) tri-tile $P\in\BHT$. Though involving only one scale, the time-frequency representation of $\underline{\L}_{m,P}$ comes with an increased level of complexity in part due to the fact that this representation\footnote{See \eqref{tilemodel}.}  involves a summation over a family of tri-tiles $\S(P)$ having four degrees of freedom. Indeed, given $P\in\BHT$ the family $\S(P)$ consists of all the tri-tiles $s=(s_1,s_2,s_3)$ with $s_j=I_{s_j}\times\omega_{s_j}$, $j\in\{1, 2, 3\}$, such that the following holds:

\begin{itemize}
\item $I_{s_j}$, $\omega_{s_j}$ intervals with $|I_{s_j}|\,|\omega_{s_j}|=1$;

\item $I_{s_3}=\frac{I_{s_1}+I_{s_2}}{2}$ and $\omega_{s_3}=\omega_{s_1}+\omega_{s_2}$;

\item $\hat{s}=P$, where here $\hat{s}$ is the $2^{am \over 2}$-enlargement of $s$ in both space and frequency.
\end{itemize}
%
%%%%%%%%%%%%%%%%%%%%%%%%%%%%%%%%%%%%%%%%%%%
%
% FIGURE 3a - picture illustrating the big tile and small tile separately
%
%%%%%%%%%%%%%%%%%%%%%%%%%%%%%%%%%%%%%%%%%%%
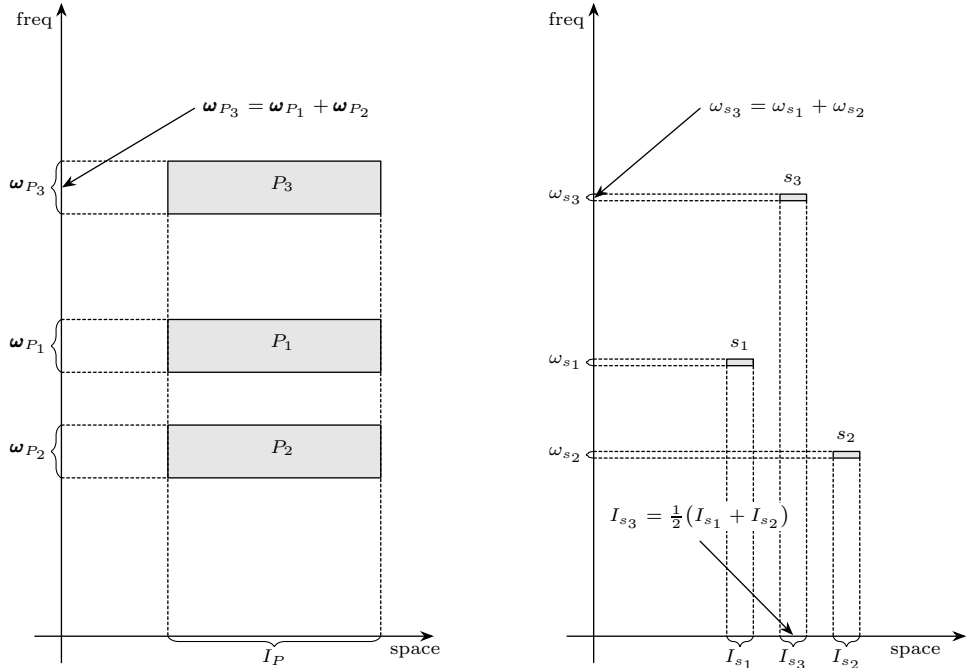
\begin{figure}[ht]
\centering
\begin{tikzpicture}[line cap=round,line join=round,>=Stealth,x=1cm,y=1cm,decoration={brace,amplitude=4pt}, scale=0.7]

\clip(-1,-0.6) rectangle (17,12);
\filldraw[line width=0.5pt,color=black,fill=black,fill opacity=0.1] (2,3) -- (6,3) -- (6,4) -- (2,4) -- cycle;
\filldraw[line width=0.5pt,color=black,fill=black,fill opacity=0.1] (2,5) -- (6,5) -- (6,6) -- (2,6) -- cycle;
\filldraw[line width=0.5pt,color=black,fill=black,fill opacity=0.1] (2,8) -- (6,8) -- (6,9) -- (2,9) -- cycle;
\filldraw[line width=0.5pt,color=black,fill=black,fill opacity=0.1] (14.5,3.375) -- (15,3.375) -- (15,3.5) -- (14.5,3.5) -- cycle;
\filldraw[line width=0.5pt,color=black,fill=black,fill opacity=0.1] (12.5,5.125) -- (13,5.125) -- (13,5.25) -- (12.5,5.25) -- cycle;
\filldraw[line width=0.5pt,color=black,fill=black,fill opacity=0.1] (13.5,8.25) -- (14,8.25) -- (14,8.375) -- (13.5,8.375) -- cycle;
\draw [->,line width=0.5pt] (-0.5,0) -- (7,0);
\draw [->,line width=0.5pt] (0,-0.5) -- (0,12);

\draw [->,line width=0.5pt] (9.5,0) -- (17,0);
\draw [->,line width=0.5pt] (10,-0.5) -- (10,12);

\draw [line width=0.5pt,dash pattern=on 1pt off 1pt] (2,8)-- (2,0);
\draw [line width=0.5pt,dash pattern=on 1pt off 1pt] (6,8)-- (6,0);
\draw [line width=0.5pt,dash pattern=on 1pt off 1pt] (14.5,3.375)-- (14.5,0);
\draw [line width=0.5pt,dash pattern=on 1pt off 1pt] (15,3.375)-- (15,0);
\draw [line width=0.5pt,dash pattern=on 1pt off 1pt] (12.5,5.125)-- (12.5,0);
\draw [line width=0.5pt,dash pattern=on 1pt off 1pt] (13,5.215)-- (13,0);

\draw [line width=0.5pt,dash pattern=on 1pt off 1pt] (13.5,8.25)-- (13.5,0);
\draw [line width=0.5pt,dash pattern=on 1pt off 1pt] (14,8.25)-- (14,0);

\draw [line width=0.5pt,dash pattern=on 1pt off 1pt] (0,3) -- (2,3);
\draw [line width=0.5pt,dash pattern=on 1pt off 1pt] (0,4) -- (2,4);
\draw [line width=0.5pt,dash pattern=on 1pt off 1pt] (0,5) -- (2,5);
\draw [line width=0.5pt,dash pattern=on 1pt off 1pt] (0,6) -- (2,6);
\draw [line width=0.5pt,dash pattern=on 1pt off 1pt] (0,8) -- (2,8);
\draw [line width=0.5pt,dash pattern=on 1pt off 1pt] (0,9) -- (2,9);
\draw [line width=0.5pt,dash pattern=on 1pt off 1pt] (10,3.375) -- (14.5,3.375);
\draw [line width=0.5pt,dash pattern=on 1pt off 1pt] (10,3.5) -- (14.5,3.5);
\draw [line width=0.5pt,dash pattern=on 1pt off 1pt] (10,5.125) -- (12.5,5.125);
\draw [line width=0.5pt,dash pattern=on 1pt off 1pt] (10,5.25) -- (12.5,5.25);
\draw [line width=0.5pt,dash pattern=on 1pt off 1pt] (10,8.25) -- (13.5,8.25);
\draw [line width=0.5pt,dash pattern=on 1pt off 1pt] (10,8.375) -- (13.5,8.375);
\draw (3.75,8.25) node[anchor=south west] {\scriptsize $P_3$};
\draw (3.75,5.25) node[anchor=south west] {\scriptsize $P_1$};
\draw (3.75,3.25) node[anchor=south west] {\scriptsize $P_2$};
\draw (12.75,5.25) node[anchor=south] {\scriptsize $s_1$};
\draw (13.75,8.375) node[anchor=south] {\scriptsize $s_3$};
\draw (14.75,3.5) node[anchor=south] {\scriptsize $s_2$};

\draw [decorate, color=black] (0,5) -- (0,6)
	node [midway, anchor=east, fill=white, inner sep=1pt, outer sep=4pt]{\scriptsize $\pmb{\omega}_{P_1}$};
\draw [decorate, color=black] (0,3) -- (0,4)
	node [midway, anchor=east, fill=white, inner sep=1pt, outer sep=4pt]{\scriptsize $\pmb{\omega}_{P_2}$};
\draw [decorate, color=black] (0,8) -- (0,9)
	node [midway, anchor=east, fill=white, inner sep=1pt, outer sep=4pt]{\scriptsize $\pmb{\omega}_{P_3}$};
\draw (2.5,10) node [anchor=west]{\scriptsize $\pmb{\omega}_{P_3} = \pmb{\omega}_{P_1} + \pmb{\omega}_{P_2}$};
\draw [->,line width=0.5pt] (2.5,10) -- (0.01,8.5);
	
\draw [decorate, color=black] (10,3.375) -- (10,3.5)
	node [midway, anchor=east, fill=white, inner sep=1pt, outer sep=3pt]{\scriptsize $\omega_{s_2}$};
\draw [decorate, color=black] (10,5.125) -- (10,5.25)
	node [midway, anchor=east, fill=white, inner sep=1pt, outer sep=3pt]{\scriptsize $\omega_{s_1}$};
\draw [decorate, color=black]  (10,8.25) -- (10,8.375)
	node [midway, anchor=east, fill=white, inner sep=1pt, outer sep=3pt]{\scriptsize $\omega_{s_3}$};
\draw (12,10) node[anchor=west] {\scriptsize $\omega_{s_3}  = \omega_{s_1} + \omega_{s_2}$};
\draw [->,line width=0.5pt] (12,10) -- (10.01,8.3);

\draw [decorate, color=black] (6,0) -- (2,0)
	node[midway, anchor=north, fill=white, inner sep=1pt, outer sep=4pt]{\scriptsize $I_{P}$};
 \draw [decorate, color=black] (13,0) -- (12.5,0)
	node[midway, anchor=north, fill=white, inner sep=1pt, outer sep=4pt]{\scriptsize $I_{s_1}$};
\draw [decorate, color=black] (14,0) -- (13.5,0)
	node[midway, anchor=north, fill=white, inner sep=1pt, outer sep=4pt] {\scriptsize $I_{s_3}$};
\draw [decorate, color=black] (15,0) -- (14.5,0)
	node[midway, anchor=north, fill=white, inner sep=1pt, outer sep=4pt]{\scriptsize $I_{s_2}$};
	
\draw [->, line width=0.5pt] (12,1.8) -- (13.75,0.02);
\draw (12,1.8) node[anchor=south, fill=white, inner sep=1pt, outer sep=4pt] {\scriptsize $I_{s_3} = \frac{1}{2}\big(I_{s_1} + I_{s_2}\big)$};
 
\draw (6,0) node[anchor=north west] {\scriptsize $\text{space}$};
\draw (15.4,0) node[anchor=north west] {\scriptsize $\text{space}$};
\draw (0,12) node[anchor=north east] {\scriptsize $\text{freq}$};
\draw (10,12) node[anchor=north east] {\scriptsize $\text{freq}$};
\end{tikzpicture}
\caption{\footnotesize Time-frequency portrait of tiles $P$ and $s$ such that $\hat{s}=P$.} \label{figure:tiles_a}
\end{figure}
%
%%%%%%%%%%%%%%%%%%%%%%%%%%%
%
%
%%%%%%%%%%%%%%%%%%%%%%%%%%%%%%%%%%%%%%%%%%%
%
% FIGURE 3b - picture illustrating the relationship between the big tile and small tiles
%
%%%%%%%%%%%%%%%%%%%%%%%%%%%%%%%%%%%%%%%%%%%
\begin{figure}[ht]
\centering
\input{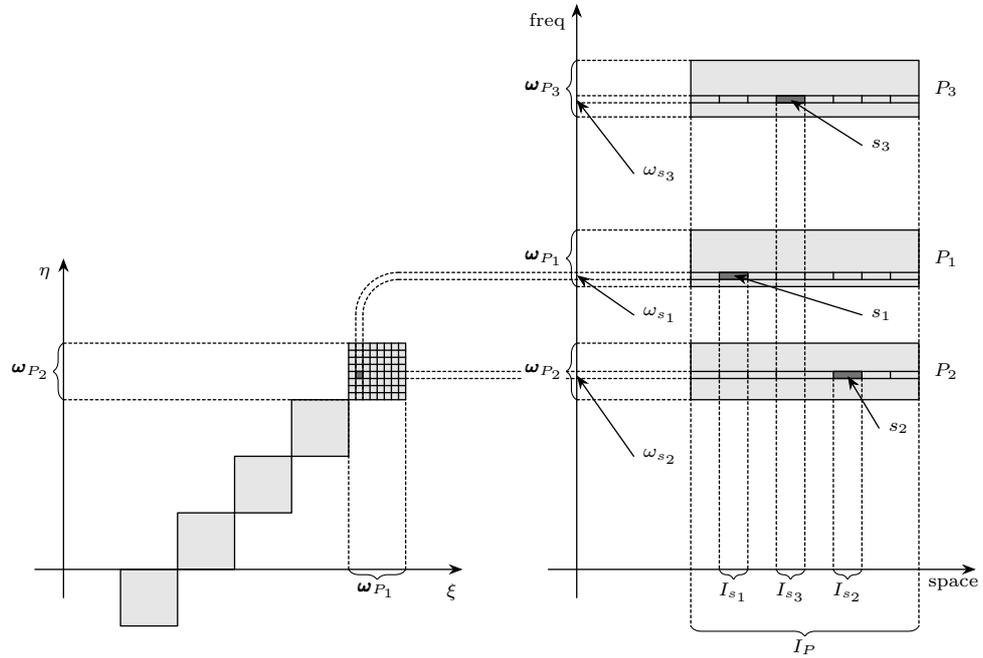}
\caption{\footnotesize Relationship between tiles $s$ and a tile $P$ with $\hat{s}=P$.} \label{figure:tiles_b}
\end{figure}
%
%%%%%%%%%%%%%%%%%%%%%%%%%%%
%

A comparison of $\S(P)$ and $\BHT$ reveals that
\begin{itemize}
\item on the one hand $\S(P)$ is less rich than $\BHT$ since all the tri-tiles $s$ have the same scale;

\item on the other hand $\S(P)$ has higher complexity than $\BHT$ since $\S(P)$ is a family with four degrees of freedom, whereas the rank-one family $\BHT$ has only three; and

\item if $P=(P_1,P_2,P_3)\in \BHT$ and $s=(s_1,s_2,s_3)\in \S(P)$ then any of the tiles $s_j$ has the Heisenberg uncertainty area one as opposed to area $2^{am}$ for any of the $P_j$'s.
\end{itemize}

For every fixed tri-tile $P\in\BHT$, the first two stages of the methodology described in \ref{step:high} will have as an output the decomposition
\begin{equation}\label{tilemodel}
\underline{\L}_{m,P}(f,g,h)=  \frac{1}{2^{\frac{am}{4}}\,|I_P|^{1 \over 2}} \sum_{s\in \S(P)} \langle f, \phi_{s_1} \rangle  \langle g, \phi_{s_2}  \rangle \langle h, \underline{\phi}_{s_3} \rangle\,,
\end{equation}
where here $\phi_{s_1}$, $\phi_{s_2}$ are wave-packets adapted to $s_1$ and $s_2$, respectively, while $\underline{\phi}_{s_3}$ stands for a suitable weighted wave-packet accounting for the joint information carried by the linearizing function $\l$ and the tile $s_3$.\footnote{For the exact formulations one is invited to consult \eqref{tmk}-\eqref{eq:modeltilee}, \eqref{s12} and \eqref{s3}.}

\subsection{(I) The high resolution, single-scale analysis -- taking advantage of the curvature}\label{HR10}

We design this approach in order to exploit the cancellation within each of the dilated ``big square'' (see Figure \ref{figure:BigCubes}). Thus, we fix in what follows the frequency square \eqref{freqsq}.

\subsubsection{The LGC-methodology: initiation}

As briefly mentioned in Section \ref{TRAI}, the starting point for the high-resolution analysis is triggered by the methodology in \cite{lvUnif}. Detailing the discussion in \ref{step:high} we proceed as follows:

In the first stage, as a result of the multiplier's phase linearization process we obtain the partitioning of the big square into $2^{am}$ same-size squares -- see Figure \ref{figure:SmallCubes}.

The second stage consists from an adapted \emph{Gabor frame} discretization of the input functions:

Let $f_{\ell+1}$ be the Fourier projection of $f$ onto the interval $\big[ 2^{am-k} (\ell+1), 2^{am-k} (\ell +2)  \big]$ (and similarly $g_{\ell-1})$. Choose a suitably smooth compactly supported function $\psi_1$ and define the $L^2$-normalized wave-packet
\beq\label{defv}
\varphi^{u,p}_{\frac{a m}{2}-k, \ell}(\xi):=\frac{1}{2^{\frac{1}{2} (\frac{a m}{2}-k )}}\psi_1\big(\frac{\xi}{2^{\frac{a m}{2}-k}}-u- 2^{{am \over 2}}\left( \ell -1 \right) \big) e^{-i \frac{p \xi}{2^{\frac{a m}{2}-k}}}.
 \eeq
Next, for $m, k, \ell$ fixed, we expand $f_{\ell+1}$ as
\beq\label{Gfd}
f_{\ell+1}=\sum_{\substack{2\cdot 2^{am\over 2}< u \leq 3 \cdot 2^{am \over 2} \\ \bar p\in \Z}}\langle f, \check{\varphi}_{{am \over 2}- k, \ell}^{u, \bar p}  \rangle  \check{\varphi}_{{am \over 2}- k, \ell}^{u, \bar p} ,
\eeq
with the obvious analogue for $g_{\ell-1}$ given by \footnote{Observe that the shift between the frequency supports of $f_{\ell+1}$ and $g_{\ell-1}$ is encoded in the range of the frequency parameters $u,v$.}
\beq\label{Ggd}
g_{\ell-1}=\sum_{\substack{0< v \leq 2^{am \over 2} \\ \bar p\in \Z}}\langle g, \check{\varphi}_{{am \over 2}- k, \ell}^{v, \bar p}  \rangle  \check{\varphi}_{{am \over 2}- k, \ell}^{v, \bar p} .
\eeq

With the above procedure completed, $T_{m, k}$ becomes\footnote{\label{footnote54} For notational simplicity, throughout the paper, given $v$ a real parameter and $A>0$ we write $v\sim A$ in order to designate that there exists an absolute constant $c>1$ (allowed to change from line to line) such that $ |v|\in [\frac{A}{c}, A c]$. If $v$ is discrete (hence $v\in\Z$) the interpretation of the above notation is simply $ |v|\in [\frac{A}{c}, A c]\cap \N$. Also we write $v\sim_a A$ if the constant $c$ above depends on $a$ and it is in our intention to stress this dependence.}
\begin{equation}
\begin{aligned}\label{tmk}
& T_{m, k}(f, g)(x) = \sum_{\ell \in \Z} \sum_{n, \bar p \in \Z} \sum_{u, v \sim 2^{ am \over 2}} \langle f, \check{\varphi}_{{am \over 2}- k, \ell}^{u-v, \bar p-n}  \rangle   \langle g, \check{\varphi}_{{am \over 2}- k, \ell}^{u+v, \bar p+n}  \rangle   \check{\varphi}_{{am \over 2}- k, 2\ell-1}^{2u, \bar p}(x) \\
& \qquad \cdot\frac{\rho_{am-ak}( \lambda(x))}{2^{am +2 k \over 4}} w_{k, n, v}^e(\l)(x) dx,
\end{aligned}
\end{equation}
where
\begin{equation}
\label{def:eq:weight-oscill0}
w_{k, n, v}^e (\lambda)(x):=\int_{\R} \psi( \xi + 2v ) e^{i n \xi} e^{i c_a 2^{ ({am  \over 2}- k) \frac{a}{a-1}} \xi^{\frac{a}{a-1}} \lambda(x)^{- {1 \over {a-1}}}} \psi\Big(\frac{\xi}{2^{\frac{am}{2}}}\Big) d \xi
\end{equation}
is regarded as a suitable weight whose behavior will play an important role in our subsequent analysis.

\subsubsection{The phase-linearized wave-packet model}

The previous reasonings provide us naturally with a wave-packet discretized model of our operator. However, in order to further simplify it, we have to consider the \emph{time-frequency correlations} arising among the parameters $u, v, n, \bar{p}$ and $\lambda(x)$ appearing in \eqref{tmk}.  A first expression of this phenomenon is revealed by the estimate\footnote{Here $\bar c_a\in\R$ is a suitable constant depending only on $a$.}
\begin{equation}
\label{eq:decay-weight-intro}
\big|  \rho_{am-ak}( \lambda(x)) w_{k, n, v}^e (\lambda)(x) \big| \lesssim w_{k, n, v} (\lambda)(x):=
  \frac{\rho_{am-ak}( \lambda(x))}{ 1+  \Big|\bar c_a \frac{n^{a-1}  \l(x)}{2^{\frac{a(am-2k)}{2}}} + 2v \Big|^2},
\end{equation}
which is the result of the non-stationary phase principle applied to \eqref{def:eq:weight-oscill0}. Thus, as a consequence of  \eqref{eq:decay-weight-intro} we can reduce the range of $n$ to $n \sim_a 2^{am \over 2}$. Further on, applying another time-frequency correlation reasoning this time relying on almost orthogonality arguments, we deduce that the range of $\bar p$ should also be restricted to intervals of length $\sim_a 2^{am \over 2}$. Thus, replacing $\bar p$ by the parameter $p_r$ with
\[
p_r= 2^{am \over 2} (r-1)+p, \quad \text{with} \quad r \in \Z \quad     \text{and} \quad 2^{am \over 2} \leq p < 2 \cdot 2^{am \over 2},
\]
and dualizing our expression in \eqref{tmk}, we arrive at the final form of our wave-packet discretized model:
\beq\label{lmk0}
\underline{\L}_{m}(f, g, h):=\sum_{k\in\Z} \underline{\L}_{m,k}(f, g, h),
\eeq
where
\beq\label{lmk}
\underline{\L}_{m,k}(f, g, h):=\sum_{r,\ell\in\Z} \underline{\L}_{m,k,\ell,r}(f, g, h) ,
\eeq
and
\begin{equation}
\begin{aligned}\label{eq:modeltilee}
\underline{\L}_{m,k,\ell,r}(f, g, h) &:= \sum_{n, p \sim 2^{am \over 2}} \sum_{u, v \sim 2^{ am \over 2}} \frac{1}{2^{am +2 k \over 4}} \langle f, \check{\varphi}_{{am \over 2}- k, \ell}^{u-v, p_r-n}  \rangle   \langle g, \check{\varphi}_{{am \over 2}- k, \ell}^{u+v, p_r+n}  \rangle  \\
& \hspace{1em} \cdot \int_{\R}  \check{\varphi}_{{am \over 2}- k, 2\ell-1}^{2u, p_r}(x)  h(x) \rho_{am-ak}( \lambda(x)) w_{k, n, v}^e(\l)(x) dx.
\end{aligned}
\end{equation}

Equivalently, one can rewrite \eqref{lmk0} (and hence \eqref{eq:modeltilee}) in a tile format using the dictionary developed in Section \ref{tfp}. Indeed, given a tri-tile $P=(P_1,P_2,P_3)=P(k,\ell,r)\in\BHT$ verifying \eqref{eq:tile:1}-\eqref{eq:tile:3} we can now set
\begin{equation}\label{Ptile}
\underline{\L}_{m,P}:=\underline{\L}_{m,k,\ell,r}\;.
\end{equation}
Furthermore for every such tri-tile $P$ there exists a unique $1$-to-$1$ map between $\S(P)$ and the collection of parameters $(u,v,n,p)$ along which the sum in \eqref{eq:modeltilee} is computed.\footnote{See \eqref{eq:def:time:frequency:regions} for a precise description of it.} Reparameterizing the expression in \eqref{eq:modeltilee} according to this dictionary one recovers \eqref{tilemodel} with
\begin{equation}\label{s12}
\phi_{s_1}:=\check{\varphi}_{{am \over 2}- k, \ell}^{u-v, p_r-n} ,\:\:\:\:\:\:\:\:\phi_{s_2}:=\check{\varphi}_{{am \over 2}- k, \ell}^{u+v, p_r+n}
\end{equation}
and
\begin{equation}\label{s3}
\phi_{s_3}:=\check{\varphi}_{{am \over 2}- k, 2\ell-1}^{2u, p_r}\:\:\:\textrm{and}\:\:\:\underline{\phi}_{s_3}:=\phi_{s_3}  \rho_{am-ak}( \lambda)  w_{k, n, v}^e(\l)
\end{equation}
thus obtaining the tile formulation of \eqref{lmk0} in the form
\begin{equation}
\begin{aligned}\label{lmk0tileform}
\underline{\L}_{m}(f, g, h)&=\sum_{P\in \BHT} \underline{\L}_{P}(f, g, h) \\
&=\sum_{P\in \BHT} \frac{1}{2^{\frac{am}{4}}\,|I_P|^{1 \over 2}} \sum_{s\in \S(P)} \langle f, \phi_{s_1} \rangle  \langle g, \phi_{s_2}  \rangle \langle h, \underline{\phi}_{s_3} \rangle .
\end{aligned}
\end{equation}

Once at this point we recall that our main objective is to prove \eqref{mainrm}. As a starting key step in achieving it one is required to obtain the single tile analogue of \eqref{mainrm} hence an $m-$decaying control over $\underline{\L}_{m,P}$.

The first natural attempt in fulfilling this task is to investigate the absolute summable behavior of \eqref{tilemodel} via a more refined exploitation of \eqref{eq:decay-weight-intro} which implies that for fixed $n \sim 2^{am \over 2}$ and $x$ so that $\lambda(x) \sim 2^{am-ak}$ the main contribution comes from those $v \sim 2^{am \over 2}$ satisfying
\begin{equation}\label{tfc}
2v \approx \bar c_a \frac{n^{a-1}  \l(x)}{2^{\frac{a(am-2k)}{2}}}\,.
\end{equation}
In particular, \eqref{tfc} reduces the degrees of freedom of the family $\S(P)$ by one. Using the latter, standard almost-orthogonality arguments imply the estimate
\begin{equation}
\label{tilemodelabs}
\begin{aligned}
|\underline{\L}_{m,P}(f,g,h)| & \lesssim \frac{1}{2^{\frac{am}{2}}} \sum_{s\in \S(P)} |I_s|^{-\frac{1}{2}} |\langle f, \phi_{s_1} \rangle|  |\langle g, \phi_{s_2}  \rangle| |\langle h, {\underline{\phi}}_{s_3} \rangle| \\
& \lesssim \|f\|_{L^p} \|g\|_{L^q} \|h\|_{L^{r'}}.
\end{aligned}
\end{equation}

However, as previously noted, our goal is to prove such an estimate with an additional $2^{-\d a m}$ decay -- without which the entire strategy outlined in Section \ref{Heur}, in particular the statement in \eqref{mainrm}, fails. As it turns out, see Appendix \ref{sec:counterexample}, one can in fact show that the second inequality in \eqref{tilemodelabs} is \emph{sharp}.\footnote{If \eqref{tfc} were independent of or depend smoothly on the spatial variable $x$, one would have been able to obtain the desired $m$-decaying bound in \eqref{tilemodelabs}.}

This reveals that the decomposition in  \eqref{lmk0tileform} does not represent an \emph{$m$-decaying absolute summable model}, but -- as we will see -- only an \emph{$m$-decaying conditional summable model}.\footnote{One can envision here a parallelism with the well known counterexample regarding the lack of absolute summability for $r<\frac{2}{3}$ of the bilinear Hilbert transform model employed in \cite{lt2}, though of course this latter phenomenon is significantly more subtle. In particular, not only that the BHT model lacks absolute summability, but moreover, as a consequence of a randomization process, it diverges in the $\ell^2$-sense (see \cite{Lacey}).}

This brings us to another novelty of our paper -- that of operating with and extracting cancellation from a discretization that is only \emph{conditionally} well behaved. The latter requires the introduction of a second model operator that is regrouping the information carried by the first model above.

\subsubsection{The continuous phase-linearized spatial model}

The lack of the $m$-decaying absolute summability of \eqref{lmk0tileform} reveals that we can not afford to ignore the signum of the Gabor/local Fourier coefficients associated to the input functions. As a consequence, we need to find a way to reshape the information in the above model such that the newly formed terms capture enough cancellation within themselves in order for the resulting model to become an $m$-decaying absolute summable model.\footnote{For a revelatory discussion on this issue the reader is invited to consult Section \ref{motint}.}

The first step in accomplishing the above is to rewrite  $\underline{\L}_{m,P}=\underline{\L}_{m,k, \ell,r}$ as follows\footnote{
Notice that it is at this point in our argument where we break the symmetry between the inputs $(f,g)$ and $h$, recasting the trilinear behavior of the form $\underline{\L}_{m,k,\ell,r}(f, g, h)$ into an examination of the corresponding bilinear behavior of the operator $S_{k,\ell,r}^{n, v, p}(f, g)$. Later however, see Section \ref{HR20} below, we will need to make up for this asymmetry by a more involved analysis that will end up reinstating the role played by $h$.}:
\begin{equation}
\begin{aligned}
\label{eq:trilinear:form:withS}
& \underline{\L}_{m,k, \ell,r}(f, g, h) \\
& \hspace{1em} = \sum_{n, p \sim 2^{am \over 2}} \sum_{ v \sim 2^{ am \over 2}} \int_{\R} S_{k,\ell,r}^{n, v, p}(f, g)(x) h(x) \rho_{am-ak}( \lambda(x)) w_{k, n, v}^e(\l)(x) dx,
\end{aligned}
\end{equation}
with
\begin{equation}
\begin{aligned}
\label{eq:def:S_n,v,p-intro}
& S_{k,\ell,r}^{n, v, p}(f, g)(x)\\
& \hspace{1em} := \sum_{u \sim 2^{ am \over 2}} \frac{1}{2^{am +2 k \over 4}} \langle f_{\ell +1} , \check{\varphi}_{{am \over 2}- k, \ell}^{u-v, p_r-n}  \rangle   \langle g_{\ell -1} , \check{\varphi}_{{am \over 2}- k, \ell}^{u+v, p_r+n}  \rangle \check{\varphi}_{{am \over 2}- k, 2\ell-1}^{2u, p_r}(x) .
\end{aligned}
\end{equation}

The key insight now is provided by the fact that, at the heuristic level, one has for every $x \in I_k^{p_r}$
\begin{equation}
\begin{aligned}\label{heurist}
S_{k,\ell,r}^{n, v, p}(f, g)(x)& \approx \tilde{S}_{k}^{n, v}(f_{\ell+1}, g_{\ell-1})(x) \\
& \hspace{1em} :=\frac{e^{2inv}}{2^k} \int_{I_k^n} f_{\ell +1}(x-t) g_{\ell -1}(x+t)  e^{-i 2^{{am \over 2}-k} 2v t} \,dt ,
\end{aligned}
\end{equation}
where here we used that $I_k^n:= \Big[  \frac{n}{2^{{am \over 2}-k}},    \frac{n+1}{2^{{am \over 2}-k}}  \Big]$ and the convention made for the complex exponential in Footnote \ref{footnote:2pi}.

As a result, combining \eqref{eq:def:S_n,v,p-intro} and \eqref{heurist}, we obtain a new model (equivalent to \eqref{lmk0tileform}) that is based   on the spatial discretization of the input functions in a shape that resembles more the original form of our operator $BC^a$:

\begin{equation}\label{contimod}
\begin{aligned}
&\underline{\L}_{m,k}(f, g, h)=  \sum_{r,\ell\in\Z} \underline{\L}_{m,k,\ell,r}(f, g, h)\\
&\approx \sum_{\ell,r \in \Z} \sum_{n,p \sim 2^\frac{am}{2}} \sum_{v \sim 2^\frac{am}{2}} \int_{\R} \tilde{S}_{k}^{n, v}(f_{\ell+1}, g_{\ell-1})(x)\, \one_{I_k^{p_r}} (x)\, \rho_{am-ak}(\l(x)) \\
& \hspace{21em} \cdot w_{k, n,v}^e(\l)(x)\,h(x)\, dx.
\end{aligned}
\end{equation}

As it turns out, the above represents an $m$-decaying absolute summable model -- an essential feature that will be exploited in the next section.

\subsubsection{The LGC-methodology: finalization}

This component of our proof represents the last, key moment where we will still exploit the phase's curvature: the actual extraction of the cancellation will be operated on the continuous phase-linearized spatial model \eqref{contimod} while once this is fulfilled it will be transferred to (a slight variant of) the phase-linearized wave-packet model \eqref{lmk0tileform} that will be needed when implementing the low-resolution multi-scale analysis -- see Sections \ref{impas} and \ref{LRMS}.

Turning now our focus on \eqref{contimod}, we remark that the presence of the weight function \eqref{def:eq:weight-oscill0} in the expression of $\underline{\L}_{m,k,\ell,r}(f, g, h)$  breaks the symmetry between the treatment of, on the one hand, $f, g$ and, on the other hand, $h$. This explains why in this first stage of the single scale estimate the time-frequency analysis of $h$ plays no significant role and also why we need to appeal to a transitional step -- see Sections \ref{impas} and \ref{HR20} -- in which we refine and symmetrize the single scale estimate in order to prepare for the low-resolution multi-scale analysis from Section \ref{LRMS}.

Combining now \eqref{eq:trilinear:form:withS}-\eqref{heurist}, one deduces
\begin{equation*}
|\underline{\L}_{m,k, \ell,r}(f, g, h)|\lesssim \Lambda_{m,k, \ell,r}(f,g) \|h\|_{L^2} ,
\end{equation*}
with
\begin{align*}
& \Lambda_{m,k, \ell,r}(f,g)
\\
& \qquad :=\Big(\sum_{p\sim 2^{\frac{a m}{2}}}\int_{I_k^{p_r}} \Big(\sum_{n\sim 2^{\frac{a m}{2}}}
\big|\sum_{ v \sim 2^{ am \over 2}} \tilde{S}_{k}^{n, v}(f_{\ell+1}, g_{\ell-1})(x)  w_{k, n, v}^e(\l)(x)\big|\Big)^2 dx\Big)^{\frac{1}{2}}
\end{align*}
being controlled by an expression that is essentially of the form
$$\Big(\sum_{p\sim 2^{\frac{a m}{2}}}\int_{I_k^{p_r}} \Big(\sum_{n\sim 2^{\frac{a m}{2}}}
\big|2^{-k} \int_{I_k^n} f_{\ell+1}(x-t) g_{\ell-1}(x+t) e^{i \frac{\bar c_a n^{a-1}}{2^{\frac{a(a-1) m}{2}}} \l(x) t} dt\big|\Big)^2 dx\Big)^{\frac{1}{2}}\:.$$

Obtaining an $m$-decaying upper bound for the above expression constitutes the most subtle argument in our paper: after a discussion on the size distribution of the input functions $f_{\ell +1}$ and $g_{\ell -1}$, we invoke a double $TT^{*}$-argument that reduces our task to a problem involving both elements of analytic number theory -- via Weyl sums/discrete Van der Corput, and elements of harmonic analysis via phase level set analysis. This approach discussed in Section  \ref{sec:onescale} -- see Proposition \ref{sgscalelambda} therein --
will allow us to extract the cancellation offered by the phase curvature and achieve the $m$-exponential decay single scale estimate:
\beq\label{mainrmm}
|\Lambda_{m,k, \ell,r}(f,g)|\lesssim_{a} 2^{-\delta_0 a m}  2^{\frac{k}{2}}  \|f_{\ell+1} \cdot \ci_{[r 2^k, (r+1) 2^k]}\|_{L^2} \|g_{\ell-1} \cdot \ci_{[r 2^k, (r+1) 2^k]}\|_{L^2}\:,
\eeq
for some $\delta_0>0$.

\subsection{A symmetry conflict and the necessity of separating the behaviors of $B$ and $C^a$ within $BC^{a}$}\label{impas}

As already mentioned, a recurrent challenge throughout our proof is the ``antagonistic behavior of the coexistent non-zero and zero curvature features of our operator''. While the design of a two-resolution approach provides us with a key insight in addressing this issue, it requires us to put together
\begin{itemize}
\item the high-resolution, non-zero curvature methods induced by the Carleson-type component $C^a$, and

\item the low resolution, time-frequency methods induced by the bilinear Hilbert component $B$.
 \end{itemize}

Thus, the initial conflict is now transferred into:
\begin{itemize}
\item On the one hand, the structure of our operator $T_{m}$ has a special type of asymmetry induced by the presence of the maximal modulation -- inherited from $C^a$ -- that in particular gives rise to the linearized function $\l(x)$. This also explains the statement and the proof of the single-scale estimate \eqref{mainrmm} which in particular does not provide a well-localized time-frequency information on the dualizing function $h$.\footnote{Later on it will become transparent that this asymmetry may be explained by the fact that the desired $m$-decay in the single scale estimate can only be obtained by exploiting the behavior of $\rho_{am-ak}(\lambda(x)) w_{k, n,v}^e(\lambda)$ -- see \eqref{def:eq:weight-oscill0} and \eqref{eq:decay-weight-intro}.}

\item On the other hand, the proof of the boundedness of the bilinear Hilbert transform exploits -- via its dualized form -- a symmetric role of the  time-frequency information carried by the input functions $f, g$ and $h$.\footnote{This is transparent in the ``size-energy'' formulation stated below in Proposition \ref{prop:size/energy0}, Section \ref{LRMS}. }
\end{itemize}

The appeasement of this conflict is obtained through a ``half-way'' compromise:

\begin{itemize}

\item  The first action, from the $T_m$ side, is to substantially refine \eqref{mainrmm}, in particular increasing its symmetry by also producing the time-frequency localization for $h \cdot w_{k, n, v}^e (\lambda)$.

\item The second action comes from the direction of the BHT analysis and relies on modifying/adapting the original concepts of size and energy to the presence of the weight $w_{k, n, v}^e (\lambda)$ and, consequently, to the distinct roles played by $(f,g)$ and $h$  in our problem.
\end{itemize}

Now the first action is part of the transitional step to be discussed in the next section, while the second action is part of the multi-scale time-frequency approach that will be the subject of Section \ref{LRMS}.

We end this section by providing another angle on the above mentioned conflict that follows from inspecting the structure of $BC^a$: on the multiplier side -- see \emph{e.g.} \eqref{multm} -- there is a merging between the bilinear Hilbert transform behavior focusing on the Whitney decomposition relative to the $\xi-\eta=0$ ``singularity'' and the Carleson-type behavior encapsulated in the expression $ e^{ i c  (\xi-\eta)^{\frac{a}{a-1}}\l(x)^{-\frac{1}{a-1}}}$. These two operators end up on the multiplier side by acting on the same variable $\xi-\eta$. Equivalently, on the physical space side, there is an intertwinement between the structure of the input functions -- $f(x-t) g(x+t)$ -- and that corresponding to the kernel $\frac{e^{i \l(x) t^a}}{t}$, respectively. This phenomenon makes more difficult the process of identifying a way to decouple/untwist the behaviors of $B$ and $C^a$ within $BC^a$, which plays a key role in our approach.\footnote{In contrast with this -- see the discussion within Section \ref{Otherdirections}, II.4 -- the decoupling process within the analysis of the Bi-Carleson operator \eqref{def:BiCarleson:series} is easier since in this latter situation the BHT singularity and the Carleson-type singularity do not act on the same variable.}

\subsection{(II) A transitional step: refinement of the single-scale estimate}\label{HR20}

The main task here is to obtain a trilinear dualized version of \eqref{mainrmm} with a sharp time-frequency localization of the input functions $f,g$ and $h$ that simultaneously quantifies the presence of the linearizing function $\l$ and separates the information carried by the input functions $f,g$ from the information carried by $h$ and $ w_{k, n, v}^e(\l)$.

The first key insight here is to notice that at a heuristic level one can pretend that
\beq\label{constancy}
\underline{S}_{k,\ell,r}^{n, v, p'}(f, g)(x):=S_{k,\ell,r}^{n, v, p'}(f, g)(x) e^{-i 2\ell 2^{am \over 2}(2^{{am\over 2} -k} x -p_r)}\:\:\textrm{is constant on}\:\underline{I}_{k,r}^{p'}\:,
\eeq
where here we used the notation $\ds \underline{I}_{k,r}^{p'}:=\Big[  2^k(r-1) + \frac{p'}{2^{am-k}}, 2^k(r-1) + \frac{p'+1}{2^{am -k}}   \Big]$.

As a consequence, for any positive $w_{k,n,v}\in L^1_{loc}(\R)$, one may think that
\begin{equation}
\begin{aligned}
\label{eq:time-freq-localization1}
\int_{\R} \underline{S}_{k,\ell,r}^{n, v, p'}(f_{\ell +1}, & g_{\ell -1})(x) e^{i 2\ell 2^{am \over 2}(2^{{am\over 2} -k} x -p_r)}   h(x) \rho_{am-ak}( \lambda(x)) w_{k, n, v}^e(\l)(x) dx \\
&\approx 2^{-\frac{am-k}{2}} \big\langle h   \rho_{am-ak}(\lambda) w_{k, n, v}^e(\l), \check{\varphi}_{am-k,2\ell-1}^{0,2^{am}(r-1)+p'}  \big\rangle \\
& \hspace{5em} \cdot  \frac{\int_{\underline{I}_{k,r}^{p'}} \underline{S}_{k,\ell,r}^{n, v, p'}(f_{\ell +1}, g_{\ell -1})(x)  w_{k,n,v}(x)dx}{ \int_{\underline{I}_{k,r}^{p'}}  w_{k,n,v}(x)dx} .
\end{aligned}
\end{equation}
After a Cauchy-Schwarz argument, one then concludes that
\begin{equation}
\begin{aligned}
\label{eq:est:w:aver1}
&\big |\underline{\L}_{m,k,\ell,r}(f, g, h) \big| \\
& \hspace{1em}\lesssim  \sum_{n, v \sim 2^{am \over 2}} \sum_{p' \sim 2^{am}}   \frac{ \big| \langle h   \rho_{am-ak}(\lambda(\cdot)) w_{k, n, v}^e(\l), \check{\varphi}_{am-k,2\ell-1}^{0,2^{am}(r-1)+p'}  \rangle \big|}{\big( \aver{\underline{I}_{k, r}^{p'}}  w_{k,n,v}(x)  dx \big)^\frac{1}{2} } \\
&\hspace{2em} \cdot \left( \int_{\underline{I}_{k, r}^{p'}} \big|\underline{S}_{k,\ell,r}^{n, v, p'}(f_{\ell +1}, g_{\ell -1})(x)\big|^2  w_{k,n,v}(x)dx \right)^\frac{1}{2}.
\end{aligned}
\end{equation}

More precisely, \eqref{eq:time-freq-localization1} and \eqref{eq:est:w:aver1} have to be understood on the right hand side as a superposition of similar terms with fast decaying amplitudes -- that will be made more rigorous in Section \ref{sec:HR2}.

As a quick remark, the liberty of choosing any positive $w_{k,n,v}\in L^1_{loc}(\R)$ in \eqref{eq:est:w:aver1} will offer us great flexibility later when analyzing the behavior of each of the terms involved in decomposition \eqref{decsgsc0} below.

Once at this point we will need to consider three key parameters\footnote{Recall relation \eqref{eq:decay-weight-intro}.}:
\begin{itemize}
\item the ``mass'' of the weight $w_{k, n, v}(\l)$, roughly defined as $\aver{\underline{I}_{k, r}^{p'}} w_{k, n, v}(\lambda)(x)   dx$ and encoding the behavior of the linearizing function $\lambda$;
\item the size distribution of the input function $f$; and
\item the size distribution of the input function $g$.
\end{itemize}

The behavior of these parameters will determine the partitioning of the set of indices
$$\I:=\Big\{ (p', n, v) |  p' \sim 2^{am}, n, v \sim 2^{am \over 2} \Big\}$$ into three components
$$\I=\mathfrak{l}\cup \mathfrak{u}\cup\mathfrak{c}$$ with $\mathfrak{l}$ the light mass set, $\mathfrak{u}$ the heavy mass-uniform input distribution set and $\mathfrak{c}$ the heavy mass-clustered input distribution set. This will reflect into the corresponding splitting
\begin{align}
\underline{\L}_{m,k,\ell,r}(f,g,h) & = \underline{\L}_{m,P}(f,g,h) \nonumber \\
& =  \underline{\L}_{m,P}^{\mathfrak{l}}(f,g,h) + \underline{\L}_{m,P}^{\mathfrak{u}}(f,g,h) + \underline{\L}_{m,P}^{\mathfrak{c}}(f,g,h)\:, \label{decsgsc0}
\end{align}
with each of the above forms analyzed accordingly:
\begin{itemize}
\item the light mass component $\underline{\L}_{m,P}^{\mathfrak{l}}(f,g,h)$ is composed of the terms that correspond to the indices where the associated mass of $w_{k, n, v}(\lambda)$ is small;
%encapsulates
\item the heavy mass-uniform input component $ \underline{\L}_{m,P}^{\mathfrak{u}}(f,g,h)$ addresses precisely the situation that captures the cancellation offered by the non-zero curvature and relies in a key fashion on the estimate \eqref{mainrmm}; and

\item the heavy mass-clustered input $\underline{\L}_{m,P}^{\mathfrak{c}}(f,g,h)$ captures the situation when the size of the input functions $f, g$ is larger than the expected average value.
\end{itemize}

The conclusion of this analysis will be provided by Proposition \ref{refinedsgscale}, Section \ref{sec:bilinear:analysis}, which states that
\begin{equation}
\label{eq:esti:L^P:light:21}
\underline{\L}_{m,P}^{\mathfrak{l}}(f, g, h) \lesssim |I_P|^{-\frac{1}{2}} f(P_1) g(P_2) h_{\mathfrak{l}}(P_3),
\end{equation}
\begin{equation}
\label{eq:esti:L^P:uniform:21}
\underline{\L}_{m,P}^{\mathfrak{u}}(f, g, h) \lesssim 2^{- \frac{\delta_1 a m}{2}} |I_P|^{-\frac{1}{2}} f(P_1) g(P_2) h_{\mathfrak{u}}(P_3),
\end{equation}
for some $\delta_1>0$, and
\begin{equation}
\label{eq:esti:L^P:cluster:21}
\underline{\L}_{m,P}^{\mathfrak{c}}(f, g, h) \lesssim |I_P|^{-\frac{1}{2}} f(P_1) g(P_2) h_{\mathfrak{c}}(P_3) ,
\end{equation}
where here $f(P_1)$ and $g(P_2)$ represent local $\ell^2$-sums involving the Fourier coefficients of the corresponding function adapted to the family $\{s\}_{s\in\S(\P)}$ while $h_{*}(P_3)$ quantifies a suitable weighted version of similar local $\ell^2$-sums associated to $h$ and restricted to $*\in\{\mathfrak{l},\mathfrak{u},\mathfrak{c}\}$ and involving both the function $h$ itself and the weight $w_{\cdot,\cdot,\cdot}(\lambda)$. The decay obtained in \eqref{eq:esti:L^P:uniform:21} is a direct manifestation of the non-zero curvature analysis performed at stage \ref{step:high}, while for \eqref{eq:esti:L^P:light:21} and \eqref{eq:esti:L^P:cluster:21} the $m$-exponential decay is hidden in the structure of $h_{*}(P_3)$.

This puts an end to the role played by the non-zero curvature in our problem and concurrently to the first part of our proof.

\subsection{(III) The low resolution, multi-scale analysis}\label{LRMS}

In this final part of the proof's exposition we address the underlying modulation invariance structure of our problem encoded in \ref{key:m}. As already explained in Section \ref{TRAI}, our plan is to appeal to the time-frequency techniques originating in the work of Lacey and Thiele as modeled by the following general abstract result\footnote{For the definitions of mass and energy (and other related concepts) please see Section \ref{sec:bilinear:analysis}.}\footnote{This general abstract result stated above was necessary when treating objects having a structure more complex than the standard BHT as revealed by the following examples: 1) the Bi-est operator regarded as BHT composed with another BHT, or, 2) the Bi-Carleson operator $Bi\text{-}C$ already mentioned in Section \ref{BiFS} (see \eqref{def:BiCarleson:series}) that can be seen as a composition between BHT and the Carleson operator.}:

\begin{proposition}[Proposition 6.5 of \cite{MTTBiest2}] \label{prop:size/energy0}
Let $\P \subset \BHT$  be a finite collection of tri-tiles and $(a_{P_j})_{P \in \P}$ a sequence of complex numbers. Then for any $\theta_1,\theta_2,\theta_3\in[0,1)$ with $\theta_1+\theta_2+\theta_3=1$
\begin{align*}
\big|\sum_{P \in \P} \frac{1}{|I_P|^{\frac{1}{2}}} a_{P_1} a_{P_2} a_{P_3}  \big| \lesssim  \prod_{j=1}^3 \big(   \size^{\P}_j((a_{P_j})_{P \in \P}) \big)^{\theta_j} \big(   \energy^{\P}_j((a_{P_j})_{P \in \P}) \big)^{1-\theta_j}.
\end{align*}
\end{proposition}

Defining now $\underline{\L}^{*}_{m}(f,g,h):=\sum_{P\in\BHT} \underline{\L}_{m,P}^{*}(f, g, h)$ we insert
 \eqref{eq:esti:L^P:light:21}-\eqref{eq:esti:L^P:cluster:21} into the above proposition -- with the latter used for the most part as a black-box -- where for $*\in\{\mathfrak{l},\mathfrak{u},\mathfrak{c}\}$ and $P\in\BHT$ we set $a_{P_1}:= f(P_1)$, $a_{P_2}:=g(P_2)$ and $a_{P_3}:= h_{*}(P_3)$.

In order to obtain though \eqref{mainrm}, we have to further address the symmetry conflict exposed in Section \ref{impas} and, in particular, the different scaling manifested in the ``two-resolution analysis'': the above proposition involves the $2^{am}$-area tiles part of the definition of tri-tiles $P$  while the structure of the coefficients $f(P_1)$, $g(P_2)$ and $h_{*}(P_3)$ depends on the interaction among area one tiles part of the tri-tile family $\S(P)$.

In answering the above we will have to re-work the classical energy and size estimates and adapt them to our context. In particular, due to the high-resolution discretization -- see Sections \ref{highr} and \ref{HR20} -- the sizes for $f$ and $g$ will be local $L^2$ averages, with the associated energy still remaining an $L^2$-based quantity; for $h$, we will need to consider different sizes that also take into account the information carried by $w_{k, n, v}(\lambda)$, each developing from the light, uniform, or clustered components.
%there will be

The key step for obtaining our desired $m$-decay is revealed in Proposition \ref{prop:aim} of Section \ref{sec:bilinear:analysis} whose proof goes through estimating size and energy quantities and can be summarized as follows: For any subcollection $\P\subset \BHT$ we have
\begin{itemize}
\item $L^2$ control over the size of the uniform component $\{h_\mathfrak{u}(P_3)\}_{P\in\P}$; \textit{i.e.},
\begin{equation}
\label{eq:est:size-h+0}
\size^{\P}_3(\{h_\mathfrak{u}(P_3)\}_{P\in\P}) \lesssim  \sup_{P \in \P} \frac{1}{|I_P|^{\frac{1}{2}}} \|h \cdot \ci_{I_P}\|_{L^2} .
\end{equation}
The simple but relevant observation here is that due to the presence of $\l$, the orthogonality among the involved terms is a consequence of the disjointness of the scales.

\item $m$-decay within the $L^2$ control over the size of the light component $\{h_\mathfrak{l}(P_3)\}_{P\in\P}$; \textit{i.e.},
\begin{equation}
\label{eq:est:size-h-0}
\size^{\P}_3(\{h_\mathfrak{l}(P_3)\}_{P\in\P}) \lesssim 2^{- \frac{\delta_1   a m}{4}} \sup_{P \in \P} \frac{1}{|I_P|^{\frac{1}{2}}} \|h \cdot \ci_{I_P}\|_{L^2} .
\end{equation}
This is a direct consequence of the ``lightness'' of the weight $w_{\cdot,\cdot,\cdot}$.

\item $m$-decay within the $L^2$ control over the size of the clustered component $\{h_\mathfrak{c}(P_3)\}_{P\in\P}$, \textit{i.e.} for some $\mu>0$
\begin{equation}
\label{eq:est:size-h:cluster0}
\size^{\P}_3(\{h_\mathfrak{c}(P_3)\}_{P\in\P}) \lesssim 2^{- \frac{\mu   a m}{4}} \sup_{P \in \P} \frac{1}{|I_P|^{\frac{1}{2}}} \|h \cdot \ci_{I_P}\|_{L^2} .
\end{equation}
Here one relies on the fact that there are at most $O(2^{(1-\mu) \frac{am}{2}})$ values of $n$ for which there exists $v \sim 2^{am \over 2}$ such that $(p', n, v) \in \mathfrak{c}$, where $\mu$ is a numerical constant chosen small enough.
\end{itemize}

Putting \eqref{eq:esti:L^P:light:21}-\eqref{eq:est:size-h:cluster0} all together, we obtain -- via the proposition above -- the desired estimate \eqref{mainrm} within the local-$L^2$ range. A standard interpolation argument, having as end-points our earlier obtained local-$L^2$ estimate and the result in \cite{Lacey} proves the result claimed by our Theorem \ref{main} in the interior of the full range -- see Section \ref{sec:bilinear:analysis}. We point out that in order to get the estimates corresponding to one of the exponents $p$ or $q$ in Theorem \ref{main} being $\infty$, we need to go through an extra interpolation step at the level of the spatially localized ``sizes'', see Subsection \ref{sec:size:interpolation}. 

%
%%%%%%%%%%%%%%%%%%%%%%%%%%%%%%%
%
% END OF INTRODUCTION/OVERVIEW
%
%%%%%%%%%%%%%%%%%%%%%%%%%%%%%%%
%
\section{Initial reductions}
\label{sec:initial:reductions}
In this section we will operate via various decompositions several simplifications on our bilinear Hilbert-Carleson operator $T:=BC^{a}$. Indeed, departing from the strategy outlined in Section \ref{Heur} -- recall here the key relations \eqref{stdec}--\eqref{thi} -- we will
address the following items:
\begin{itemize}[leftmargin=5.5mm]
\item the rigorous decomposition of $T$ into the sum of two components: the low oscillatory component $T_0$ and the high oscillatory component $T_{\mathrm{osc}}$;

\item the complete treatment of the low oscillatory component $T_0$;

\item the rigorous decomposition of the high oscillatory component into a sum of operators according to the level sets of the phase of the multiplier thus clarifying the decomposition $T_{\mathrm{osc}}=\sum_{m\in \N} T_{m}$ evoked in \eqref{thi};

\item the reduction to the local-$L^2$ case for the boundedness range of $T_m$;

\item the scale decomposition $T_{m}=\sum_{k\in\Z} T_{m,k}$ and the identification of the main component of the multiplier associated with each $T_{m,k}$.
\end{itemize}

\subsection{The decomposition $T=T_{0} + T_{\mathrm{osc}}$}

By a standard argument, we start by constructing a smooth non-negative even bump function $\rho$ supported in $[-2,-1/2]\cup [1/2,2]$ such that for every $t \neq 0$
\begin{equation}
\sum_{k \in \mathbb{Z}}\rho(2^{-k} t) = 1. \label{eq:smooth_dyadic_decomposition}
\end{equation}
Similarly, we can identify a smooth non-negative even bump function $\tilde{\rho} = \tilde{\rho}_a$ supported in $[-2^a,-2^{-a}]\cup [2^{-a},2^a]$ such that for every $\lambda\neq 0$
\[ \sum_{j \in \mathbb{Z}}\tilde{\rho}(2^{-aj} \lambda) = 1.   \]
Next we set
\[ \chi(\lambda,t) := \sum_{\substack{j,k \in \mathbb{Z} \\ j+k \leq 0}} \tilde{\rho}(2^{-aj} \lambda)\rho(2^{-k} t) \]
and notice that this function has the following properties: 1) $\chi(\lambda,t) \lesssim 1$ for all $\lambda, t \neq 0$, and 2) if $\chi(\lambda,t) \neq 0$ then $|\lambda t^a| \lesssim 1$.\\
Selecting now in the r\^{o}le of $\l$ the function $\l(x)$ arising from the linearization of the supremum in the definition of $BC^{a}$, we have that
\begin{equation*}
 BC^a(f,g)(x) =:T(f,g)(x)=\int f(x-t)g(x+t)e^{i \lambda(x) t^a}  \frac{dt}{t}
\end{equation*}
which decomposes as 
\begin{equation}
\begin{aligned}
& \int f(x-t)g(x+t)e^{i \lambda(x) t^a} \chi(\lambda(x), t) \frac{dt}{t} \\
& \hspace{1em} + \int f(x-t)g(x+t)e^{i \lambda(x) t^a} (1-\chi(\lambda(x),t)) \frac{dt}{t} \\
& =: T_0(f,g)(x) + T_{\mathrm{osc}}(f,g)(x),
\end{aligned} \label{oosc}
\end{equation}
thus completing the decomposition of our operator $T$ into a sum of low and high oscillatory components as stated in the introduction.

\subsection{The boundedness of $T_0$}

In this section we show that $T_{0}$ is bounded in the same range as the one corresponding to the bilinear Hilbert transform. This will be achieved by showing that
\beq\label{THM}
|T_{0}(f,g)| \lesssim H_{\ast}(f,g) + M(f,g) ,
\eeq
where here
\begin{itemize}
\item $H_{\ast}(f,g)$ stands for the maximal (smooth) truncation of the bilinear Hilbert transform
\begin{equation} H_{\ast}(f,g)(x):= \sup_{r > 0} \Big|\int f(x-t)g(x+t) \theta(t / r) \frac{dt}{t}\Big|, \label{Lmax}
\end{equation}
with $\theta$ any smooth bump function supported near the origin, and

\item $M(f,g)$ stands for the bilinear maximal function
\begin{equation}
M(f,g)(x):= \sup_{r > 0} \frac{1}{2r} \int_{-r}^{r} |f(x-t)||g(x+t)| dt .  \label{eq:bilinearmaxfunction}
\end{equation}
\end{itemize}
Notice that, indeed, if \eqref{THM} holds, then the claim about the boundedness range for $T_0$ follows from Lacey's work in \cite{Lacey} (see also \cite{DemeterTaoThiele} for further generalisations).

Turning now towards the proof of \eqref{THM}, we split $T_0$ as
\begin{align*}
T_0(f,g)(x) &= \int f(x-t)g(x+t) \chi(\lambda(x), t) \frac{dt}{t} \\
& \hspace{1em} + \int f(x-t)g(x+t)(e^{i \lambda(x) t^a}-1) \chi(\lambda(x), t) \frac{dt}{t}\\
&=:I + II .
\end{align*}
We now see that the first term on the right-hand side is bounded by $H_{\ast}(f,g)(x)$:
\begin{align*}
|I|&=\Big| \int f(x-t)g(x+t) \chi(\lambda(x), t) \frac{dt}{t}\Big| \\
&  = \Big|\sum_{j \in \mathbb{Z}} \tilde{\rho}(2^{-aj} \lambda(x))\; \int f(x-t)g(x+t) \sum_{k \leq -j} \rho(2^{-k}t)   \frac{dt}{t}\Big| \\
& = \Big| \sum_{j \in \mathbb{Z}} \tilde{\rho}(2^{-aj} \lambda(x))\; \int f(x-t)g(x+t) \theta(2^{j}t)   \frac{dt}{t} \Big| \\
& \leq \Big( \sum_{j \in \mathbb{Z}} \tilde{\rho}(2^{-aj} \lambda(x))\Big)   H_{\ast}(f,g)(x) = H_{\ast}(f,g)(x),
\end{align*}
where here we have let $\theta(t):= \sum_{k \leq 0} \rho(2^{-k}t)$ (a smooth bump function supported near the origin).

For the second term, we use the relations $|e^{is}-1|\leq |s|$ and $\chi(\lambda,t) \lesssim \mathbf{1}_{[-C,C]}(\lambda t^a)$ in order to deduce that
\begin{align*}
|II|&\lesssim \int_{\{t : |\lambda(x) t^a| \lesssim 1\}}|f(x-t)||g(x+t)| |\lambda(x) t^a| \frac{dt}{|t|}\\
&\lesssim \sum_{\ell \in \mathbb{N}} 2^{-\ell} \int_{\{t : |\lambda(x) t^a| \sim 2^{-\ell} \}} |f(x-t)||g(x+t)| \frac{dt}{|t|}\\
&\lesssim \sum_{\ell \in \mathbb{N}} 2^{-\ell} \frac{1}{2^{-\ell/a} |\lambda(x)|^{-1/a}}\int_{\{t : |t| \lesssim 2^{-\ell/a}|\lambda(x)|^{-1/a}\}} |f(x-t)||g(x+t)|  dt \\
&\lesssim M(f,g)(x).
\end{align*}
This ends the proof of \eqref{THM}.

\subsection{Level set partition of the multiplier's phase: the decomposition $T_{\mathrm{osc}}=\sum_{m\in \N} T_{m}$}

We start by splitting the operator into components in which both the spatial scale and the parameter\footnote{Since any bound we will obtain below will be independent of the measurable function $\lambda(\cdot)$, from now on we consider it fixed once and for all and omit it when convenient from the notation.} $\lambda$ are localized. Then we will reassemble the pieces according to the magnitude of the oscillatory phase -- it is in this latter parameter that we will obtain decay in our norm bounds, allowing us to sum all the contributions together.\par

Observe that by construction we have\footnote{Notice that $1 - \chi(\lambda,t)$ nicely factors into terms in which $\lambda$ is localized and terms in which $t$ is localized, this being the very reason for which we defined $\chi(\lambda,t)$ the way we did, instead of something simpler such as $\theta(\lambda t^a)$ with $\theta$ a bump function supported near the origin (which would have sufficed for the purposes of the previous subsection).}
\[ 1  -\chi(\lambda,t) = \sum_{\substack{j,k \in \mathbb{Z} \\ j+k > 0}} \tilde{\rho}(2^{-aj}\lambda)\rho(2^{-k}t), \]
and therefore we can decompose $T_{\mathrm{osc}}$ accordingly as
\[ T_{\mathrm{osc}}(f,g)(x) = \sum_{\substack{j,k \in \mathbb{Z} \\ j+k > 0}} \tilde{\rho}(2^{-aj}\lambda(x)) \int f(x-t)g(x+t) e^{i\lambda(x) t^a} \rho(2^{-k}t) \frac{dt}{t}. \]

Letting $m = k +j$ we see that we can collect the terms according to $m$, obtaining
\begin{align}
T_{\mathrm{osc}}(f,g)(x) &= \sum_{m \in \mathbb{N}} \sum_{k \in \mathbb{Z}} \tilde{\rho}(2^{-a(m-k)}\lambda(x)) \int f(x-t)g(x+t) e^{i\lambda(x) t^a} \rho(2^{-k}t) \frac{dt}{t} \nonumber \\
&=: \sum_{m \in \mathbb{N}} \sum_{k \in \mathbb{Z}} T_{m,k}(f,g)(x)=: \sum_{m \in \mathbb{N}} T_{m}(f,g)(x).  \label{eq:Tmk}
\end{align}
Observe that in each of the integrals above for a given $m$ we have $|\lambda(x)t^a| \sim 2^{am}$; in other words, we have localized the magnitude of the phase thus achieving the decomposition we desired.

\subsection{The boundedness range of $T_m$ -- reduction to the local-$L^2$ case}
\label{sec:reduction:local:L2}

 As explained in the introduction, the rest of the paper will be devoted to the study of
\[ T_{m}(f,g) := \sum_{k \in \mathbb{Z}} T_{m,k}(f,g) . \]
Given $m\in \N$, our final aim is to show that $T_m$ obeys $L^p \times L^q \to L^r$ bounds that decay exponentially in the parameter $m$.

More precisely, our main result will be the following:

\begin{theorem}\label{main_theorem_pxq_to_r}
For any exponents $p,q,r$ such that $1<p,q \leq \infty$, $2/3<r<\infty$ and $1/p + 1/q = 1/r$ there exists $\delta = \delta(p,q)>0$ such that for all $f \in L^p(\mathbb{R}), g \in L^q(\mathbb{R})$ and for all $m \in \mathbb{N}$
\begin{equation}\label{fullrange}
\|T_m (f,g)\|_{L^r} \lesssim_{a,p,q} 2^{-\delta a m} \|f\|_{L^p} \|g\|_{L^q}.
\end{equation}
Moreover, the implicit constant is independent of the fixed measurable function $\lambda$.
\end{theorem}

To begin with, we notice that each $T_m$ is controlled pointwise by the bilinear maximal function $M$ defined in \eqref{eq:bilinearmaxfunction}:
\begin{equation}
\begin{aligned}
|T_{m}(f,g)(x)| &= \Big|\sum_{k \in \mathbb{Z}} \tilde{\rho}(2^{-a(m-k)}\lambda(x)) \int f(x-t)g(x+t) e^{i\lambda(x) t^a} \rho(2^{-k}t) \frac{dt}{t}\Big| \\
&\leq \sum_{k \in \mathbb{Z}} \tilde{\rho}(2^{-a(m-k)}\lambda(x)) \int |f(x-t)||g(x+t)||\rho(2^{-k}t)| \frac{dt}{|t|} \\
&\lesssim \sum_{k \in \mathbb{Z}} \tilde{\rho}(2^{-a(m-k)}\lambda(x)) M(f,g)(x) \lesssim M(f,g)(x).
\end{aligned}\label{eq:T:m:cont:M}
\end{equation}
Hence the boundedness of $T_m$ follows from that of $M$; in particular, from \cite{Lacey}, $T_m$ is $L^p \times L^q \to L^r$ bounded for any $1< p, q \leq \infty$, $2/3 <r <\infty$ with $1/p+1/q=1/r$. Therefore we also have uniformly in $m$
\[ \|T_m(f,g)\|_{L^r} \lesssim_{a,p,q} \|f\|_{L^p} \|g\|_{L^q}. \]

Because of this and by means of interpolation,\footnote{See for example \cite{multilinear-Marcink-constant} for a precise statement about multilinear interpolation which keeps track of constants} identity \eqref{fullrange} for a certain tuple $(p_0, q_0, r_0)$ -- with $1<p_0, q_0<\infty, 2/3 < r_0< \infty$ -- easily extends across the full \emph{open} range of Theorem \ref{main_theorem_pxq_to_r}. This leaves out the boundary tuples $(p, \infty, p)$ and $(\infty, q, q)$, but these cases will also follow via (linear) interpolation, once \eqref{fullrange} is known to hold for some tuples $(p_0, \infty, p_0)$ or $(\infty, q_0, q_0)$.

In order to obtain the full range in Theorem \ref{main_theorem_pxq_to_r}, we proceed in the following manner:
\begin{itemize}[leftmargin=5.5mm]
\item we show that \eqref{fullrange} holds for all tuples $(p, q, r)$ belonging to the \emph{strict} local-$L^2$ range, i.e. such that $2<p, q, r'< \infty$;
\item interpolating this with the estimates we have for $M$, we obtain \eqref{fullrange} for all tuples $(p, q, r)$, except for endpoints corresponding to $p=\infty$ or $q=\infty$;
\item using a different type of interpolation -- this time at the level of localized maximal averages (see Section \ref{sec:size:interpolation}) -- we  prove that \eqref{fullrange} extends to all tuples $(p, q, r)$ with $2<p, q \leq \infty, 1<r<\infty$, $1/p+1/q=1/r$;
\item finally, we use interpolation on the boundary\footnote{We interpolate between $\|T_m\|_{L^{p_1} \times L^\infty \to L^{p_1}} \lesssim_{a, p_1} 2^{-\delta a m}$ for $p_1>2$ and $\|T_m\|_{L^{p_2} \times L^\infty \to L^{p_2}} \lesssim 1$ for $1<p_2<2$. This reduces to linear interpolation for the operator $f \mapsto T_m(f, g)$, where $g \in L^\infty$ is fixed.} of the range to deduce the full result of Theorem \ref{main_theorem_pxq_to_r}.
\end{itemize}

As a consequence of the above discussion, Theorem \ref{main_theorem_pxq_to_r} follows from a similar result for input functions in the local-$L^2$ range:

\begin{theorem}[\textsf{local-$L^2$ input functions}]\label{main_theorem_local_L2}
For any exponents $p,q,r$ such that $2<p,q \leq \infty, 1<r<\infty$ and $1/p + 1/q = 1/r$, there exists $\delta = \delta(p,q)>0$ such that for all $f \in L^p(\mathbb{R}), g \in L^q(\mathbb{R})$ and for all $m \in \mathbb{N}$
\begin{equation}\label{loc2}
\|T_m (f,g)\|_{L^r} \lesssim_{a,p,q} 2^{-\delta   a m} \|f\|_{L^p} \|g\|_{L^q},
\end{equation}
where the constant is independent of the fixed measurable function $\lambda$.
\end{theorem}

The central point of our paper will be the proof of the $m$-decaying estimates within the local-$L^2$ range.

\subsection{The main component of $T_{m,k}$: stationary phase analysis.}
\label{subsection_stationary_phase_analysis}

In this section we will focus on the analysis of the multiplier symbol associated with $T_{m,k}$ as represented by the formula
\[ T_{m,k}(f,g)(x) = \tilde{\rho}\Big(\frac{\lambda(x)}{2^{a(m-k)}}\Big) \iint \widehat{f}(\xi)\widehat{g}(\eta) e^{i (\xi + \eta) x} \mathfrak{M}_{m,k}(\xi - \eta;\lambda(x))  d\xi  d\eta , \]
where
\begin{equation} \mathfrak{M}_{m,k}(\zeta;\lambda) := \int e^{i \lambda t^a} e^{-i \zeta t}  \frac{\rho(2^{-k}t)}{t} dt. \label{bigM}
\end{equation}

Since $\mathfrak{M}_{m,k}$ is an oscillatory integral, we will make use of the stationary phase principle in order to decompose our operator $T_{m,k}$ into a principal component that will become the main object of interest for the remainder of the paper and an error term that will be controlled by the bilinear maximal function defined in \eqref{eq:bilinearmaxfunction}. In the present section we will only provide a heuristic argument for identifying the dominant term of $T_{m,k}$ (respectively $\mathfrak{M}_{m,k}$), while the detailed but largely standard analysis involving more laborious computations and reductions can be find in Appendix \ref{appendix:multiplier}.

These being said, we turn out attention towards \eqref{bigM} and notice that the phase of $\mathfrak{M}_{m,k}$, that is $\mu(t) := \lambda t^a - \zeta t$, has -- up to a sign -- a single critical real point at $|t_c| := \left|\zeta / a\lambda\right|^{1/(a-1)}$. A heuristic application of the stationary phase principle (see for example Ch. VIII of \cite{BigStein}) suggests that the main contribution to the oscillatory integral $\mathfrak{M}_{m,k}$ is given, up to an absolute constant, by
\[ e^{i \mu(t_c)} \frac{\rho(2^{-k} t_c)}{t_c (\mu''(t_c))^{1/2}}, \]
provided that $\zeta$ is such that $|t_c| \sim 2^k$. Expanding the definitions of $t_c$ and $\mu$ in the above formula and taking into account that $|t_c| \sim 2^k, \lambda \sim 2^{a(m-k)}$ and $|\zeta| \sim 2^{am-k}$, we necessarily have that the second factor above is of magnitude $\sim 2^{-am/2}$, and thus the main contribution to $\tilde{\rho}(\lambda 2^{-a(m-k)}) \mathfrak{M}_{m,k}(\zeta;\lambda)$ is morally of the form\footnote{Here one can compute exactly the value of $c_a$ as given by the formula $c_a := a^{-a/(a-1)}-a^{-1/(a-1)}$.}
\begin{equation}
\mathfrak{m}_{m,k}(\zeta;\lambda) := \frac{1}{2^{am/2}}   e^{i c_a \zeta^{a/(a-1)} \lambda^{-1/(a-1)}}  \rho\Big(\frac{\lambda}{2^{a(m-k)}}\Big)  \psi\Big(\frac{\zeta}{2^{am-k}}\Big) \label{smallM}
\end{equation}
for some bump functions $\rho(\lambda),\psi(\zeta)$ supported in $|\lambda| \sim 1$, $|\zeta| \sim 1$ respectively.\footnote{Notice that we are abusing notation once again: the latter $\rho$ is not necessarily the same one as that introduced in \eqref{eq:smooth_dyadic_decomposition}.} For later convenience, we will denote from now on $\rho_{am-ak}(\lambda) := \rho(\lambda 2^{-a(m-k)})$. \\

Based on the above heuristic (once again made precise in Appendix \ref{appendix:multiplier}) we have that Theorem \ref{main_theorem_local_L2} will follow from the following:
\begin{theorem}\label{mdec}
With $\mathfrak{m}_{m,k}$ defined by \eqref{smallM} we let (with a further abuse of notation)
\begin{equation}\label{decmk}
 T_{m,k}(f,g)(x):= \iint \widehat{f}(\xi)\widehat{g}(\eta) e^{i (\xi + \eta) x} \mathfrak{m}_{m,k}(\xi - \eta;\lambda(x))  d\xi  d\eta
 \eeq
and $T_{m}(f,g) := \sum_{k \in \mathbb{Z}} T_{m,k}(f,g)$.\\
For any $p,q,r$ such that $2<p,q \leq \infty, 1<r<\infty$ and $1/p + 1/q = 1/r$, there exists $\delta = \delta(p,q)>0$ such that for all $m \in \mathbb{N}$\footnote{Here the implicit constant depends at most polynomially on the $C^1$-norms of $\rho, \psi$ appearing in \eqref{smallM}.}
\begin{equation}\label{decm1}
\Big\|T_{m}(f,g)\Big\|_{L^r}\lesssim_{a,p,q} 2^{-\d a m} \|f\|_{L^p} \|g\|_{L^q}.
\end{equation}
\end{theorem}

This result will be derived as a consequence of Proposition \ref{prop:constraint:n}, Corollary \ref{cor:conclusion:trilinear:form:L}, and the extension of the range of boundedness performed in Section \ref{sec:size:interpolation}.

\begin{remark} We end this section by mentioning that in the non-resonant regime, our operator $BC^{a}$ may be thought of as a bilinear modulation invariant\footnote{In the sense of \ref{key:m}.} version  of Stein-Wainger's (polynomial) Carleson operator discussed in \cite{SteinWainger}. This parallelism is preserved when performing the stationary phase analysis of the multiplier as one can  notice by comparing the discussion in our Appendix \ref{appendix:multiplier} with the corresponding analysis in \cite{SteinWainger} or \cite{Stein}.
\end{remark}

\section{Two-resolution analysis: Discretization}
\label{sec:discretization}

\subsection{Preliminaries}\label{Prel}  In this section we will introduce the model operators that are critical for the study of $BC^a$. Generically speaking, the role of such a model is to express the original operator, say $T$, as a superposition of objects of lower complexity (thus easier to be analyzed) while properly capturing the key features of $T$ (\textit{e.g.} dilation, translation, modulation invariance with all these regarded as part of the larger interaction between the structure of the input functions and that of the kernel). Of course, the geometry of the time-frequency regions associated with these lower complexity objects plays the determinative role in the partition, organization of the terms and finally in regrouping the information carried by them in order to obtain the global control over $T$.

Loosely speaking, through our discretization procedure(s) we aim to handle the two competing features of $BC^a$: the non-zero curvature (phase) and the modulation invariance (input).

Indeed, the current section will be divided into two major components:
\begin{itemize}
\item \emph{the high resolution modelling}: designed with the purpose of extracting the cancellation encoded in the curvature of the phase of our kernel/multiplier within every single scale, its analysis is further split into three stages:
%in order to extract
\begin{itemize}
\item construction of the \emph{discrete phase-linearized wave-packet model}: this is performed in Section \ref{HR1} and follows a methodology -- detailed in Section \ref{HR1lgc} -- that involves the phase linearization of the multiplier and the Gabor discretization of the input functions. While this model will be essential in the later stages of our proof when we need to perform a multiscale analysis of BHT-flavor it comes with a major constraint: it is \emph{not} an $m$-decaying absolutely summable model thus making difficult the task of capturing the cancellation at the single scale level. This is the reason for which we need to appeal to a second, auxiliary model, that will be described immediately below.
\item construction of the \emph{continuous phase-linearized spatial model}: discussed in  Section \ref{HR2} -- see also the helpful heuristic in Section \ref{BEV} -- this relies on a linearization process that is applied directly on the kernel (spatial) side of our operator but \emph{without} performing any Fourier analysis (wave-packet discretization) on the input functions. It is this coarser model that will be employed in Section \ref{sec:onescale} in order to obtain the desired single scale estimate with exponential decay in the parameter $m$.

\item transition between the two models above: this is explained in Section \ref{passage} and will become an important ingredient in Section \ref{sec:HR2} when refining the single scale estimate in order to capture the time-frequency localization properties of the dualizing function $h$ and thus prepare the implementation of the low resolution multi-scale analysis.
\end{itemize}

\noindent Finally, a very concise antithetical view on these two high resolution models that summarizes the need for both of them is discussed in Section \ref{HR3}.
\medskip

\item \emph{the low resolution modelling}: presented in Section \ref{LR}, this model will be quintessential in approaching the zero-curvature feature of our operator $BC^a$. Its design relies on properly regrouping the information contained in the discrete phase-linearized wave-packet model in order to be able to extract the almost orthogonality  across multiple scales thus involving the multi-scale nature of our problem. It is through this model that we will be able to exploit the bilinear Hilbert type features of our operator in order to propagate the exponential decay in $m$ from the local level (single scale) to the global level (multi-scale model).
\end{itemize}

The success of our approach depends on the compatibility between the high and low resolution analysis as it is this joint contribution that will deliver the desired result.

\subsection{The high resolution model (I): a phase-linearized wave-packet discretization}\label{HR1}

In the present section we construct a first model operator relying on a wave-packet discretization of the input functions.

\subsubsection{Preparatives}\label{prep2R}

Recalling the decompositions \eqref{oosc} and \eqref{eq:Tmk} and relation \eqref{decmk} in Section \ref{sec:initial:reductions}, throughout this section we fix $m\in\N$, $k\in\Z$ and  focus our attention on the operator $T_{m, k}$ given by
\begin{equation}
\begin{aligned}
& T_{m,k}(f,g)(x) := \\
&\iint \widehat{f}(\xi)\widehat{g}(\eta) e^{i (\xi + \eta) x} \frac{\rho_{am-ak}(\lambda(x))}{2^{am\over 2}}  e^{i c_a (\xi-\eta)^{a'} \lambda(x)^{-\frac{1}{a-1}}} \psi \Big(\frac{\xi- \eta}{2^{am-k}}\Big)   d\xi  d\eta,
\end{aligned}\label{def:disc:Tm,k}
\end{equation}
where from now on, for notational simplicity, we set $a' = a/(a-1)$.

Notice that in the formula above the frequency information is essentially supported in the strip $\{ \xi-\eta \sim 2^{am-k} \}$. As a first step we partition this strip in square regions, thus decoupling the non-oscillatory information in $\xi$ and $\eta$. For this, we consider a suitable partition of unity
\[
1=\sum_{\ell \in \Z} \phi\Big(\zeta - \ell\Big),
\]
where $\phi$ is a smooth function which is compactly supported say in $[\frac{1}{2},2]$. Inserting this into the formula of our operator, in the first instance, we get
\begin{equation*}
\begin{aligned}
T_{m,k}(f,g)(x)=\sum_{\ell, \ell' \in \Z} \iint \widehat{f}(\xi)\widehat{g}(\eta) \phi\Big({\xi \over 2^{am-k}} - \ell\Big)
\phi\Big({\eta \over 2^{am-k}} - \ell'\Big)\\
e^{i (\xi + \eta) x} \frac{\rho_{am-ak}(\lambda(x))}{2^{am\over 2}}  e^{i c_a (\xi-\eta)^{a'} \lambda(x)^{-\frac{1}{a-1}}} \psi \Big(\frac{\xi- \eta}{2^{am-k}}\Big)  d\xi  d\eta.
\end{aligned}
\end{equation*}

Since the function $\psi$ introduced in the stationary phase analysis of the multiplier in \eqref{bigM} is supported around $ \zeta \sim 1$, simple considerations\footnote{By writing each of the functions $\psi$ and $\phi$ as a superposition of, say, $20$ functions, we can further shrink the support of our functions and assume without loss of generality that $\psi$ is supported in $[2-\frac{1}{10}, 2+\frac{1}{10}]$ and $\phi$ in $[\frac{3}{4},\frac{5}{4}]$.} imply that without loss of generality we may assume that
\begin{equation}
\begin{aligned}
T_{m,k}(f,g)(x)=\sum_{\ell\in \Z} \iint \widehat{f}(\xi)\widehat{g}(\eta) \phi\Big({\xi \over 2^{am-k}} - (\ell+1)\Big)
\phi\Big({\eta \over 2^{am-k}} - (\ell-1)\Big)\\
e^{i (\xi + \eta) x} \frac{\rho_{am-ak}(\lambda(x))}{2^{am\over 2}}  e^{i c_a (\xi-\eta)^{a'} \lambda(x)^{-\frac{1}{a-1}}} \psi \Big(\frac{\xi- \eta}{2^{am-k}}\Big)  d\xi  d\eta.
\end{aligned}
\label{eq:Tm,k:first}
\end{equation}
Having reached this point, we turn our attention towards the crucial part of the multiplier -- its oscillatory component -- the one that will dictate the further discretization of our operator. Before entering into details, we make a short detour in order to explain the general methodology that we will follow.

\subsubsection{The LGC-method outlined}\label{HR1lgc}

In this section we briefly discuss the key characteristics of the LGC-methodology employed in \cite{lvUnif}; these are as follows: phase linearization, Gabor (local Fourier) analysis and almost orthogonality via time-frequency correlation\footnote{Throughout the paper the notion of time-frequency correlation designates the existence of a conditionality/relationship among the time and frequency parameters appearing in the Gabor discretizations of the input functions. It should be noted that in general the ``correlation'' aspect of the LGC method can also incorporate the standard frequency correlation of a given function relative to a Fourier mode; however, this latter notion of correlation is not used in the present paper.} and phase level set analysis. Before elaborating more on the above elements we emphasise that each of them has been commonly used in the math literature: for example phase linearization is extensively exploited in PDE, Gabor decompositions are standard tools in signal processing and more recently in time-frequency analysis while the study of the frequency correlations proves its utility in several areas with a special highlight on additive combinatorics. Thus, the novelty and power of the method\footnote{This was introduced in order to provide a unifying perspective on several themes belonging to the non-zero curvature realm. For more on this, the reader is invited to consult Section \ref{Hist} in the Introduction.} does not rely on the nature of each of its constitutive elements but rather on their combined effect in analyzing problems that belong to the general class of singular and maximal integral operators.

With these being said, we now present a very concise overview of these key characteristics:
\begin{enumerate}
\item[(L)] \label{item:L} \emph{linearization}: this represents the design of the spatial/frequency discretization procedure through which the phase of the operator's kernel (if regarded from the spatial side) or the phase of its associated multiplier (if regarded from the frequency side) is confined to oscillate at a linear level (when restricted to the regions arising from this discretization);

\item[(G)] \label{item:G}\emph{adapted Gabor frame analysis}: within each of the regions obtained at the first item, one performs an adapted Gabor frame decomposition of the functions on which the operator acts; this produces a new expression which now involves a sum of local Fourier (Gabor) coefficients multiplied by variable weights resulted from the kernel/multiplier discretization discussed at the previous step;

\item[(C)] \label{item:C} \emph{cancellation via $TT^{*}$-method and (time-frequency) correlation}: the resulting discrete operator is now analyzed at the purely $L^2$ level via a $TT^{*}$ argument. This analysis relies on the cancellation hidden in the oscillation of the discretized kernel/multiplier, and has as an output an expression involving $\ell^2$ spatial/frequency coefficients associated to the input functions, multiplied by elementary expressions that depend on the time and frequency parameters involved in the discretization. Eventually, one exploits the created time-frequency correlations by employing various elements such as\footnote{This list is not complete, and depending on the nature of the problem, not all these elements need to appear in the proof.} the analysis of the size distribution of the (input dependent) coefficients, phase level set estimates, elements of combinatorics and number theory (see \emph{e.g.} discrete Van der Corput or more general bounds for exponential sums) thus achieving the desired generic decay estimate shaped by \eqref{stdeccee}.
\end{enumerate}

With these explained, we are now ready to follow the above steps in order to achieve our first model operator.

\subsubsection{Phase linearization}\label{PL}

In this section, for a fixed $m\in\N$ and $k,\ell\in\Z$ we are going to subdivide each of the frequency squares $[2^{am-k}(\ell+1),  2^{am-k}(\ell+2)]\times[2^{am-k}(\ell-1),  2^{am-k}\ell]$ obtained in Section \ref{prep2R} into a collection of smaller dyadic squares of a prescribed length $2^{\frac{am}{2} - k}$, designed such that the phase of the multiplier oscillates at a linear level within each such smaller square.

In what follows we assume that $(x,\xi,\eta)$ are such that $\mathfrak{m}_{m,k}(\xi-\eta;\lambda(x))\not=0$, see \eqref{smallM}. Then, if we let $\mu(\xi,\eta)$ stand for the phase of the multiplier $\mathfrak{m}_{m,k}(\xi-\eta;\lambda(x))$ as given by \eqref{smallM}, we notice that
\[ \mu(\xi,\eta)=\mu(\xi-\eta)= c_a (\xi-\eta)^{a/(a-1)} \lambda(x)^{-1/(a-1)}, \]
with  $|\lambda(x)| \sim 2^{a(m-k)}$ and $|\xi-\eta| \sim 2^{am-k}$.

From the above we deduce that
\begin{equation} \label{secderiv}
\begin{aligned}|\partial^2_{\xi}\mu(\xi,\eta)|= |\partial^2_{\eta}\mu(\xi,\eta)|=|\partial^2_{\xi\eta}\mu(\xi,\eta)| & =
|c_a|  |\xi-\eta|^{\frac{2-a}{a-1}} |\lambda(x)|^{-\frac{1}{a-1}} \\
& \sim 2^{2k-am} = (2^{\frac{am}{2} - k})^{-2} .
\end{aligned}
\end{equation}

Consequently, if we want to impose that the second order term in the Taylor expansion of $\mu(\xi,\eta)$ around a given point $(\xi_0,\eta_0)$ is at most $O(1)$, we need to impose the condition $|\xi-\xi_0|, |\eta-\eta_0|\lesssim 2^{\frac{am}{2} - k}$. Now, within such a smaller square region we have that the phase can be decomposed as a linear component plus an -- essentially non-oscillatory -- nonlinear component.

In light of the above it becomes natural to subdivide the frequency support of  $\mathfrak{m}_{m,k}(\xi-\eta;\lambda(x))$ into (essentially) squares of size-length $2^{\frac{am}{2}-k}$. This process will be made precise in the next section.

\subsubsection{Gabor frame discretization}

Following the above description we will perform a Gabor frame (windowed Fourier series) decomposition of the input functions that visualize the frequency universe in ``unit'' lengths of size $2^{\frac{am}{2}-k}$. Concretely, we first introduce the $L^2$-normalized Gabor frame (wave-packet)
\begin{equation}
\label{def:gdec3}
\varphi^{u,n}_{\frac{a m}{2}-k,\ell}(\xi):=\frac{1}{2^{\frac{am-2k}{4}}} \varphi\Big(\frac{\xi}{2^{\frac{a m}{2}-k}}-u-2^{\frac{am}{2}} (\ell-1)\Big) e^{-i n \xi    2^{-(\frac{a m}{2}-k)}}
\end{equation}
with $\varphi$ a suitable smooth function compactly supported and then decompose
%\footnote{Here we are making a slight notational abuse: strictly speaking, the function's coefficients in our Gabor decomposition(s) should be of the form $\langle \widehat{f}, \varphi^{u,n}_{1,\frac{am}{2}-k,\ell}$ where $\varphi_{1}=\varphi \psi_1=\psi_1$. However, since our arguments are insensitive to this detail, for notational simplicity we prefer two write our Gabor decompositions as above.}
\begin{equation}\label{Gf}
\begin{aligned}
&\qquad\qquad\qquad\qquad\widehat{f}_{\ell+1}(\xi):=\widehat{f}(\xi) \phi({\xi \over 2^{am-k}} -(\ell+1))\\
=&\sum_{u,n \in \mathbb{Z}}  \langle \widehat{f}_{\ell+1}, \varphi^{u,n}_{\frac{am}{2}-k,\ell}\rangle  \varphi^{u,n}_{\frac{am}{2}-k,\ell}(\xi)
=\sum_{2 \cdot 2^{\frac{am}{2}}< u \leq  3 \cdot 2^{\frac{am}{2}}}\sum_{n \in \mathbb{Z}}  \langle \widehat{f}_{\ell+1}, \varphi^{u,n}_{\frac{am}{2}-k,\ell}\rangle \varphi^{u,n}_{\frac{am}{2}-k,\ell}(\xi) ,
\end{aligned}
\end{equation}
and
\begin{equation}\label{Gg}
\begin{aligned}
&\qquad\qquad\quad\quad\widehat{g}_{\ell-1}(\eta):=\widehat{g}(\eta) \phi({\eta \over 2^{am-k}} -(\ell-1))\\
=&\sum_{v,p \in \mathbb{Z}} \langle \widehat{g}_{\ell-1}, \varphi^{v,p}_{\frac{am}{2}-k,\ell}\rangle  \varphi^{v,p}_{\frac{am}{2}-k,\ell}(\eta)=\sum_{1 \leq v \leq 2^{\frac{am}{2}}}    \sum_{p \in \mathbb{Z}} \langle \widehat{g}_{\ell-1}, \varphi^{v,p}_{\frac{am}{2}-k,\ell}\rangle \varphi^{v,p}_{\frac{am}{2}-k,\ell}(\eta) ,
\end{aligned}
\end{equation}
where in the above we made use of the support properties for $\widehat{f}_{\ell+1}$, $\widehat{g}_{\ell-1}$ and $\varphi$.

If we now insert \eqref{Gf} and \eqref{Gg} back to \eqref{eq:Tm,k:first} we deduce that
\begin{align}\label{Trepr}
T_{m,k}(f,g)(x)=\sum_{\ell \in \Z} \sum_{n, p \in \Z} & \sum_{u, v \in \Z} \langle \widehat{f}_{\ell+1}, \varphi^{u,n}_{\frac{am}{2}-k,\ell}\rangle \langle \widehat{g}_{\ell-1}, \varphi^{v,p}_{\frac{am}{2}-k,\ell}\rangle   \K^{u,v,n,p}_{m,k,\ell}(x) ,
\end{align}
with
\begin{equation}
\begin{aligned}
\K^{u,v,n,p}_{m,k,\ell}(x) &:= \frac{1}{2^{\frac{a m}{2}}} \rho_{a m-a k}(\l(x))   \iint  e^{i c_a (\xi-\eta)^{a'} \lambda(x)^{-\frac{1}{a-1}}} e^{i \xi x} e^{i \eta x} \\
& \hspace{7em}\cdot\varphi^{u,n}_{\frac{a m}{2}-k,\ell}(\xi)  \varphi^{v,p}_{\frac{a m}{2}-k,\ell}(\eta) \psi \Big(\frac{\xi- \eta}{2^{am-k}}\Big) d\xi d\eta .
\end{aligned}\label{disc:gdec6}
\end{equation}

This ends the linearized wave-packet decomposition of our operator, of course, up to a better understanding of the kernel $\K^{u,v,n,p}_{m,k,\ell}$.

\subsubsection{Analyzing the kernel}

In this section we focus on obtaining an explicit form for the kernel $\K^{u,v,n,p}_{m,k,\ell}$ by exploiting the time-frequency localization of the wave-packets in relation to the oscillatory factor $\ds  e^{i c_a (\xi-\eta)^{a'} \lambda(x)^{-\frac{1}{a-1}}}$ in \eqref{disc:gdec6}.

Notice that the change of variables $\xi \rightarrow 2^{\frac{a m}{2}-k} \xi$ and $\eta \rightarrow 2^{\frac{a m}{2}-k} \eta$ in \eqref{disc:gdec6} together with the definition of the wave-packets \eqref{def:gdec3} allows us to rewrite the above as
\begin{align*}
\K^{u,v,n,p}_{m,k,\ell}(x) = &\frac{\rho_{am-ak}(\l(x)) }{2^{k}} \iint  e^{ i  c_a  \frac{\big(2^{\frac{a m}{2}-k} (\xi-\eta)\big)^{a'}}{\l(x)^{1/(a-1)}}} \cdot e^{i(2^{\frac{a m}{2}-k} \xi x + 2^{\frac{a m}{2}-k} \eta x - \xi n - \eta p)}  \\
& \cdot \varphi(\xi-u-2^{\frac{am}{2}} (\ell-1) ) \varphi(\eta-v-2^{\frac{am}{2}}(\ell-1)) \psi \Big(\frac{\xi- \eta}{2^{\frac{am}{2}}}\Big) d\xi d\eta.
\end{align*}
Applying now a second change of variables $\xi-\eta=\xi_1$ and $\xi+\eta=\eta_1$ we rewrite the kernel $\K^{u,v,n,p}_{m,k,\ell}(x)$ as
\begin{align}\label{ker1}
&\frac{\rho_{am-ak}(\l(x)) }{2^{k}}  \iint  e^{i c_a  (2^{\frac{a m}{2}-k}  \xi_1)^{a'} \l(x)^{-\frac{1}{a-1}}}  e^{i 2^{\frac{a m}{2}-k} \eta_1 x}   e^{-i \frac{\xi_1+\eta_1}{2}  n} e^{-i  \frac{\eta_1-\xi_1}{2}  p} \\
& \cdot \varphi\Big(\frac{\xi_1+\eta_1}{2}-u-2^{\frac{am}{2}} (\ell-1) \Big) \varphi\Big(\frac{\eta_1-\xi_1}{2}-v-2^{\frac{am}{2}}(\ell-1)\Big) \psi \Big(\frac{\xi_1}{2^{\frac{am}{2}}}\Big) d\xi_1 d\eta_1.\nonumber
\end{align}
We claim  that the expression under the integral sign behaves essentially as a tensor product of the form $a(\xi_1) b(\eta_1)$ for some suitable functions $a$ and $b$. Indeed, given the support localization for the function $\varphi$ our kernel can be safely approximated as
\begin{equation}\label{kernel}
\K^{u,v,n,p}_{m,k,\ell}(x)\approx\frac{\rho_{am-ak}(\l(x)) }{2^{\frac{am+2k}{4}}}   w_{k,\frac{p-n}{2}, \frac{v-u}{2}}^e (\lambda)(x)   \check{\varphi}_{\frac{am}{2}- k, 2\ell-1}^{u+v, \frac{p+n}{2}}(x) ,
\end{equation}
where here\footnote{For simplicity of the exposition, we abuse the notation and in the definitions of $w_{k,\frac{p-n}{2}, \frac{v-u}{2}}^e (\lambda)(x)$ and $\frac{1}{2^{\frac{am-2k}{4}}} \check{\varphi}_{\frac{am}{2}- k, 2\ell-1}^{u+v, \frac{p+n}{2}}(x)$ we are using the same symbol $\varphi$ as the one introduced earlier though of course, strictly speaking, the functions involved in these definitions while having similar properties do not need to agree with $\varphi$.}
\begin{equation}\label{weight}
\begin{aligned}
w_{k,\frac{p-n}{2}, \frac{v-u}{2}}^e (\lambda)(x) :=  \int \varphi ( \xi_1 - (u-v) )  e^{i   \frac{p-n}{2} \xi_1} e^{i c_a 2^{ ({am  \over 2}- k) a'} \xi_1^{a'} \lambda(x)^{- {1 \over {a-1}}}} \psi \Big(\frac{\xi_1}{2^{\frac{am}{2}}}\Big)  d \xi_1,
\end{aligned}
\end{equation}
and the wave-packet  $\frac{1}{2^{\frac{am-2k}{4}}} \check{\varphi}_{\frac{am}{2}- k, 2\ell-1}^{u+v, \frac{p+n}{2}}(x)$ stands for the expression
\begin{equation} \label{eq:modified_wave_packet}
\int e^{i 2^{\frac{a m}{2}-k} \eta_1 x}  e^{-i\frac{p+n}{2} \eta_1} \varphi(\eta_1- (u+v)- 2^\frac{am}{2} (2 \ell -2)) d\eta_1 .
\end{equation}

The rigorous, formal approach to the informal statement \eqref{kernel} follows some standard reasonings in Fourier analysis and thus, in what follows, we will only provide an outline of them:

Since $\varphi$ is smooth and compactly supported -- for the sake of the argument one may assume without loss of generality that
$\textrm{supp} \varphi\subseteq [0,1]$ -- we can write pointwise within $[0,1]$ the equality between $\varphi$ and its Fourier series, that is
\[
\varphi(\zeta)= \sum_{s \in \Z} c_s e^{i  s   \zeta},\:\:\:\:\forall\:\zeta\in [0,1] ,
\]
where the coefficients $(c_s)_s$ are fast-decaying (in a way that agrees with the regularity of $\varphi$).

Deduce from the above that on the support of the function\footnote{Here we abuse the notation as we allow the function $\varphi$ to slightly change from line to line.} $(\xi_1, \eta_1) \mapsto \varphi(\xi_1-(u-v)) \varphi(\eta_1-(u+v)-2^{am \over 2}(2\ell-2))$ one has that
$$ \varphi(\frac{\xi_1+\eta_1}{2}-u-2^{\frac{am}{2}} (\ell-1) ) \varphi(\frac{\eta_1-\xi_1}{2}-v-2^{\frac{am}{2}}(\ell-1))$$ equals
\begin{align*}
&\sum_{s, t \in \Z} c_s c_t e^{i s (\frac{\xi_1+\eta_1}{2}-u-2^{\frac{am}{2}} (\ell-1) )} e^{i t (\frac{\eta_1-\xi_1}{2}-v-2^{\frac{am}{2}}(\ell-1))}.
\end{align*}

Thus, \eqref{kernel} becomes
\begin{equation*}
\begin{aligned}
 \K^{u,v,n,p}_{m,k,\ell}(x) &= \sum_{s,t\in \Z} c_s c_t
\frac{\rho_{am-ak}(\l(x)) }{2^{\frac{am+2k}{4}}}
 w_{k,\frac{s-t}{2}+\frac{p-n}{2}, \frac{v-u}{2}}^e (\lambda)(x)   \check{\varphi}_{\frac{am}{2}- k, 2\ell-1}^{u+v, \frac{p+n}{2}-\frac{s+t}{2}}(x)\\
& \qquad \qquad  \quad e^{-i s (u+2^{\frac{am}{2}} (\ell-1))} e^{-i t  (v+2^{\frac{am}{2}}(\ell-1))}
\\
&=:\sum_{s,t\in\Z} c_s c_t \tilde{\K}^{u,v,n+t,p+s}_{m,k,\ell}(x) .
\end{aligned}
\end{equation*}
Due to the fast decay of the coefficients $\{c_s\}_s$ we can focus on the case $s=t=0$, justifying our claim \eqref{kernel}.

We conclude this section with the following remark: by a simple inspection of \eqref{kernel} we notice that the behavior of the rough linearizing function is now separated and isolated within the term $w_{k,\frac{p-n}{2}, \frac{v-u}{2}}^e (\lambda)$ which can be regarded as a weight attached to the standard wave-packet $\check{\varphi}_{\frac{am}{2}- k, 2\ell-1}^{u+v, \frac{p+n}{2}}$.

\subsubsection{Achieving the final form of our first model}
\label{subsec:final:form:disc}
% the reasonings from the previous section
Based on the previous section's reasonings, we were able to reduce the initial operator $T_{m,k}(f, g)$ defined in \eqref{def:disc:Tm,k} to the study of an expression realized as a superposition of simpler operators that reads as\footnote{In this precise context, for notational simplicity, we are abusing our convention on the notation $v\sim 2^{am \over 2}$ (which as defined originally is equivalent with $|v|\sim 2^{am \over 2}$). Indeed, in formula \eqref{Tmkred} the domain of the summation $\{u, v \sim 2^\frac{am}{2}\}$ is in fact restricted to $2^{\frac{am}{2}}<u\leq 2 \cdot 2^\frac{am}{2}$ and $-\frac{3}{2} 2^{am\over 2}< v \leq -\frac{1}{2} 2^\frac{am}{2}$.}

\begin{equation}\label{Tmkred}
\begin{aligned}
\sum_{\ell \in \Z} \  \sum_{n, p \in \Z} \ \sum_{u, v \in \Z}  & \langle \hat{f}_{\ell+1}, \varphi^{u,n}_{\frac{am}{2}-k,\ell}\rangle \langle \hat{g}_{\ell-1}, \varphi^{v,p}_{\frac{am}{2}-k,\ell}\rangle \\
& \cdot \frac{\rho_{am-ak}(\l(x)) }{2^{\frac{am+2k}{4}}}   \check{\varphi}_{\frac{am}{2}- k, 2\ell-1}^{u+v, \frac{p+n}{2}}(x)  w_{k,\frac{p-n}{2}, \frac{v-u}{2}}^e (\lambda)(x)\\
=\sum_{\ell \in \Z} \ \sum_{n, p \in \Z} \ \sum_{u, v \sim 2^\frac{am}{2}} &  \langle \hat{f}_{\ell+1}, \varphi^{u,n}_{\frac{am}{2}-k,\ell}\rangle \langle \hat{g}_{\ell-1}, \varphi^{v,p}_{\frac{am}{2}-k,\ell}\rangle \\
& \cdot \frac{\rho_{am-ak}(\l(x)) }{2^{\frac{am+2k}{4}}}   \check{\varphi}_{\frac{am}{2}- k, 2\ell-1}^{u+v, \frac{p+n}{2}}(x)  w_{k,\frac{p-n}{2}, \frac{v-u}{2}}^e (\lambda)(x).
\end{aligned}
\end{equation}
By applying Parseval's identity and the change of variables $(u,v,n,p) \to (u-v,u+v,p-n,p+n)$ we have that the expression in \eqref{Tmkred} may be rewritten as

\begin{equation}\label{Tmkred1}
\begin{aligned}
\sum_{\ell \in \Z} \sum_{n,p \in \Z} \sum_{u, v\in \Z} &\langle f_{\ell+1}, \check{\varphi}_{\frac{am}{2}- k, \ell}^{u-v,p-n}  \rangle   \langle g_{\ell-1}, \check{\varphi}_{\frac{am}{2}- k, \ell}^{u+v,p+n}  \rangle \\
 &\cdot\frac{\rho_{am-ak}(\l(x)) }{2^{\frac{am+2k}{4}}}    \check{\varphi}_{\frac{am}{2}- k, 2\ell-1}^{2u, p}(x) w_{k, n,v}^e(\l)(x).
\end{aligned}
\end{equation}

The first task for this section is to show that the study of \eqref{Tmkred1} can be further reduced to a simpler expression -- which, for notational simplicity, will be redefined as $T_{m,k}$ -- and which involves the summation only over $n \sim_a 2^\frac{am}{2}$.\footnote{We remind here that the definition of $\sim_a$ was provided in footnote \ref{footnote54}.}\footnote{This will be an immediate consequence of the structure of the decaying term in the right-hand side of \eqref{eq:decay-weight}.} Indeed, our claim is that it is enough to consider\footnote{From now on, for notational simplicity, we will drop the $a$ dependence in $\sim_a$. However, it should be noted that, strictly speaking, whenever we read $n,p \sim 2^\frac{am}{2}$ we have in fact that $n,p \sim_a 2^\frac{am}{2}$.}

\begin{equation}
\begin{aligned}
T_{m,k}(f, g)(x)= \sum_{\ell \in \Z} \sum_{\substack{n \sim 2^\frac{am}{2} \\ p \in \Z}}  \sum_{u, v \in\Z} \langle f_{\ell+1}, \check{\varphi}_{\frac{am}{2}- k, \ell}^{u-v,p-n}  \rangle   \langle g_{\ell-1}, \check{\varphi}_{\frac{am}{2}- k, \ell}^{u+v,p+n}  \rangle &  \\
 \cdot \frac{\rho_{am-ak}(\l(x)) }{2^{\frac{am+2k}{4}}}    \check{\varphi}_{\frac{am}{2}- k, 2\ell-1}^{2u, p}(x) w_{k, n,v}^e(\l)(x) & .
\end{aligned}\label{eq:model}
\end{equation}

If is to believe for a moment \eqref{eq:model}, then the restriction $n \sim 2^\frac{am}{2}$ naturally\footnote{Notice the structure $p-n$ and $p+n$ in the Gabor coefficients $\langle f_{\ell+1}, \check{\varphi}_{\frac{am}{2}- k, \ell}^{u-v,p-n}  \rangle $ and  $\langle g_{\ell-1}, \check{\varphi}_{\frac{am}{2}- k, \ell}^{u+v,p+n}  \rangle$, respectively. The deeper, geometric motivation for the regrouping of the sum in $p$ will become transparent in Section \ref{LR}.} invites to a regrouping of the sum in $p$ in subcomponents spanning precisely the reference length $2^\frac{am}{2}$. More precisely, we will write the set of integers as a disjoint union of intervals containing $2^\frac{am}{2}$ consecutive integers, and using the notation
\begin{equation}
\label{eq:not:p_r}
p_r:= 2^\frac{am}{2} (r-1) + p, \quad \text{ with } \quad r\in \Z \quad \textrm{and} \quad 2^{am \over 2}\leq p < 2\cdot 2^{am \over 2}
\end{equation}
we will rewrite \eqref{eq:model} as
\begin{equation}
\begin{aligned}
T_{m,k}(f, g)(x)= \sum_{\ell,r \in \Z} \sum_{n,p \sim 2^\frac{am}{2}}  \sum_{u, v\in\Z} \langle f_{\ell+1}, \check{\varphi}_{\frac{am}{2}- k, \ell}^{u-v,p_r-n}  \rangle   \langle g_{\ell-1}, \check{\varphi}_{\frac{am}{2}- k, \ell}^{u+v,p_r+n}  \rangle & \\
 \cdot \frac{\rho_{am-ak}(\l(x)) }{2^{\frac{am+2k}{4}}}    \check{\varphi}_{\frac{am}{2}- k, 2\ell-1}^{2u, p_r}(x) w_{k, n,v}^e(\l)(x) &.
\end{aligned}\label{eq:model0}
\end{equation}

Returning now to \eqref{eq:model}, the key for proving our claim resides in the behavior of the oscillatory term
\begin{equation}
\begin{aligned}
&\rho_{am-ak}(\l(x)) w_{k, n,v}^e(\l)(x) \\
&= \rho_{am-ak}(\l(x)) \int \varphi ( \xi + 2v )  e^{i n\xi} e^{i c_a 2^{ ({am  \over 2}- k) a'} \xi^{a'} \lambda(x)^{- {1 \over {a-1}}}} \psi \Big(\frac{\xi}{2^{\frac{am}{2}}}\Big)   d \xi .
\end{aligned} \label{eq:w_k}
\end{equation}
Denote now the phase in \eqref{eq:w_k} by
\begin{equation}
 \nu(\xi):= n\xi + c_a 2^{ ({am  \over 2}- k) a'} \xi^{a'} \lambda(x)^{- {1 \over {a-1}}} , \label{eq:phase:disc}
\end{equation}
and integrate by parts twice in the right-hand side of \eqref{eq:w_k} in order to obtain under the integral sign a decaying factor of the form $\frac{1}{1+|\nu'(\xi)|^2}$. Once here, denoting by $\xi_c$ the\footnote{Up to a sign, this real number is uniquely determined.} critical point of $\nu$ (\emph{i.e.}, $\nu'(\xi_c)=0$)  we have two situations:
\begin{itemize}
\item either $|\xi_c|\sim 2^{\frac{am}{2}}$ case in which, using the phase linearization performed in Section \ref{PL} we have by construction that $|\nu''(\xi)| \sim 1$ for $|\xi|\sim 2^{\frac{am}{2}}$ and hence, via a Taylor series argument, we deduce that $ |\nu'(\xi)| \sim |\xi-\xi_c|\sim |\xi_c+2v|$;

\item or else we must have $ |\nu'(\xi)|\gtrsim 2^{\frac{am}{2}}$.
\end{itemize}

In either of the above situations, observing that
\begin{equation}
\label{eq:phase:c:point}
\xi_c= \bar c_a \frac{n^{a-1}  \l(x)}{2^{a(\frac{am}{2}-k)}},
\end{equation}
we obtain a first time-frequency correlation condition as given by
\begin{equation}
\begin{aligned}
\big|  \rho_{am-ak}( \lambda(x)) w_{k, n, v}^e (\lambda)(x) \big| &\lesssim
\rho_{am-ak}( \lambda(x))  \frac{1}{ 1+  \Big|\bar c_a \frac{n^{a-1}  \l(x)}{2^{a(\frac{am}{2}-k)}} +2v \Big|^2} \\
& \hspace{1em} =: w_{k, n, v}(\lambda)(x)
\end{aligned}\label{eq:decay-weight}
\end{equation}

%\eqref{eq:phase:c:point}
Now the decay in \eqref{eq:decay-weight} offers us the intuition that the main contribution in \eqref{Tmkred1} is expected to come from parameters $n$ and $v$ satisfying
\[
-2v \approx \bar c_a \frac{n^{a-1}  \l(x)}{2^{a(\frac{am}{2}-k)}}.
\]

However, since $v \sim 2^\frac{am}{2}$, that effectively restricts $n$ to $|n| \sim 2^\frac{am}{2}$ which by symmetry may be reduced to $n \sim 2^\frac{am}{2}$ thus providing an informal argument supporting our claim in \eqref{eq:model}.

The formal, rigorous argument is presented now. First, recall from Section \ref{sec:reduction:local:L2} that we only need to prove operator bounds for $\sum_{k} T_{m,k}$ in the local-$L^2$ range (\emph{i.e.} $2<p,q\leq \infty$) with a $2^{-\delta a m}$ decay. Thus, our claim is a consequence of the following dualized statement:

\begin{proposition}  \label{prop:constraint:n}
Let $\mathfrak{N} = \mathfrak{N}_m$ be the subset of indices\footnote{Here we use $A \ll_a B$ to denote the inequality $A \leq c_a B$ for a constant $c_a>0$ depending on $a$ and sufficiently small that any subsequent arguments appealing to this inequality are enabled (with the obvious modifications for the situation $A \gg_a B$).}
\[ \mathfrak{N}:=\{n\in \mathbb{Z} : |n|\ll_a 2^\frac{am}{2} \textrm{ or } |n| \gg_a 2^\frac{am}{2}\} . \] For each $k\in\Z$ and $m\in\N$  define now
 the trilinear form\footnote{This is obtained from \eqref{Tmkred1} by restricting $n$ to $\mathfrak{N}$ and then dualizing against the function $h$.}
\begin{align*}
\underline{\L}_{m, k}^{\mathfrak{N}} (f, g, h) :=\sum_{ \ell \in \Z} & \sum_{ \substack{n \in \mathfrak{N} \\ p \in \Z}} \sum_{u,v \sim 2^{ am \over 2}} \frac{1}{2^{\frac{am+2k}{4}}} \langle f_{\ell+1}, \check{\varphi}_{\frac{am}{2}- k, \ell}^{u-v, p-n}  \rangle   \langle g_{\ell-1}, \check{\varphi}_{\frac{am}{2}- k, \ell}^{u+v, p+n}  \rangle   \\
  & \cdot \int  \check{\varphi}_{\frac{am}{2}- k, 2\ell-1}^{2u, p}(x)   h(x)  \rho_{am-ak}( \lambda(x)) w_{k, n, v}^e(\l)(x) dx .
\end{align*}
Then for every $N >0$ and any exponents $p,q,r$ such that $2<p,q\leq \infty$ and $1/p + 1/q = 1/r$ we have
\[ \big\| \sum_{k \in \mathbb{Z}} \underline{\L}_{m, k}^{\mathfrak{N}} \big\|_{L^p \times L^q \times L^{r'} \to {\mathbb C}} \lesssim_{a,p,q,N} 2^{-N a m} \]
uniformly in $m \in \mathbb{N}$.
\end{proposition}

\begin{proof}
Recalling the definition of $\nu(\xi)$ in \eqref{eq:phase:disc} we have
\[ \nu'(\xi) = n + c_a a' 2^{ ({am  \over 2}- k) a'} \Big(\frac{\xi}{\lambda(x)}\Big)^{{1 \over {a-1}}} \]
with $|\xi| \sim 2^\frac{am}{2}$ (because of the support of $\psi$); for $n \in \mathfrak{N}$ the phase function has no critical point, since
$|\nu'(\xi)| \sim 2^\frac{am}{2}$ if $|n|\ll_a 2^\frac{am}{2}$ and
$|\nu'(\xi)| \sim |n|$ if $|n|\gg_a 2^\frac{am}{2}$. We can thus summarize this as
\[ |\nu'(\xi)| \sim (2^\frac{am}{2} +|n|) . \]
Therefore, repeated integration by parts in \eqref{eq:w_k} yields for any\footnote{The value of $N$ is allowed to vary from line to line.} $N>0$
\[ \rho_{am-ak}( \lambda(x)) |w_{k, n, v}^e (\lambda)(x)| \lesssim_N (2^\frac{am}{2} +|n|)^{-N} \rho_{am-ak}( \lambda(x)) . \]
For any interval $I$, we introduce now the $L^\infty$-normalised weight adapted to $I$ as
\begin{equation}
\label{eq:def:bump}
 \ci_{I}(x) := \Big(1 + \frac{\mathrm{dist}(x,I)}{|I|}\Big)^{-M}
 \end{equation}
for some large $M>0$ -- also allowed to change from line to line. Recalling the definition of our wave-packet in \eqref{def:gdec3} we have
\[ |\check{\varphi}^{2u,p}_{\frac{am}{2}-k, 2\ell -1}(x)| \lesssim 2^{\frac{am-2k}{4}}  \ci_{I^{p}_{k}}(x), \]
where $I^{p}_{k}$ denotes the interval
\begin{equation} I^{p}_{k} := \{ x \in \mathbb{R} : 2^{\frac{am}{2}-k} x - p \in [0,1] \} = \Big[ \frac{p}{2^{\frac{am}{2}-k}}, \frac{p+1}{2^{\frac{am}{2}-k}} \Big]. \label{eq:Ip} \end{equation}
Taking absolute values and applying the pointwise bounds above we have then for the trilinear form
\begin{align*}
|\underline{\L}_{m, k}^{\mathfrak{N}} (f, g, h)| \lesssim  2^{-k} \sum_{\ell \in \mathbb{Z}} & \sum_{ \substack{n \in \mathfrak{N} \\ p \in \Z}} \sum_{u, v \sim 2^{ am \over 2}} |\langle f_{\ell+1}, \check{\varphi}_{\frac{am}{2}- k, \ell}^{u-v, p-n}  \rangle|   |\langle g_{\ell-1}, \check{\varphi}_{\frac{am}{2}- k, \ell}^{u+v, p+n}  \rangle| \\
& \cdot (2^\frac{am}{2} +|n|)^{-N} \int  \ci_{I^{p}_{k}}(x) | h(x) | \rho_{am-ak}( \lambda(x)) dx,
\end{align*}
where the last factor does not depend anymore on $u,v,\ell$. Using Cauchy-Schwarz in $u,v$ and $\ell$ yields without difficulty
\begin{align*}
|\underline{\L}_{m,k}^{\mathfrak{N}} (f, g, h)| & \lesssim  2^{\frac{am}{2}-k}   \sum_{ \substack{n \in \mathfrak{N} \\ p \in \Z}} \Big(\int |f(y)|^2 \ci_{I_k^{p-n}}(y)   dy\Big)^{1/2}  \Big(\int |g(z)|^2 \ci_{I_k^{p+n}}(z)   dz\Big)^{1/2} \\
& \hspace{4em} \cdot (2^\frac{am}{2} +|n|)^{-N} \int  \ci_{I_k^{p}}(x) | h(x) | \rho_{am-ak}( \lambda(x)) dx.
\end{align*}

Now for every $x\in \R$, standard reasonings imply that
\begin{align*}
\Big(\int |f(y)|^2 \ci_{I_k^{p-n}}(y)   dy\Big)^{1/2} \leq 2^{\frac{1}{2}(k-{am \over 2})}\Big(|n|+ \frac{\dist(x,I_k^p)}{2^{{am\over 2}-k}}\Big)^{1\over 2} M_2(f)(x),
\end{align*}
with a similar inequality for the average on $g$ (here $M_2 (f) = M(|f|^2)^{1/2}$).
Deduce from here that
\begin{align*}
|\underline{\L}_{m,k}^{\mathfrak{N}} (f, g, h)| & \lesssim \sum_{p \in \mathbb{Z}}\sum_{n \in \mathfrak{N}}
(2^\frac{am}{2} +|n|)^{-N} \int \Big(|n|+ \frac{\dist(x,I_k^p)}{2^{{am\over 2}-k}}\Big) M_2(f)(x) \\
& \qquad \qquad \cdot M_2(g)(x) \ci_{I_k^{p}}(x) | h(x) | \rho_{am-ak}( \lambda(x)) dx
 \\
& \lesssim 2^{-N a m} \sum_{p \in \mathbb{Z}}
 \int M_2(f)(x) M_2(g)(x) \ci_{I_k^{p}}(x) | h(x) | \rho_{am-ak}( \lambda(x)) dx.
\end{align*}
Given the definition of $\ci_{I_k^p}$ the sum in $p$ is $\lesssim 1$. The factor of $\rho_{am-ak}( \lambda(x))$ enables us to further sum over $k$ and thus
\begin{align*}
\sum_{k} |\underline{\L}_{m, k}^{\mathfrak{N}} (f, g, h)|& \lesssim 2^{-N a m} \sum_{k \in \Z}
\int  M_2(f)(x) M_2(g)(x) | h(x) | \rho_{am-ak}( \lambda(x)) dx \\
& \lesssim 2^{-N a m}
\int  M_2(f)(x) M_2(g)(x) | h(x) |  dx.
\end{align*}
Via H\"{o}lder, the latter expression is now trivially bounded in the stated range $2 < p, q \leq \infty$ and $1/p + 1/q = 1/r$ thus ending our proof.
\end{proof}

We now record the dualized version of what we achieved via our reductions:

\begin{equation}
\label{def:trilinear:form0}
\begin{aligned}
&\quad\underline{\mathcal{L}}(f,g,h):= \sum_{m\in \mathbb{N}} \underline{\mathcal{L}}_{m}(f,g,h):= \sum_{m\in \mathbb{N}}\sum_{k\in \mathbb{Z}} \underline{\mathcal{L}}_{m,k}(f,g,h):= \sum_{m\in \mathbb{N}}\sum_{k\in \mathbb{Z}}\langle T_{m,k}(f,g), h \rangle=\\
&\sum_{{m\in \mathbb{N}}\atop{k,\ell,r \in \Z}} \sum_{n,p,u,v\sim 2^\frac{am}{2}}
\langle f_{\ell+1}, \check{\varphi}_{\frac{am}{2}- k, \ell}^{u-v,p_r-n}  \rangle   \langle g_{\ell-1}, \check{\varphi}_{\frac{am}{2}- k, \ell}^{u+v,p_r+n}  \rangle \langle h,\frac{\rho_{am-ak}(\l) }{2^{\frac{am+2k}{4}}}    \check{\varphi}_{\frac{am}{2}- k, 2\ell-1}^{2u, p_r} w_{k, n,v}^e(\l)\rangle .
\end{aligned}
\end{equation}

As already mentioned in Section \ref{Prel}, the major problem that we are confronted with when dealing with the above model is the
failure of the $m$-decay absolute summability property -- see the Appendix A for a proof. As a consequence, we have to design an alternative route that will help us in Section \ref{sec:onescale} to extract the cancellation encoded in the curvature of the phase in our operator's kernel. This route is the subject of our next section.

\subsection{The high resolution model (II): a continuous phase-linearized spatial model}\label{HR2}
\subsubsection{Motivation and intuition}\label{motint}

The aim of this subsection is to highlight a simple though deep manifestation of the fact the trigonometric system is not an unconditional basis for $L^p(\TT)$ whenever $1\leq p\not=2\leq \infty$. The example that we will provide below reveals the subtle cancellation hidden in the Fourier coefficients of a given function and serves as a memento that many of otherwise simple objects when expressed as sums involving Fourier coefficients give rise to only conditionally summable models. Then the challenge in operating with such a model is how to properly regroup the terms (see \emph{e.g.} the case of Littlewood-Paley theory) such that enough cancellation can be captured within the newly formed terms.

However, as we will soon see, sometimes reshaping the entire problem from the frequency to the physical side completely removes this difficult issue of conditional summability. It is this last remark that we would particularly like to emphasise via the example below. Its key underlying philosophy was inspirational for our problem, and, in our search for proving the single scale estimate in Section \ref{sec:onescale} motivated the idea of substituting the discrete phase-linearized wave-packet model by the continuous phase-linearized spatial model.\footnote{For this substitution/transition between the two models one is also invited to consult Section \ref{passage}.}

$\newline$

Let $F, G\in L^2(\TT)$ with their Fourier representation given by
$$F(t)\simeq \sum_{n\in\Z} a_n e^{i n t}\:\:\:\textrm{and}\:\:\:G(t)\simeq \sum_{k\in\Z} b_k e^{i k t} .$$
Then, a simple application of H\"older implies
\beq\label{conv}
F G\in L^1(\TT) ,
\eeq
while from the basic properties of the Fourier series we have
\beq\label{Pars}
(F G)(t)\simeq \sum_{n\in\Z} (\sum_{k\in\Z} a_{n-k} b_k) e^{i n t} .
\eeq
Taking now $H\in L^{\infty}(\TT)$ with $H(t)\simeq \sum_{n\in\Z} c_n e^{i n t}$ we let
\begin{equation}
\label{Example1}
\Lambda(F,G,H):=((F G)*H)(0)=\int_{\R} F(t) G(t) H(-t) dt .
\end{equation}
Now, we conclude that
\begin{itemize}
\item on the one hand, from \eqref{conv}, \eqref{Pars} and \eqref{Example1}, we have
\begin{equation*}
\begin{aligned}
&\Lambda(F,G,H)\simeq \sum_{n,k\in\Z}a_{n-k} b_k c_n ,\\
&|\Lambda(F,G,H)|\leq \|F\|_{L^2(\TT)} \|G\|_{L^2(\TT)} \|H\|_{L^{\infty}(\TT)} .
\end{aligned}
\end{equation*}
\item on the other hand, for generic $\|F\|_{L^2(\TT)}=\|G\|_{L^2(\TT)}=\|H\|_{L^{\infty}(\TT)}=1$
\begin{equation*}
\sum_{n,k\in\Z}|a_{n-k}| |b_k| |c_n|=\infty.
\end{equation*}
\end{itemize}

As already hinted earlier, the valuable lesson that we are learning from the above example is that the sums involving (product) of Fourier coefficients (of some functions) are encoding hidden cancellations that can be often exploited by reframing the problem from the Fourier (frequency) setting into a spatial setting.

With this in mind, we return to the  high resolution wave-packet discretized model as reflected in \eqref{eq:model0} and notice that we are in a very similar situation: our model may be visualised as a superposition of \emph{product of local Fourier coefficients} attached to $f$ and $g$ (and with a modified weight for $h$ if regarded in the dualized version).

Thus, judging by analogy with the expression in \eqref{Pars}, we are invited to consider\footnote{In this motivational argument we are fixing all the parameters but $u$.} the following elementary block within \eqref{eq:model0}:
\begin{equation}
\begin{aligned}
 S_{k, \ell,r}^{n, v, p}(f, g)(x):= \sum_{u\in\Z} \frac{1}{2^{\frac{am+2k}{4}}}  \langle f_{\ell+1} , \check{\varphi}_{\frac{am}{2}- k, \ell}^{u-v, p_r-n}  \rangle    \langle g_{\ell-1}, \check{\varphi}_{\frac{am}{2}- k, \ell}^{u+v, p_r+n}  \rangle  \check{\varphi}_{\frac{am}{2}- k, 2\ell-1}^{2u, p_r}(x) ,
\end{aligned} \label{eq:def:S_n,v,p}
\end{equation}
whose joint frequency support regarded from the input side, \textit{i.e.} the support of  $(\hat{f}, \hat{g})$, is depicted in Figure \ref{figure:support_S^nvp}.\\

Recalling now \eqref{def:gdec3} and \eqref{eq:modified_wave_packet} we further deduce that
\begin{equation}\label{Fcoef}
\begin{aligned}
 & S_{k, \ell,r}^{n, v, p}(f, g)(x) \\
 &= \frac{1}{2^k} \Big(  \sum_{u\in\Z} \langle f_{\ell+1} , \check{\varphi}_{\frac{am}{2}- k, \ell}^{u-v, p_r-n}  \rangle  \langle g_{\ell-1}, \check{\varphi}_{\frac{am}{2}- k, \ell}^{u+v, p_r+n}  \rangle  e^{i (2^{\frac{am}{2}-k} x-p_r)2u }  \Big) \\
& \hspace{2em}\cdot \check {\varphi}(2^{\frac{a m}{2}-k} x - p_r)   e^{i 2^{\frac{am}{2}}(2^{\frac{am}{2}-k} x-p_r) (2\ell -2)}.
\end{aligned}
\end{equation}

%%%%%%%%%%%%%%%%%%%%%%%%%%%%%%%%%%%%%%%%%%%
%
% FIGURE 4 - picture of $S_{k, \ell,r}^{n, v, p}$
%
%%%%%%%%%%%%%%%%%%%%%%%%%%%%%%%%%%%%%%%%%%%
\begin{figure}[ht]
\centering
\input{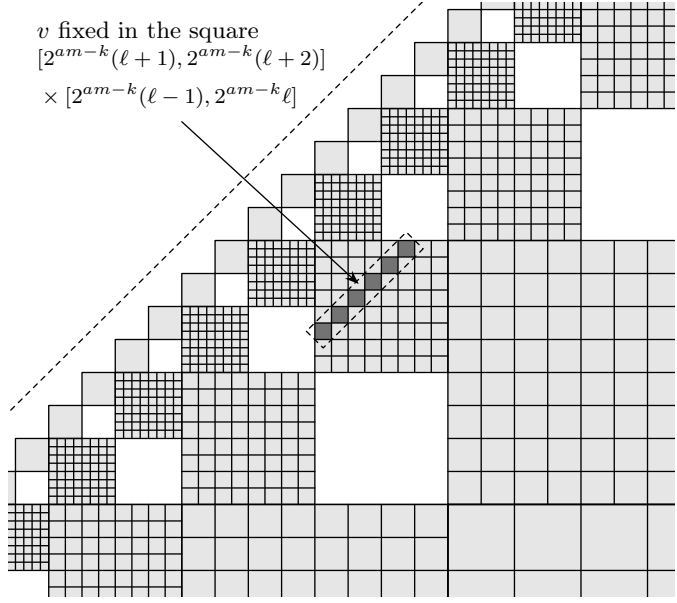}
\caption{\footnotesize Frequency portrait of the information associated with $S_{k,\ell,r}^{n,v,p}$. It is at this level that we manage to exploit the curvature of the phase (see Section \ref{sec:onescale}) and also obtain the frequency localization of $h \cdot w_{k, n, v}^e$ (see Section \ref{sec:HR2}).} \label{figure:support_S^nvp}
\end{figure}
%%%%%%%%%%%%%%%%%%%%%%%%%%%
%

Analyzing the right-hand side of \eqref{Fcoef} we deduce the following time-frequency characteristics:
\begin{itemize}
\item in space,  thanks to the presence of $\check{\varphi}$, the information is concentrated within $ I_k^{p_r}$;
\item in frequency -- see Figure \ref{figure:support_S^nvp} -- the joint support in $(\xi,\eta)$ encoded within the wave-packets contributing to the sum stays within the strip $$\xi-\eta = -2^{\frac{am}{2}-k}   2v + O(2^{\frac{am}{2}-k}) .$$
\end{itemize}
Thus, inspecting \eqref{Fcoef} in parallel with \eqref{Pars}, intuitively we expect
\begin{equation}\label{spacetransl}
S_{k, \ell,r}^{n, v, p}(f,g)(x)\approx\frac{1}{2^k} e^{2inv} \Big(\int_{I_k^n} f_{\ell +1}(x-t) g_{\ell -1}(x+t)  e^{-i 2^{\frac{am}{2}-k} 2v t} dt \Big) \one_{I_k^{p_r}} (x) .
\end{equation}
The rigorous realization of \eqref{spacetransl} is the subject of our next section.

\subsubsection{Transition from the discrete phase-linearized wave-packet model to the continuous phase-linearized spatial model}\label{passage}

In this section we focus on making \eqref{spacetransl} precise. For reasons that will become transparent momentarily, we will reshape the information in \eqref{eq:def:S_n,v,p} by isolating the Gabor coefficient of $f$, which we write as:
\[ \langle f_{\ell+1} , \check{\varphi}_{\frac{am}{2}- k, \ell}^{u-v, p_r-n}  \rangle
=\int_{\R} f_{\ell+1}(s)  \overline{\check{\varphi}_{\frac{am}{2}- k, \ell}^{0, p_r-n}}(s) e^{- i (2^{\frac{am}{2}- k} s-(p_r-n)) (u-v)} \,ds. \]
After the change of variable $2^{\frac{am}{2}- k} s-(p_r-n)=y$ and a periodization argument (similar with the one used in order to prove the Poisson summation formula) this becomes
\[ \frac{1}{2^{\frac{am-2k}{4}}}  \int_{\TT} \Big(\sum_{j_1\in\Z} f_{\ell+1}\big(\frac{y+j_1+(p_r-n)}{2^{\frac{am}{2}- k}}\big)  \overline{\check{\varphi}_{0, \ell}^{0, 0}}(y+j_1)\Big) e^{- i y (u-v)} dy  , \]
where here $\check{\varphi}_{0, \ell}^{0, 0}(y)=\check{\varphi}(y)   e^{i y  2^{\frac{am}{2}}( \ell -1)}$.

Applying a similar reasoning for the coefficient $\langle g_{\ell-1}, \check{\varphi}_{\frac{am}{2}- k, \ell}^{u+v, p_r+n}  \rangle$ we can now rewrite \eqref{eq:def:S_n,v,p} as
\begin{equation}\label{eq:def:S_n,v,p3}
\begin{aligned}
&S_{k, \ell,r}^{n, v, p}(f, g)(x)= \frac{1}{2^{\frac{am+2k}{4}}}  \frac{1}{2^{\frac{am}{2}- k}}  \check{\varphi}_{\frac{am}{2}- k, 2\ell-1}^{0, p_r}(x) \\
\qquad &\cdot \sum_{u\in\Z} \Big(\int_{\TT} \Big(\sum_{j_1\in\Z} f_{\ell+1}\big(\frac{y+j_1+(p_r-n)}{2^{\frac{am}{2}- k}}\big)  \overline{\check{\varphi}_{0, \ell}^{0, 0}}(y+j_1)\Big) e^{- i y (u-v)} dy \Big)\\
\qquad &\cdot  \Big(\int_{\TT}\Big(\sum_{j_2\in\Z} g_{\ell-1}\big(\frac{y+j_2+(p_r+n)}{2^{\frac{am}{2}- k}}\big)  \overline{\check{\varphi}_{0, \ell}^{0, 0}}(y+j_2)\Big) e^{- i y (u+v)} dy \Big)  e^{i 2u (2^{\frac{a m}{2}-k} x - p_r)} .
\end{aligned}
\end{equation}
%we are making use of the following
Once at this point, observe that we have the following identity: if $a, b \in L^2(\TT)$, then
\begin{equation}\label{keyblinFouriercoef}
\begin{aligned}
&\sum_{u\in\Z} \Big(\int_{\TT} a(y) e^{- i y (u-v)} dy \Big)\cdot\Big(\int_{\TT} b(y) e^{- i y (u+v)} dy \Big)  e^{i 2u x}\\
&=\int_{\TT} a(x-t) b(x+t) e^{-2 i v t} dt .
\end{aligned}
\end{equation}
Applying now \eqref{keyblinFouriercoef} for $a(y):=\sum_{j_1\in\Z} f_{\ell+1}\big(\frac{y+j_1+(p_r-n)}{2^{\frac{am}{2}- k}}\big)  \overline{\check{\varphi}_{0, \ell}^{0, 0}}(y+j_1)$ and $b(y):=\sum_{j_2\in\Z} g_{\ell-1}\big(\frac{y+j_2+(p_r+n)}{2^{\frac{am}{2}- k}}\big)  \overline{\check{\varphi}_{0, \ell}^{0, 0}}(y+j_2)$ and inserting the result in \eqref{eq:def:S_n,v,p3} we deduce
\begin{equation}\label{eq:def:S_n,v,p4}
\begin{aligned}
&S_{k, \ell,r}^{n, v, p}(f, g)(x)= \sum_{j_1, j_2\in\Z} \frac{1}{2^{\frac{am+2k}{4}}}  \frac{1}{2^{\frac{am}{2}- k}}  \check{\varphi}_{\frac{am}{2}- k, 2\ell-1}^{0, p_r}(x) \\
&\cdot \int_{\TT} \Big( f_{\ell+1}\big(\frac{2^{\frac{a m}{2}-k} x - p_r-t+j_1+(p_r-n)}{2^{\frac{am}{2}- k}}\big)  \overline{\check{\varphi}_{0, \ell}^{0, 0}}(2^{\frac{a m}{2}-k} x - p_r-t+j_1)\Big)\\
&\cdot  \Big( g_{\ell-1}\big(\frac{2^{\frac{a m}{2}-k} x - p_r+t+j_2+(p_r+n)}{2^{\frac{am}{2}- k}}\big)  \overline{\check{\varphi}_{0, \ell}^{0, 0}}(2^{\frac{a m}{2}-k} x - p_r+t+j_2)\Big) e^{- 2 i v t} dt .
\end{aligned}
\end{equation}

Finally, after applying the change of variable $t \mapsto 2^{\frac{a m}{2}-k} t-n$, we conclude
\begin{equation}\label{eq:def:S_n,v,pfinal}
\begin{aligned}
 &S_{k, \ell,r}^{n, v, p}(f, g)(x)= \sum_{j_1,j_2\in\Z} \frac{1}{2^{k}} \check{\varphi}(2^{\frac{am}{2}-k} x - p_r)   e^{i (2^{\frac{a m}{2}-k} x - p_r) 2^{\frac{am}{2}}( 2\ell -2)} \\
 &\cdot \int_{I^n_k}  f_{\ell+1}\big(x-t+\frac{j_1}{2^{\frac{am}{2}- k}}\big)  \overline{\check{\varphi}_{0, \ell}^{0, 0}}(2^{\frac{a m}{2}-k} (x-t) - p_r+n+j_1) \\
&\cdot g_{\ell-1}\big(x+t+\frac{j_2}{2^{\frac{am}{2}- k}}\big)
\overline{\check{\varphi}_{0, \ell}^{0, 0}}(2^{\frac{a m}{2}-k}(x+t) - p_r-n+j_2) e^{- 2 i v (2^{\frac{a m}{2}-k} t-n)} \,dt\\
&=:\sum_{j_1,j_2\in\Z} \tilde{S}_{k,p_r}^{n, v}[j_1,j_2](f_{\ell+1}, g_{\ell-1})(x) .
\end{aligned}
\end{equation}

\begin{observation}\label{convred}
Formulation \eqref{eq:def:S_n,v,pfinal} ends the rigorous transition from the discrete wave-packet model to the continuous spatial model. As already mentioned, we will make use of it in Section \ref{sec:onescale} when proving the single estimate \eqref{sgsc} and then in Section \ref{sec:HR2} when working at a refined version of \eqref{sgsc}. In all these instances one could work directly with the formulation \eqref{eq:def:S_n,v,pfinal} proved above. However, for notational simplicity and fluency of the exposition we will reduce \eqref{eq:def:S_n,v,pfinal} to a simpler, approximative form -- represented by \eqref{eq:def:S_n,v,pFinal} below -- and this last formulation will be embraced from now on
in our paper. While we will present an outline of why indeed it is enough to only work with \eqref{eq:def:S_n,v,pFinal} we will leave the technical details to the reader.
\end{observation}

In light of the above observation, we now provide a brief argumentation for the following claim:

If $x\in I^{p_r}_k$ then without loss of generality one can assume from now on that
\begin{equation}\label{eq:def:S_n,v,pFinal}
 S_{k, \ell,r}^{n, v, p}(f, g)(x)= \tilde{S}_{k}^{n, v}(f_{\ell+1,r}, g_{\ell-1,r})(x)  ,
\end{equation}
with
\begin{equation} \tilde{S}_{k}^{n, v}(F,G)(x):=\frac{1}{2^{k}}  \int_{I^n_k} F\big(x-t\big)  G\big(x+t\big) e^{- 2 i v (2^{\frac{a m}{2}-k} t-n)} dt. \label{eq:tildeS} \end{equation}
and
\begin{equation}\label{eq:def:S_n,v,pFinalprep}
\begin{aligned}
%&\tilde{S}_{k,p_r}^{n, v}(F,G)(x):=\frac{1}{2^{k}} \check{\varphi}(2^{\frac{am}{2}-k} x - p_r)  \int_{I^n_k} F\big(x-t\big)  G\big(x+t\big) e^{- 2 i v (2^{\frac{a m}{2}-k} t-n)} dt ,\\
f_{\ell+1,r}:= \sum_{u\in\Z} \sum_{p \sim 2^{am/2}} \langle f_{\ell+1}, \check \varphi_{\frac{am}{2}-k, \ell}^{u, p_r}   \rangle  \check \varphi_{\frac{am}{2}-k, \ell}^{u, p_r}= \sum_{u,p \sim 2^{am/2}} \langle f_{\ell+1}, \check \varphi_{\frac{am}{2}-k, \ell}^{u, p_r}   \rangle  \check \varphi_{\frac{am}{2}-k, \ell}^{u, p_r} ,\\
g_{\ell-1,r}:= \sum_{u\in\Z} \sum_{p \sim 2^{am/2}} \langle g_{\ell-1}, \check \varphi_{\frac{am}{2}-k, \ell}^{u, p_r}   \rangle  \check \varphi_{\frac{am}{2}-k, \ell}^{u, p_r}= \sum_{u,p \sim 2^{am/2}} \langle g_{\ell-1}, \check \varphi_{\frac{am}{2}-k, \ell}^{u, p_r}   \rangle  \check \varphi_{\frac{am}{2}-k, \ell}^{u, p_r} .
\end{aligned}
\end{equation}

Indeed, we first start by noticing that since $\check{\varphi}_{0, 1}^{0, 0}(y)=\check{\varphi}(y)$, we have
\begin{equation*}
\begin{aligned}
&\overline{\check{\varphi}_{0, \ell}^{0, 0}}(2^{\frac{a m}{2}-k} (x-t) - p_r+n+j_1) \overline{\check{\varphi}_{0, \ell}^{0, 0}}(2^{\frac{a m}{2}-k}(x+t) - p_r-n+j_2)\\
= &\overline{\check{\varphi}}(2^{\frac{a m}{2}-k} (x-t) - p_r+n+j_1)  \overline{\check{\varphi}}(2^{\frac{a m}{2}-k}(x+t) - p_r-n+j_2) \\
& \qquad \qquad \qquad  \qquad \qquad \qquad  \cdot  e^{-i (2(2^{\frac{a m}{2}-k} x-p_r)+j_1+j_2)  2^{\frac{am}{2}}( \ell -1)} .
\end{aligned}
\end{equation*}

Isolating now the first term in the last line above and ignoring the conjugation we use the information that $\check{\varphi}$ is compactly supported in frequency for applying a standard Taylor series argument -- here $x\in I^{p_r}_k$ is assumed to be fixed -- that has as an effect the tensorization (separation of variables in $x$ and $t$):
\begin{equation*}
\begin{aligned}
\check{\varphi}(2^{\frac{a m}{2}-k} (x-t) - p_r+n+j_1)=\int_{\R} e^{i  \xi (2^{\frac{a m}{2}-k} (x-t) - p_r+n+j_1)} \varphi(\xi) d\xi\\
=\sum_{k_1,k_2\in\Z} \int_{\R} \frac{\left(i \xi (2^{\frac{a m}{2}-k} x - p_r+1)\right)^{k_1}}{k_1!}
\frac{\left(-i \xi (2^{\frac{a m}{2}-k} t - n +1)\right)^{k_2}}{k_2!} e^{i \xi j_1} \varphi(\xi) d\xi\\
=\sum_{k_1,k_2\in\Z} \frac{\left(2^{\frac{a m}{2}-k} x - p_r+1 \right)^{k_1}}{k_1!}
\frac{\left(2^{\frac{a m}{2}-k} t - n+1\right)^{k_2}}{k_2!} \check{\varphi}_{k_1,k_2}(j_1) ,
\end{aligned}
\end{equation*}
where in the last line $\check{\varphi}_{k_1,k_2}$ is a Schwartz function depending on $k_1, k_2$ such that $|\check{\varphi}_{k_1,k_2}(j_1)|\lesssim \frac{C^{k_1+k_2}}{|j_1|^{10}+1}$ with $C>0$ a suitable absolute constant.

Deduce thus from here that for any $j_1, j_2\in\Z$ and $x\in I^{p_r}_k$ we have
\begin{equation}\label{exactform}
\begin{aligned}
 & \qquad\qquad\qquad\qquad\qquad\tilde{S}_{k,p_r}^{n, v}[j_1,j_2](f_{\ell+1}, g_{\ell-1})(x)=\\
 &\frac{1}{2^{k}}   e^{-i (j_1+j_2) 2^{\frac{am}{2}}(\ell -1)} \sum_{{k_1,l_1\in\N}\atop{k_2,l_2\in\N}} \frac{\left(2^{\frac{a m}{2}-k} x - p_r +1 \right)^{k_1+l_1}}{k_1! l_1!} \check{\varphi}(2^{\frac{am}{2}-k} x - p_r) \check{\varphi}_{k_1,k_2}(j_1) \check{\varphi}_{l_1,l_2}(j_2) \\
 &\cdot \int_{I^n_k}  f_{\ell+1}\big(x-t+\frac{j_1}{2^{\frac{am}{2}- k}}\big)
 g_{\ell-1}\big(x+t+\frac{j_2}{2^{\frac{am}{2}- k}}\big)
e^{- 2 i v (2^{\frac{a m}{2}-k} t-n)} \frac{\left(2^{\frac{a m}{2}-k} t - n +1\right)^{k_2+l_2}}{k_2! l_2!} dt
\end{aligned}
\end{equation}

A simple inspection of \eqref{exactform} reveals the following:
\begin{itemize}
\item the presence of the term $\check{\varphi}_{k_1,k_2}(j_1) \check{\varphi}_{l_1,l_2}(j_2)$ provides a polynomial decay in terms of the size of the parameters $j_1, j_2$ thus allowing us to sum in $j_1$ and $j_2$ in \eqref{eq:def:S_n,v,pfinal};

\item for $x\in I^{p_r}_k$ the expressions $2^{\frac{a m}{2}-k} x - p_r +1$ and $\check{\varphi}(2^{\frac{am}{2}-k} x - p_r)$ are of size $\approx 1$;

\item for $t\in I^n_k$ the expression $2^{\frac{a m}{2}-k} t - n +1$ is of size $\approx 1$.\footnote{It might be instructive to verify -- for example -- the proof of \eqref{triangineq} in Section \ref{sec:onescale} to see why factors such as $\frac{\left(2^{\frac{a m}{2}-k} t - n +1\right)^{k_2+l_2}}{k_2! l_2!}$ and $\frac{\left(2^{\frac{a m}{2}-k} x - p_r +1 \right)^{k_1+l_1}}{k_1! l_1!}$ are indeed harmless.}
\end{itemize}
The above explains why, for $x\in I^{p_r}_k$, the right-hand side in \eqref{eq:def:S_n,v,pfinal} can be reduced to the (dominant) term $\tilde{S}_{k,p_r}^{n, v}[0,0](f_{\ell+1}, g_{\ell-1})$ and moreover, why the latter can be further reduced to the expression
$\tilde{S}_{k}^{n, v}(f_{\ell+1}, g_{\ell-1})$ which, given the imposed range for $x$, is well approximated\footnote{Strictly speaking $\one_{[r 2^k, (r+1) 2^k)}(x) \tilde{S}_{k}^{n, v}(f_{\ell+1}, g_{\ell-1})(x)$ is a superposition of terms of the form  $c_{r-r'} \tilde{S}_{k}^{n, v}(f_{\ell+1,r'}, g_{\ell-1,r'})(x)$, $r'\in\Z$ with $\{c_{r'}\}_{r'\in\Z}\in \ell^{\infty}(\Z)$ a fast-decaying sequence obeying $c_{r'}=O(|r'|^{-2})$ as $|r'|\rightarrow \infty$.} by $\tilde{S}_{k}^{n, v}(f_{\ell+1,r}, g_{\ell-1,r})(x)$, thus certifying \eqref{eq:def:S_n,v,pFinal}.

\subsubsection{The continuous phase-linearized spatial model revisited -- bird's eye view. The heuristic equivalence of the two high resolution models}\label{BEV}

This section is meant solely for presenting an alternate, more direct way of visualizing our continuous model that in particular reveals the important role played by the linearization process. Indeed, returning to the original form of our operator $T_{m,k}$ we show the very natural way in which one can achieve a decomposition of $T_{m,k}$ in elementary blocks resembling the expression in the right-hand side of \eqref{spacetransl}. Since the previous section was covered in great detail, the description provided here will be expository, and hence stripped of technicalities.

These being said, let us recall the original formulation
\begin{equation}\label{tmkorig}
\begin{aligned}
T_{m,k}(f,g)(x)&=\rho_{a m- a k}(\l(x)) \int f(x-t)g(x+t) e^{i\lambda(x) t^a} \rho(2^{-k}t) \frac{dt}{t}\\
&\approx \frac{1}{2^k} \rho_{a m- a k}(\l(x)) \int_{[2^k, 2^{k+1}]}  f(x-t)  g(x+t)  e^{i\lambda(x) t^a} dt .
\end{aligned}
\end{equation}
Following the same phase linearization approach as the one in Section \ref{PL} with the only distinction that we are now working in space as opposed to frequency there, we first isolate the phase of our kernel
$$\tau_x(t):=\lambda(x) t^a ,$$
and notice that in the non-trivial regime deduced from \eqref{tmkorig} we must have $|t|\approx 2^k$ and $|\l(x)|\approx  2^{a(m-k)}$. As a consequence, we further have that
$$\tau_x''(t)=a (a-1) \lambda(x) t^{a-2}\approx_a 2^{am-2k} ,$$
and hence, in order to focus only on the linear growth around a given point $t_0\in [2^k, 2^{k+1}]$ one must have the size of the spatial region restricted to $|t-t_0|\lesssim 2^{k} 2^{-\frac{am}{2}}$.

Thus, recalling that $I^{n}_{k}$ stands for $\Big[ \frac{n}{2^{\frac{am}{2}-k}},  \frac{n+1}{2^{\frac{am}{2}-k}} \Big]$, we have that
\begin{equation}\label{tmkorig1}
T_{m,k}(f,g)(x)\approx \frac{1}{2^k} \rho_{a m- a k}(\l(x)) \sum_{n\sim 2^{\frac{am}{2}}}\int_{I^n_k} f(x-t)  g(x+t)  e^{i\lambda(x) t^a} dt .
\end{equation}
Given the very construction of $I^{n}_{k}$, taking  $t^n_k:=n 2^{k-\frac{am}{2}}$ and applying a Taylor series argument around the chosen point $t^n_k$, we have that
\beq\label{t04}
\tau_x(t)=\lambda(x) t^a= \lambda(x) (t^n_k)^a + (t-t^n_k) a \l(x) (t^n_k)^{a-1} +  O(1) ,
\eeq
as long as $|\l(x)|\approx 2^{a(m-k)}$ and $t\in  I^{n}_{k}$.

Thus, ignoring\footnote{This can be done rigorously via a Taylor series argument.} the $O(1)$ term appearing in the phase, we have reached
the final form of our (continuous) spatial linearized model which is given by
\begin{equation}\label{tmkspacelinearmodel}
\begin{aligned}
T_{m,k}(f,g)(x)\approx
\frac{1}{2^k} \rho_{a m- a k}(\l(x)) &\sum_{p\in\Z} \sum_{n\sim 2^{\frac{am}{2}}} e^{i (1-a) \lambda(x)  (t^n_k)^a} \chi_{I^p_k}(x)\\
&\cdot\Big(\int_{I^n_k} f(x-t)  g(x+t)  e^{i a \lambda(x) (t^n_k)^{a-1}  t} dt\Big) .
\end{aligned}
\end{equation}

We end this section with an informative discussion on how can one visualize the equivalence between the continuous model expressed above in \eqref{tmkspacelinearmodel} and the discretized wave-packet model expressed in \eqref{eq:model0}. Starting from the latter and employing definition \eqref{eq:def:S_n,v,p} we have
\begin{equation}
\begin{aligned}
T_{m,k}(f, g)(x)=& \sum_{\ell,r \in \Z} \sum_{n,p \sim 2^\frac{am}{2}}  \sum_{u, v \sim 2^\frac{am}{2}} \langle f_{\ell+1}, \check{\varphi}_{\frac{am}{2}- k, \ell}^{u-v,p_r-n}  \rangle   \langle g_{\ell-1}, \check{\varphi}_{\frac{am}{2}- k, \ell}^{u+v,p_r+n}  \rangle  \\
& \cdot \frac{\rho_{am-ak}(\l(x)) }{2^{\frac{am+2k}{4}}}    \check{\varphi}_{\frac{am}{2}- k, 2\ell-1}^{2u, p_r}(x) w_{k, n,v}^e(\l)(x)\\
=&  \sum_{\ell,r \in \Z} \sum_{n,p \sim 2^\frac{am}{2}} \sum_{v \sim 2^\frac{am}{2}} S_{k, \ell,r}^{n, v, p}(f, g)(x)  \rho_{am-ak}(\l(x))  w_{k, n,v}^e(\l)(x) .
\end{aligned}\label{eq:model01}
\end{equation}

Now based on the heuristic in \eqref{spacetransl} we further have
\begin{equation}\label{eq:model02}
\begin{aligned}
T_{m,k}(f, g)(x)\approx &\sum_{\ell,r \in \Z} \sum_{n,p \sim 2^\frac{am}{2}} \sum_{v \sim 2^\frac{am}{2}} \one_{I_k^{p_r}} (x) \rho_{am-ak}(\l(x))  w_{k, n,v}^e(\l)(x)\\
&\cdot \frac{1}{2^k} e^{2inv} \left(\int_{I_k^n} f_{\ell +1}(x-t) g_{\ell -1}(x+t)  e^{-i 2^{\frac{am}{2}-k} 2v t} dt \right) .
\end{aligned}
\end{equation}
Once at this point, relying on the time-frequency correlation as captured by \eqref{eq:decay-weight}, we notice that the main contribution is attained when
\begin{equation}\label{tfcor}
-2v\approx \bar c_a \frac{n^{a-1}  \l(x)}{2^{a(\frac{am}{2}-k)}} ,
\end{equation}
while from relations \eqref{eq:w_k}--\eqref{eq:decay-weight} via a stationary phase argument we deduce
\begin{equation}\label{stph}
\rho_{am-ak}(\l(x)) w_{k, n,v}^e(\l)(x)\approx \rho_{am-ak}(\l(x))  e^{i (1-\frac{1}{a'}) \bar  c_a \frac{n^{a}  \l(x)}{2^{a(\frac{am}{2}-k)}}}   \varphi \big( \xi + \bar c_a \frac{n^{a-1}  \l(x)}{2^{a(\frac{am}{2}-k)}} \big) .
 \end{equation}
Putting together \eqref{eq:model02}, \eqref{tfcor} and \eqref{stph}, at the informal level, we conclude

\begin{equation}\label{eq:model03}
\begin{aligned}
T_{m,k}(f, g)(x)&\approx \sum_{\ell,r \in \Z} \sum_{n,p \sim 2^\frac{am}{2}} \one_{I_k^{p_r}} (x) \rho_{am-ak}(\l(x)) \frac{1}{2^k} e^{-i \frac{1}{a'} \bar c_a \frac{n^{a}  \l(x)}{2^{a(\frac{am}{2}-k)}}}\\
& \hspace{3em} \cdot \Big(\int_{I_k^n} f_{\ell +1}(x-t) g_{\ell -1}(x+t)  e^{i 2^{\frac{am}{2}-k} \bar c_a \frac{n^{a-1}  \l(x)}{2^{a(\frac{am}{2}-k)}} t} dt \Big) ,
\end{aligned}
\end{equation}
thus certifying for $\bar c_a=a$ the equivalence with the formulation given in \eqref{tmkspacelinearmodel} as claimed.

\subsection{A brief contrasting conclusion regarding the two high resolution models} \label{HR3}

Our intention here is to summarize the need for utilizing both phase-linearized high-resolution models constructed earlier:
\begin{itemize}
\item On the one hand -- see Section \ref{sec:onescale}, the integral formulation in \eqref{tmkspacelinearmodel} captures enough cancellation to yield the desired single scale version of \eqref{mainrm}, that is
\beq\label{mainrmk}
\|T_{m,k}(f,g)\|_{L^r}\lesssim_{a,p,q} 2^{-\d a m}  \|f\|_{L^p} \|g\|_{L^q} ,
\eeq
even if one takes absolute values in the summand of \eqref{tmkspacelinearmodel}.
 
As already explained, \eqref{mainrmk} or its dual correspondent cannot be though attained if one ignores the signs of the Gabor coefficients in either \eqref{eq:model} or in its dual version \eqref{def:trilinear:form0}, respectively.

\item  On the other hand -- see Sections \ref{sec:HR2} and \ref{sec:bilinear:analysis}, for upgrading the single scale estimate \eqref{mainrmk} to the global estimate \eqref{mainrm} one has to involve the interactions among the wave-packet associated with the input functions $f,g$ and $h$, thus necessitating the use of the first model operator discussed in Section \ref{HR1} that will now need to be reshaped in order to exploit the BHT-like structure of our operator -- this latter key step is part of the design of the low resolution model introduced in the next section.
\end{itemize}

\subsection{The low resolution model: identifying the underlying BHT structures}\label{LR}

In this section we highlight the built-in BHT frame as suggested by the zero-curvature features \ref{key:td} and \ref{key:m} and set the scene for the multi-scale time-frequency analysis that will be performed later on in Section \ref{sec:bilinear:analysis}. Based on the previous reasonings, we are naturally led to consider two kinds of time-frequency tiling:
\begin{itemize}
\item the first one captures the local, single-scale behavior\footnote{As the name suggests, this discretization that is part of the high resolution analysis concerns the local properties of our operator when restricted to a single ``big frequency square'' -- see Figures \ref{figure:BigCubes} and \ref{figure:SmallCubes}. This process has as a result the subdivision of such a big frequency square into smaller, same-scale frequency squares. Of course, each such big frequency square comes with its own family of smaller, uniform-in-scale frequency squares and as a result at the global level we will deal with a multiscale problem whose aim is to quantify the interactions among these families of smaller frequency squares.} of our operator and involves tiles of area $1$; this model is simply the geometric, tile rephrasing of the high resolution wave-packet model discussed in Section \ref{HR1}.

\item the second one accounts for the global, multi-scale behavior of our operator and consists of tiles having area  $2^{am}$. This model, referred to as the low resolution wave-packet model, is the expression of the BHT features of $BC^a$, and amounts to regrouping the terms within the first model above in order to later exploit the almost orthogonality along multiple scales via tree structures.
\end{itemize}

\subsubsection{The bilinear Hilbert transform frequency portrait: identifying the structure of the time-frequency regions}
The initial decomposition ${1 \over t }= \sum_{k \in \Z} {\rho(2^{-k}t) \over t}$ of the singular kernel, followed by the level set decomposition of the phase of the multiplier for the high oscillatory component, and the stationary phase analysis led to the study of the operator $T_m=\sum_{k \in \Z} T_{m, k}$, with each individual term $T_{m, k}$ (as defined in \eqref{def:disc:Tm,k}) having frequency information supported in the region $\{ \xi-\eta \sim 2^{am-k}\}$. The subsequent step, designed for independently extracting the frequency information for $f$ and $g$, led to considering ``big square'' frequency regions
\begin{equation}
\label{eq:def:W:1}
\big[ 2^{am-k} (\ell+1), 2^{am-k} (\ell+2) \big] \times  \big[ 2^{am-k} (\ell-1), 2^{am-k} \ell\big].
\end{equation}
As $k$ and $\ell$ vary over the set of integers, this collection of squares identifies with (a part of) a Whitney decomposition of the region $\{ \xi > \eta \}$ that may be interpreted as a $2^{am}$-dilation of the classical frequency portrait of the bilinear Hilbert transform.\footnote{The multiplier associated to BHT is $sgn(\xi-\eta)$, and $\{ \xi=\eta\}$ represents its singularity set relative to which one performs the above mentioned Whitney decomposition.}  Thus, inspired by and in parallel with the standard BHT analysis, we expect the collection induced by the squares in \eqref{eq:def:W:1} to play an analogue role in the analysis of our operator $BC^a$.

However, unlike the BHT case, when analyzing the information contained in the above frequency regions \eqref{eq:def:W:1} we will not involve wave-packets adapted to the scale of these big squares. Instead, in order to be able to exploit the curvature component of our $BC^a$ which is responsible for the strategy presented in Section \ref{Heur} and particularly \eqref{mainrm}, our time-frequency analysis relies on wave-packets adapted to the smaller scale $2^{{am \over 2}-k}$ as indicated by the LGC methodology.  This produces
\begin{equation}
\begin{aligned}
T_{m,k,\ell}(f, g)(x)= \sum_{\substack{n \sim 2^\frac{am}{2} \\ \bar p \in \Z}}  \sum_{u, v \in\Z} \langle f_{\ell+1}, \check{\varphi}_{\frac{am}{2}- k, \ell}^{u-v,\bar p-n}  \rangle   \langle g_{\ell-1}, \check{\varphi}_{\frac{am}{2}- k, \ell}^{u+v,\bar p+n}  \rangle &  \\
 \cdot \frac{\rho_{am-ak}(\l(x)) }{2^{\frac{am+2k}{4}}}    \check{\varphi}_{\frac{am}{2}- k, 2\ell-1}^{2u, \bar p}(x) w_{k, n,v}^e(\l)(x)&,
\end{aligned}\label{eq:model:W:2}
\end{equation}
which has the frequency information supported in \eqref{eq:def:W:1}. In order to control the information contained in the Whitney squares, a certain spatial localization that indicates an (almost) orthogonality relation between objects having the same frequency support is necessary.

Notice that for every $\bar p \in \Z$, the spatial information for $f$ is adapted to $\Big[ \frac{\bar p-n}{2^{{am \over 2}-k}}, \frac{\bar p-n+1}{2^{{am \over 2}-k}}  \Big]$, and that of $g$ to $\Big[ \frac{\bar p+n}{2^{{am \over 2}-k}}, \frac{\bar p+n+1}{2^{{am \over 2}-k}} \Big]$. Since $n \sim 2^{am \over 2}$, a $2^{am \over 2}$-separation between the values of $\bar p$ warrants separation for not only the spatial information for $f$, but also for that of $g$ and of the output $T_{m, k, \ell}$. This leads to the regrouping mentioned in \eqref{eq:model0}, the definition \eqref{eq:not:p_r} and the representation of $T_m$ in the form of $T_m(f, g)=\sum_{k, \ell, r \in \Z} T_{m, k , \ell, r}(f, g)$, where
\begin{equation}
\begin{aligned}
T_{m,k,\ell,r}(f,g)(x) &:= \sum_{\substack{p \sim 2^{\frac{am}{2}} \\ n \sim 2^{\frac{am}{2}}}} \sum_{\substack{u \sim 2^{\frac{am}{2}} \\ v \sim 2^{\frac{am}{2}}}} \frac{1}{2^{\frac{am+2k}{4}}} \langle f_{\ell+1}, \check{\varphi}_{\frac{am}{2}- k, \ell}^{u-v, p_r-n}  \rangle   \langle g_{\ell-1}, \check{\varphi}_{\frac{am}{2}- k, \ell}^{u+v, p_r+n}  \rangle \\
& \hspace{7em} \cdot \check{\varphi}_{\frac{am}{2}- k, 2\ell-1}^{2u, p_r}(x) \rho_{am-ak}( \lambda(x)) w_{k, n, v}^e(\l)(x) .
\end{aligned}\label{eq:model:with:r}
\end{equation}

The spatial information for $f$ relevant to $T_{m,k,\ell,r}$ is thus concentrated on
\[
\bigcup_{p, n \sim 2^{am \over 2}} \Big[ \frac{p_r-n}{2^{{am \over 2}-k}}, \frac{p_r-n+1}{2^{{am \over 2}-k}}  \Big],
\]
which can be approximated\footnote{Up to considering $O(1)$ similar terms.} by $[r 2^k, (r+1) 2^k]$.
Similarly, the spatial information for $g$ is, for the most part, contained in
\[
\bigcup_{p, n \sim 2^{am \over 2}} \Big[ \frac{p_r+n}{2^{{am \over 2}-k}}, \frac{p_r+n+1}{2^{{am \over 2}-k}}  \Big],
\]
and this is also comparable to $[r 2^k, (r+1) 2^k]$.

In conclusion, for different values of $r$, the spatial information for $f$ encoded in $T_{m,k,\ell,r}(f,g)$ is located on disjoint intervals; the same holds for $g$ and for the output $T_{m,k,\ell,r}(f,g)$. On the other hand, for a fixed $r$, the spatial information for $f$, $g$ and the output $T_{m,k,\ell,r}(f,g)$ is adapted to the same interval $[r 2^k, (r+1) 2^k]$.

The regrouping in \eqref{eq:model0}, which consisted in gathering $2^{am \over 2}$ consecutive integer values of $p$, had two important consequences: (i) the time-frequency information for any of $f, g$ or $T_{m,k,\ell,r}(f,g)$ featured in $T_{m,k,\ell,r}$, uniquely determines the time-frequency information for the remaining input/output data; (ii) as $\ell$ and $r$ vary, the time-frequency information is associated to disjoint regions.

So the analogue of the tri-tiles appearing in the BHT analysis are precisely the time-frequency regions of area $2^{am}$
\begin{equation}
\begin{aligned}
P_1 & := I_P \times \pmb{\omega}_{P_1} := [r 2^k, (r+1) 2^k] \times \big[ 2^{am-k}(\ell+1), 2^{am-k}(\ell+2) \big]  \\
P_2 & := I_P \times \pmb{\omega}_{P_2} := [r 2^k, (r+1) 2^k] \times \big[ 2^{am-k} (\ell-1), 2^{am-k} \ell \big],\\
P_3 & := I_P \times \pmb{\omega}_{P_3} := [r 2^k, (r+1) 2^k] \times \big[ 2^{am-k} (2 \ell), 2^{am-k} (2\ell+1) \big],
\end{aligned}\label{eq:tri:tiles:def:S4}
\end{equation}
where $k, \ell, r$ vary in $\Z$. We denote $\BHT$ the collection of such tri-tiles $P:=(P_1, P_2, P_3)$, in clear analogy to the bilinear Hilbert transform operator:
\begin{equation}
 \BHT := \{ P=(P_1, P_2, P_3) | \text{ with } P_1, P_2, P_3 \text{  given by  } \eqref{eq:tri:tiles:def:S4}, \text{  for  } k,\ell,r \in \mathbb{Z}\}. \label{def:BHT}
\end{equation}
Since each of the $P_j$ has area $2^{am}$, we refer to them as ``expanded (tri-)tiles''.

In fact, $\BHT$ is characterized (such a collection is sometimes called of \emph{rank 1}) by the following properties:
\begin{itemize}
 \item[(i)] the collection of frequency squares $ (\pmb{\omega}_{P_1} \times \pmb{\omega}_{P_2})_{P\in \BHT}$ is a Whitney covering of the half-plane $\{\xi>\eta\}$, which means
\begin{equation}
\label{eq:Whitney}
\dist( \pmb{\omega}_{P_1} \times \pmb{\omega}_{P_2}, \Gamma)\sim |\pmb{\omega}_{P_1}|=|\pmb{\omega}_{P_2}|,
\end{equation}
where $\Gamma:=\{\xi=\eta\}$ is the boundary of the half-plane;
\item[(ii)] For $P\in \BHT$, the three tiles $P_1,P_2,P_3$ have the same spatial support $I_P$ and are all of area $2^{am}$;
\item[(iii)] For $P\in \BHT$, $\pmb{\omega}_{P_3} \sim \pmb{\omega}_{P_1} + \pmb{\omega}_{P_2}$, i.e. $\pmb{\omega}_{P_3}$ is an interval of the same length as $\pmb{\omega}_{P_1}$ and $\pmb{\omega}_{P_2}$, from the same dyadic lattice, and contained in $\pmb{\omega}_{P_1}+\pmb{\omega}_{P_2}$.
\end{itemize}

\subsubsection{Tile-rephrasing of the wave-packet information} In this section we reinterpret the discretization \eqref{eq:model:with:r} of our operator from a geometric point of view by phrasing the action of the wave-packets in terms of tiles.

We start by isolating the wave-packet $\check{\varphi}_{\frac{am}{2}- k, \ell}^{u-v, p_r-n}$ part of the term $\langle f_{\ell+1}, \check{\varphi}_{\frac{am}{2}- k, \ell}^{u-v, p_r-n}  \rangle$ in \eqref{eq:model:with:r}. We remark now that its time-frequency portrait is captured within the area one tile
\begin{align*}
s_1 &:= I^{p_r-n}_{k} \times \omega^{u_{\ell}-v}_{k} \\
&:= \Big[ (r-1) 2^k + \frac{p-n}{2^{\frac{am}{2}-k}},  (r-1) 2^k + \frac{p-n+1}{2^{\frac{am}{2}-k}} \Big] \\
& \hspace{3em} \times \Big[\frac{\ell}{ 2^{-(am -k)}}+ \frac{u-v}{2^{-(\frac{am}{2} -k)}}, \frac{\ell}{ 2^{-(am -k)}}+ \frac{u-v+1}{2^{-(\frac{am}{2} -k)}} \Big] .
\end{align*}
Notice now that for $r,k$ fixed, as $v,n,p$ range over their respective $\sim 2^{\frac{am}{2}}$ possible values, $s_1$ will ``span'' a time-frequency region which is roughly a $2^{\frac{am}{2}}$-dilate of any other chosen $s_1$ with similar characteristics, region that can be essentially identified with $P_1$ from \eqref{eq:tri:tiles:def:S4}.

Deduce that, up to a decomposition of the family of parameters into $O(1)$ subfamilies, we have\footnote{For a better geometric intuition, we refer the reader to Figure \ref{figure:tiles_b} in Section \ref{Genov}.}
\[ s_1 \subset P_1 \]
for all $n,v,p\sim 2^{\frac{am}{2}}$ .\\
Similar considerations hold for $g$ as well, and in particular we can introduce
\begin{align*}
s_2 &:= I^{p_r + n}_{k} \times \omega^{u_{\ell-1} + v}_{k}\\
&= \Big[ (r-1) 2^k + \frac{p+n}{2^{\frac{am}{2}-k}},  (r-1) 2^k + \frac{p+n+1}{2^{\frac{am}{2}-k}} \Big] \\
& \hspace{3em} \times \Big[\frac{\ell-1}{ 2^{-(am -k)}}+ \frac{u+v}{2^{-(\frac{am}{2} -k)}}, \frac{\ell-1}{ 2^{-(am -k)}}+ \frac{u+v+1}{2^{-(\frac{am}{2} -k)}} \Big]
\end{align*}
so that for $p,n,v \sim 2^{\frac{am}{2}}$ we always have $s_2 \subset P_2  $.

Turning towards the wave-packet $\check{\varphi}_{\frac{am}{2}- k, 2\ell-1}^{2u, p_r}$ we are also led to introduce
\begin{align*}
s_3 &:= I^{p_r}_{k} \times \omega^{2u_{2\ell-1}}_{k}\\
& = \Big[ (r-1) 2^k + \frac{p}{2^{\frac{am}{2}-k}},  (r-1) 2^k + \frac{p+1}{2^{\frac{am}{2}-k}} \Big] \\
& \hspace{3em} \times \Big[\frac{2\ell-1}{ 2^{-(am -k)}}+ \frac{2u}{2^{-(\frac{am}{2} -k)}}, \frac{2\ell-1}{ 2^{-(am -k)}}+ \frac{2 u +1}{2^{-(\frac{am}{2} -k)}} \Big]
\end{align*}
with (again for all $p,n,v \sim 2^{\frac{am}{2}}$) $s_3 \subset P_3 .$

We summarize the roles of the parameters appearing in the definition of the tritiles $s=(s_1, s_2, s_3)$, $P=(P_1, P_2, P_3)$, and in the model operator \eqref{eq:model:with:r}:
\begin{itemize}[leftmargin=5.5mm]
\item $k$ defines the main scale of the decomposition and determines the strip in the frequency plane;
\item the frequency parameter $\ell$ indicates the joint frequency localization of $f$ and $g$ at the scale $2^{am-k}$;
\item the spatial parameter $r$ indicates the spatial moral support of the input/output function(s) arising from the action of $T_{m,k,\ell,r}$;
\item the frequency parameters $u$ and $v$ provide a more accurate localization for $\hat f$ and $\hat g$ at the smaller scale $2^{\frac{am}{2}-k}$ (this arises from the linearization of the oscillatory phase of the multiplier symbol);
\item the spatial parameter $p$ indicates a more accurate spatial information of the output corresponding to the wave-packet $\check{\varphi}_{\frac{am}{2}- k, 2\ell-1}^{2u, p_r}$.
\item the spatial parameter $n$ accounts for the spatial localization in the continuous $t$ variable (at the same smaller scale $2^{\frac{am}{2}-k}$) exploiting the bilinear Hilbert structure of the input functions that reflects in the $\frac{2n}{2^{\frac{am}{2}-k}}$ difference\footnote{This is immediate from the spatial structure of the arguments of the input functions: $x+t$ for $f$ and $x-t$ for $g$.} between the moral spatial supports of $f$ and $g$, respectively.
\end{itemize}

We notice that $k,\ell,r\in\Z$ are associated with the ``expanded tiles'' $P$, while the regular $s$ tiles of area $1$ require also the parameters $p,n,u,v\sim 2^{am \over 2}$ in order to be located as we are zooming into a given $P$. This explains why we refer to a model involving the tri-tiles $s$ as a \emph{high resolution model} and to the model involving the ``expanded tiles'' $P$ as a \emph{low resolution model}.

The expanded tri-tile $P=(P_1,P_2,P_3)$ and their subordinated tri-tile(s) $s=(s_1, s_2, s_3)$ are represented in Figure \ref{figure:tiles_a} while their inter-relationship is illustrated in Figure \ref{figure:tiles_b}. For future reference, we collect all the above definitions :
\begin{equation}
\begin{aligned}
s_1 &:= I^{p_r-n}_{k} \times \omega^{u_{\ell}-v}_{k}, \\
s_2 &:= I^{p_r + n}_{k} \times \omega^{u_{\ell-1} + v}_{k}, \\
s_3 &:= I^{p_r}_{k} \times \omega^{2u_{2\ell-1}}_{k}, \\
P_1 &:= [r 2^k, (r+1) 2^k] \times \big[ 2^{am-k} (\ell+1), 2^{am-k} (\ell+2) \big], \\
P_2 &:= [r 2^k, (r+1) 2^k] \times \big[ 2^{am-k} (\ell-1), 2^{am-k} \ell\big], \\
P_3 &:= [r 2^k, (r+1) 2^k] \times \big[ 2^{am-k} (2 \ell), 2^{am-k} (2\ell+1) \big].
\end{aligned}
\label{eq:def:time:frequency:regions}
\end{equation}
Since later on we will prefer to drop the explicit form of both collections of tri-tiles, we also write for $1 \leq j \leq 3$, $s_j=I_{s_j} \times \omega_{s_j}$, where $I_{s_j}$ is one of the intervals obtained by dividing $I_P$ into $2^{am \over 2}$  pieces of equal length, and $\omega_{s_j}$ is one of the intervals obtained by dividing $\pmb{\omega}_{P_j}$ into $2^{am \over 2}$ identical pieces.
In such a situation we write
\begin{equation}\label{sfat}
\hat{s}_j= P_j, \quad \text{for   } 1 \leq j \leq 3, \quad \text{or simply       } \hat{s}:=P.
\end{equation}
More generally we use the notation
\begin{equation}
\label{sfat:2}
\hat J =I
\end{equation}
if $J$ is obtained by dividing $I$ into $2^{am \over 2}$ congruent subintervals.

Notice that each of the intervals $I_P$ is dyadic while each of its high-resolution associate $I_{s_j}$ obeys $\hat{I}_{s_j} =I_P$. Correspondingly, each of the intervals $\pmb{\omega}_{P_j}$ are part of a $2^{am}$ dilation of the dyadic grid with $\hat{\omega}_{s_j}=\pmb{\omega}_{P_j}$.

\begin{remark}\label{tilestory}
As mentioned previously, $\BHT$ is a rank-1 family: once $P_{i}$ determined, so are the other two tiles forming the expanded tri-tiles $P=(P_1, P_2, P_3)$ and thus only the parameters $k, r, \ell$ will be of relevance. In contrast, the collection of tri-tiles $s=(s_1, s_2, s_3)$ from \eqref{eq:def:time:frequency:regions} involves the extra parameters $n, p, u, v \sim 2^{am \over 2}$, and once $s_1$ is set, there are still $2^{am \over 2}$ choices for the spatial interval $I_{s_2}$ and $2^{am \over 2}$ choices for the frequency interval $\omega_{s_2}$ in order to fully determine $s_2$.
The rank-1 property makes $\BHT$ especially suitable for the later multi-scale time-frequency analysis, but, as we will see, we still need to rely on the tri-tiles of area 1 and on their associated wave-packets when estimating the almost orthogonality among the expanded tiles that is quintessential in obtaining the control over $T_m$ and in extenso over our original operator $BC^a$.
\end{remark}

Finally, for any $s = (s_1, s_2, s_3)$ as in \eqref{eq:def:time:frequency:regions}, we let
\begin{equation}
\begin{aligned}
\phi_{s_1} &:= \check{\varphi}_{\frac{am}{2}- k, \ell}^{u-v, p_r-n}, \quad &\rho_{I_P}(\lambda)&:= \rho_{am-ak}(\lambda(x)),\\
\phi_{s_2} &:= \check{\varphi}_{\frac{am}{2}- k, \ell}^{u+v, p_r+n},  \quad &w^e_s(\lambda)&:=w_{k, n, v}^e(\lambda),\\
\phi_{s_3} &:= \check{\varphi}_{\frac{am}{2}- k, 2\ell-1}^{2u, p_r}, \quad &|I_s|&:= 2^{-(\frac{am}{2}-k)} .
\end{aligned}\label{def:wave:p:f:g}
\end{equation}

\subsubsection{The low-resolution model} \label{subsubsec:lowmodel}

With the notations and conventions from above, we are now ready to formally introduce the low-resolution model. This will be delivered in two equivalent forms utilized depending on the context and our convenience.

\begin{itemize}[leftmargin=5.5mm]
\item \underline{\textit{the low resolution model - the parametric format}}
\medskip

In this setting, our low resolution model is given by
$$T(f,g)(x)=\sum_{m\in\N} T_m(f,g)(x)=\sum_{m\in\N}\sum_{k\in\Z}  T_{m,k}(f, g)(x) ,$$
with
\begin{equation}
\begin{aligned}
T_{m,k}(f, g)(x)& = \sum_{\ell,r \in \mathbb{Z}} T_{m,k,\ell,r}(f,g)(x), \qquad \text{ where} \\
T_{m,k,\ell,r}(f,g)(x) &:= \sum_{\substack{p \sim 2^{\frac{am}{2}}, \\ n \sim 2^{\frac{am}{2}}}} \sum_{\substack{u \sim 2^{\frac{am}{2}}, \\ v \sim 2^{\frac{am}{2}}}} \frac{1}{2^{\frac{am+2k}{4}}} \langle f_{\ell+1}, \check{\varphi}_{\frac{am}{2}- k, \ell}^{u-v, p_r-n}  \rangle   \langle g_{\ell-1}, \check{\varphi}_{\frac{am}{2}- k, \ell}^{u+v, p_r+n}  \rangle \\
& \hspace{7em} \cdot \check{\varphi}_{\frac{am}{2}- k, 2\ell-1}^{2u, p_r}(x) \rho_{am-ak}( \lambda(x)) w_{k, n, v}^e(\l)(x) .
\end{aligned}\label{eq:model:with:rr}
\end{equation}
Notice here, that for fixed $m\in\N$ and $k,\ell,r\in\Z$ the discretization of $T_{m,k,\ell,r}$ in \eqref{eq:model:with:rr} represents precisely the high-resolution linearized wave-packet discretization obtained in \eqref{eq:model0}.

Since we will work extensively with the dualization of these operators, we define the associated trilinear forms by
\begin{equation}
\label{def:trilinear:form}
\begin{aligned}
\underline{\mathcal{L}}_{m}(f,g,h) &:= \sum_{k\in\Z} \underline{\mathcal{L}}_{m,k}(f,g,h) ,\\
\underline{\mathcal{L}}_{m,k}(f,g,h) &:= \langle T_{m,k}(f,g), h \rangle= \sum_{\ell, r \in \mathbb{Z}} \underline{\mathcal{L}}_{m,k,\ell,r}(f,g,h) ,\\
\underline{\mathcal{L}}_{m,k,\ell,r}(f,g,h) &:= \langle T_{m,k,\ell,r}(f,g), h \rangle .\\
\end{aligned}
\end{equation}

\item \underline{\textit{the low resolution model - the tile format}}
\medskip

Alternatively, our operator can be written in terms of elementary building blocks associated with (expanded) tri-tiles. This format, anticipated in \eqref{tilemodel} of Section \ref{Genov}, will be used in Section \ref{sec:bilinear:analysis} where the time-frequency analysis for the low resolution model will be performed.

In this situation we have
\begin{equation}
\begin{aligned}
& T(f,g)(x)=\sum_{m\in\N} T_m(f,g)(x)=\sum_{m\in\N}\sum_{P \in \BHT} T_{m,P}(f,g)(x), \qquad \text{ where } \\
& T_{m,P}(f,g)(x):= 2^{-\frac{am}{2}} \sum_{s : \hat s =P} \frac{1}{|I_s|^{\frac{1}{2}}} \langle f, \phi_{s_1} \rangle \langle g, \phi_{s_2} \rangle \phi_{s_3}(x) \rho_{I_P}(\lambda(x)) w^e_s(\lambda)(x).
\end{aligned}   \label{eq:model:with:tiles}
\end{equation}
Here the operators $T_{m,P}$ correspond to the operators $T_{m,k,\ell,r}$ from the parametrized setting where we assumed that $P=(P_1, P_2, P_3)$ and $s=(s_1, s_2, s_3)$ are given by \eqref{eq:def:time:frequency:regions}.

Maintaining the parallelism with the parametric format we also define the associated trilinear forms
\begin{equation}
\label{def:trilinear:formtr}
\begin{aligned}
\underline{\mathcal{L}}_{m}(f,g,h) &:= \sum_{P\in\BHT} \underline{\mathcal{L}}_{m,P}(f,g,h) ,\\
\underline{\mathcal{L}}_{m,P}(f,g,h) &:= \langle T_{m,P}(f,g),h \rangle .\\
\end{aligned}
\end{equation}
\end{itemize}

We end this section with the following

\begin{remark}
\label{remark:grouping}
The low resolution analysis will be critical in Section \ref{sec:bilinear:analysis} where we will need to employ time-frequency methods similar in spirit with those used for approaching the bilinear Hilbert transform. While both \eqref{eq:model:with:rr} and \eqref{eq:model:with:tiles} offer a good intuition on the essence of the low resolution format in reality we will need a slight refinement of these formulations that is described in  Section \ref{sec:HR2}.

In what follows we will stick to the parametric format relying on the indices $k,\ell,r,p,n,u,v$ as it will be more convenient for performing a high-resolution study of the operator. Once we have completed the finer analysis, we will pass to the tile format. This latter, more geometric description, will be employed along the entire Section \ref{sec:bilinear:analysis}.
\end{remark}

\section{The high resolution, single-scale analysis (I): Extracting the cancellation encoded in the phase curvature}  \label{sec:onescale}
\subsection{Preparatives}\label{Inthr}

In this section we provide the first key ingredient for the strategy discussed in Section \ref{Genov} -- see in particular \eqref{mainrm} -- that is expressed in the following single scale $m$-exponential decay estimate:\footnote{While not needed later in a larger range, it is not difficult to see that this $L^2$-based estimate may be extended via interpolation arguments to the full Banach triangle, \textit{i.e.} $\|\underline{\L}^{a}_{m,k}\|_{L^p\times L^q\times L^{r'} \to \mathbb{C}}\lesssim_{a,p,q,r} 2^{-\d_0' am}$ with $\frac{1}{p}+\frac{1}{q}=\frac{1}{r}$ and $1<p,q<\infty$ and some exponent $\d_0'$ depending on $p,q$.}

\begin{proposition} \label{sgscale} Fix  $k,\ell,r\in\Z$ and $m\in\N$ and assume $a\in(0,\infty)\setminus\{1,2\}$ is a given parameter. With the notations from the previous section we let\footnote{While the same as the expression in Section \ref{subsubsec:lowmodel} we  restate for convenience the formula of our form here, and also -- for only this section -- insert the superscript $a$ in our formulations in order to stress the dependence of our form as well as of our estimates on this parameter.}
\begin{equation}
\begin{aligned}
&\underline{\L}^{a}_{m,k,\ell,r}(f,g,h):=\\
&\sum_{n,v,p\sim 2^{\frac{am}{2}}}\int_{\R}  S_{k, \ell,r}^{n, v, p}(f,g)(x)\, \rho_{am-ak}( \lambda(x))\, w_{k, n, v}^e(\l)(x)\,h(x)\,dx.
\end{aligned}\label{sgscdef}
\end{equation}
Then, there exists $\d_0\in(0,1)$ absolute constant such that
\begin{equation}
|\underline{\L}^a_{m,k,\ell,r}(f,g,h)|\lesssim_{a}2^{-\d_0 a m}\, 2^{-\frac{k}{2}} \|f\|_{L^2}\|g\|_{L^2}\|h\|_{L^2}.\label{sgsc}
\end{equation}
\end{proposition}

We start by elaborating on several subtleties pertaining to the above statement and its proof.

We first notice that the conclusion of Proposition \ref{sgscale} may be reshaped up to tails\footnote{\emph{E.g.} a superposition of similar flavor terms with fast decaying coefficients.}
\begin{equation}
\begin{aligned}
 |\underline{\L}_{m,k,\ell,r}^{a}(f,g,h)| \lesssim_a 2^{-\d_0 a  m} |I_{k,r}| & \Big(\aver{I_{k,r}} |\pi_{\omega_{k,\ell+1}} f|^2\Big)^{1/2} \\
& \cdot \Big(\aver{I_{k,r}} |\pi_{\o_{k,\ell-1}}g|^2\Big)^{1/2} \Big(\aver{I_{k,r}} |h|^2\Big)^{1/2},
\end{aligned}\label{TTstargum}
\end{equation}
where here $I_{k,r}:=[r2^{k}, (r+1)2^{k}]$, $\o_{k,\ell+1}:=[2^{am-k}\ell,2^{am-k}(\ell+1)]$, $r,\ell\in\Z$, $\pi_{\omega}$ stands for the frequency projection operator on the interval $\o$ and $\aver{I} f:=\frac{1}{|I|}\int_I f$.

Recalling now that $\underline{\L}_{m,k,\ell,r}^{a}(f,g,h)$ may be rewritten as $\underline{\L}_{m,P}^{a}(f,g,h)$ with $P=(P_1,P_2,P_3)\in \BHT$ the tri-tile defined by \eqref{eq:def:time:frequency:regions}, we notice that \eqref{TTstargum}  would correspond to a \emph{single tile} estimate if the function $h$ would be well-localized in time-frequency.

A more thorough inspection reveals now two key features of \eqref{TTstargum}:
\begin{enumerate}[label=(\alph*), leftmargin=15pt]
\item \label{item:one:scale:asymm} it is \emph{asymmetric} in the time-frequency localization of the inputs $(f,g,h)$: indeed, the functions $f$ and $g$ are well adapted to the expected time-frequency regions represented by the expanded tiles $P_1$ and $P_2$ respectively; however, for the function $h$ our estimate only provides the spatial localization corresponding to $P_3$;

\item \label{item:one:Scale:L2} it is a \emph{purely $L^2$}-based estimate that focuses on the norm $\|\underline{\L}_{m,P}^{a}\|_{L^2\times L^2\times L^2 \to \mathbb{C}}$ (and thus, in particular, goes beyond the usual H\"{o}lder duality range).

\end{enumerate}

As expected, these two features strongly relate with both our proof methodology and the structure of our paper:
\begin{enumerate}[label=(\Roman*), leftmargin=15pt]
\item \label{item:1scale:bilinear} as per item \ref{item:one:scale:asymm} above, the shape of the single scale estimate in \eqref{TTstargum} indicates that its proof favors the \emph{bilinear} behavior $(f,g)$ within the structure of $T_{m,k,\ell,r}^{a}$ treating the third function $h$ as an auxiliary object. This is further on reinforced  by the item \ref{item:one:Scale:L2} which hints that \eqref{TTstargum} involves a $T T^{*}$ argument in disguise, since $\|\underline{\L}_{m,k,\ell,r}^{a}\|_{L^2\times L^2\times L^2 \to \mathbb {C}}=\|T_{m,k,\ell,r}^{a}\|_{L^2\times L^2 \to L^2}$.

\item \label{item:1scale:necessity} the lack of a good time-frequency localization for the function $h$ observed in \ref{item:one:scale:asymm} hints towards the fact that one needs a further refinement of the single scale estimate \eqref{TTstargum} in order to develop the multiscale analysis required for providing the full boundedness range for our original operator $BC^a$. Indeed, as mentioned in the introduction, the treatment of $BC^a$ has to include the approach for the classical BHT with the latter relying fundamentally on the geometry/combinatorics of the (expanded) tri-tiles $P=(P_1,P_2, P_3)\in\BHT$ and hence on the good localization of the dualizing function $h$. Also the observation in \ref{item:one:Scale:L2} will offer the right setting for constructing and adapting suitable notions of sizes for each of the functions $f,g,\text{ and }h$ consistent with the multilinear time-frequency analysis for BHT.
\end{enumerate}

Based on the information provided in Sections \ref{Prel}, \ref{HR1lgc} and \ref{HR3} in what follows we will briefly recap and provide more context for
\begin{enumerate}[label=(\roman*)]
\item the features discussed at \ref{item:one:scale:asymm} and \ref{item:one:Scale:L2} and juxtaposed to \ref{item:1scale:bilinear} and \ref{item:1scale:necessity},

\item the necessity of going back and forth between the discrete linearized wave-packet (no $m$-decay absolute summability) model and the continuous linearized spatial ($m$-decay absolute summability) model.\footnote{It might be worth noticing that this is a direct consequence of the rough behavior of the linearizing function  $\l(x)$; indeed, if $\l(x)$ is smooth (say e.g. polynomial in $x$) then one can work directly with the discrete phase-linearized wave-packet model which obeys the $m$-decay absolute summability property.}
\end{enumerate}

Indeed, as already revealed in Section \ref{sec:discretization}, our initial approach is guided by the frequency discretization of our operator consistent with the phase linearization of the multiplier (see Section \ref{PL}) followed by an adapted Gabor frame discretization of the input functions $f$ and $g$. This achieves the discretization of our dualized form $\underline{\L}^{a}_{m,k}(f,g,h)$ as given by \eqref{lmk} and hence completes the first two steps (L) and (G) in our LGC-methodology described in Section \ref{HR1lgc}. However, the attempt of completing our program with the final step (C) faces a key difficulty: the wave-packet coefficients appearing in the discretized form $\underline{\L}^{a}_m(f,g,h)=\sum_{k\in\Z}\underline{\L}^{a}_{m,k}(f,g,h)$ do not produce an $m$-decaying absolute summable model. As already explained, we thus have to make a detour and reverse the Gabor discretization reshaping the linearization at the first step onto  the spatial variable (kernel side) -- see Sections \ref{passage} and \ref{BEV} and compare this with Section \ref{Tr}. The spatial linearization offers a critical advantage in that it circumvents the problem of the conditional $m$-decaying summability for $\underline{\L}_{m,k,\ell,r}^{a}(f,g,h)$ (see Appendix \ref{sec:counterexample}) since now, via \eqref{sgscdef}, the signum of the local Fourier (Gabor) coefficients within \eqref{eq:def:S_n,v,p} are encrypted through Parseval, into a bilinear oscillatory integral of the form \eqref{eq:def:S_n,v,pFinal}--\eqref{eq:tildeS}. Once at this point we pass to the (C) item exposed in Section \ref{HR1lgc} and employ what one might refer to as a \emph{bilinear $TT^{*}$-argument}  that relies on exponential sums\footnote{The fact that our treatment of a continuous nature problem is based on discrete analysis elements such as exponential sums is quite surprising; in more standard situations one usually sees the reverse applied - with continuous realm serving as a model and offering assistance to discrete analysis. In our setting, the specific shape of these exponential sums is dictated by the spatial phase linearization discretization.} and on level set estimates exploiting the time-frequency correlation and the non-zero curvature of the phase. This provides the insight to our proof of \eqref{TTstargum}, which gives us the desired exponential decay in $m$ but, as mentioned at item (a), without also providing a good time-frequency localization for the function $h$. The latter aspect however is necessary when passing from the single scale (fixed $k\in\Z$) to the multiscale analysis (summation over $k\in\Z$). This is why we need a refinement of the single scale estimate that is performed in Section \ref{sec:HR2} and that -- via the transition explained in Section \ref{passage} -- combines the features of both the present approach via spatial linearization and that of the initial  linearized wave-packet discretized model ending up by providing a single scale estimate with both exponential decay in $m$ and well localized information for the triple $(f,g,h)$.

\medskip
\subsection{Transition from the discrete phase-linearized wave-packet model to the continuous phase-linearized spatial model}\label{Tr}

This section can be thought as an adaptation of the reasonings presented in Section \ref{passage} to the specific context provided by Proposition \ref{sgscale}.

We start our reduction with the following simple observation: since our estimate \eqref{sgsc} is invariant under
translation, modulation and dilation symmetries, it is enough to only treat the case $k=0$, $\ell=1$ and $r=1$ (and hence $p_r=p$). Next, for notational simplicity, throughout this section we will refer to the wave-packets $\check{\varphi}^{u,n}_{\frac{am}{2},1}$ as simply
$\check{\varphi}^{u,n}$.
% where here $j=1,2$.

Recalling now \eqref{sgscdef}, an immediate application of Cauchy-Schwarz gives
\beq\label{dom}
|\underline{\L}^a_{m,0,1,1}(f,g,h)|\lesssim \Lambda_{m}^{a}(f,g)\,\|h\|_{L^2},
\eeq
where here we used the notation\footnote{Throughout the paper $\bar{c}_a$ is an absolute nonzero constant that is allowed to change from line to line. Also, throughout this section we set $\langle x \rangle := (1+x^2)^{1/2}$.}
\begin{equation}
\begin{aligned}
& [\Lambda_{m}^{a}(f,g)]^2:= \\
& \int_{\R} \Bigg(\frac{\rho_{am}(\l(x))}{2^{\frac{am}{4}}} \sum_{\substack{n,v \\ \sim 2^{\frac{am}{2}} }}
\frac{\Big|\sum\limits_{u,p \sim 2^{\frac{am}{2}} } \langle f,\check{\varphi}^{u-v,p-n}\rangle \langle g,\check{\varphi}^{u+v,p+n} \rangle\check{\varphi}^{2u,p}(x)\Big|}{\big\langle \bar c_a \frac{n^{a-1} \lambda(x)}{2^{\frac{a^2m}{2}}}+2v\big\rangle^2}\Bigg)^2 \,dx.
\end{aligned}\label{bilinfr}
\end{equation}
From \eqref{dom}, \eqref{bilinfr} and the symmetry invariance discussed earlier, we deduce that \eqref{sgsc} is an immediate consequence of

\beq\label{sgsc-bis}
|\Lambda_{m}^{a}(f,g)|\lesssim_a 2^{-\d_0 a m} \|f\|_{L^2}\|g\|_{L^2}.
\eeq

Notice now that since \eqref{bilinfr} is a positive expression, from now on we can assume without loss of generality that
\beq\label{levl}
\rho_{am}(\l(x))=1 \qquad \textrm{for any $x \in \mathbb{R}$}.
\eeq

Once at this point, we remark -- via \eqref{spacetransl} (or the more precise \eqref{eq:def:S_n,v,pFinal}) -- that our desired estimate \eqref{sgsc-bis} follows if we are able to prove a similar bound for the continuous version of our bilinear form $\Lambda_{m}^a(f,g)$, as given by\footnote{Here, as before, $I^n=[n\,2^{-\frac{a m}{2}},\,(n+1)\,2^{-\frac{a m}{2}}]$.}:
\begin{equation}
\begin{aligned}
& [\underline{\Lambda}_{m}^a(f,g)]^2 := \\
& \sum_{p\sim 2^{\frac{a m}{2}}}\int_{I^p} \Bigg(\sum_{n,v\sim 2^{\frac{a m}{2}}}\,
\frac{\Big|\int_{I^n} f(x-t) g(x+t) e^{-i 2^{\frac{a m}{2}} \, 2vt}\,dt\Big|}
{\big\langle \bar c_a \frac{n^{a-1}}{2^{\frac{a^2 m}{2}}}\l(x)+ 2v\big\rangle^2}\Bigg)^2\,dx.
\end{aligned}\label{bilinfr1}
\end{equation}

\begin{observation} In formulation \eqref{bilinfr1} above, for clarity purposes, we used a standard  almost-orthogonality argument in $p$ in order to only focus on the main diagonal term and discard the error term produced by the ``tails'' arising from the wave-packet localization. We will not provide here an argumentation for this reduction since there are several similar arguments carried over in detail in Sections \ref{sec:discretization} and \ref{sec:HR2}.
\end{observation}

We now claim that the behavior of $\underline{\Lambda}_{m}^a(f,g)$ may be reduced to that of\footnote{Here is where we use the time-frequency correlation provided by the expression appearing at the denominator in the right-hand side of \eqref{bilinfr1}.}

\begin{equation}
\begin{aligned}
& [\bar{\Lambda}_{m}^{a}(f,g)]^2:= \\
& \sum_{p\sim 2^{\frac{a m}{2}}}\int_{I^p} \bigg(\sum_{n\sim 2^{\frac{a m}{2}}}
\Big|\int_{I^n} f(x-t)g(x+t) e^{i\frac{\bar c_a n^{a-1}}{2^{\frac{a(a-1) m}{2}}} \l(x)\,t}\,dt\Big|\bigg)^2\,dx.
\end{aligned}\label{bilinfr2}
\end{equation}

Indeed, recalling the notation $M_{b}g(x):=e^{-i b x} g(x)$, $b\in\R$, we further claim that
\beq\label{triangineq}
\underline{\Lambda}_{m}^a(f,g)\lesssim \sum_{s \in \mathbb{Z}: |s|\leq 2^{\frac{am}{2}}} \frac{\bar{\Lambda}_{m}^{a}(f,M_{2^{\frac{am}{2}}s}g)}{s^2+1}.
\eeq

This last statement is based on a rather standard technical reasoning involving a Taylor series argument. However, for reader's convenience, we provide below a sketch of its proof: we first notice that

\begin{align*}
[\underline{\Lambda}_{m}^a(f,g)]^2 \approx \sum_{p\sim 2^{\frac{a m}{2}}}\int_{I^p} \bigg(\sum_{n,v\sim 2^{\frac{a m}{2}}}\,\sum_{|s|\leq 2^{\frac{a m}{2}}} & \one \Big(\Big|\frac{\bar{c}_a\,n^{a-1}}{ 2^{\frac{a^2 m}{2}} }\l(x) + 2v-s\Big|\lesssim 1\Big) \\
&\cdot \frac{\Big|\int_{I^n} f(x-t)g(x+t)e^{-i 2^{\frac{a m}{2}}\,2vt}\,dt\Big|}{s^2+1}\bigg)^2\,dx.
\end{align*}

Fix now $x$, and assume $\Big|\frac{\bar{c}_a n^{a-1}}{ 2^{\frac{a^2 m}{2}}  }\l(x)+2v-s\Big|\lesssim 1$. Isolating the $t$-integrand, we have
%\label{Tay}
\begin{align*}
& \Big|\int_{I^n} f(x-t)g(x+t)e^{- i2^{\frac{am}{2}}\,2vt}\,dt\Big| \\
&=\Big|\int_{I^n} f(x-t)g(x+t) e^{-i2^{\frac{am}{2}}st}e^{i2^{\frac{am}{2}}\frac{\bar{c}_a n^{a-1}}{2^{\frac{a^2 m}{2}}} \l(x) t} \\
& \hspace{13em} \cdot e^{-i 2^{\frac{am}{2}}(t-(n-1)2^{-\frac{am}{2}})(\frac{\bar{c}_a n^{a-1}}{2^{\frac{a^2 m}{2}}} \l(x) + 2v -s)}\,dt\Big|  \\
& \lesssim \sum_{k=0}^{\infty} \frac{\Big(\frac{\bar{c}_a n^{a-1}}{2^{\frac{a^2 m}{2}}} \l(x) + 2v -s\Big)^k}{k!} \Big|\int_{I^n} f(x-t) (M_{2^{\frac{am}{2}}s}g)(x+t) \\
& \hspace{13em} \cdot e^{i2^{\frac{am}{2}}\frac{\bar{c}_a n^{a-1}}{2^{\frac{a^2 m}{2}}} \l(x)t}
(2^{\frac{am}{2}}(t-(n-1)2^{-\frac{am}{2}}))^k\,dt\Big|  \\
&\lesssim \sum_{k=0}^{\infty}
\Big|\int_{I^n} f(x-t)(M_{2^{\frac{am}{2}}s}g)(x+t) e^{i \frac{\bar{c}_a n^{a-1}}{2^{\frac{a(a-1)m}{2}}}\l(x) t} \\
& \hspace{13em} \cdot \frac{(2^{\frac{am}{2}+10}(t-(n-1) 2^{-\frac{am}{2}}))^k}{k!}\,dt\Big|.
\end{align*}

Now, since $2^{\frac{am}{2}+10}\,(t-(n-1)\,2^{-\frac{am}{2}})$ is a function  that is morally constant one for $t\in I^n$ we conclude that \eqref{triangineq} holds.

With these done, based on the above considerations, the proof of \eqref{sgsc} is an immediate  consequence of Proposition \ref{sgscalelambda} below.

\subsection{The continuous phase-linearized spatial model formulation of the single scale estimate}\label{SPL}

In this section we formulate the single scale estimate in the desired continuous phase-linearized spatial form and  perform some simple reductions involving the distribution of the spatial information carried by the input functions $f$ and $g$.

\begin{proposition} \label{sgscalelambda} Let $a\in(0,\infty)\setminus\{1,2\}$ be a fixed parameter and $\l: \R \to [2^{am-10},2^{am+10}]$ be an arbitrary Lebesgue measurable function.

Then there exists\footnote{At the end of this proof we provide an explicit value for $\d_0$ though this is far from optimal. One can indeed increase the value of the exponent $\d_0$ but we prefer to trade off this improvement over the clarity of the exposition since for our purposes any non-trivial decay (i.e $\d_0>0$) is good enough.} an absolute constant $\d_0>0$ such that for any $m\in\N$, one has
\begin{equation}
\begin{aligned}
& \bigg(\sum_{p\sim 2^{\frac{am}{2}}}\int_{I^p} \bigg(\sum_{n\sim 2^{\frac{am}{2}}} \Big|\int_{I^n} f(x-t)g(x+t)e^{i \frac{\bar{c}_a n^{a-1}}{2^{\frac{a(a-1)m}{2}}}\l(x)t}\,dt\Big|\bigg)^2\,dx\bigg)^{\frac{1}{2}} \\
& = \bar{\Lambda}_{m}^{a}(f,g) \lesssim_a 2^{-\d_0 a m}\|f\|_{L^2}\|g\|_{L^2}.
\end{aligned}\label{sgsce}
\end{equation}
\end{proposition}
\begin{proof}

The proof of this statement will be split in two cases depending on the distribution of the spatial information carried by the functions $f$ and $g$:

Fix $\mu\in (0,1)$ a small parameter to be chosen later and define
\begin{itemize}
\item the \emph{$f$-clustered}  (heavy) component as
\[ C_{\mu}(f):=\Big\{n\sim 2^{\frac{am}{2}} : \int_{3 I^n}|f|^2 \geq 2^{(\mu-1)\frac{am}{2}} \|f\|_{L^2}^2\Big\}. \]
Notice here that
\begin{equation}\label{sgsce1}
\#C_{\mu}(f)\lesssim 2^{(1-\mu)\frac{am}{2}}.
\end{equation}
\item the \emph{$f$-uniform}  (light) component as
\begin{equation}\label{unifset}
U_{\mu}(f):= \Big\{n\sim 2^{\frac{am}{2}} : \int_{3 I^n}|f|^2< 2^{(\mu-1)\frac{am}{2}} \|f\|_{L^2}^2 \Big\}.
\end{equation}
\end{itemize}

A similar decomposition applies to $g$.

In direct correspondence with the above we analyze the behavior of our form $\bar{\Lambda}_{m}^{a}$ according to:
\begin{itemize}
\item the \emph{clustered} component defined as
\begin{equation}
\begin{aligned}
& \bar{\Lambda}_{m}^{a,C_{\mu}}(f,g):=\\
& \bigg(\sum_{p\sim 2^{\frac{am}{2}}}\int_{I^p} \bigg(\sum_{\substack{n\sim 2^{\frac{am}{2}}: \\ p-n\in C_{\mu}(f) \\ \text{ or } \\
p+n\in C_{\mu}(g)}}
\Big|\int_{I^n} f(x-t)g(x+t)e^{i\frac{\bar{c}_a n^{a-1}}{2^{\frac{a(a-1)m}{2}}}\l(x) t}\,dt\Big|\bigg)^2\,dx\bigg)^{\frac{1}{2}},
\end{aligned}\label{clust}
\end{equation}

\item the \emph{uniform} component defined as
\begin{equation}
\begin{aligned}
& \bar{\Lambda}_{m}^{a,U_{\mu}}(f,g):= \\
& \bigg(\sum_{p\sim 2^{\frac{am}{2}}}\int_{I^p} \bigg(\sum_{\substack{n\sim 2^{\frac{am}{2}} : \\ p-n\in U_{\mu}(f) \\ \text{ and } \\ p+n\in U_{\mu}(g)}}
\Big|\int_{I^n} f(x-t)g(x+t)e^{i\frac{\bar{c}_a n^{a-1}}{2^{\frac{a(a-1)m}{2}}}\l(x) t}\,dt\Big|\bigg)^2 \,dx\bigg)^{\frac{1}{2}},
\end{aligned}\label{defunif}
\end{equation}
\end{itemize}
and notice that
\beq\label{rel}
\bar{\Lambda}_{m}^{a}(f,g)\leq \bar{\Lambda}_{m}^{a,C_{\mu}}(f,g)\,+\,\bar{\Lambda}_{m}^{a,U_{\mu}}(f,g).
\eeq

With these done, we are now ready to perform our analysis.\\
\vspace{0.3cm}

\noindent\textbf{Case 1:} \underline{\textsf{The clustered component.}}
\vspace{0.3cm}

The clustered component has a straightforward treatment. Indeed, relying on the estimate \eqref{sgsce1} and applying (twice) a Cauchy-Schwarz argument we get
\begin{equation}
\begin{aligned}
& [\bar{\Lambda}_{m}^{a,C_{\mu}}(f,g)]^2 \\
& \lesssim 2^{(1-\mu)\frac{am}{2}} \sum_{p, n \sim 2^{\frac{am}{2}}}\int_{I^p}
\Big(\int_{I^n} |f(x-t)|^2\,dt\Big)\Big(\int_{I^n} |g(x+t)|^2\,dt\Big)\,dx \\
& \lesssim 2^{-\mu \frac{am}{2}}\sum_{p,n\sim 2^{\frac{am}{2}}} \int_{3 I^{p-n}}|f|^2 \int_{3I^{p+n}}|g|^2 \lesssim 2^{-\mu \frac{am}{2}}\|f\|^2_{L^2} \|g\|^2_{L^2}.
\end{aligned} \label{CS}
\end{equation}

$\newline$

\noindent\textbf{Case 2:} \underline{\textsf{The uniform component.}}
$\newline$

The treatment of the uniform component is the main step for our present single scale estimate and in particular for \eqref{sgsce}. It is within this chapter of our journey that the non-zero curvature plays a critical r\^{o}le. Because of its later relevance the proof of the key estimate for the uniform component will be delivered as part of an independent statement - see Proposition \ref{prop:uniform} below. In what follows we show that if the latter holds then relation \eqref{sgsce} also holds.

For this, it is enough to notice that a simple Cauchy--Schwarz argument gives us

\begin{equation}
\begin{aligned}
& [\bar{\Lambda}_{m}^{a,U_{\mu}}(f,g)]^2 \\
& \lesssim 2^{\frac{am}{2}} \sum_{p\sim 2^{\frac{am}{2}}}\sum_{\substack{n\sim 2^{\frac{am}{2}} \\ p-n\in U_{\mu}(f) \\ p+n\in U_{\mu}(g)}}\int_{I^p}
\Big|\int_{I^n} f(x-t) g(x+t) e^{i \frac{\bar{c}_a n^{a-1}}{2^{\frac{a(a-1)m}{2}}} \l(x) t}\,dt\Big|^2 \,dx.
\end{aligned}\label{unif-0}
\end{equation}

However, from  Proposition \ref{prop:uniform}, the right-hand side of the above inequality is bounded from above by $2^{-2 \delta_1 a m}$ for some absolute constant $\delta_1>0$ and hence
\beq\label{uniff}
[\bar{\Lambda}_{m}^{a,U_{\mu}}(f,g)]^2 \lesssim_a 2^{-2 \delta_1 a m} \|f\|^2_{L^2} \|g\|^2_{L^2} .
\eeq

Putting together \eqref{rel}, \eqref{CS} and \eqref{uniff} we conclude that
\beq\label{final2}
|\bar{\Lambda}_{m}^{a}(f,g)|\lesssim_a 2^{-\d_0 a m}\|f\|_{L^2}\|g\|_{L^2},
\eeq
where here\footnote{In effect -- see \eqref{finaunifconclude} and \eqref{d1case3} -- the proof of Proposition \ref{prop:uniform} will give us that, up to a polynomial term in $m$, \eqref{uniff} holds for  $\d_1=\frac{1}{68}\,(1-16 \mu)$ for $a\in(0,\infty)\setminus\{1,2,3\}$ and $\d_1=\frac{1}{6}\,(\frac{1}{114}-\frac{4}{3}\mu)$ for $a=3$. Thus, ignoring polynomial growth in $m$ factors, \eqref{final2} follows with $\d_0=\frac{1}{1292}$ (for $\mu=\frac{1}{323}$).} $\d_0:=\min\{\delta_1,\frac{\mu}{4}\}$.
\end{proof}

We are now left with proving Proposition \ref{prop:uniform}.

\begin{observation} As already mentioned in Section \ref{Inthr}, while Proposition \ref{sgscale} is to be regarded as providing the desired control over the single scale component of our operator, the key ingredient to be used later will be offered by the specialized version stated as Proposition \ref{prop:uniform} below (together with its consequence expressed as Corollary \ref{cor:uniform}).
\end{observation}

\subsection{Cancellation via $TT^{*}$ method, exponential sums and phase level set analysis}\label{subsec:ttstar}

In this section we present the crux of our argument which relies on extracting the cancellation hidden in the non-zero curvature of the phase by applying a $TT^{*}$ argument relying on elements of number theory (Weyl sums estimates). Within this argument a key r\^{o}le -- see Section \ref{levset} -- is played by the good control over the level sets of suitable exponential phases where these phases are treated as ``fewnomials'' having as coefficients expressions involving the linearizing frequency function $\l(\cdot)$.

\begin{proposition} \label{prop:uniform} With the previous notations, let
\begin{equation}
\begin{aligned}
& \Lambda_{m}^{a,U_{\mu}}(f,g):= \\
& \bigg(2^{\frac{am}{2}} \sum_{p\sim 2^{\frac{am}{2}}}\sum_{\substack{n\sim 2^{\frac{am}{2}} \\ p-n\in U_{\mu}(f) \\ p+n\in U_{\mu}(g)}} \int_{I^p}
\Big|\int_{I^n} f(x-t)g(x+t)e^{i \frac{\bar{c}_a n^{a-1}}{2^{\frac{a(a-1)m}{2}}} \l(x) t}\,dt\Big|^2\,dx \bigg)^{\frac{1}{2}}.
\end{aligned}\label{unif}
\end{equation}
Then, for a small enough choice of $\mu>0$ there exists a constant $\delta_1=\delta_1(\mu)>0$ such that for any $a\in(0,\infty)\setminus\{1,2\}$  and any $f,g\in L^2$ the following holds
\beq\label{finaunif}
\Lambda_{m}^{a,U_{\mu}}(f,g)\lesssim_a 2^{-\delta_1 a m}\|f\|_{L^2}\|g\|_{L^2}.
\eeq
\end{proposition}

Pairing Proposition \ref{prop:uniform} with the reasoning involved in proving \eqref{triangineq}, which reduces the study of $\underline{\Lambda}_{m}^a$ to that of $\bar{\Lambda}_{m}^{a}$, we deduce the following corollary:

\begin{corollary} \label{cor:uniform} With the exponent $\delta_1$ from Proposition \ref{prop:uniform}, we have that for all functions $f,g\in L^2$ and all $m\in \N$
\begin{align*}
& \bigg(2^{\frac{am}{2}} \sum_{v\sim 2^{\frac{am}{2}}}\sum_{\substack{n,p\sim 2^{\frac{am}{2}} \\ p-n\in U_{\mu}(f) \\ p+n\in U_{\mu}(g)}} \int_{I^p}
\frac{\big|\int_{I^n} f(x-t)g(x+t) e^{-i 2^{\frac{a m}{2}}\,2vt}\,dt\big|^2}
{\Big\langle \frac{\bar c_a n^{a-1}}{2^{\frac{a^2 m}{2}}} \l(x) + 2v\Big\rangle^2}\,dx\bigg)^{\frac{1}{2}} \\
& \hspace{20em} \lesssim_a 2^{-\delta_1 a m} \|f\|_{L^2} \|g\|_{L^2}.
\end{align*}
In particular, recalling the definitions of $w_{k,n,v}(\lambda)$ and $\tilde S_k^{n, v}$ -- see \eqref{eq:decay-weight} and \eqref{eq:tildeS}, respectively -- we deduce that
\begin{align*}
& \Big(2^{\frac{am}{2}} \sum_{v\sim 2^{\frac{am}{2}}}\sum_{\substack{n,p\sim 2^{\frac{am}{2}} \\ p-n\in U_{\mu}(f) \\ p+n\in U_{\mu}(g)}} \int_{I^p}
|\tilde S_0^{n, v}(f,g)(x)|^2 w_{0,n,v}(\lambda)(x) dx\Big)^{\frac{1}{2}} \\
& \hspace{20em} \lesssim_a 2^{-\delta_1 a m} \|f\|_{L^2} \|g\|_{L^2}.
\end{align*}
\end{corollary}

All the remaining part of this section is devoted to the proof of Proposition \ref{prop:uniform}.

\begin{proof}[Proof of Proposition \ref{prop:uniform}]
Without loss of generality, in what follows we assume $f,\,g\geq0$ and $\|f\|_{L^2}=\|g\|_{L^2}=1$.

\subsubsection{The $TT^{*}$ argument and the reduction to an exponential sum estimate}\label{ttstar}

This novel approach within the realm of our topic treats a problem originally formulated in a continuous setting via elements of analytic number theory.\footnote{This route was hinted at by the simpler discrete Van der Corput argument used as an alternative for analyzing the $L^2$ bound of the bilinear Hilbert transform along curves - see Section 8.2. in the Appendix of \cite{L1}.}

Returning now to \eqref{unif}, we square the expression therein and, applying a double change of variable, we obtain

\begin{equation}
\begin{aligned}
& [\Lambda_{m}^{a,U_{\mu}}(f,g)]^2 \\
&=  2^{\frac{am}{2}}\sum_{p\sim 2^{\frac{am}{2}}}\sum_{\substack{n\sim 2^{\frac{am}{2}} \\ p-n\in U_{\mu}(f) \\ p+n\in U_{\mu}(g)}} \int_{I^0}
\Big|\int_{I^0} f(x-t+(p-n)2^{-\frac{am}{2}}) \\
& \hspace{7em} \cdot g(x+t+(p+n)2^{-\frac{am}{2}}) e^{i \frac{\bar{c}_a n^{a-1}}{2^{\frac{a(a-1)m}{2}}}\l(x+p\,2^{-\frac{am}{2}})t}\,dt\Big|^2 \,dx \\
&= 2^{\frac{am}{2}}\sum_{{u\in U_{\mu}(f)}\atop{v\in U_{\mu}(g)}}\int_{I^0}
\Big|\int_{I^0} f(x-t+u 2^{-\frac{am}{2}}) g(x+t+v 2^{-\frac{am}{2}}) \\
& \hspace{16em} \cdot e^{i\frac{\bar{c}_a (\frac{v-u}{2})^{a-1}}{2^{\frac{a(a-1)m}{2}}} \l(x+\frac{u+v}{2} 2^{-\frac{am}{2}}) t}\,dt\Big|^2 \,dx.
\end{aligned}\label{uniff0}
\end{equation}
Set now $f_u(t):=f(t+u 2^{-\frac{am}{2}})$ and
\beq\label{lambda_x}
\underline{\l}_x(v):=\l(x+v\,2^{-\frac{am}{2}-1}),
\eeq
and redenote for notational simplicity $\I:=[\Lambda_{m}^{a,U_{\mu}}(f,g)]^2$. Thus
\beq\label{mathf}
\I=\sum_{\substack{u\in U_{\mu}(f) \\ v\in U_{\mu}(g)}} \frac{1}{|I^0|}\,\int_{I^0}
\Big|\int_{I^0} f_u(x-t) g_v(x+t) e^{i \frac{\bar{c}_a (v-u)^{a-1}}{2^{\frac{a(a-1)m}{2}}} \underline{\l}_x(u+v) t}\,dt\Big|^2\,dx.
\eeq

Rewrite the expression above as
\begin{equation}
\begin{aligned}
\I& =\sum_{\substack{u\in U_{\mu}(f) \\ v\in U_{\mu}(g)}}
\frac{1}{|I^0|} \int_{I^0} \int\limits_{|s|\leq 2^{-\frac{am}{2}}}\int_{I^0} f_u(x-t) g_v(x+t) \\
& \hspace{2em} \cdot f_u(x-(t-s))g_v(x+t-s)\one_{I^0}(t-s)
e^{i \frac{\bar{c}_a (v-u)^{a-1}}{2^{\frac{a(a-1)m}{2}}} \underline{\l}_x(u+v) s}\,dt\,ds\,dx,
\end{aligned}\label{decsq}
\end{equation}
where $\one_{I^0}$ denotes the characteristic function of $I^0$.

Fix now $\nu\in(0,1)$ a small parameter that will be chosen later. We next decompose
\beq\label{decsq1}
\I= \I^{<}+\I^{>},
\eeq
where
\begin{equation}
\begin{aligned}
\I^{<} &:= \sum_{\substack{u\in U_{\mu}(f) \\ v\in U_{\mu}(g)}}
\frac{1}{|I^0|} \int_{I^0} \int\limits_{|s|\leq 2^{-(\frac{1}{2}+\nu) a m}}\int_{I^0}  f_u(x-t)g_v(x+t) \\
& \hspace{2em} \cdot f_u(x-(t-s)) g_v(x+t-s)\one_{I^0}(t-s) e^{i\frac{\bar{c}_a (v-u)^{a-1}}{2^{\frac{a(a-1)m}{2}}} \underline{\l}_x(u+v)s}\,dt \,ds\,dx,
\end{aligned} \label{decsq1bis}
\end{equation}
and
\begin{equation}
\begin{aligned}
\I^{>} &:=\sum_{\substack{u\in U_{\mu}(f) \\ v\in U_{\mu}(g)}}
\frac{1}{|I^0|} \int_{I^0} \int\limits_{2^{-(\frac{1}{2}+\nu) a m}<|s|\leq 2^{-\frac{am}{2}}} \int_{I^0}  f_u(x-t) g_v(x+t) \\
& \hspace{2em} \cdot f_u(x-(t-s)) g_v(x+t-s) \one_{I^0}(t-s) e^{i\frac{\bar{c}_a (v-u)^{a-1}}{2^{\frac{a(a-1)m}{2}}} \underline{\l}_x(u+v)s}\,dt\,ds\,dx.
\end{aligned}\label{decsq2}
\end{equation}
We start with term $\I^{<}$ which is much easier to treat.

Setting $f_{u,x}(-t):=f_{u}(x-t)$ and $g_{v,x}(t):=g_v(x+t)$ we apply a Cauchy-Schwarz inequality under the integral signs in order to deduce that
%& := \sum_{{u\in U_{\mu}(f)}\atop{v\in U_{\mu}(g)}}\, 2^{-\nu m}\,\int_{I^0}\,\left(\int_{\R}  M(f_{u,x}(-\cdot)\,g_{v,x}(\cdot)\,\one_{I_0}(\cdot))^2\,dt\right)\,dx \\
\begin{equation}
\begin{aligned}
\I^{<} & \lesssim 2^{-\nu a m} \sum_{\substack{u\in U_{\mu}(f) \\ v\in U_{\mu}(g)}} \int_{I^0} \int_{I^0}  f_{u,x}(-t)^2  g_{v,x}(t)^2 \,dt\,dx  \\
& \lesssim 2^{-\nu a m} \sum_{\substack{u\in U_{\mu}(f) \\ v\in U_{\mu}(g)}} \int_{3 I^0} \int_{3 I^0}  f(u 2^{-\frac{am}{2}} +y)^2 g^2(v 2^{-\frac{am}{2}}+z)\,dy\,dz  \\
& \lesssim 2^{-\nu a m}.
\end{aligned}\label{max}
\end{equation}

We pass now to the treatment of $\I^{>}$. Define the expression
\begin{equation}
\begin{aligned}
\I(x)&:=\sum_{\substack{u\in U_{\mu}(f) \\ v\in U_{\mu}(g)}} \int_{2^{-(\frac{1}{2}+\nu) a m}<|s|\leq 2^{-\frac{a}{2} m}} \int_{I^0}  f_u(x-t) g_v(x+t) \\
& \hspace{2em} \cdot f_u(x-(t-s))g_v(x+t-s) \one_{I_0}(t-s) e^{i\frac{\bar{c}_a (v-u)^{a-1}}{2^{\frac{a(a-1)m}{2}}} \underline{\l}_x(u+v) s}\,dt\,ds
\end{aligned}\label{mathf1}
\end{equation}
and deduce that
\beq\label{mathf2}
\I^{>}=\frac{1}{|I^0|}\,\int_{I^0} \I(x)\,dx\:.
\eeq
We will show that there exists an absolute constant $\bar{\ep}>0$ such that the following uniform pointwise estimate holds:
\beq\label{mathf3-toshow}
 \I(x)\lesssim_a 2^{-\bar{\ep} a m }, \qquad \forall x\in I^0.
\eeq

For notational simplicity we will take $x=0$ and redenote $f_u(-t)$ by $f_u(t)$ and $\underline{\l}_0(u+v)$ by $\underline{\l}(u+v)$. Also set $J^0:=\{s : 2^{-(\frac{1}{2}+\nu) a m}<|s|\leq 2^{-\frac{am}{2}}\}$. Thus
\begin{equation}
\begin{aligned}
& \I(0)= \\
& \sum_{\substack{u\in U_{\mu}(f) \\ v\in U_{\mu}(g)}}
\int_{J^0} \int_{I^0}  f_u(t) g_v(t) f_u(t-s) g_v(t-s) \one_{I^0}(t-s) e^{i \frac{\bar{c}_a (v-u)^{a-1}}{2^{\frac{a(a-1)m}{2}}} \underline{\l}(u+v) s}\,dt\,ds.
\end{aligned}\label{mathf4-I0}
\end{equation}
Using now \eqref{unifset} and applying repeatedly Cauchy-Schwarz, one has
\begingroup
\allowdisplaybreaks
\begin{align*}
\I(0)& \leq \Big(\sum_{u\in U_{\mu}(f)}
\iint\limits_{J^0\times I^0} |f_u(t) f_u(t-s)|^2 \one_{I^0}(t-s)\,dt\,ds\Big)^{\frac{1}{2}} \\
& \hspace{1em} \cdot \bigg(\iint\limits_{J^0\times I^0} \sum_{u\sim 2^{\frac{am}{2}}} \Big|\sum_{v\in U_{\mu}(g)} g_v(t) g_v(t-s) e^{i \frac{\bar{c}_a (v-u)^{a-1}}{2^{\frac{a(a-1)m}{2}}} \underline{\l}(u+v) s}\Big|^2 \one_{I^0}(t-s)\,dt\,ds\bigg)^{\frac{1}{2}} \\
& \lesssim 2^{(\mu-1)\frac{am}{4}} \bigg(\iint\limits_{J^0\times I^0} \sum_{v,v_1\in U_{\mu}(g)} g_v(t) g_v(t-s) g_{v_1}(t) g_{v_1}(t-s) \\
& \hspace{5em} \cdot \bigg(\sum_{u\sim 2^{\frac{am}{2}}}
e^{i \bar{c}_a \Big[\frac{(v-u)^{a-1}}{2^{\frac{a(a-1)m}{2}}}
\underline{\l}(u+v) - \frac{(v_1-u)^{a-1}}{2^{\frac{a(a-1)m}{2}}}
\underline{\l}(u+v_1)\Big] s}\bigg) \one_{I^0}(t-s)\,dt\,ds\bigg)^{\frac{1}{2}} \\
& \lesssim 2^{(\mu-1)\frac{am}{4}}\bigg(\iint\limits_{J^0\times I^0} \Big(\sum_{v,v_1\in U_{\mu}(g)} |g_v(t) g_v(t-s)g_{v_1}(t) g_{v_1}(t-s)|^2\Big)^{\frac{1}{2}} \\
& \cdot \bigg(\sum_{\substack{v,v_1 \\ \in U_{\mu}(g)}} \bigg|\sum_{u\sim 2^{\frac{am}{2}}}
e^{i\bar{c}_a \Big[\frac{(v-u)^{a-1}}{2^{\frac{a(a-1)m}{2}}}
\underline{\l}(u+v)-\frac{(v_1-u)^{a-1}}{2^{\frac{a(a-1)m}{2}}}
\underline{\l}(u+v_1)\Big]s}\bigg|^2\bigg)^{\frac{1}{2}}\one_{I^0}(t-s)\,dt\,ds\bigg)^{\frac{1}{2}}.
\end{align*}
\endgroup
Denote now by
\beq\label{keytr}
\J(s):= \sum_{v,v_1\sim 2^{\frac{am}{2}}} \bigg|\sum_{u\sim 2^{\frac{am}{2}}}
e^{i \bar{c}_a \Big[\frac{(v-u)^{a-1}}{2^{\frac{a(a-1)m}{2}}}
\underline{\l}(u+v)-\frac{(v_1-u)^{a-1}}{2^{\frac{a(a-1)m}{2}}}
\underline{\l}(u+v_1)\Big]s}\bigg|^2.
\eeq
Then, based on the above reasonings, we deduce
\beq\label{keytr1}
\I(0)^2\leq 2^{(\mu-1)\frac{am}{2}}\int_{J^0\times I^0} \sum_{v\in U_{\mu}(g)}|g_v(t)g_v(t-s)|^2 |\J(s)|^{\frac{1}{2}} \one_{I^0}(t-s)\,dt\,ds.
\eeq

Our aim now is to show that we can do better than the trivial estimate for the term $\J(s)$, that is, to prove that there exists an absolute constant
$\ep'>0$ such that
\beq\label{keytr2}
\J(s) \lesssim_a 2^{(2-\ep') a m}.
\eeq

We start by noticing that
\begin{equation}
\begin{aligned}
& \J(s) = \\
& \sum_{u,v,u_1,v_1\sim 2^{\frac{am}{2}}}
\exp\Big( i \bar{c}_a \Big[\frac{(v-u)^{a-1}}{2^{\frac{a(a-1)m}{2}}}
\underline{\l}(u+v) - \frac{(v_1-u)^{a-1}}{2^{\frac{a(a-1)m}{2}}}
\underline{\l}(u+v_1)\\
& \hspace{8em} - \frac{(v-u_1)^{a-1}}{2^{\frac{a(a-1)m}{2}}}
\underline{\l}(u_1+v)+\frac{(v_1-u_1)^{a-1}}{2^{\frac{a(a-1)m}{2}}}
\underline{\l}(u_1+v_1)\Big] s \Big).
\end{aligned}\label{j1}
\end{equation}

Next, we are applying a change of variables as follows: $u+v=p$, $u+v_1=q$ and $u_1+v=r$. Define now
\begin{equation}
\begin{aligned}
P_{p,q,r}(v)&:=\frac{(2v-p)^{a-1}}{2^{\frac{a(a-1)m}{2}}}
\tilde{\l}(p)-\frac{(2v+q-2p)^{a-1}}{2^{\frac{a(a-1)m}{2}}}
\tilde{\l}(q) \\
& \hspace{2em}-\frac{(2v-r)^{a-1}}{2^{\frac{a(a-1)m}{2}}}
\tilde{\l}(r)+\frac{(2v+q-p-r)^{a-1}}{2^{\frac{a(a-1)m}{2}}}
\tilde{\l}(q+r-p),
\end{aligned}\label{polyn0}
\end{equation}
where here we set $\tilde{\l}:=\bar{c}_a\,\underline{\l}\,2^{-\frac{am}{2}}$. Notice that in particular
\beq\label{lam}
\tilde{\l}(p)\sim 2^{\frac{am}{2}}\qquad \text{ for any } p\sim 2^{\frac{am}{2}}.
\eeq
Our approach will be conducted by regarding $P_{p,q,r}(v)$ as a piecewise smooth function in $v$ with coefficients depending on $p,q$ and $r$.\footnote{Notice that if $a\in\N$ then $P_{p,q,r}(v)$ is a polynomial in $v$.}

Once at this point we rewrite the expression in \eqref{j1} as
\beq\label{j3}
\J(s)= \sum_{p,q,r\sim 2^{\frac{a m}{2}}}
\sum_{v\sim_{p,q,r} 2^{\frac{a m}{2}}} e^{i 2^{\frac{am}{2}} s P_{p,q,r}(v)},
\eeq
where in the above we used the notation $v\sim_{p,q,r} 2^{\frac{a m}{2}}$ to state that simultaneously $v\sim 2^{\frac{a m}{2}}$ and $a_1(p,q,r)\leq v\leq a_2(p,q,r)$.\footnote{ Here $a_2(p,q,r)$ is of the form $\min\{p-2^{\frac{a m}{2}},\,r-2^{\frac{a m}{2}},\,2^{\frac{a m}{2}+1}+p-q\}$ while
$a_1(p,q,r)$ is of the form $\max\{p-2^{\frac{a m}{2}+1},\,r-2^{\frac{a m}{2}+1},\,2^{\frac{a m}{2}}+p-q\}$. If the two relations in the definition of $v\sim_{p,q,r} 2^{\frac{a m}{2}}$ cannot be satisfied simultaneously, then by convention the $v$ summation in \eqref{j3} is taken to be zero.}

Fix now in what follows $p,q$ and $r$. We rewrite \eqref{polyn0} in the form
\beq\label{polyn0f}
P_{p,q,r}(v) = F_4^{a-1}\Big(\frac{2v}{2^{\frac{a m}{2}}}\Big),
\eeq
where here we set the function
\beq\label{FF}
F_4^{a-1} := x\mapsto \sum_{k=1}^4 a_k (x+\a_k)^{a-1},
\eeq
with
\begin{alignat*}{7}
\a_1 & =-\frac{p}{2^{\frac{a m}{2}}},\quad &&\a_2 &&=\frac{q-2p}{2^{\frac{a m}{2}}},\quad &&\a_3&& =-\frac{r}{2^{\frac{a m}{2}}},\quad &&\a_4 &&=\frac{q-p-r}{2^{\frac{a m}{2}}}, \\
a_1 &=\tilde{\l}(p), && a_2 &&=-\tilde{\l}(q), && a_3&&=-\tilde{\l}(r), &&a_4 &&=\tilde{\l}(q+r-p).
\end{alignat*}

\subsubsection{Analysis of the phase level sets}\label{levset}

Our aim in this section is to obtain good upper bounds on the size of the level sets of the phase $P_{p,q,r}(\cdot)$.\footnote{The level set analysis (this time applied to the phase of the multiplier) had a similar importance in the problem treated in \cite{lvUnif} -- see the crucial estimate on the number of heavy pairs provided by Proposition 49 via Lemma 51 therein.} The main result here is stated as Lemma \ref{phasecont1} and is a consequence of the non-zero curvature of the original phase in \eqref{Top} -- hence the restriction of the parameter $a$ to the non-resonant case.

\begin{lemma} \label{zeroc}  Let $n\in\N$, $d\in\R^{*}$ and $\a\in\R$ be fixed. Assume we are given $\{\a_k\}_{k=1}^{n}\subset(\a,\infty)$ pairwise distinct and $\{a_k\}_{k=1}^n \subset\R$.

Define the function $F_n^d: (-\a,\infty) \to \R$ by
\beq\label{F}
F_n^d(x):=\sum_{k=1}^n a_k (x+\a_k)^{d}.
\eeq
Then either $F_n^d$ is identically zero or it has at most $n$ real distinct zeros.
\end{lemma}

\begin{proof}
Without loss of generality we assume that
\beq\label{Fnotz}
F_n^d \text{ is not identically } 0 \text{ and } d\in\R\setminus\{0,1\},
\eeq
as otherwise there is nothing to prove.

We now perform an induction on $n$. Notice that the base case $n=1$ is trivially true.
Assume now that the statement holds for $n-1$ and assume by contradiction that there exists $F_{n}^d$ as in \eqref{F} having (at least) $n+1$ real distinct roots, say $\{x_j\}_{j=1}^{n+1}$ with $-\a<x_1<x_2<\ldots<x_{n+1}<\infty$.

Deduce then that the function
\beq\label{G}
G_n^d(x):=\frac{F_n^d(x)}{(x+\a_1)^d}=a_1+\sum_{k=2}^n a_k \frac{(x+\a_k)^{d}}{(x+\a_1)^{d}},
\eeq
has the same zeros as $F_n^d$ within $(-\a,\infty)$.

Applying now Rolle's theorem on each interval $(x_j, x_{j+1})$ we deduce the existence of $y_j\in (x_j, x_{j+1})$ such that
$({G^d_n})'(y_j):=\frac{d}{dx} G_n^d(y_j)=0$ for every $j\in\{1,\ldots,n\}$. At this point it is worth noticing that $(G^d_n)'$ can not be identically zero as otherwise this would contradict our initial assumption \eqref{Fnotz}.

Next, we notice that
%\beq \label{G1}
\[ (G^d_n)'(x)=-\sum_{k=2}^n a_k d \cdot \frac{\a_k-\a_1}{(x+\a_1)^2} \cdot \frac{(x+\a_k)^{d-1}}{(x+\a_1)^{d-1}}, \]
%\eeq
hence we must have that the function
\beq \label{F2}
 F_{n-1}^{d-1}(x):=\sum_{k=2}^n a_k d (\a_k-\a_1) (x+\a_k)^{d-1} ,
\eeq
obeys: 1) $d-1\in \R^{*}$, 2) it is not identically zero, and 3) has the above $\{y_j\}_{j=1}^{n}$ among its zeros. However, this contradicts the induction hypothesis at step $n-1$, thus concluding our proof.
\end{proof}

\begin{lemma} \label{levelset} Let $F_n^d$ be a non-zero function whose form is prescribed by \eqref{F} with $n\geq 2$. With the notations from above, assume for simplicity that
\beq\label{assumesim}
\{a_k\}_{k=1}^n \subset \R\setminus\{0\}, \text{ and }\{|\a_k|\}_{k=1}^n \subset \Big[0,\frac{1}{2}\Big] \text{ with } \{\a_k\}_{k=1}^{n}\text{  pairwise distinct}.
\eeq
Then, given a parameter $\eta>0$, we have
\begin{itemize}
\item if $d\in\N$ with $d< n-1$ then
\beq\label{LFF}
\big|\{x\in [1,2] :  |F_n^d(x)|<\eta\}\big|\leq 100 d^2 \Big(\frac{\eta}{\|F_n^d\|_{L^{\infty}([1,2])}}\Big)^{\frac{1}{d}};
\eeq
\item otherwise,
\beq\label{LF}
\big|\{x\in [1,2] :  |F_n^d(x)| < \eta\}\big|\leq K(F_n^d)
\eta^{\frac{1}{n-1}},
\eeq
with $K(F_n^d):=\frac{3^{2|d|+2n+10}}{\left(|a_n|
\left(\prod_{j=1}^{n-1}|d-j+1|\,|\a_n-\a_j|\right)\,\right)^{\frac{1}{n-1}}}$.
\end{itemize}
\end{lemma}

\begin{proof}
First we notice that \eqref{LFF} is a simple application of Lagrange's interpolation formula for polynomials (see e.g.  Lemma B in the Appendix of \cite{lv3}).

Turning our attention towards \eqref{LF} we will prove this via induction. Set $\eta_n:=\eta$ and define
\beq\label{LF0}
J_n:=\{x\in [1,2] :  |F_n^d(x)|<\eta_n\}.
\eeq
Notice now that $J_n$ is a subset of
\beq\label{LF1}
\tilde{J}_n:=\{x\in [1,2] :  |G_n^d(x)|<\tilde{\eta}_n\},
\eeq
where $G_n^d(x)$ is given by \eqref{G} and $\tilde{\eta}_n:= 3^{|d|}\eta_n$.

From Lemma \ref{zeroc} we deduce that $\tilde{J}_n$ can be written as a union of at most $2(n+1)$ intervals and, moreover
\beq\label{LF2}
\int_{\tilde{J}_n}|(G^d_n)'(x)|\,dx < 4(n+1) \tilde{\eta}_n.
\eeq
Let now $0<\tilde{\eta}_{n-1}$ (to be properly chosen later) and write $J_n^1:=\{x\in \tilde{J}_n : |(G^d_n)'(x)|<\tilde{\eta}_{n-1}\}$ and $J_n^2:=\tilde{J}_n\setminus J_n^1$.
From \eqref{LF2} we deduce that
\beq\label{LF3}
|J_n^2|< 4(n+1) \frac{\tilde{\eta}_n}{\tilde{\eta}_{n-1}},
\eeq
while $J_n^1\subseteq J_{n-1}$ where, recalling \eqref{F2}, we let
\beq\label{LF4}
J_{n-1}:=\{x\in [1,2] : |F_{n-1}^{d-1}(x)|< \eta_{n-1}\},
\eeq
with $\eta_{n-1}:=3^{|d|+2}\tilde{\eta}_{n-1}$.

Putting together \eqref{LF0}--\eqref{LF4} we get
\beq\label{LF5}
|J_n|\leq 4(n+1) 3^{2|d|+2} \frac{\eta_n}{\eta_{n-1}} + |J_{n-1}| .
\eeq
Iterating \eqref{LF5} we get
\begin{equation}
\begin{aligned}
 |J_n| \leq 4(n+1) 3^{2|d|+2}\frac{\eta_n}{\eta_{n-1}} &+ 4n 3^{2|d-1|+2}\frac{\eta_{n-1}}{\eta_{n-2}} + \ldots \\
& + 4\cdot3\cdot 3^{2|d-(n-2)|+2} \frac{\eta_{2}}{\eta_{1}} + |J_{1}|,
\end{aligned}\label{LF6}
\end{equation}
where
%\beq\label{LF7}
\[ J_{1}:=\{x\in [1,2] : |F_{1}^{d-n+1}(x)|< \eta_{1}\}, \]
%\eeq
with
%\beq\label{F8}
\[ F_{1}^{d-n+1}(x) := a_n\Big(\prod_{j=1}^{n-1}(d-j+1)(\a_n-\a_j)\Big)(x+\a_n)^{d-n+1}. \]
%\eeq
Taking now $\eta_1:=|a_n| \big(\prod_{j=1}^{n-1}|d-j+1||\a_n-\a_j|\big)(\frac{1}{2})^{|d-n|+1}$ we see that $|J_1|=0$.

Choosing now $\eta_2, \ldots, \eta_{n-1}$ such that the terms in the right-hand side of \eqref{LF6} are pairwise equal, we conclude that
\eqref{LF} holds.
\end{proof}

Applying now Lemma \ref{zeroc} and Lemma \ref{levelset} we deduce

\begin{lemma} \label{phasecont} Let $F_4^{a-1}$ be a generic function as given by \eqref{F} and \eqref{assumesim} with $n=4$. Fix now $\eta>0$. We then have that the set
\beq\label{levsetp}
\Big\{x\in [1,2] : \Big|\frac{d}{dx} F_4^{a-1}\Big|<\eta\Big\}
\eeq
can be covered by at most ten intervals of length $L_{\eta}^{a-1}$ such that
\begin{itemize}
\item \textbf{Case 1:} if $a=3$ then
\beq\label{L1-a3}
L_{\eta}^2\lesssim \frac{\eta}{|A|+|B|},
\eeq
with $A=\sum_{k=1}^4 a_k$ and $B=\sum_{k=1}^4 a_k \a_k$.

\item \textbf{Case 2:} if $a=4$ then
\beq\label{L1}
L_{\eta}^3 \lesssim \Big(\frac{\eta}{|A|+|B|+|C|}\Big)^{\frac{1}{2}},
\eeq
with $A=\sum_{k=1}^4 a_k$, $B=\sum_{k=1}^4 a_k \a_k$ and $C=\sum_{k=1}^4 a_k \a_k^2$.

\item \textbf{Case 3:} if $a\in(0,\infty)\setminus\{1,2,3,4\}$ then

\beq\label{L2}
L_{\eta}^{a-1}\lesssim 20^a \Big(\frac{\eta}{|a-1| |a-2| |a-3| |a-4|\sup_{1\leq j\leq 4} \big(|a_j| \prod_{k\not=j} |\a_j-\a_k|\big)}\Big)^{\frac{1}{3}}.
\eeq
\end{itemize}
\end{lemma}

From Lemma \ref{phasecont} we further have that

\begin{lemma} \label{phasecont1} In the specific setting provided by the choice of the $F_4^{a-1}$'s parameters as indicated by \eqref{FF}, we have that
\beq\label{levsetp1}
\Big\{x\in [1,2] : \Big|\frac{d}{dx}F_4^{a-1}\Big|<\eta\Big\}
\eeq
can be covered by at most ten intervals of length $L_{\eta}^{a-1}$ such that
\begin{itemize}

\item \textbf{Case 1:} if $a=3$ then
\beq\label{L011}
L_{\eta}^2(p,q,r)\lesssim \frac{\eta 2^{\frac{a m}{2}}}{|(r-q)\tilde{\l}(p)+ (p-r)\tilde{\l}(q)+(q-p)\tilde{\l}(r)|}.
\eeq

\item \textbf{Case 2:} if $a=4$ then
\beq\label{L11}
L_{\eta}^3(p,q,r)\lesssim \Big(\frac{\eta 2^{\frac{a m}{2}}}{|p-r| |p-q|}\Big)^{\frac{1}{2}}.
\eeq

\item \textbf{Case 3:}  if $a\in(0,\infty)\setminus\{1,2,3,4\}$ then

\beq\label{L21}
L_{\eta}^{a-1}(p,q,r)\lesssim 20^a \Big(\frac{\eta 2^{a m}}{|a-1| |a-2| |a-3| |a-4| |p-r| |q-p| |r-q|}\Big)^{\frac{1}{3}}.
\eeq
\end{itemize}
\end{lemma}

\begin{observation}\label{nonresonantresctriction} The form of our lemma above is suggestive for the restriction $a\in (0,\infty)\setminus\{1,2\}$ in our main result -- see Theorem \ref{main}. The fact that for the resonant values $a\in\{1,2\}$  no decay is possible in \eqref{sgscdef} is a consequence of the generalized modulation invariance of our operator $BC^a$ -- see for this the discussion in Section \ref{openquest}. In particular, the failure of \eqref{keytr2} can be directly verified by inspecting \eqref{j1}.
\end{observation}

\begin{proof}
We first address the cases $a\in \{3, 4\}$. In this situation, with the notations from \eqref{FF} and \eqref{L1}, we have

\begin{align}
A& =\tilde{\l}(p)- \tilde{\l}(q)-\tilde{\l}(r)+\tilde{\l}(q+r-p), \label{coef01} \\
B&=\frac{1}{2^{\frac{am}{2}}}\big(-p \tilde{\l}(p)- (q-2p)\tilde{\l}(q)+r \tilde{\l}(r)+(q-p-r)\tilde{\l}(q+r-p)\big), \label{coef02} \\
C&=\frac{1}{2^{am}}\big(p^2\tilde{\l}(p)- (q-2p)^2\tilde{\l}(q)-r^2\tilde{\l}(r)+(q-p-r)^2\tilde{\l}(q+r-p)\big). \label{coef03}
\end{align}
From \eqref{coef01} and \eqref{coef02} we have
\beq\label{coef021}
\Big|\frac{p+r-q}{2^{\frac{am}{2}}} A\Big| + |B|\geq \frac{1}{2^{\frac{am}{2}}}|(r-q)\tilde{\l}(p)+ (p-r)\tilde{\l}(q)+(q-p)\tilde{\l}(r)|,
\eeq
which immediately implies \eqref{L011}.

Next, from \eqref{coef02} and \eqref{coef03}, we get
\beq\label{coef022}
\Big|\frac{p+r-q}{2^{\frac{am}{2}}} B\Big| + |C|\geq \frac{1}{2^{am}} |p(q-r)\tilde{\l}(p)+ (q-2p)(p-r)\tilde{\l}(q)+r(p-q)\tilde{\l}(r)|.
\eeq
Finally, combining \eqref{coef021} and \eqref{coef022} we deduce that
\begin{equation}
\begin{aligned}
& \frac{|p|}{2^{\frac{am}{2}}}\Big(\Big|\frac{p+r-q}{2^{\frac{am}{2}}}A \Big| + |B|\Big) + \Big|\frac{p+r-q}{2^{\frac{am}{2}}} B \Big| + |C| \\
& \hspace{8em} \geq \frac{1}{2^{am}}|p-r||q-p||\tilde{\l}(q)+\tilde{\l}(r)|,
\end{aligned}\label{coef031}
\end{equation}
from which we deduce\footnote{Notice that relation \eqref{coef031a} may be regarded as another manifestation of the time-frequency correlations induced by the non-zero curvature, this times relating the phase parameter $\l$ to the parameters $p,\,q,\,r$ arising from the spatial localization of the input functions (after a double $TT^{*}$ argument).}
\beq\label{coef031a}
|A|+|B|+|C|\gtrsim 2^{\frac{am}{2}} \frac{|p-r|}{2^{\frac{am}{2}}}  \frac{|q-p|}{2^{\frac{am}{2}}}.
\eeq
Conclude now that \eqref{L11} follows from \eqref{L1} and \eqref{coef031a}.

The case $a\in(0,\infty)\setminus\{1,2,3,4\}$ follows trivially from  \eqref{lam}, \eqref{FF} and \eqref{L2} once we exploit the fact that $\sup_{1\leq j\leq 4} \prod_{k\not=j} |\a_j-\a_k|\geq \prod_{k\not=1} |\a_1-\a_k|$.
\end{proof}

\subsubsection{The Cases 2 and 3: $a\in (0,\infty)\setminus\{1,2,3\}$. A discrete Van der Corput argument}

We are now ready to evaluate expression \eqref{j3}. Our analysis will depend on the amount of cancellation encoded by the rate of change of the phase associated with our exponential sum. Consequently, for a suitable small parameter $\bar{\eta}\in (0,1)$ (to be chosen later) we will split our exponential sum into two components:
\beq\label{jsplit}
\J(s)=\J_L(s) + \J_H(s),
\eeq
with
\begin{itemize}
\item the \emph{low oscillatory} component defined by
\beq\label{js}
\J_L(s):=\sum_{p,q,r\sim 2^{\frac{am}{2}}}\Big|
\sum_{\substack{v\sim_{p,q,r} 2^{\frac{am}{2}} \\ |P'_{p,q,r}(v)|< \bar{\eta}}} e^{i 2^{\frac{am}{2}} s P_{p,q,r}(v)}\Big|,
\eeq
\item the \emph{high oscillatory} component defined by
\beq\label{jl}
\J_H(s):=\sum_{p,q,r\sim 2^{\frac{am}{2}}}\Big|
\sum_{\substack{v\sim_{p,q,r} 2^{\frac{am}{2}} \\ |P'_{p,q,r}(v)|\geq  \bar{\eta}}} e^{i 2^{\frac{am}{2}} s P_{p,q,r}(v)}\Big|.
\eeq
\end{itemize}

We now analyze these two terms separately.

$\newline$
\noindent \textsf{The low oscillatory component.}
$\newline$

This is the easy situation in which we can simply ignore the cancellations coming from the exponential sum. Indeed, from \eqref{polyn0f} we have
\[ P'_{p,q,r}(v)=\frac{1}{2^{\frac{a m}{2}-1}} \Big(\frac{d}{dx}\Big) F_4^{a-1}(\frac{2v}{2^{\frac{a m}{2}}}) \]
and hence, applying\footnote{Here trivial considerations allow us to extend the applicability of our  Lemma \ref{phasecont1} -- with the obvious adaptations -- for an extended range of $F_4^{a-1}$, say $\Big[\frac{1}{10},10\Big]$.} Lemma \ref{phasecont1} with $\eta=2^{\frac{a m}{2}-1}\bar{\eta}$, we further have
\beq\label{js1}
\J_L(s) \leq \sum_{p,q,r\sim 2^{\frac{am}{2}}}
\sum_{\substack{v\sim 2^{\frac{am}{2}} \\ |P'_{p,q,r}(v)|< \bar{\eta}}} 1\lesssim 2^{\frac{am}{2}} \sum_{p,q,r\sim 2^{\frac{am}{2}}} \min\{1, L_{\eta}^{a-1}(p,q,r)\}.
\eeq
Deduce
\begin{itemize}
\item In Case 2: from \eqref{L11} we have
 \beq\label{js110}
\J_L(s)\lesssim_{a} 2^{2 am} \bar{\eta}^{\frac{1}{2}}.
\eeq
\item In Case 3: from \eqref{L21} we have
 \beq\label{js111}
\J_L(s)\lesssim_{a} 2^{2 am} \bar{\eta}^{\frac{1}{3}}.
\eeq
\end{itemize}

$\newline$
\noindent  \textsf{The high oscillatory component.}
$\newline$

We start by recalling the following classical number theoretical result which may be regarded as particular case of the discrete Van der Corput lemma (see \cite[Theorem 2.1]{Graham_Kolesnik}):

\begin{lemma} \label{VdC} [\textsf{Discrete Van der Corput}] Let $J\subset\R$ be a compact interval and $\a\in (0,1)$. Assume we are given a smooth function $F:\,J\,\mapsto\,\R$ such that $F$ is a continuously differentiable with\footnote{Here $\|x\|$ stands for the distance from $x$ to the nearest integer.} $\|F'(x)\|\geq \a$ for all $x\in J$ and $F'$ monotonic. Then\footnote{Recall that throughout the paper the display $e^{i x}$ stands for $e^{2\pi i x}$.}
\beq\label{VDC1}
\Big|\sum_{n\in J\cap \Z} e^{i F(n)}\Big|\lesssim \a^{-1}.
\eeq
\end{lemma}

In treating our high oscillatory component $\J_H(s)$ we are going to apply the above lemma to the function $F(x):=2^{\frac{am}{2}} s P_{p,q,r}(x)$, given that we are now working in the regime $\{v : |P'_{p,q,r}(v)|\geq \bar{\eta}\}$.

Next, taking  \eqref{mathf}, we notice that by subdividing\footnote{Say for example in $2^{100(a^2+1)}$ equidistant intervals.} the region of integration in $t$ followed by a Cauchy-Schwarz argument, one may reduce the range of $s$ so that it obeys the condition $|2^{\frac{am}{2}}s|<2^{-50(a^2+1)}$. Since $F'(v)=2 s \Big(\frac{d}{dx}F_4^{a-1}\Big)\Big(\frac{2v}{2^{\frac{a m}{2}}}\Big)$, we can now ensure that
%\beq\label{1der}
\[ \frac{1}{2}>\|F'(v)\| = \|P'_{p,q,r}(v)  2^{\frac{am}{2}} s\| \gtrsim \bar{\eta} 2^{\frac{am}{2}}|s|. \]

Finally, based on \eqref{polyn0f} and Lemma \ref{zeroc}, the function $F$ can be decomposed in at most ten intervals such that $F'$ is monotonic within each such interval.

Thus, as desired, we are now in the setting prescribed by Lemma \ref{VdC}. Consequently, from \eqref{VDC1} we deduce that

%\beq\label{jstr121}
\[ \J_H(s)\lesssim 2^{\frac{3a m}{2}}(\bar{\eta} 2^{\frac{am}{2}}|s|)^{-1}\lesssim 2^{\frac{3a m}{2}} 2^{\nu a m}(\bar{\eta})^{-1}. \]
%\eeq
where in the last inequality we used the fact that $s\in J_0$.
\vspace{1cm}

\noindent\textsf{For Case 2:} From \eqref{js110}, for $a=4$, we have

%\beq\label{js1100}
\[ \J(s)\lesssim 2^{\frac{3a m}{2}} 2^{\nu a m} (\bar{\eta})^{-1} + 2^{2 am}\bar{\eta}^{\frac{1}{2}}, \]
%\eeq
and hence after a variational argument, taking $\bar{\eta}= 2^{\frac{2}{3}(\nu-\frac{1}{2}) a m}$
%\beq\label{js11000}
\[ \J(s)\lesssim 2^{(\frac{11}{6} + \frac{\nu}{3}) a m}. \]
%\eeq

Recalling now \eqref{keytr1} we deduce
\begin{align*}
\I(0)^2 & \lesssim 2^{(\mu-1)\frac{am}{2}} 2^{(\frac{11}{12} + \frac{\nu}{6})a m} \int_{J^0\times I^0} \sum_{v\in U_{\mu}(g)} | g_v(t) g_v(t-s) |^2 \one_{I^0}(t-s)\,dt\,ds \\
 & \lesssim 2^{(\mu-1)\frac{am}{2}} 2^{(\frac{11}{12} + \frac{\nu}{6})am}2^{(\mu-1)\frac{am}{2}}\leq 2^{-(\frac{1}{12}-\mu -\frac{\nu}{6}) am}.
\end{align*}
From \eqref{mathf2}, we further have
\beq\label{mathf2s}
\I^{>}\lesssim 2^{-(\frac{1}{24}-\frac{\mu }{2}-\frac{\nu}{12}) am}.
\eeq
Finally, using \eqref{decsq1} and applying another variational argument between \eqref{max} and \eqref{mathf2s}, we conclude
\beq\label{mathf2ss}
\I\lesssim 2^{-\frac{6}{13}(\frac{1}{12}-\mu)am},
\eeq
where here we picked $\nu=\frac{6}{13}(\frac{1}{12}-\mu)$.
$\newline$

\noindent\textsf{For Case 3:} From \eqref{js111} we have
%\beq\label{js1112}
\[ \J(s)\lesssim_a 2^{\frac{3a m}{2}} 2^{\nu a m} (\bar{\eta})^{-1} + 2^{2 am}\bar{\eta}^{\frac{1}{3}}, \]
%\eeq
and hence  taking $\bar{\eta}= 2^{\frac{3}{4}(\nu-\frac{1}{2}) am}$ we get
%\beq\label{js110012}
\[ \J(s)\lesssim_a 2^{(\frac{15}{8} + \frac{\nu}{4})am}. \]
%\eeq
As before, from \eqref{mathf2}, we have
\beq\label{mathf2s2}
\I^{>}\lesssim_a 2^{-(\frac{1}{32}-\frac{\mu }{2}-\frac{\nu}{16})am}.
\eeq
Further on, from a variational argument between \eqref{max} and \eqref{mathf2s2} via \eqref{decsq1} we have
\beq\label{mathf2ss2}
\I\lesssim_a 2^{-\frac{1}{34}(1-16\mu)am},
\eeq
where here $\nu=\frac{1}{34}(1-16\mu)$.

Conclude from \eqref{mathf2ss} and \eqref{mathf2ss2} that
\beq\label{finaunifconclude}
|\underline{\Lambda}_{m}^{a,U_{\mu}}(f,g)|\lesssim_a 2^{-\frac{1}{68}(1-16 \mu) a m} \|f\|_{L^2}\|g\|_{L^2},
\eeq
and hence that \eqref{finaunif} holds when $a\in(0,\infty)\setminus\{1,2,3\}$ for $\d_1=\frac{1}{68}(1-16 \mu)$.

\subsubsection{The Case 1: $a=3$}

The analysis in this case turns out to be more delicate since the decay offered by \eqref{L011} depends on the structure of the linearizing function $\tilde{\l}$. Indeed, a closer inspection of \eqref{L011} reveals the key underlying philosophical difficulty: as $\tilde{\l}$ gets ``closer'' to a straight line the strength of the decay weakens. Thus it becomes natural to perform a quantitative analysis of ``the amount of linearity'' encoded into the graph of $\tilde{\l}$.

%%%%%%%%%%%%%%%%%%%%%%%%%%%%%%%%%%%%%%%%%%%
%
% FIGURE - picture illustrating the tubes
%
%%%%%%%%%%%%%%%%%%%%%%%%%%%%%%%%%%%%%%%%%%%
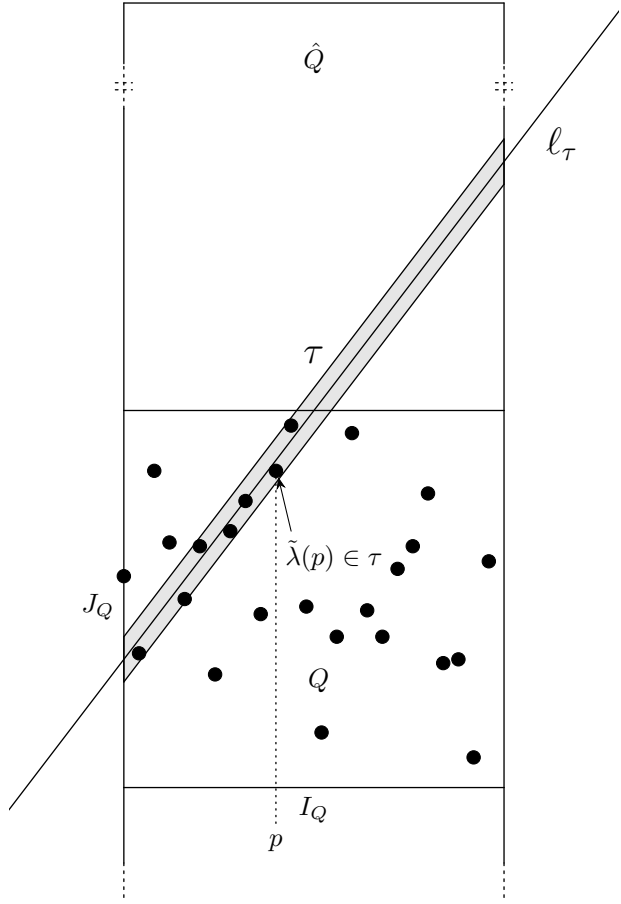
\begin{figure}[h]
\centering
\begin{tikzpicture}[line cap=round,line join=round,>=Stealth,x=1cm,y=1cm]
\clip(-1.5,-1.5) rectangle (6.5,10.5);

%% Q square
\fill[line width=0.5pt,fill opacity=0] (0,0) -- (5,0) -- (5,5) -- (0,5) -- cycle;
\draw [line width=0.5pt] (0,0)-- (5,0);
\draw [line width=0.5pt] (5,0)-- (5,5);
\draw [line width=0.5pt] (5,5)-- (0,5);
\draw [line width=0.5pt] (0,5)-- (0,0);
\draw (2.3,1.7) node[anchor=north west] {$Q$};
\draw (2.5,0) node[anchor=north] {$I_Q$};
\draw (0,2.4) node[anchor=east] {$J_Q$};
%% parallelogram containing l_qr
%\fill[line width=0.5pt,color=black,fill=black,fill opacity=0.1] (0,-0.4) -- (0,-1) -- (5,5) -- (5,5.6) -- cycle;
%\draw [line width=0.5pt,color=black] (0,-0.4)-- (0,-1);
%\draw [line width=0.5pt,color=black] (0,-1)-- (5,5);
%\draw [line width=0.5pt,color=black] (5,5)-- (5,5.6);
%\draw [line width=0.5pt,color=black] (5,5.6)-- (0,-0.4);
%% parallelogram high slope
\fill[line width=0.5pt,color=black,fill=black,fill opacity=0.1] (0,2) -- (0,1.4) -- (5,8) -- (5,8.6) -- cycle;
\draw [line width=0.5pt,color=black] (0,2)-- (0,1.4);
\draw [line width=0.5pt,color=black] (0,1.4)-- (5,8);
\draw [line width=0.5pt,color=black] (5,8)-- (5,8.6);
\draw [line width=0.5pt,color=black] (5,8.6)-- (0,2);
\draw (2.5,5.5) node[anchor=south] {\LARGE $\tau$};
\draw (2,3.4) node[anchor=north west] {$\tilde{\lambda}(p) \in \tau$};
\draw [->] (2.2,3.4)-- (2.03,4.12);
%% parallelogram low slope
%\fill[line width=0.5pt,color=black,fill=black,fill opacity=0.1] (0,3.2) -- (0,2.6) -- (5,1.4) -- (5,2) -- cycle;
%\draw [line width=0.5pt,color=black] (0,3.2)-- (0,2.6);
%\draw [line width=0.5pt,color=black] (0,2.6)-- (5,1.4);
%\draw [line width=0.5pt,color=black] (5,1.4)-- (5,2);
%\draw [line width=0.5pt,color=black] (5,2)-- (0,3.2);

\draw [line width=0.5pt, dotted] (2,4.2)--(2,-0.5);
\draw (2,-0.5) node[anchor=north] {$p$};

%% middle line l_tau
\draw [line width=0.5pt,domain=-5.444349942765321:15.557738432698038] plot(\x,{(--8.5--6.6*\x)/5});
\draw (5.42907169843202,8.87523323912829) node[anchor=north west] {\LARGE $\ell_{\tau}$};
%% l_qr line
%\draw [line width=0.5pt,domain=-5.444349942765321:15.557738432698038] plot(\x,{(--2.2186666666666666-3.6746666666666665*\x)/-3.2});
%\draw (5.4,5.6) node[anchor=north west] {$\ell_{q,r}$};

%% lower half of \hat{Q}
\draw [line width=0.5pt] (0,0)-- (0,-1);
\draw [line width=0.5pt] (5,0)-- (5,-1);
\draw [line width=0.5pt,dash pattern=on 1pt off 2pt] (0,-1)-- (0,-1.5);
\draw [line width=0.5pt,dash pattern=on 1pt off 2pt] (5,-1)-- (5,-1.5);

\draw (2.5,10) node[anchor=north] {$\hat{Q}$};

\draw [line width=0.5pt] (0,5)-- (0,9);
\draw [line width=0.5pt,dash pattern=on 1pt off 2pt] (0,9)-- (0,9.598606253100192);
\draw [line width=0.5pt,dash pattern=on 1pt off 2pt] (-0.11171323603969871,9.252755944866891) -- (0.11171323603969871,9.252755944866891);
\draw [line width=0.5pt,dash pattern=on 1pt off 2pt] (-0.11171323603969871,9.345850308233304) -- (0.11171323603969871,9.345850308233304);
\draw [line width=0.5pt,dash pattern=on 1pt off 2pt] (5,9)-- (5,9.598606253100192);
\draw [line width=0.5pt,dash pattern=on 1pt off 2pt] (4.888286763960302,9.252755944866891) -- (5.111713236039699,9.252755944866891);
\draw [line width=0.5pt,dash pattern=on 1pt off 2pt] (4.888286763960302,9.345850308233304) -- (5.111713236039699,9.345850308233304);
\draw [line width=0.5pt] (5,9.598606253100192)-- (5,10.403385899706668);
\draw [line width=0.5pt] (5,10.403385899706668)-- (0,10.403385899706668);
\draw [line width=0.5pt] (0,10.403385899706668)-- (0,9.598606253100192);
\draw [line width=0.5pt] (5,9)-- (5,5);
%\draw [line width=0.5pt,dotted] (4,-0.5)-- (4,3.9);
%\draw [line width=0.5pt,dotted] (4,3.9)-- (5.3,3.9);
%\draw [line width=0.5pt,dotted] (0.8,-0.5)-- (0.8,0.22533333333333333);
%\draw [line width=0.5pt,dotted] (0.8,0.22533333333333333)-- (-0.3,0.22533333333333333);
%\draw (0.8,-0.5) node[anchor=north] {$r$};
%\draw (4,-0.5) node[anchor=north] {$q$};
%\draw (5.3,3.9) node[anchor=west] {$\tilde{\lambda}^{\mathcal{N}}(q)$};
%\draw (-0.3,0.22533333333333333) node[anchor=east] {$\tilde{\lambda}^{\mathcal{N}}(r)$};
\begin{scriptsize}

%%%%%%%%%%%%%%%%%%%%%%%
% points in L
%%%%%%%%%%%%%%%%%%%%%%%
\draw [fill=black] (0,2.8030905719740478) circle (2.5pt);
\draw [fill=black] (0.2,1.78) circle (2.5pt);
\draw [fill=black] (0.8,2.5) circle (2.5pt);
\draw [fill=black] (1,3.2) circle (2.5pt);
\draw [fill=black] (1.4,3.4) circle (2.5pt);
\draw [fill=black] (1.6,3.8) circle (2.5pt);
\draw [fill=black] (1.8,2.3) circle (2.5pt);
\draw [fill=black] (2,4.2) circle (2.5pt);
\draw [fill=black] (2.2,4.8) circle (2.5pt);
\draw [fill=black] (2.4,2.4) circle (2.5pt);
\draw [fill=black] (2.8,2) circle (2.5pt);
\draw [fill=black] (3.2,2.35) circle (2.5pt);
\draw [fill=black] (4.2,1.65) circle (2.5pt);
\draw [fill=black] (4.4,1.7) circle (2.5pt);
\draw [fill=black] (3.4,2) circle (2.5pt);

%%%%%%%%%%%%%%%%%%%%%
% points in N
%%%%%%%%%%%%%%%%%%%%%

\draw [fill=black] (0.4,4.2) circle (2.5pt);
%\draw [fill=white] (0.4,4.2) circle (1.5pt);

\draw [fill=black] (0.6,3.25) circle (2.5pt);

\draw [fill=black] (1.2,1.5) circle (2.5pt);
%\draw [fill=white] (1.2,1.5) circle (1.5pt);

\draw [fill=black] (2.6,0.73) circle (2.5pt);
%\fill [white] (2.6,0.73) circle (1.5pt);

\draw [fill=black] (3,4.7) circle (2.5pt);
%\draw [fill=white] (3,4.7) circle (1.5pt);

\draw [fill=black] (3.6,2.9) circle (2.5pt);
%\draw [fill=white] (3.6,2.9) circle (1.5pt);

\draw [fill=black] (3.8,3.2) circle (2.5pt);
%\draw [fill=white] (3.8,3.2) circle (1.5pt);

\draw [fill=black] (4,3.9) circle (2.5pt);
%\draw [fill=white] (4,3.9) circle (1.5pt);

\draw [fill=black] (4.6,0.4) circle (2.5pt);
%\draw [fill=white] (4.6,0.4) circle (1.5pt);

\draw [fill=black] (4.8,3) circle (2.5pt);
\end{scriptsize}
\end{tikzpicture}
\caption{\footnotesize Graph of $\tilde{\lambda}_x$ and tube $\tau$} \label{figure:tubes_1}
\end{figure}

%%%%%%%%%%%%%%%%%%%%%%%%%%%

We start on our way by considering the (spatial) cube $Q=I_{Q}\times J_Q$ where $I_Q=J_Q:=[2^{\frac{3m}{2}-1}, 2^{\frac{3m}{2}+1}]$ and notice that without loss of generality we may assume that $p,q,r$ and $v$ take values inside $I_Q$.

Recalling notation \eqref{lam} we let for $x\in I^0$ and $p\in I_Q\cap \Z$
\beq\label{lam1}
\tilde{\l}_x(p)= c\underline{\l}_x(p) 2^{-\frac{3m}{2}},
\eeq
where here $c>0$ some suitable absolute constant. As above, we remark that we may assume wlog that $\{\tilde{\l}_x(p)\}_{p\in I_Q\cap \Z}\subseteq J_Q$ for all $x\in I^0$.

Let $\epsilon>0$ be a small parameter to be chosen later. Define now $\mathcal{T}$ the family of all tubes (parallelograms) $\t=(\a_{\t},\o_{\t}, I_{\t})$ inside $\hat{Q}:=I_{Q}\times 2^{2\ep m+2}\,J_Q$ such that the vertical edges $\a_{\t}, \o_{\t}$ are dyadic intervals of unit length, $I_{\t}= I_Q$, and their central line\footnote{The
central line $\ell_{\t}$ of $\t$ is obtained by joining $c(\a_{\t})$ with $c(\o_{\t})$ -- the centers of the intervals $\a_{\t}$ and $\o_{\t}$, respectively.} $\ell_{\t}$ has a slope $sl_{\t}$ that obeys $|sl_{\t}|:=\frac{|c(\a_{\t})-c(\o_{\t})|}{|I_{\t}|}\leq 2^{2 \ep m}$. \footnote{The reason for introducing the above elements in the definition for the family  $\mathcal{T}$  will become transparent towards the end of this section when approaching relation \eqref{js1LN}.}

In what follows we will use the following notation: we say that
\beq\label{paral}
\tilde{\l}_x(p)\in\t \qquad \text{ if and only if }\qquad -\frac{1}{2}\leq\tilde{\l}_x(p)-\ell_{\t}(p)<\frac{1}{2}.
\eeq

Also, if $\J\subseteq I_Q\cap \Z$, we introduce
\beq\label{paral1}
E_x^{\J}(\t):=\{p\in \J :  \tilde{\l}_x(p)\in\t \},
\eeq
and define the $(x,\J)$-density of a tube $\t$ by\footnote{Notice the close parallelism with the dictionary involved in the time-frequency discretization of the Quadratic Carleson operator in \cite{lv1}.}
\beq\label{dens}
\Delta_x^{\J}(\t):=\frac{\# E_x^{\J}(\t)}{\# (I_{\t}\cap \Z)}.
\eeq

Notice here that we always have
\[ 0\leq \Delta_x^{\J}(\t)\leq 1. \]

From now on we fix $x\in I^0$. Also fix $\ep>0$ to be chosen later. In what follows we will split the domain of our linearizing function into two components
\beq\label{lindec}
I_Q\cap \Z=\L_{x}\cup \n_{x},
\eeq
where
\begin{itemize}
\item $\L_{x}$ stands for the set of points where $ \tilde{\l}_x(\cdot)$ behaves \emph{linearly} so that the graph of
$\tilde{\l}_x\,\big|_{\L_x}$ can be well approximated by big pieces of Lipschitz graphs;

\item $\n_x$ represents the set of points where $ \tilde{\l}_x(\cdot)$ is \emph{non-flat} such that the graph of
$\tilde{\l}_x\,\big|_{\n_x}$ \emph{can not} be covered by a (suitably) \emph{small} number of tubes $\t \in \mathcal{T}$.
\end{itemize}

We remark here that the decomposition in \eqref{lindec} is not unique and also that it is possible that one of the two sets therein is empty. With these done, we pass to the explicit, quantitative realization of \eqref{lindec}:

$\newline$
\noindent\textsf{Construction of $\L_{x}$ and $\n_{x}$}
$\newline$

This construction is realized via the following greedy algorithm:
\begin{enumerate}
\item set $\L_x,\,\mathcal{T}_x^{H}=\emptyset$;

\item set $\n_x=(I_Q\cap \Z)\setminus \L_x$;

\item select $\bar{\t}\in \mathcal{T}\setminus\mathcal{T}_x^{H}$ such that $\Delta_x^{\n_x}(\bar{\t})\geq 2^{-2 \ep m}$ maximal with this property. If no such $\bar{\t}$ exists then Stop; otherwise move to the next item.

\item update $\L_x=\L_x^{old}\cup E_x^{\n_x}(\bar{\t})$;

\item update $\mathcal{T}_x^{H}=\mathcal{T}_x^{H,old}\cup \{\bar{\t}\}$;

\item return to item 2.
\end{enumerate}

Notice that this algorithm must stop in at most $2^{2 \ep m}$ steps.

$\newline$
\noindent\textsf{Properties of $\L_{x}$ and $\n_{x}$}
$\newline$

There exists $N_x\in\N$ with
\beq\label{Nx-def}
N_x\leq 2^{2\ep m}\,,
\eeq
such that $\mathcal{T}_x^{H}=\{\t_{j}\}_{j=1}^{N_x}$ where here the index $j$ reflects the position in the selection of $\t_j$ in the above algorithm. For $j\in\{1,\ldots,\, N_x\}$, we set\footnote{Here, for a generic $\t\in \mathcal{T}$, we denote with $E_x(\t):=E_x^{I_Q\cap \Z}(\t)$ and by convention let $ E_x(\t_{0})=\emptyset$.}
\[ \bar{E}_x(\t_j):=E_x(\t_j)\setminus \bigcup_{k=1}^{j-1}E_x(\t_{k}) \quad \text{ and } \quad \bar{\Delta}_x(\t_j):=\frac{\#\bar{E}_x(\t_j)}{\# (I_{\t_j}\cap \Z)}. \]

Now from the algorithm above we deduce that the set $\L_x$ can be covered disjointly by a small collection of ``heavy'' tubes, \textit{i.e.} each such tube has a high density:
\beq\label{Lx}
\L_{x}=\bigcup_{j=1}^{N_x} \bar{E}_x(\t_j)\quad \text{ with }\quad \bar{\Delta}_x(\t_j)\geq 2^{2\ep m}.
\eeq

In contrast with the above, the complement set $\n_x$ can not concentrate in a suitably small neighbourhood of any given tube $\t\in\mathcal{T}$, that is

\beq\label{Nx}
\forall \t\in\mathcal{T}, \hspace{1cm} \frac{\# \{p\in \n_x :  |\tilde{\l}_x(p)-\ell_{\t}(p)|\leq 2^{\ep m-10}\}}{\# (I_{\t}\cap \Z)}\leq 2^{-\ep m}.
\eeq

Indeed, if \eqref{Nx} would fail then via a pigeonhole principle argument we would be able to identify at least one $\t\in\mathcal{T}$ such that
\beq\label{Nx1}
\#E_x^{\n_x}(\t) \geq 2^{-2\ep m+5}\,\# (I_{\t}\cap \Z).
\eeq
However \eqref{Nx1} contradicts the construction algorithm of $\L_{x}$ and $\n_{x}$.

$\newline$
\noindent\textsf{Construction of a family of auxiliary linearizing functions}
$\newline$

In this subsection, for a fixed $x\in I^0$,  we introduce two families of linearizing functions:
\begin{enumerate}
\item The first one corresponds to the set of points $\L_x$ and consists of functions $\{\tilde{\l}_x^{\L,j}\}_{j=1}^{2^{2\ep m}}$ so that for each such $j$ the graph of $\tilde{\l}_x^{\L,j}$ approximates the central line of $\t_j$.

\item The second one consists of a single element, $\tilde{\l}_x^{\n}$, that incorporates the behavior of the initial linearizing function $\tilde{\l}_x$ on the set $\n_{x}$ thus isolating the ``curved'' (quadratic) behavior of $\tilde{\l}_x$.
\end{enumerate}
We emphasize here that our later number theoretical reasonings depend in a key fashion on the fact that each of the above linearizing functions is extended (and hence well defined) on the full domain $I_Q\cap \Z$ while preserving its key flat/nonflat features. This extension can be operated due to the positivity of the expression $\I(\tilde{\l})$ in \eqref{mathf3} below which facilitates the fundamental relation \eqref{mathf4} thus allowing an independent treatment of the linear and curved cases.

Concretely, with the previous notations, we define the following functions:

For $1\leq j\leq 2^{2\ep m}$ we define $\tilde{\l}_x^{\L,j}:\,I_{Q}\cap \Z\,\mapsto\,J_{Q}$ by
\begin{equation}
\tilde{\l}_x^{\L,j}(p):= \begin{cases}
\tilde{\l}_x(p) & \text{if } 1\leq j\leq N_x \text{ and } p\in \bar{E}_x(\t_j), \\
\ell_{\t_j}(p) &\text{if } 1\leq j\leq N_x \text{ and } p\in (I_{Q}\cap \Z)\setminus \bar{E}_x(\t_j), \\
p  & \text{if } N_x < j\leq 2^{2\ep m}\text{ and } p\in (I_{Q}\cap \Z).
\end{cases}\label{lamextL}
\end{equation}
Also we define  $\tilde{\l}_x^{\n}:\,I_{Q}\cap \Z\,\mapsto\,J_{Q}$ by
\begin{equation}
\tilde{\l}_x^{\n}(p):= \begin{cases}
\tilde{\l}_x(p) & \text{if } p\in \n_x, \\
2^{\frac{3m}{2}-1} + 2^{-\frac{3m}{2}}(p-2^{\frac{3m}{2}})^2   &\text{if } p\in (I_{Q}\cap \Z)\setminus \n_x.
\end{cases} \label{lamextN}
\end{equation}

%%%%%%%%%%%%%%%%%%%%%%%%%%%%%%%%%%%%%%%%%%%
%
% FIGURE - picture illustrating the lambda^L case and lambda^N graph
%
%%%%%%%%%%%%%%%%%%%%%%%%%%%%%%%%%%%%%%%%%%%
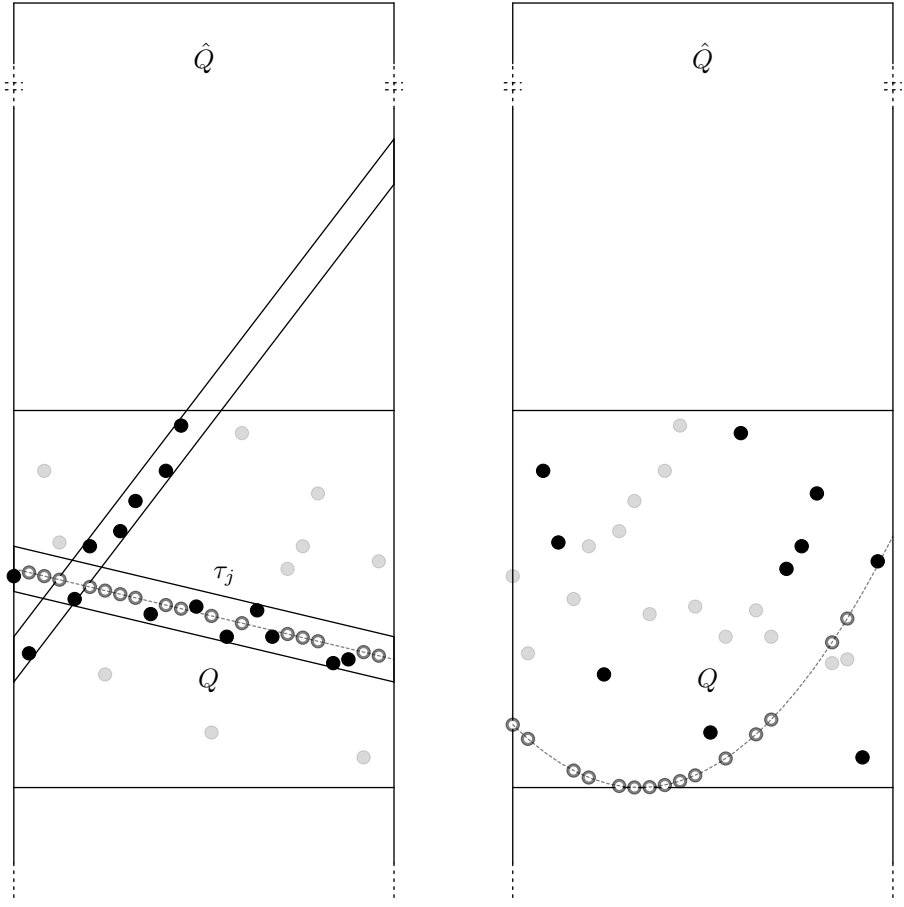
\begin{figure}[h]
\centering
\begin{minipage}{0.47\textwidth}
\centering
\begin{tikzpicture}[line cap=round,line join=round,>=triangle 45,x=1cm,y=1cm]
\clip(-0.5,-1.5) rectangle (6.5,10.5);

%% Q square
\fill[line width=0.5pt,fill opacity=0] (0,0) -- (5,0) -- (5,5) -- (0,5) -- cycle;
\draw [line width=0.5pt] (0,0)-- (5,0);
\draw [line width=0.5pt] (5,0)-- (5,5);
\draw [line width=0.5pt] (5,5)-- (0,5);
\draw [line width=0.5pt] (0,5)-- (0,0);
\draw (2.3,1.7) node[anchor=north west] {$Q$};
%\draw (2.5,0) node[anchor=north] {$I_Q$};
%\draw (0,2.4) node[anchor=east] {$J_Q$};

%% parallelogram containing l_qr
%\fill[line width=0.5pt,color=black,fill=black,fill opacity=0.1] (0,-0.4) -- (0,-1) -- (5,5) -- (5,5.6) -- cycle;
%\draw [line width=0.5pt,color=black] (0,-0.4)-- (0,-1);
%\draw [line width=0.5pt,color=black] (0,-1)-- (5,5);
%\draw [line width=0.5pt,color=black] (5,5)-- (5,5.6);
%\draw [line width=0.5pt,color=black] (5,5.6)-- (0,-0.4);

%% parallelogram high slope
%\fill[line width=0.5pt,color=black,fill=black,fill opacity=0.1] (0,2) -- (0,1.4) -- (5,8) -- (5,8.6) -- cycle;
\draw [line width=0.5pt,color=black] (0,2)-- (0,1.4);
\draw [line width=0.5pt,color=black] (0,1.4)-- (5,8);
\draw [line width=0.5pt,color=black] (5,8)-- (5,8.6);
\draw [line width=0.5pt,color=black] (5,8.6)-- (0,2);
%\draw (2.5,5.5) node[anchor=south] {$\tau$};
%\draw (2,4.2) node[anchor=north west] {$\tilde{\lambda}(p) \in \tau$};

%% parallelogram low slope
%\fill[line width=0.5pt,color=black,fill=black,fill opacity=0.1] (0,3.2) -- (0,2.6) -- (5,1.4) -- (5,2) -- cycle;
\draw [line width=0.5pt,color=black] (0,3.2)-- (0,2.6);
\draw [line width=0.5pt,color=black] (0,2.6)-- (5,1.4);
\draw [line width=0.5pt,color=black] (5,1.4)-- (5,2);
\draw [line width=0.5pt,color=black] (5,2)-- (0,3.2);
\draw (2.5,2.55) node[anchor=south west] {$\tau_j$};

%% middle line l_tau
%\draw [line width=0.5pt,domain=-5.444349942765321:15.557738432698038] plot(\x,{(--8.5--6.6*\x)/5});
%\draw (5.42907169843202,8.87523323912829) node[anchor=north west] {$\ell_\tau$};
%% l_qr line
%\draw [line width=0.5pt,domain=-5.444349942765321:15.557738432698038] plot(\x,{(--2.2186666666666666-3.6746666666666665*\x)/-3.2});
%\draw (5.4,5.6) node[anchor=north west] {$\ell_{q,r}$};

%% lower half of \hat{Q}
\draw [line width=0.5pt] (0,0)-- (0,-1);
\draw [line width=0.5pt] (5,0)-- (5,-1);
\draw [line width=0.5pt,dash pattern=on 1pt off 2pt] (0,-1)-- (0,-1.5);
\draw [line width=0.5pt,dash pattern=on 1pt off 2pt] (5,-1)-- (5,-1.5);

\draw (2.5,10) node[anchor=north] {$\hat{Q}$};

\draw [line width=0.5pt] (0,5)-- (0,9);
\draw [line width=0.5pt,dash pattern=on 1pt off 2pt] (0,9)-- (0,9.598606253100192);
\draw [line width=0.5pt,dash pattern=on 1pt off 2pt] (-0.11171323603969871,9.252755944866891) -- (0.11171323603969871,9.252755944866891);
\draw [line width=0.5pt,dash pattern=on 1pt off 2pt] (-0.11171323603969871,9.345850308233304) -- (0.11171323603969871,9.345850308233304);
\draw [line width=0.5pt,dash pattern=on 1pt off 2pt] (5,9)-- (5,9.598606253100192);
\draw [line width=0.5pt,dash pattern=on 1pt off 2pt] (4.888286763960302,9.252755944866891) -- (5.111713236039699,9.252755944866891);
\draw [line width=0.5pt,dash pattern=on 1pt off 2pt] (4.888286763960302,9.345850308233304) -- (5.111713236039699,9.345850308233304);
\draw [line width=0.5pt] (5,9.598606253100192)-- (5,10.403385899706668);
\draw [line width=0.5pt] (5,10.403385899706668)-- (0,10.403385899706668);
\draw [line width=0.5pt] (0,10.403385899706668)-- (0,9.598606253100192);
\draw [line width=0.5pt] (5,9)-- (5,5);
%\draw [line width=0.5pt,dotted] (4,-0.5)-- (4,3.9);
%\draw [line width=0.5pt,dotted] (4,3.9)-- (5.3,3.9);
%\draw [line width=0.5pt,dotted] (0.8,-0.5)-- (0.8,0.22533333333333333);
%\draw [line width=0.5pt,dotted] (0.8,0.22533333333333333)-- (-0.3,0.22533333333333333);
%\draw (0.8,-0.5) node[anchor=north] {$r$};
%\draw (4,-0.5) node[anchor=north] {$q$};
%\draw (5.3,3.9) node[anchor=west] {$\tilde{\lambda}^{\mathcal{N}}(q)$};
%\draw (-0.3,0.22533333333333333) node[anchor=east] {$\tilde{\lambda}^{\mathcal{N}}(r)$};
\begin{scriptsize}

%%%%%%%%%%%%%%%%%%%%%%%
% points in L
%%%%%%%%%%%%%%%%%%%%%%%
\draw [fill=black] (0,2.8030905719740478) circle (2.5pt);
\draw [fill=black] (0.2,1.78) circle (2.5pt);
\draw [fill=black] (0.8,2.5) circle (2.5pt);
\draw [fill=black] (1,3.2) circle (2.5pt);
\draw [fill=black] (1.4,3.4) circle (2.5pt);
\draw [fill=black] (1.6,3.8) circle (2.5pt);
\draw [fill=black] (1.8,2.3) circle (2.5pt);
\draw [fill=black] (2,4.2) circle (2.5pt);
\draw [fill=black] (2.2,4.8) circle (2.5pt);
\draw [fill=black] (2.4,2.4) circle (2.5pt);
\draw [fill=black] (2.8,2) circle (2.5pt);
\draw [fill=black] (3.2,2.35) circle (2.5pt);
\draw [fill=black] (4.2,1.65) circle (2.5pt);
\draw [fill=black] (4.4,1.7) circle (2.5pt);
\draw [fill=black] (3.4,2) circle (2.5pt);

%%%%%%%%%%%%%%%%%%%%%
% points in N
%%%%%%%%%%%%%%%%%%%%%
\draw [fill=black, opacity=0.15] (0.4,4.2) circle (2.5pt);

\draw [fill=black, opacity=0.15] (0.6,3.25) circle (2.5pt);

\draw [fill=black, opacity=0.15] (1.2,1.5) circle (2.5pt);

\draw [fill=black, opacity=0.15] (2.6,0.73) circle (2.5pt);

\draw [fill=black, opacity=0.15] (3,4.7) circle (2.5pt);

\draw [fill=black, opacity=0.15] (3.6,2.9) circle (2.5pt);

\draw [fill=black, opacity=0.15] (3.8,3.2) circle (2.5pt);

\draw [fill=black, opacity=0.15] (4,3.9) circle (2.5pt);

\draw [fill=black, opacity=0.15] (4.6,0.4) circle (2.5pt);

\draw [fill=black, opacity=0.15] (4.8,3) circle (2.5pt);
%
%%%%%%%%%%%%%%%%%%%%%%%%%%
%% points in the tube that are added in the definition of \lambda^L
%%%%%%%%%%%%%%%%%%%%%%%%%%
%\draw [fill=black,opacity=0.5] (5,1.7) circle (2.5pt);
%\fill [white] (5,1.7) circle (1.5pt);
\draw [fill=black,opacity=0.5] (0.2,2.852) circle (2.5pt);
\fill [white] (0.2,2.852) circle (1.5pt);
\draw [fill=black,opacity=0.5] (0.4,2.804) circle (2.5pt);
\fill [white] (0.4,2.804) circle (1.5pt);
\draw [fill=black,opacity=0.5] (0.6,2.756) circle (2.5pt);
\fill [white] (0.6,2.756) circle (1.5pt);
\draw [fill=black,opacity=0.5] (1,2.66) circle (2.5pt);
\fill [white] (1,2.66) circle (1.5pt);
\draw [fill=black,opacity=0.5] (1.2,2.612) circle (2.5pt);
\fill [white] (1.2,2.612) circle (1.5pt);
\draw [fill=black,opacity=0.5] (1.4,2.564) circle (2.5pt);
\fill [white] (1.4,2.564) circle (1.5pt);
\draw [fill=black,opacity=0.5] (1.6,2.516) circle (2.5pt);
\fill [white] (1.6,2.516) circle (1.5pt);
\draw [fill=black,opacity=0.5] (2,2.42) circle (2.5pt);
\fill [white] (2,2.42) circle (1.5pt);
\draw [fill=black,opacity=0.5] (2.2,2.372) circle (2.5pt);
\fill [white] (2.2,2.372) circle (1.5pt);
\draw [fill=black,opacity=0.5] (2.6,2.276) circle (2.5pt);
\fill [white] (2.6,2.276) circle (1.5pt);
\draw [fill=black,opacity=0.5] (3,2.18) circle (2.5pt);
\fill [white] (3,2.18) circle (1.5pt);
\draw [fill=black,opacity=0.5] (3.6,2.036) circle (2.5pt);
\fill [white] (3.6,2.036) circle (1.5pt);
\draw [fill=black,opacity=0.5] (3.8,1.988) circle (2.5pt);
\fill [white] (3.8,1.988) circle (1.5pt);
\draw [fill=black,opacity=0.5] (4,1.94) circle (2.5pt);
\fill [white] (4,1.94) circle (1.5pt);
\draw [fill=black,opacity=0.5] (4.6,1.796) circle (2.5pt);
\fill [white] (4.6,1.796) circle (1.5pt);
\draw [fill=black,opacity=0.5] (4.8,1.748) circle (2.5pt);
\fill [white] (4.8,1.748) circle (1.5pt);
\draw [line width=0.5pt,dash pattern=on 1pt off 1pt, opacity=0.5] (0,2.9)-- (5,1.7);
\end{scriptsize}
\end{tikzpicture}
\end{minipage}\hspace{0.5cm}
\begin{minipage}{0.47\textwidth}
\centering
\begin{tikzpicture}[line cap=round,line join=round,>=triangle 45,x=1cm,y=1cm]
\clip(-0.5,-1.5) rectangle (6.5,10.5);

%% Q square
\fill[line width=0.5pt,fill opacity=0] (0,0) -- (5,0) -- (5,5) -- (0,5) -- cycle;
\draw [line width=0.5pt] (0,0)-- (5,0);
\draw [line width=0.5pt] (5,0)-- (5,5);
\draw [line width=0.5pt] (5,5)-- (0,5);
\draw [line width=0.5pt] (0,5)-- (0,0);
\draw (2.3,1.7) node[anchor=north west] {$Q$};
\draw [line width=0.5pt] (0,0)-- (0,-1);
\draw [line width=0.5pt] (5,0)-- (5,-1);
\draw [line width=0.5pt,dash pattern=on 1pt off 2pt] (0,-1)-- (0,-1.5);
\draw [line width=0.5pt,dash pattern=on 1pt off 2pt] (5,-1)-- (5,-1.5);

\draw (2.5,10) node[anchor=north] {$\hat{Q}$};

\draw [line width=0.5pt] (0,5)-- (0,9);
\draw [line width=0.5pt,dash pattern=on 1pt off 2pt] (0,9)-- (0,9.598606253100192);
\draw [line width=0.5pt,dash pattern=on 1pt off 2pt] (-0.11171323603969871,9.252755944866891) -- (0.11171323603969871,9.252755944866891);
\draw [line width=0.5pt,dash pattern=on 1pt off 2pt] (-0.11171323603969871,9.345850308233304) -- (0.11171323603969871,9.345850308233304);
\draw [line width=0.5pt,dash pattern=on 1pt off 2pt] (5,9)-- (5,9.598606253100192);
\draw [line width=0.5pt,dash pattern=on 1pt off 2pt] (4.888286763960302,9.252755944866891) -- (5.111713236039699,9.252755944866891);
\draw [line width=0.5pt,dash pattern=on 1pt off 2pt] (4.888286763960302,9.345850308233304) -- (5.111713236039699,9.345850308233304);
\draw [line width=0.5pt] (5,9.598606253100192)-- (5,10.403385899706668);
\draw [line width=0.5pt] (5,10.403385899706668)-- (0,10.403385899706668);
\draw [line width=0.5pt] (0,10.403385899706668)-- (0,9.598606253100192);
\draw [line width=0.5pt] (5,9)-- (5,5);

%% coordinates of q,r points
%\draw [line width=0.5pt,dotted] (4,-0.5)-- (4,3.9);
%\draw [line width=0.5pt,dotted] (4,3.9)-- (5.3,3.9);
%\draw [line width=0.5pt,dotted] (0.8,-0.5)-- (0.8,0.22533333333333333);
%\draw [line width=0.5pt,dotted] (0.8,0.22533333333333333)-- (-0.3,0.22533333333333333);
%\draw (0.8,-0.5) node[anchor=north] {$r$};
%\draw (4,-0.5) node[anchor=north] {$q$};
%\draw (5.3,3.9) node[anchor=west] {$\tilde{\lambda}^{\mathcal{N}}(q)$};
%\draw (-0.3,0.22533333333333333) node[anchor=east] {$\tilde{\lambda}^{\mathcal{N}}(r)$};
\begin{scriptsize}

%%%%%%%%%%%%%%%%%%%%%%%
% points in L
%%%%%%%%%%%%%%%%%%%%%%%
\draw [fill=black, opacity=0.15] (0,2.8030905719740478) circle (2.5pt);
\draw [fill=black, opacity=0.15] (0.2,1.78) circle (2.5pt);
\draw [fill=black, opacity=0.15] (0.8,2.5) circle (2.5pt);
\draw [fill=black, opacity=0.15] (1,3.2) circle (2.5pt);
\draw [fill=black, opacity=0.15] (1.4,3.4) circle (2.5pt);
\draw [fill=black, opacity=0.15] (1.6,3.8) circle (2.5pt);
\draw [fill=black, opacity=0.15] (1.8,2.3) circle (2.5pt);
\draw [fill=black, opacity=0.15] (2,4.2) circle (2.5pt);
\draw [fill=black, opacity=0.15] (2.2,4.8) circle (2.5pt);
\draw [fill=black, opacity=0.15] (2.4,2.4) circle (2.5pt);
\draw [fill=black, opacity=0.15] (2.8,2) circle (2.5pt);
\draw [fill=black, opacity=0.15] (3.2,2.35) circle (2.5pt);
\draw [fill=black, opacity=0.15] (4.2,1.65) circle (2.5pt);
\draw [fill=black, opacity=0.15] (4.4,1.7) circle (2.5pt);
\draw [fill=black, opacity=0.15] (3.4,2) circle (2.5pt);

%%%%%%%%%%%%%%%%%%%%%
% points in N
%%%%%%%%%%%%%%%%%%%%%
\draw [fill=black] (0.4,4.2) circle (2.5pt);

\draw [fill=black] (0.6,3.25) circle (2.5pt);

\draw [fill=black] (1.2,1.5) circle (2.5pt);

\draw [fill=black] (2.6,0.73) circle (2.5pt);

\draw [fill=black] (3,4.7) circle (2.5pt);

\draw [fill=black] (3.6,2.9) circle (2.5pt);

\draw [fill=black] (3.8,3.2) circle (2.5pt);

\draw [fill=black] (4,3.9) circle (2.5pt);

\draw [fill=black] (4.6,0.4) circle (2.5pt);

\draw [fill=black] (4.8,3) circle (2.5pt);
%
%%%%%%%%%%%%%%%%%%%%%%%%%%
%% points in the tube that are added in the definition of \lambda^L
%%%%%%%%%%%%%%%%%%%%%%%%%%
%\draw [fill=black] (5,1.7) circle (2.5pt);
%\draw [fill=black] (0.2,2.852) circle (2.5pt);
%\draw [fill=black] (0.4,2.804) circle (2.5pt);
%\draw [fill=black] (0.6,2.756) circle (2.5pt);
%\draw [fill=black] (1,2.66) circle (2.5pt);
%\draw [fill=black] (1.2,2.612) circle (2.5pt);
%\draw [fill=black] (1.4,2.564) circle (2.5pt);
%\draw [fill=black] (1.6,2.516) circle (2.5pt);
%\draw [fill=black] (2,2.42) circle (2.5pt);
%\draw [fill=black] (2.2,2.372) circle (2.5pt);
%\draw [fill=black] (2.6,2.276) circle (2.5pt);
%\draw [fill=black] (3,2.18) circle (2.5pt);
%\draw [fill=black] (3.6,2.036) circle (2.5pt);
%\draw [fill=black] (3.8,1.988) circle (2.5pt);
%\draw [fill=black] (4,1.94) circle (2.5pt);
%\draw [fill=black] (4.6,1.796) circle (2.5pt);
%\draw [fill=black] (4.8,1.748) circle (2.5pt);
%
%%%%%%%%%%%%%%%%%%%%%%%%%%%%
%% points on the parabola that are added in the definition of \lambda^N
%%%%%%%%%%%%%%%%%%%%%%%%%%%%

\draw [fill=black, opacity=0.5] (0,0.8333333333333334) circle (2.5pt);
\fill [white] (0,0.8333333333333334) circle (1.5pt);

\draw [fill=black, opacity=0.5] (0.2,0.6453333333333334) circle (2.5pt);
\fill [white] (0.2,0.6453333333333334) circle (1.5pt);

\draw [fill=black, opacity=0.5] (0.8,0.22533333333333333) circle (2.5pt);
\fill [white] (0.8,0.22533333333333333) circle (1.5pt);

\draw [fill=black, opacity=0.5] (1,0.13333333333333336) circle (2.5pt);
\fill [white] (1,0.13333333333333336) circle (1.5pt);

\draw [fill=black, opacity=0.5] (1.4,0.021333333333333357) circle (2.5pt);
\fill [white] (1.4,0.021333333333333357) circle (1.5pt);

\draw [fill=black, opacity=0.5] (1.6,0.0013333333333333329) circle (2.5pt);
\fill [white] (1.6,0.0013333333333333329) circle (1.5pt);

\draw [fill=black, opacity=0.5] (1.8,0.005333333333333331) circle (2.5pt);
\fill [white] (1.8,0.005333333333333331) circle (1.5pt);

\draw [fill=black, opacity=0.5] (2,0.03333333333333332) circle (2.5pt);
\fill [white] (2,0.03333333333333332) circle (1.5pt);

\draw [fill=black, opacity=0.5] (2.2,0.08533333333333336) circle (2.5pt);
\fill [white] (2.2,0.08533333333333336) circle (1.5pt);

\draw [fill=black, opacity=0.5] (2.4,0.16133333333333327) circle (2.5pt);
\fill [white] (2.4,0.16133333333333327) circle (1.5pt);

\draw [fill=black, opacity=0.5] (2.8,0.38533333333333314) circle (2.5pt);
\fill [white] (2.8,0.38533333333333314) circle (1.5pt);

\draw [fill=black, opacity=0.5] (3.2,0.7053333333333334) circle (2.5pt);
\fill [white] (3.2,0.7053333333333334) circle (1.5pt);

\draw [fill=black, opacity=0.5] (3.4,0.9013333333333331) circle (2.5pt);
\fill [white] (3.4,0.9013333333333331) circle (1.5pt);

\draw [fill=black, opacity=0.5] (4.2,1.9253333333333331) circle (2.5pt);
\fill [white] (4.2,1.9253333333333331) circle (1.5pt);

\draw [fill=black, opacity=0.5] (4.4,2.2413333333333334) circle (2.5pt);
\fill [white] (4.4,2.2413333333333334) circle (1.5pt);

%\draw [fill=black, opacity=0.5] (5,3.3333333333333326) circle (2.5pt);
%\fill [white] (5,3.3333333333333326) circle (1.5pt);

\draw [black, line width=0.5pt,dash pattern=on 1pt off 1pt, opacity=0.5] plot[domain=0:5] (\x,{(3/10)*(\x-(5/3))*(\x-(5/3))});

\end{scriptsize}
\end{tikzpicture}
\end{minipage}
\caption{\footnotesize On the left, the graph of $\tilde{\lambda}_{x}^{\mathcal{L},j}$ and associated tube $\tau_j$. The black points are the points in $\mathcal{L}$, the shaded points are the points in $\mathcal{N}$ and the points with empty interior are the auxiliary points on the line $\ell_{\tau_j}$ introduced by the definition of $\tilde{\lambda}_{x}^{\mathcal{L},j}$. On the right, the graph of $\tilde{\lambda}_{x}^{\mathcal{N}}$, where this time the black points are the ones in $\mathcal{N}$, the shaded points are those in $\mathcal{L}$ and the points with empty interior are the auxiliary points on the parabola introduced by the definition of $\tilde{\lambda}_{x}^{\mathcal{N}}$.}
\end{figure}

%%%%%%%%%%%%%%%%%%%%%%%%%%%

With this, returning to \eqref{mathf}, we have that
\begin{equation}
\begin{aligned}
& (\underline{\Lambda}_{m}^{3,U_{\mu}}(f,g))^2= \I(\tilde{\l}) \\
&:=\sum_{\substack{u\in U_{\mu}(f) \\ v\in U_{\mu}(g)}} \frac{1}{|I^0|}\int_{I^0}
\Big|\int_{I^0} f_u(x-t) g_v(x+t) e^{i \frac{(v-u)^{2}}{2^{3 m}}\tilde{\l}_x(u+v) 2^{\frac{3m}{2}} t}\,dt\Big|^2 \,dx.
\end{aligned}\label{mathf3}
\end{equation}
Based on \eqref{lamextL}--\eqref{mathf3}  we immediately deduce that
\beq\label{mathf4}
\I(\tilde{\l})\lesssim \I(\tilde{\l}^{\n}) + \sum_{j=1}^{2^{2\ep m}} \I(\tilde{\l}^{\L,j})=:\I(\tilde{\l}^{\n}) + \I(\tilde{\l}^{\L}).
\eeq
Once at this point we follow the strategy described in Section \ref{ttstar}:
\begin{itemize}
\item with the obvious correspondences we perform the analogue decomposition to \eqref{decsq1}, that is
\beq\label{dec1}
\I(\tilde{\l}^{*})=:\I^{<}(\tilde{\l}^{*}) + \I^{>}(\tilde{\l}^{*}),
\eeq
where here $\{*\}\in\{\L,\n\}$. We emphasize here that we will use two distinct cut-off parameters $\nu_{\L},\,\nu_{\n}>0$ corresponding to the parameter $\nu$ considered in the decomposition \eqref{decsq1}--\eqref{decsq2}.

\item since the argument in \eqref{max} is independent on the structure of the linearizing function $\tilde{\l}$ we can maintain the same reasonings in order to deduce that
\beq\label{dec2}
\I^{<}(\tilde{\l}^{\L})\lesssim 2^{2\ep m} 2^{- 3\nu_{\L} m}\qquad\text{ and }\qquad
\I^{<}(\tilde{\l}^{\n})\lesssim 2^{- 3\nu_{\n} m}.
\eeq

\item for the second component, $\I^{>}(\tilde{\l}^{*})$, we apply the analogue of \eqref{mathf2}, and, with the obvious correspondences we write
\beq\label{dec3}
\I^{>}(\tilde{\l}^{*}):=\frac{1}{|I^0|} \int_{I^0} \I(\tilde{\l}^{*})(x)\,dx,
\eeq
with
\beq\label{decc3}
\I(\tilde{\l}^{\L})(x):= \sum_{j=1}^{2^{2\ep m}} \I(\tilde{\l}^{\L,j})(x).
\eeq

\item as in the case $a\in(0,\infty)\setminus\{1,2,3\}$, one can prove a pointwise decay for $\I(\tilde{\l}^{*})(\cdot)$ (thus in effect estimating $\|\I(\tilde{\l}^{*})(x)\|_{L_x^{\infty}(I^0)}$) and hence without loss of generality we will set $x=0$.

\item taking for simplicity $j=1$, $\tilde{\l}_0^{\L,1}=\tilde{\l}^{\L,1}$ and $\tilde{\l}_0^{\n}=\tilde{\l}^{\n}$  it thus remains to estimate the expressions
\begin{equation}
\begin{aligned}
&\I(\tilde{\l}^{\L,1})(0)= \sum_{\substack{u\in U_{\mu}(f) \\ v\in U_{\mu}(g)}}
\int_{J^0}\Big(\int_{I^0}  f_u(t) g_v(t) f_u(t-s) g_v(t-s) \\
& \hspace{12em} \cdot \one_{I^0}(t-s) e^{i \frac{(v-u)^{2}}{2^{3m}}\tilde{\l}^{\L,1}(u+v)2^{\frac{3m}{2}} s}\,dt\Big)\,ds,
\end{aligned}\label{dec4}
\end{equation}
and
\begin{equation}
\begin{aligned}
& \I(\tilde{\l}^{\n})(0)= \sum_{\substack{u\in U_{\mu}(f) \\ v\in U_{\mu}(g)}}
\int_{J^0} \Big(\int_{I^0}  f_u(t) g_v(t) f_u(t-s) g_v(t-s)\\
& \hspace{12em} \cdot \one_{I^0}(t-s) e^{i \frac{(v-u)^{2}}{2^{3m}} \tilde{\l}^{\n}(u+v) 2^{\frac{3m}{2}} s} \,dt\Big)\,ds.
\end{aligned}\label{dec5}
\end{equation}

\item Following similar reasonings to those in \eqref{mathf4-I0}--\eqref{j3} and letting $\{*\}\in\{(\L,1),\n\}$, we reduce our task to that of estimating the exponential sums

\beq\label{j3lin}
\J^{*}(s)= \sum_{p,q,r\sim 2^{\frac{3 m}{2}}}
\sum_{v\sim_{p,q,r} 2^{\frac{3 m}{2}}} e^{i 2^{\frac{3m}{2}} s P_{p,q,r}^{*}(v)},
\eeq
with
\begin{equation}
\begin{aligned}
P_{p,q,r}^{*}(v)&:=\frac{(2v-p)^{2}}{2^{3 m}} \tilde{\l}^{*}(p) - \frac{(2v+q-2p)^{2}}{2^{3 m}} \tilde{\l}^{*}(q)\\
&\hspace{1em} - \frac{(2v-r)^{2}}{2^{3 m}} \tilde{\l}^{*}(r) + \frac{(2v+q-p-r)^{2}}{2^{3 m}} \tilde{\l}^{*}(q+r-p),
\end{aligned}\label{polyn0lin}
\end{equation}

As before, regarding now $P_{p,q,r}^{*}(v)$ as a polynomial in $v$ and writing
\beq\label{Pstar}
P_{p,q,r}^{*}(v)= 4 A v^2 + 4 B v + C ,
\eeq
we have that $A,\,B,\,C$ obey \eqref{coef01}--\eqref{coef03} with the obvious adaptations to our above setting.

\end{itemize}

$\newline$
\noindent\textsf{Treatment of the term $\I(\tilde{\l}^{\L,1})(0)$}
$\newline$

In this specific situation, relying on the fact that the graph of $\tilde{\l}^{\L,1}$ is by construction ``close'' to a line, the key insight is to notice that, by rewriting $P_{p,q,r}^{\L,0}(v)$ as a function of $p$ as opposed to $v$ and hence regarding \eqref{j3lin} as a discrete oscillatory sum in $p$, one can reduce the matters to the previous analysis.

More precisely, based on \eqref{lamextL}, we may assume without loss of generality that
\begin{equation}
\tilde{\l}^{\L,1}(p):= \begin{cases}
\tilde{\l}_0(p) &\text{if }  p\in \bar{E}_0(\t_1), \\
\ell_{\t_1}(p) &\text{if } p\in (I_{Q}\cap \Z)\setminus \bar{E}_0(\t_1).
\end{cases}\label{lamextL0}
\end{equation}

Moreover, since from the above definitions we know that
\beq\label{dif}
|\tilde{\l}_0(p)-\ell_{\t_1}(p)|\leq \frac{1}{2}\quad \text{ for all } p\in \bar{E}_0(\t_1),
\eeq
a standard Taylor argument in the same spirit as the one discussed in Section \ref{Tr} allows us to reduce our discussion to the case in which
\beq\label{lamextL1}
\tilde{\l}^{\L,1}(p)=\ell_{\t_1}(p)\quad\text{ for any }p\in (I_{Q}\cap \Z).
\eeq

Now we notice that due to the linearity of $\tilde{\l}^{\L,1}$ one has in \eqref{Pstar} that
\beq\label{coef01lin}
A=B=0,
\eeq
while from \eqref{coef031} and \eqref{coef01lin} one has
\beq\label{coef02lin}
C=Q_{q,r}(p):=\frac{1}{2^{3m}}(p-r)(q-p)(\tilde{\l}^{\L,1}(q)+\tilde{\l}^{\L,1}(r))=:F_3^2\Big(\frac{p}{2^{\frac{3m}{2}}}\Big),
\eeq
where, as before -- recall \eqref{FF}, one has
\beq\label{FFa3}
F_3^{2}(x)=\sum_{k=1}^3 a_k (x+\a_k)^{2}\:,
\eeq
with
\begin{alignat*}{3}
\a_1&=-\frac{r}{2^{\frac{3 m}{2}}}, \quad && a_1&& =-(\tilde{\l}^{\L,1}(q)+\tilde{\l}^{\L,1}(r)),\\
\a_2 &=-\frac{q}{2^{\frac{3 m}{2}}}, \quad && a_2&&=-\frac{\tilde{\l}^{\L,1}(q)+\tilde{\l}^{\L,1}(r)}{4}, \\
\a_3&=\frac{q-2r}{2^{\frac{3 m}{2}}}, \quad && a_3 &&=\frac{\tilde{\l}^{\L,1}(q) + \tilde{\l}^{\L,1}(r)}{4}.
\end{alignat*}

At this point we rewrite \eqref{j3lin} as
\beq\label{a3lin}
\J^{\L,1}(s)= \sum_{q,r,v\sim 2^{\frac{3 m}{2}}}
\sum_{p\sim_{q,r,v} 2^{\frac{3 m}{2}}}e^{i 2^{\frac{3m}{2}} s Q_{q,r}(p)},
\eeq
and follow the analogue of \eqref{jsplit}, that is
\beq\label{jsplitL}
|\J^{\L,1}(s)| \leq \J_L^{\L,1}(s) + \J_H^{\L,1}(s) ,
\eeq
where here, as expected
\beq\label{jsL}
 \J_L^{\L,1}(s):=\sum_{q,r,v\sim 2^{\frac{3m}{2}}}\bigg|
\sum_{\substack{p\sim_{q,r,v} 2^{\frac{3m}{2}} \\ |Q'_{q,r}(p)|< \bar{\eta}}} e^{i 2^{\frac{3m}{2}} s Q_{q,r}(p)}\bigg|,
\eeq
and
\beq\label{jlL}
\J_H^{\L,1}(s):=\sum_{q,r,v\sim 2^{\frac{3m}{2}}}\bigg|
\sum_{\substack{p\sim_{q,r,v} 2^{\frac{3m}{2}} \\ |Q'_{q,r}(p)|\geq \bar{\eta}}} e^{i 2^{\frac{3m}{2}} s Q_{q,r}(p)}\bigg|.
\eeq

For the treatment of the low oscillatory component $\J_L^{\L,1}$ we are precisely in the setting of Lemma \ref{phasecont}, Case 1, with $\eta=2^{\frac{3 m}{2}}\bar{\eta}$ and\footnote{The parameters $A$ and $B$ are now adapted to the new context, that is, they are defined relative to $Q_{q,r}$.} $|A|+|B|\sim |\tilde{\l}^{\L,1}(q) + \tilde{\l}^{\L,1}(r)|\sim 2^{\frac{3m}{2}}$.

Thus, we deduce that
\beq\label{js1L}
\J_L^{\L,1}(s) \lesssim \sum_{q,r,v\sim 2^{\frac{3m}{2}}} 2^{\frac{3m}{2}}\min\{1,L_{\eta}^{2}(q,r)\}\lesssim 2^{6 m}\bar{\eta}.
\eeq

For the high oscillatory component $\J_H^{\L,1}$ we apply the same Lemma \ref{VdC} in order to deduce that
\beq\label{js1H}
\J_H^{\L,1}(s)\lesssim 2^{\frac{9 m}{2}}(\bar{\eta}2^{\frac{3 m}{2}}|s|)^{-1} \lesssim 2^{\frac{9 m}{2}} 2^{3 \nu_{\L} m} (\bar{\eta})^{-1}.
\eeq

Putting now together \eqref{jsplitL}, \eqref{js1L} and \eqref{js1H} and applying a variational argument in $\bar{\eta}$ -- whose optimal choice turns out to be $2^{\frac{3m}{2}(\nu_{\L}-\frac{1}{2})}$ -- we have that
\beq\label{conclL}
|\J^{\L,1}(s)|\lesssim 2^{6 m} 2^{\frac{3m}{2}(\nu_{\L}-\frac{1}{2})}.
\eeq
Conclude now that
\beq\label{concLm}
\I^{>}(\tilde{\l}^{\L,1})\lesssim  2^{\frac{3m}{8}(4\mu+\nu_{\L}-\frac{1}{2})}.
\eeq

$\newline$
\noindent\textsf{Treatment of the term $\I(\tilde{\l}^{\n})(0)$}
$\newline$

We start by recalling \eqref{dec5} and notice that, due to \eqref{Nx} and \eqref{lamextN} (for $x=0$) and for $\ep>0$ small enough, we have the following key property:
\beq\label{KeyP}
\forall \t\in\mathcal{T}, \qquad \frac{\# \{p\in I_{Q}\cap \Z : |\tilde{\l}^{\n}(p)-\ell_{\t}(p)|\leq 2^{\ep m-10}\}}{\# (I_{Q}\cap \Z)}\leq 2^{-\ep m}.
\eeq

Now, with $\J^{\n}(s)$ and $P_{p,q,r}^{\n}(v)$ defined by \eqref{j3lin}--\eqref{Pstar}, we transpose \eqref{jsplit} into
\beq\label{jsplitN}
|\J^{\n}(s)|\leq\J_L^{\n}(s) + \J_H^{\n}(s),
\eeq
where
\beq\label{jsN}
\J_L^{\n}(s):=\sum_{p,q,r\sim 2^{\frac{3m}{2}}}\bigg|
\sum_{\substack{v\sim_{p,q,r} 2^{\frac{3m}{2}} \\ |\frac{d}{dx}P_{p,q,r}^{\n}(v)|< \bar{\eta}}} e^{i 2^{\frac{3m}{2}} s P_{p,q,r}^{\n}(v)}\bigg|,
\eeq
and
\beq\label{jlN}
\J_H^{\n}(s):=\sum_{p,q,r\sim 2^{\frac{3m}{2}}}\bigg|
\sum_{\substack{v\sim_{p,q,r} 2^{\frac{3m}{2}} \\ |\frac{d}{dx}P_{p,q,r}^{\n}(v)|\geq \bar{\eta}}} e^{i 2^{\frac{3m}{2}} s P_{p,q,r}^{\n}(v)}\bigg|.
\eeq
%\end{proof}

For the low oscillatory component $\J_L^{\n}$ we apply Lemma \ref{phasecont1}, Case 1, for $\eta=2^{\frac{3 m}{2}-1}\bar{\eta}$.

Thus, using \eqref{KeyP}, we deduce that
\begin{equation}
\begin{aligned}
\J_L^{\n}(s) & \lesssim
\sum_{p,q,r\sim 2^{\frac{3m}{2}}} 2^{\frac{3m}{2}} \min\{1, L_{\eta}^{2}(p,q,r)\} \\
& \lesssim 2^{(6-2\ep) m} +
2^{\frac{3m}{2}} \sum_{\substack{q,r\sim 2^{\frac{3m}{2}} \\ |q-r|\gtrsim 2^{(\frac{3}{2}-2\ep)m}}}  \sum_{p\sim 2^{\frac{3m}{2}}}\min\bigg\{1, \\
& \hspace{7em} \frac{2^{3m}\bar{\eta}}{|r-q|\Big|\tilde{\l}^{\n}(p)- \Big(p\frac{\tilde{\l}^{\n}(r)-\tilde{\l}^{\n}(q)}{r-q} + \frac{r \tilde{\l}^{\n}(q) - q \tilde{\l}^{\n}(r)}{r-q}\Big)\Big|}\bigg\} \\
& \lesssim 2^{(6-\ep) m} + 2^{3m} \sum_{\substack{q,r\sim 2^{\frac{3m}{2}} \\ |q-r|\gtrsim 2^{(\frac{3}{2}-2\ep)m}}} \frac{2^{3m} \bar{\eta}}{|r-q|2^{\ep m}} \\
& \lesssim  2^{(6-\ep) m}(1 + m 2^{\frac{3m}{2}} \bar{\eta}).
\end{aligned}\label{js1LN}
\end{equation}

%%%%%%%%%%%%%%%%%%%%%%%%%%%%%%%%%%%%%%%%%%%
%
% FIGURE - picture illustrating the lambda^N case
%
%%%%%%%%%%%%%%%%%%%%%%%%%%%%%%%%%%%%%%%%%%%
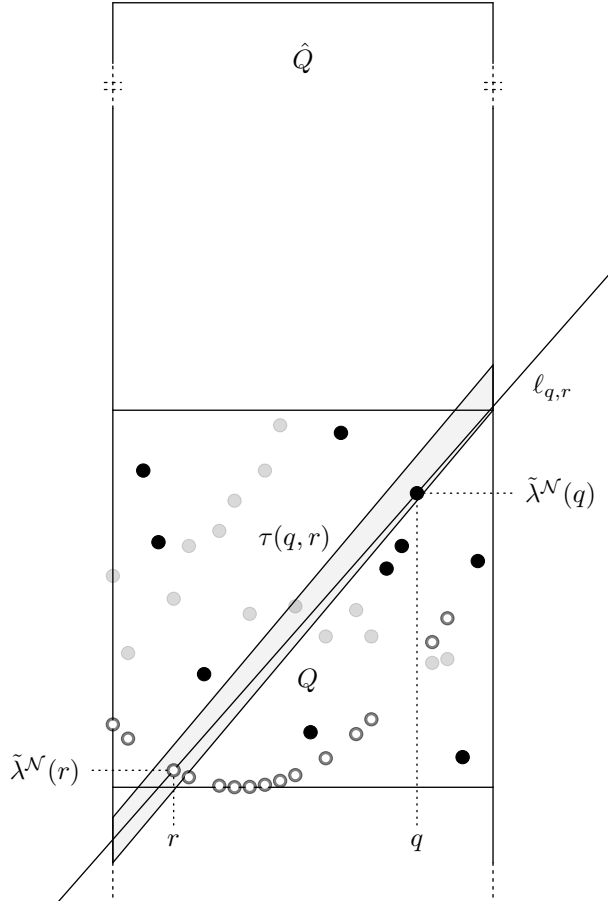
\begin{figure}[h]
\centering
\begin{tikzpicture}[line cap=round,line join=round,>=triangle 45,x=1cm,y=1cm]
\clip(-1.5,-1.5) rectangle (6.5,10.5);

%% Q square
\fill[line width=0.5pt,fill opacity=0] (0,0) -- (5,0) -- (5,5) -- (0,5) -- cycle;
\draw [line width=0.5pt] (0,0)-- (5,0);
\draw [line width=0.5pt] (5,0)-- (5,5);
\draw [line width=0.5pt] (5,5)-- (0,5);
\draw [line width=0.5pt] (0,5)-- (0,0);
\draw (2.3,1.7) node[anchor=north west] {$Q$};
%\draw (2.5,0) node[anchor=north] {$I_Q$};
%\draw (0,2.4) node[anchor=east] {$J_Q$};

%% parallelogram containing l_qr
\fill[line width=0.5pt,color=black,fill=black,fill opacity=0.05] (0,-0.4) -- (0,-1) -- (5,5) -- (5,5.6) -- cycle;
\draw [line width=0.5pt,color=black] (0,-0.4)-- (0,-1);
\draw [line width=0.5pt,color=black] (0,-1)-- (5,5);
\draw [line width=0.5pt,color=black] (5,5)-- (5,5.6);
\draw [line width=0.5pt,color=black] (5,5.6)-- (0,-0.4);
\draw (3,3) node[anchor=south east] {$\tau(q,r)$};

%% parallelogram high slope
%\fill[line width=0.5pt,color=black,fill=black,fill opacity=0.1] (0,2) -- (0,1.4) -- (5,8) -- (5,8.6) -- cycle;
%\draw [line width=0.5pt,color=black] (0,2)-- (0,1.4);
%\draw [line width=0.5pt,color=black] (0,1.4)-- (5,8);
%\draw [line width=0.5pt,color=black] (5,8)-- (5,8.6);
%\draw [line width=0.5pt,color=black] (5,8.6)-- (0,2);
%\draw (2.5,5.5) node[anchor=south] {$\mathcal{T}$};
%\draw (2,4.2) node[anchor=north west] {$\tilde{\lambda}(p) \in \mathcal{T}$};

%% parallelogram low slope
%\fill[line width=0.5pt,color=black,fill=black,fill opacity=0.1] (0,3.2) -- (0,2.6) -- (5,1.4) -- (5,2) -- cycle;
%\draw [line width=0.5pt,color=black] (0,3.2)-- (0,2.6);
%\draw [line width=0.5pt,color=black] (0,2.6)-- (5,1.4);
%\draw [line width=0.5pt,color=black] (5,1.4)-- (5,2);
%\draw [line width=0.5pt,color=black] (5,2)-- (0,3.2);

%% middle line l_tau
%\draw [line width=0.5pt,domain=-5.444349942765321:15.557738432698038] plot(\x,{(--8.5--6.6*\x)/5});
%\draw (5.42907169843202,8.87523323912829) node[anchor=north west] {$\ell_{\mathcal{T}}$};
%% l_qr line
\draw [line width=0.5pt,domain=-5.444349942765321:15.557738432698038] plot(\x,{(--2.2186666666666666-3.6746666666666665*\x)/-3.2});
\draw (5.4,5.6) node[anchor=north west] {$\ell_{q,r}$};

%% lower half of \hat{Q}
\draw [line width=0.5pt] (0,0)-- (0,-1);
\draw [line width=0.5pt] (5,0)-- (5,-1);
\draw [line width=0.5pt,dash pattern=on 1pt off 2pt] (0,-1)-- (0,-1.5);
\draw [line width=0.5pt,dash pattern=on 1pt off 2pt] (5,-1)-- (5,-1.5);

\draw (2.5,10) node[anchor=north] {$\hat{Q}$};

\draw [line width=0.5pt] (0,5)-- (0,9);
\draw [line width=0.5pt,dash pattern=on 1pt off 2pt] (0,9)-- (0,9.598606253100192);
\draw [line width=0.5pt,dash pattern=on 1pt off 2pt] (-0.11171323603969871,9.252755944866891) -- (0.11171323603969871,9.252755944866891);
\draw [line width=0.5pt,dash pattern=on 1pt off 2pt] (-0.11171323603969871,9.345850308233304) -- (0.11171323603969871,9.345850308233304);
\draw [line width=0.5pt,dash pattern=on 1pt off 2pt] (5,9)-- (5,9.598606253100192);
\draw [line width=0.5pt,dash pattern=on 1pt off 2pt] (4.888286763960302,9.252755944866891) -- (5.111713236039699,9.252755944866891);
\draw [line width=0.5pt,dash pattern=on 1pt off 2pt] (4.888286763960302,9.345850308233304) -- (5.111713236039699,9.345850308233304);
\draw [line width=0.5pt] (5,9.598606253100192)-- (5,10.403385899706668);
\draw [line width=0.5pt] (5,10.403385899706668)-- (0,10.403385899706668);
\draw [line width=0.5pt] (0,10.403385899706668)-- (0,9.598606253100192);
\draw [line width=0.5pt] (5,9)-- (5,5);

% coordinates of q,r points
\draw [line width=0.5pt,dotted] (4,-0.5)-- (4,3.9);
\draw [line width=0.5pt,dotted] (4,3.9)-- (5.3,3.9);
\draw [line width=0.5pt,dotted] (0.8,-0.5)-- (0.8,0.22533333333333333);
\draw [line width=0.5pt,dotted] (0.8,0.22533333333333333)-- (-0.3,0.22533333333333333);
\draw (0.8,-0.5) node[anchor=north] {$r$};
\draw (4,-0.5) node[anchor=north] {$q$};
\draw (5.3,3.9) node[anchor=west] {$\tilde{\lambda}^{\mathcal{N}}(q)$};
\draw (-0.3,0.22533333333333333) node[anchor=east] {$\tilde{\lambda}^{\mathcal{N}}(r)$};
\begin{scriptsize}

%%%%%%%%%%%%%%%%%%%%%%%
% points in L
%%%%%%%%%%%%%%%%%%%%%%%
\draw [fill=black, opacity=0.15] (0,2.8030905719740478) circle (2.5pt);
\draw [fill=black, opacity=0.15] (0.2,1.78) circle (2.5pt);
\draw [fill=black, opacity=0.15] (0.8,2.5) circle (2.5pt);
\draw [fill=black, opacity=0.15] (1,3.2) circle (2.5pt);
\draw [fill=black, opacity=0.15] (1.4,3.4) circle (2.5pt);
\draw [fill=black, opacity=0.15] (1.6,3.8) circle (2.5pt);
\draw [fill=black, opacity=0.15] (1.8,2.3) circle (2.5pt);
\draw [fill=black, opacity=0.15] (2,4.2) circle (2.5pt);
\draw [fill=black, opacity=0.15] (2.2,4.8) circle (2.5pt);
\draw [fill=black, opacity=0.15] (2.4,2.4) circle (2.5pt);
\draw [fill=black, opacity=0.15] (2.8,2) circle (2.5pt);
\draw [fill=black, opacity=0.15] (3.2,2.35) circle (2.5pt);
\draw [fill=black, opacity=0.15] (4.2,1.65) circle (2.5pt);
\draw [fill=black, opacity=0.15] (4.4,1.7) circle (2.5pt);
\draw [fill=black, opacity=0.15] (3.4,2) circle (2.5pt);

%%%%%%%%%%%%%%%%%%%%%
% points in N
%%%%%%%%%%%%%%%%%%%%%
\draw [fill=black] (4,3.9) circle (2.5pt);

\draw [fill=black] (0.4,4.2) circle (2.5pt);

\draw [fill=black] (0.6,3.25) circle (2.5pt);

\draw [fill=black] (1.2,1.5) circle (2.5pt);

\draw [fill=black] (2.6,0.73) circle (2.5pt);

\draw [fill=black] (3,4.7) circle (2.5pt);

\draw [fill=black] (3.6,2.9) circle (2.5pt);

\draw [fill=black] (3.8,3.2) circle (2.5pt);

\draw [fill=black] (4.6,0.4) circle (2.5pt);

\draw [fill=black] (4.8,3) circle (2.5pt);
%
%%%%%%%%%%%%%%%%%%%%%%%%%%
%% points in the tube that are added in the definition of \lambda^L
%%%%%%%%%%%%%%%%%%%%%%%%%%
%\draw [fill=black] (5,1.7) circle (2.5pt);
%\draw [fill=black] (0.2,2.852) circle (2.5pt);
%\draw [fill=black] (0.4,2.804) circle (2.5pt);
%\draw [fill=black] (0.6,2.756) circle (2.5pt);
%\draw [fill=black] (1,2.66) circle (2.5pt);
%\draw [fill=black] (1.2,2.612) circle (2.5pt);
%\draw [fill=black] (1.4,2.564) circle (2.5pt);
%\draw [fill=black] (1.6,2.516) circle (2.5pt);
%\draw [fill=black] (2,2.42) circle (2.5pt);
%\draw [fill=black] (2.2,2.372) circle (2.5pt);
%\draw [fill=black] (2.6,2.276) circle (2.5pt);
%\draw [fill=black] (3,2.18) circle (2.5pt);
%\draw [fill=black] (3.6,2.036) circle (2.5pt);
%\draw [fill=black] (3.8,1.988) circle (2.5pt);
%\draw [fill=black] (4,1.94) circle (2.5pt);
%\draw [fill=black] (4.6,1.796) circle (2.5pt);
%\draw [fill=black] (4.8,1.748) circle (2.5pt);
%
%%%%%%%%%%%%%%%%%%%%%%%%%%%%
%% points on the parabola that are added in the definition of \lambda^N
%%%%%%%%%%%%%%%%%%%%%%%%%%%%

\draw [fill=black, opacity=0.5] (0,0.8333333333333334) circle (2.5pt);
\fill [white] (0,0.8333333333333334) circle (1.5pt);

\draw [fill=black, opacity=0.5] (0.2,0.6453333333333334) circle (2.5pt);
\fill [white] (0.2,0.6453333333333334) circle (1.5pt);

\draw [fill=black, opacity=0.5] (0.8,0.22533333333333333) circle (2.5pt);
\fill [white] (0.8,0.22533333333333333) circle (1.5pt);

\draw [fill=black, opacity=0.5] (1,0.13333333333333336) circle (2.5pt);
\fill [white] (1,0.13333333333333336) circle (1.5pt);

\draw [fill=black, opacity=0.5] (1.4,0.021333333333333357) circle (2.5pt);
\fill [white] (1.4,0.021333333333333357) circle (1.5pt);

\draw [fill=black, opacity=0.5] (1.6,0.0013333333333333329) circle (2.5pt);
\fill [white] (1.6,0.0013333333333333329) circle (1.5pt);

\draw [fill=black, opacity=0.5] (1.8,0.005333333333333331) circle (2.5pt);
\fill [white] (1.8,0.005333333333333331) circle (1.5pt);

\draw [fill=black, opacity=0.5] (2,0.03333333333333332) circle (2.5pt);
\fill [white] (2,0.03333333333333332) circle (1.5pt);

\draw [fill=black, opacity=0.5] (2.2,0.08533333333333336) circle (2.5pt);
\fill [white] (2.2,0.08533333333333336) circle (1.5pt);

\draw [fill=black, opacity=0.5] (2.4,0.16133333333333327) circle (2.5pt);
\fill [white] (2.4,0.16133333333333327) circle (1.5pt);

\draw [fill=black, opacity=0.5] (2.8,0.38533333333333314) circle (2.5pt);
\fill [white] (2.8,0.38533333333333314) circle (1.5pt);

\draw [fill=black, opacity=0.5] (3.2,0.7053333333333334) circle (2.5pt);
\fill [white] (3.2,0.7053333333333334) circle (1.5pt);

\draw [fill=black, opacity=0.5] (3.4,0.9013333333333331) circle (2.5pt);
\fill [white] (3.4,0.9013333333333331) circle (1.5pt);

\draw [fill=black, opacity=0.5] (4.2,1.9253333333333331) circle (2.5pt);
\fill [white] (4.2,1.9253333333333331) circle (1.5pt);

\draw [fill=black, opacity=0.5] (4.4,2.2413333333333334) circle (2.5pt);
\fill [white] (4.4,2.2413333333333334) circle (1.5pt);

%\draw [fill=black, opacity=0.5] (5,3.3333333333333326) circle (2.5pt);
%\fill [white] (5,3.3333333333333326) circle (1.5pt);

\end{scriptsize}
\end{tikzpicture}
\caption{\footnotesize Graph of $\tilde{\lambda}^{\mathcal{N}}$ with line $\ell_{q,r}$ and tube $\tau(q,r)$.} \label{figure:tubes_3}
\end{figure}

%%%%%%%%%%%%%%%%%%%%%%%%%%%

We mention here that for the third inequality, we used the following: for $q,r$ fixed, there exists a tube $\tau=\tau(q,r)\in \mathcal{T}$ such that\footnote{Here we use the fact that, due to the restriction $|q-r|\gtrsim 2^{(\frac{3}{2}-2\ep)m}$, we can ensure that the slope of $\ell_{q,r}(p)$ is in absolute value bounded by $2^{2\ep m}$.}
\begin{equation}
|\ell_\tau(p)- \ell_{q,r}(p)|\leq \frac{1}{2}\qquad \forall p\in I_{\tau}\cap\Z
 \label{eq:slo}
 \end{equation}
where
\[ \ell_{q,r}(p):=p \frac{\tilde{\l}^{\n}(r) - \tilde{\l}^{\n}(q)}{r-q} + \frac{r \tilde{\l}^{\n}(q) - q \tilde{\l}^{\n}(r)}{r-q} \]
is the line joining the points $(r,\tilde{\l}^{\n}(r))$ and $(q,\tilde{\l}^{\n}(q))$.

Now, as suggested by \eqref{KeyP} and \eqref{eq:slo}, we split the sum over $p$ according to the two subsets
\[  \{p\in I_{Q}\cap \Z : |\tilde{\l}^{\n}(p)-\ell_{\t}(p)|\leq 2^{\ep m-10}\} \quad \text{ and } \quad \{p\in I_{Q}\cap \Z : |\tilde{\l}^{\n}(p)-\ell_{\t}(p)|> 2^{\ep m-10}\}. \]
The summation in $q,r$ over the contribution of the first subset is controlled by $2^{(6-\ep)m}$ thanks to \eqref{KeyP}. For the second subset, using the restriction\footnote{On the complement set $|q-r|\lesssim 2^{(\frac{3}{2}-2\ep)m}$ the summation in $q,r$ and $p$ is bounded by the same expression $2^{(6-\ep)m}$.} $|q-r|\gtrsim 2^{(\frac{3}{2}-2\ep)m}$ and \eqref{eq:slo} one has
\[ |\tilde{\l}^{\n}(p)- \ell_{q,r}(p)| \gtrsim  2^{\ep m}, \]
which allows us to conclude the estimate in \eqref{js1LN}.

For the high oscillatory component $\J_H^{\n}$, via Lemma \ref{VdC}, we have the same estimate as the one given by \eqref{js1H}:

\beq\label{js1HN}
\J_H^{\n}(s)\lesssim 2^{\frac{9 m}{2}}(\bar{\eta} 2^{\frac{3 m}{2}} |s|)^{-1} \lesssim 2^{\frac{9 m}{2}} 2^{3 \nu_{\n} m}(\bar{\eta})^{-1}.
\eeq

Now, from \eqref{jsplitN}, \eqref{js1LN} and \eqref{js1HN}, we obtain for $\bar{\eta}=\frac{1}{\sqrt{m}} 2^{\frac{m}{2}(3\nu_{\n}+\ep-3)}$ that
\beq\label{conclN}
|\J^{\n}(s)|\lesssim 2^{6 m} \sqrt{m} 2^{\frac{m}{2}(3\nu_{\n}-\ep)}.
\eeq

Conclude now that
\beq\label{concNm}
\I^{>}(\tilde{\l}^{\n}) \lesssim m^{\frac{1}{8}} 2^{\frac{m}{8}(12\mu+3\nu_{\n}-\ep)}.
\eeq
\\
\noindent\textsf{The final estimates}\par
\vspace{0.5cm}
From \eqref{mathf4}--\eqref{decc3}, \eqref{concLm} and \eqref{concNm} we deduce that
\beq\label{IL}
\I(\tilde{\l}^{\L})\lesssim 2^{2\ep m}(2^{- 3\nu_{\L} m} + 2^{\frac{3m}{8}(4\mu+\nu_{\L}-\frac{1}{2})}),
\eeq
which for $\nu_{\L}:=\frac{1}{9}(\frac{1}{2}-4\mu)$ becomes
\beq\label{IL1}
\I(\tilde{\l}^{\L})\lesssim 2^{2\ep m} 2^{-\frac{1}{3}(\frac{1}{2}-4\mu) m}.
\eeq
while
\beq\label{IN}
\I(\tilde{\l}^{\n})\lesssim 2^{- 3\nu_{\n} m} + m^{\frac{1}{8}} 2^{\frac{m}{8}(12\mu+3\nu_{\n}-\ep)}.
\eeq
which for $\nu_{\n}:=\frac{1}{27}(\ep-12\mu)$ becomes
\beq\label{IN1}
\I(\tilde{\l}^{\n})\lesssim  m^{\frac{1}{8}} 2^{- \frac{\ep-12\mu}{9} m}.
\eeq
Combining now \eqref{IL1} with \eqref{IN1} we conclude
\beq\label{Itot}
\I(\tilde{\l})\lesssim \I(\tilde{\l}^{\n}) + \I(\tilde{\l}^{\L})\lesssim  m^{\frac{1}{8}} 2^{- \frac{\ep-12\mu}{9} m} + 2^{2\ep m} 2^{-\frac{1}{3}(\frac{1}{2}-4\mu) m}\lesssim m^{\frac{1}{8}} 2^{-(\frac{1}{114}-\frac{4}{3}\mu) m} ,
\eeq
where in the last inequality we choose $\ep=\frac{3}{38}$.

Finally, deduce that, up to a polynomial term in $m$, \eqref{finaunif} holds when $a=3$ for
\beq\label{d1case3}
 \d_1=\frac{1}{6} (\frac{1}{114}-\frac{4}{3}\mu).
 \eeq
\end{proof} 
\section{The high resolution, single-scale analysis (II): Time-frequency localization via weight mass and input size distribution}
\label{sec:HR2}

As already noted before -- recall the discussion in Section \ref{Inthr} with a special stress on the formulation \eqref{TTstargum} -- Proposition \ref{sgscale} captures the cancellation hidden in the non-zero curvature of the $BC^a$ kernel's phase thus providing the single scale $m$-exponential decay of  $\|\underline{\L}_{m,P}\|_{L^2\times L^2\times L^2\to {\mathbb C}}$ for every given $P = P(k,\ell,r)\in \BHT$. The difficulty now lies in  putting the various scales together while preserving a similar type decay in $m$, thus aiming for a \emph{global} estimate of the form
\begin{equation}\label{glob}
\|\underline{\L}_{m}\|_{L^p\times L^q\times L^r\to {\mathbb C}}\lesssim 2^{- \delta(p,q,r)\, a m}\,,
\end{equation}
where, of course, this time $(p,q,r)$ needs to obey the natural H\"older homogeneity condition and $\delta(p,q,r)>0$. As hinted by the modulation invariance property of $BC^a$ in \ref{key:m}, it is natural to expect that time-frequency analysis methods inspired by the bilinear Hilbert transform approach in \cite{lt2} would be required for achieving \eqref{glob}. These type of methods are designed for capturing the almost-orthogonality among suitable families of \emph{modulated Calder\'on-Zygmund (sub)operators}\footnote{Represented as \emph{trees} in the time-frequency plane, these objects constitute the elementary building blocks within any modulation invariant analysis problem.} as opposed to the standard  Calder\'on-Zygmund theory that exhibits a simpler, almost orthogonality relation among scales.

In the context given by our problem we need to combine the almost orthogonality methods employed in Section \ref{sec:onescale} and relying on a suitable bilinear $TT^{*}$ argument involving exponential sums with the almost orthogonality methods developed for controlling tree-like structures that are part of the modulation invariance time-frequency analysis and that will be the focus of Section \ref{sec:bilinear:analysis}. However, in order to be able to perform this latter analysis we will need a very precise (Heisenberg) time-frequency localization for all the three functions $f,g$ and $h$ that are involved within a generic trilinear form $\underline{\L}_{m,P}(f,g,h)$, $P\in\BHT$. This latter fact will be required when quantifying the interactions among the tree structures associated with each of the input functions.

It is this last aspect that becomes the central theme of our present section: in what follows we will focus on the refinement of the single scale estimate proved in Section \ref{sec:onescale} thus addressing the item (a) -- see also (II) -- in Section \ref{Inthr}.

\subsection{Intuition -- a heuristic}\label{heurhighres}

In this section we briefly review with few extra accents the outline provided in Section \ref{HR20}. We first remind the concrete setting for our problem:

Fix  $P = P(k,\ell,r)\in \BHT$; then, from \eqref{def:trilinear:form} we have that
\begin{equation}\label{lmp}
\begin{aligned}
& \underline{\L}_{m,P}(f, g, h)=\underline{\L}_{m,k,\ell, r}(f, g, h) \\
&= \sum_{n, p \sim 2^{am \over 2}} \sum_{ v  \sim 2^{ am \over 2}} \int_{\R} S_{k, \ell,r}^{n, v, p}(f, g)(x) h(x) \rho_{am-ak}( \lambda(x)) w_{k, n, v}^e(\l)(x) dx,
\end{aligned}
\end{equation}
where $S_{k, \ell,r}^{n, v, p}$ is given by
\begin{equation}
\begin{aligned}\label{eq:def:S_n,v,p-bis}
& S_{k, \ell,r}^{n, v, p}(f, g)(x) \\
& = \sum_{u \sim 2^{ am \over 2}} \frac{1}{2^{am +2 k \over 4}} \langle f_{\ell+1}, \check{\varphi}_{{am \over 2}- k, \ell}^{u-v, p_r-n}  \rangle   \langle g_{\ell-1}, \check{\varphi}_{{am \over 2}- k, \ell}^{u+v, p_r+n}  \rangle \check{\varphi}_{{am \over 2}- k, 2\ell-1}^{2u, p_r}(x).
\end{aligned}
\end{equation}

The challenge is to find a consensus between the following two competing factors:
\begin{itemize}
\item \textsf{(i)} the extraction of accurate time-frequency information for the triple input function;

\item \textsf{(ii)} the exploitation of the non-zero curvature in the oscillatory phase, now subsumed in the expression of $w_{k,n,v}^e(\lambda)(x)$ -- see \eqref{lmp} above.
\end{itemize}

Each of the above items comes with its own set of difficulties:

\begin{itemize}
\item \textsf{for (i)}: while the time-frequency localization for $f$ and $g$ is readily available via the wave-packet coefficients resulted from their Gabor decomposition -- see \eqref{eq:def:S_n,v,pFinalprep} and \eqref{TTstargum}, no similar approach is available for $h$ \emph{per se} due to the implicit presence of the stopping time function $\lambda(x)$ via the ``oscillatory weight'' $w_{k, n, v}^e(\l)$;

\item \textsf{for (ii)}:  in order to preserve the decay provided by the single scale estimate in Proposition \ref{sgscale} the oscillatory weight $w_{k, n, v}^e(\l)$ has to be paired with the output of the bilinear coefficient $\tilde S_k^{n, v}$ as revealed by Corollary \ref{cor:uniform}.
\end{itemize}

The solution to \textsf{(i)} coincides with ``the path of least resistance'' following the most natural resolution: to simply switch the focus from the time-frequency localization of $h$ to the time-frequency localization of the compound expression $h\,\rho_{am-ak}( \lambda)\, w_{k, n, v}^e(\l)$.

The solution to \textsf{(ii)} relies on the remark provided in \eqref{constancy} and also displayed for reader's convenience below:
$\newline$

\noindent\textbf{Claim:} [\textsf{informal}] \emph{With the previous notations, set}
\begin{equation}
\begin{aligned}\label{eq:def:underline:S}
& \underline{S}_{k,\ell,r}^{n, v, p}(f, g)(x) \\
&:=  2^{-k} \sum_{u \sim 2^{am \over 2}} \langle f_{\ell+1}, \check{\varphi}_{{am \over 2}- k, \ell}^{u-v, p_r-n}  \rangle   \langle g_{\ell-1}, \check{\varphi}_{{am \over 2}- k, \ell}^{u+v, p_r+n} \rangle \, e^{i (2^{{am \over 2}-k}x-p_r) 2u}.
\end{aligned}
\end{equation}
\emph{Then, the expression $\underline{S}_{k,\ell,r}^{n, v, p}(f, g)(x)$ is morally constant on any spatial interval of length $\lesssim 2^{-(am-k)}$.}
$\newline$

The heuristic motivation of this is based on the following observation: inspecting \eqref{eq:def:underline:S} we notice that $\underline{S}_{k,\ell,r}^{n, v, p}(f, g)(x)$ is a trigonometric polynomial that can be expressed as a linear combination of elements belonging to the family of complex exponentials $\{e^{i (2^{{am \over 2}-k}x-p_r) 2u}\}_{u\sim 2^{\frac{am}{2}}}$. Thus, the frequency support of such a trigonometric polynomial lies within an interval of the form $[c_1\,2^{am-k},\, c_2\,2^{am-k}]$ for some suitable absolute constants $c_1$ and $c_2$, which, based on a reproducing kernel argument, implies that the spatial information of $\underline{S}_{k,\ell,r}^{n, v, p}(f, g)$ varies no more than $O(1)$ within any given interval of length  $\lesssim 2^{-(am-k)}$.

%%%%%%%%%%%%%%%%%%%%%%%%%%%%%%%%%%%%%%%%%%%
%
% FIGURE 5 - picture illustrating the new small tiles
%
%%%%%%%%%%%%%%%%%%%%%%%%%%%%%%%%%%%%%%%%%%%
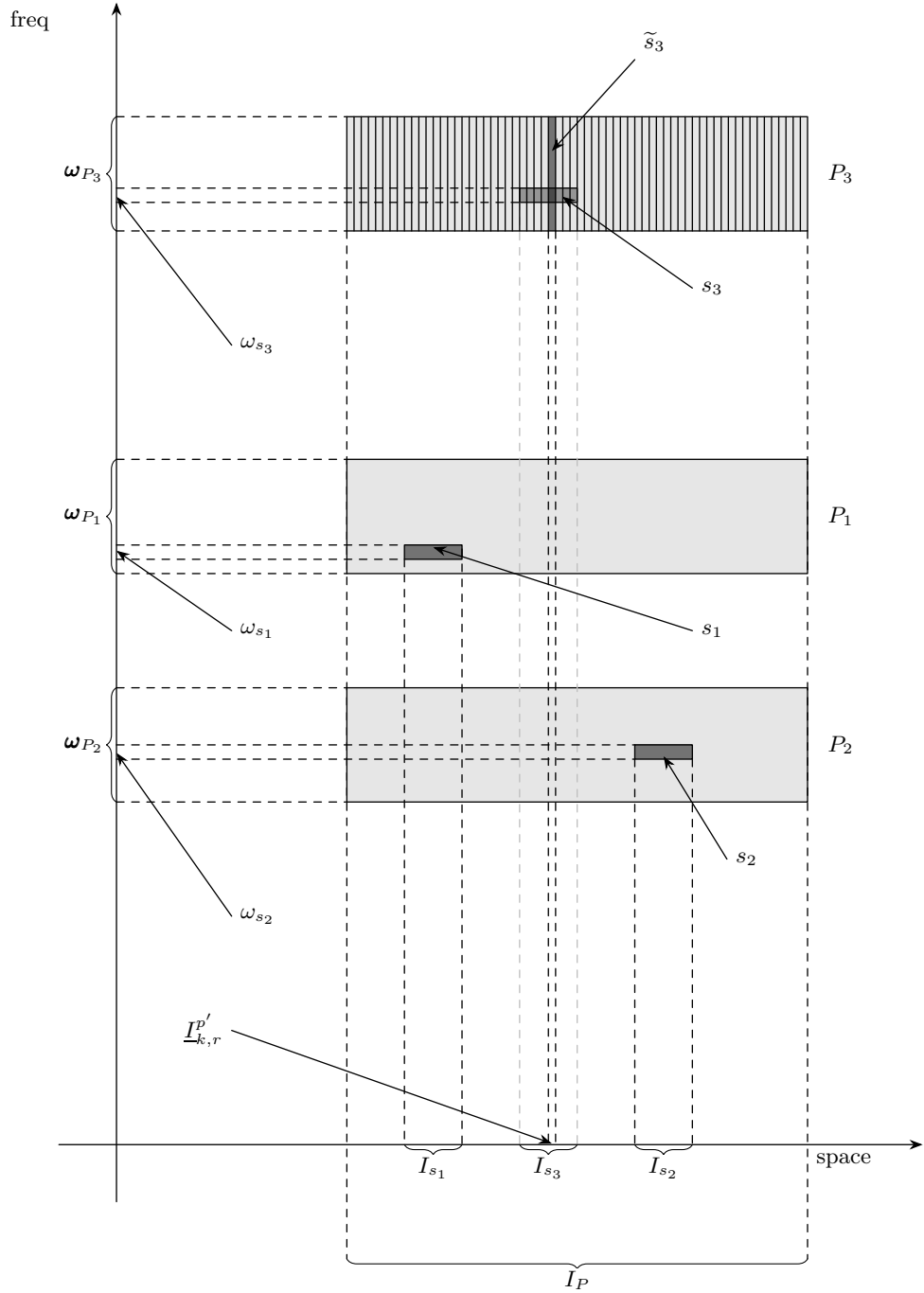
\begin{figure}[t]
\centering
\begin{tikzpicture}[line cap=round,line join=round,>=Stealth,x=1cm,y=1cm, decoration={brace,amplitude=4pt}, scale=1.6]

\clip(-1,-1.5) rectangle (7.1,10);
% P_2
\filldraw[line width=0.5pt,color=black,fill=black,fill opacity=0.1] (2,3) -- (6,3) -- (6,4) -- (2,4) -- cycle;
% P_1
\filldraw[line width=0.5pt,color=black,fill=black,fill opacity=0.1] (2,5) -- (6,5) -- (6,6) -- (2,6) -- cycle;
% P_3
\filldraw[line width=0.5pt,color=black,fill=black,fill opacity=0.1] (2,8) -- (6,8) -- (6,9) -- (2,9) -- cycle;
%
% selected small tiles
%
% s_2
\filldraw[line width=0.5pt,color=black,fill=black,fill opacity=0.5] (4.5,3.375) -- (5,3.375) -- (5,3.5) -- (4.5,3.5) -- cycle;
% s_1
\filldraw[line width=0.5pt,color=black,fill=black,fill opacity=0.5] (2.5,5.125) -- (3,5.125) -- (3,5.25) -- (2.5,5.25) -- cycle;
% s_3
\filldraw[line width=0.5pt,color=black,fill=black,fill opacity=0.35] (3.5,8.25) -- (4,8.25) -- (4,8.375) -- (3.5,8.375) -- cycle;

% vertical dashed lines down to I_P
\draw [line width=0.5pt,dash pattern=on 3pt off 3pt] (2,8)-- (2,-1);
\draw [line width=0.5pt,dash pattern=on 3pt off 3pt] (6,8)-- (6,-1);
%
% horizontal dashed lines to \omega_{P_2}
\draw [line width=0.5pt,dash pattern=on 3pt off 3pt] (0,3)-- (2,3);
\draw [line width=0.5pt,dash pattern=on 3pt off 3pt] (0,4)-- (2,4);
%
% horizontal dashed lines to \omega_{P_1}
\draw [line width=0.5pt,dash pattern=on 3pt off 3pt] (0,5)-- (2,5);
\draw [line width=0.5pt,dash pattern=on 3pt off 3pt] (0,6)-- (2,6);
%
% horizontal dashed lines to \omega_{P_3}
\draw [line width=0.5pt,dash pattern=on 3pt off 3pt] (0,8)-- (2,8);
\draw [line width=0.5pt,dash pattern=on 3pt off 3pt] (0,9)-- (2,9);
% 
%% border s_2
%\draw [line width=0.5pt,color=black] (4.5,2.375)-- (5,2.375);
%\draw [line width=0.5pt,color=black] (5,2.375)-- (5,2.5);
%\draw [line width=0.5pt,color=black] (5,2.5)-- (4.5,2.5);
%\draw [line width=0.5pt,color=black] (4.5,2.5)-- (4.5,2.375);
%%
%% border s_1
%\draw [line width=0.5pt,color=black] (2.5,5.125)-- (3,5.125);
%\draw [line width=0.5pt,color=black] (3,5.125)-- (3,5.25);
%\draw [line width=0.5pt,color=black] (3,5.25)-- (2.5,5.25);
%\draw [line width=0.5pt,color=black] (2.5,5.25)-- (2.5,5.125);
%%
%% border s_3
%\draw [line width=0.5pt,color=black] (3.5,7.5)-- (4,7.5);
%\draw [line width=0.5pt,color=black] (4,7.5)-- (4,7.75);
%\draw [line width=0.5pt,color=black] (4,7.75)-- (3.5,7.75);
%\draw [line width=0.5pt,color=black] (3.5,7.75)-- (3.5,7.5);
%
% vertical dashed lines down to I_{s_2}
\draw [line width=0.5pt,dash pattern=on 3pt off 3pt] (4.5,3.375)-- (4.5,0);
\draw [line width=0.5pt,dash pattern=on 3pt off 3pt] (5,3.375)-- (5,0);
%
% vertical dashed lines down to I_{s_1}
\draw [line width=0.5pt,dash pattern=on 3pt off 3pt] (2.5,5.125)-- (2.5,0);
\draw [line width=0.5pt,dash pattern=on 3pt off 3pt] (3,5.215)-- (3,0);
%
% vertical dashed lines down to I_{s_3}
\draw [line width=0.5pt, color=lightgray, dash pattern=on 3pt off 3pt] (3.5,8.25)-- (3.5,0);
\draw [line width=0.5pt, color=lightgray, dash pattern=on 3pt off 3pt] (4,8.25)-- (4,0);
%
%
% x-xi axes
%
\draw [->,line width=0.5pt] (-0.5,0) -- (7,0);
\draw [->,line width=0.5pt] (0,-0.5) -- (0,10);
%
% horizontal dashed lines to s_2
%
\draw [line width=0.5pt,dash pattern=on 3pt off 3pt] (0,3.375) -- (4.5,3.375);
\draw [line width=0.5pt,dash pattern=on 3pt off 3pt] (0,3.5) -- (4.5,3.5);
%
% horizontal dashed lines to s_1
%
\draw [line width=0.5pt,dash pattern=on 3pt off 3pt] (0,5.125) -- (2.5,5.125);
\draw [line width=0.5pt,dash pattern=on 3pt off 3pt] (0,5.25) -- (2.5,5.25);
\foreach \i in {1,...,63}
{
\draw [line width=0.5pt,color=black] (2 + \i*0.0625,8) -- (2 + \i*0.0625,9);
}
%
% small tile \tilde{s}_3
%
\fill[color=black,fill=black,fill opacity=0.5] (3.75,8) -- (3.8125,8) -- (3.8125,9) -- (3.75,9) -- cycle;
%
% arrow and label for \tilde{s}_3
%
\draw [->,line width=0.5pt] (4.5,9.5) -- (3.78,8.7);
\draw (4.5,9.5) node[anchor=south west] {\small $\widetilde{s}_3$};
%
% vertical dashed lines to \underline{I}^{p'}_{k,r}
%
\draw [line width=0.5pt,dash pattern=on 3pt off 3pt] (3.75,8) -- (3.75,0);
\draw [line width=0.5pt,dash pattern=on 3pt off 3pt] (3.8125,8) -- (3.8125,0);
\draw [->,line width=0.5pt] (1,1) -- (3.78,0.01);
\draw (1,1) node[anchor=east] {\small $\underline{I}^{p'}_{k,r}$};
\draw [decorate, color=black] (0,5) -- (0,6)
	node [midway, anchor=east, fill=white, inner sep=1pt, outer sep=4pt]{\small $\pmb{\omega}_{P_1}$};
\draw [decorate, color=black] (0,3) -- (0,4)
	node [midway, anchor=east, fill=white, inner sep=1pt, outer sep=4pt]{\small $\pmb{\omega}_{P_2}$};
	
\draw [decorate, color=black] (0,8) -- (0,9)
	node [midway, anchor=east, fill=white, inner sep=1pt, outer sep=4pt]{\small $\pmb{\omega}_{P_3}$};

%\draw [decorate, color=black] (0,2.5) -- (0,2.375)
%	node [midway, anchor=west, fill=white, inner sep=1pt, outer sep=3pt]{\small $\omega_{s_2}$};
%\draw [decorate, color=black] (0,5.25) -- (0,5.125)
%	node [midway, anchor=west, fill=white, inner sep=1pt, outer sep=3pt]{\small $\omega_{s_1}$};
%\draw [decorate, color=black] (0,8.375) -- (0,8.25)
%	node [midway, anchor=west, inner sep=1pt, outer sep=3pt]{\contour{white}{\small $\omega_{s_3}$}};
%
\draw [->,line width=0.5pt]  (1,2) -- (0,3.43);
\draw (1,2) node[anchor=west] {\small $\omega_{s_2}$};
\draw [->,line width=0.5pt]  (1,4.5) -- (0,5.2);
\draw (1,4.5) node[anchor=west] {\small $\omega_{s_1}$};
\draw [->,line width=0.5pt]  (1,7) -- (0,8.3);
\draw (1,7) node[anchor=west] {\small $\omega_{s_3}$};
%
% horizontal dashed lines projecting \omega_{s_3} on the frequency axis
%
\draw [line width=0.5pt, dash pattern=on 3pt off 3pt] (0,8.25)-- (3.5,8.25);
\draw [line width=0.5pt, dash pattern=on 3pt off 3pt] (0,8.375)-- (3.5,8.375);

\draw (6.1,8.5) node[anchor=west] {\small $P_3$};
\draw (6.1,5.5) node[anchor=west] {\small $P_1$};
\draw (6.1,3.5) node[anchor=west] {\small $P_2$};
%\draw (3.75,7.75) node[anchor=south] {\small $s_3$};
%\draw (2.75,5.25) node[anchor=south] {\small $s_1$};
%\draw (4.75,2.5) node[anchor=south] {\small $s_2$};
\draw [->,line width=0.5pt]  (5,7.5) -- (3.87,8.3);
\draw (5,7.5) node[anchor=west] {\small $s_3$};
\draw [->,line width=0.5pt]  (5,4.5) -- (2.75,5.23);
\draw (5,4.5) node[anchor=west] {\small $s_1$};
\draw [->,line width=0.5pt]  (5.3,2.5) -- (4.75,3.4);
\draw (5.3,2.5) node[anchor=west] {\small $s_2$};

\draw [decorate, color=black] (6,-1) -- (2,-1)
 node[midway, anchor=north, fill=white, inner sep=1pt, outer sep=4pt]{\small $I_{P}$};
 \draw [decorate, color=black] (3,0) -- (2.5,0)
 node[midway, anchor=north, fill=white, inner sep=1pt, outer sep=4pt]{\small $I_{s_1}$};
\draw [decorate, color=black] (4,0) -- (3.5,0)
 node[midway, anchor=north, fill=white, inner sep=1pt, outer sep=4pt]{\small $I_{s_3}$};
\draw [decorate, color=black] (5,0) -- (4.5,0)
 node[midway, anchor=north, fill=white, inner sep=1pt, outer sep=4pt]{\small $I_{s_2}$};

\draw (6,0) node[anchor=north west] {\small $\text{space}$};
\draw (-1,10) node[anchor=north west] {\small $\text{freq}$};

\end{tikzpicture}
\caption{\footnotesize Time-frequency localization on the input/output variables in $\underline{S}_{k\ell,r}^{n,v,p}$.} \label{figure:5}
\end{figure}
%%%%%%%%%%%%%%%%%%%%%%%%%%%
%

Now since the moral spatial support of $\underline{S}_{k,\ell,r}^{n, v, p}(f, g)(x)$ is given by\footnote{Recall \eqref{eq:Ip}, \eqref{eq:not:p_r}. Also, deduce from \eqref{eq:def:S_n,v,p-bis} that the moral spatial support of $S_{k, \ell,r}^{n, v, p}(f, g)(x)$ is the same with that of  $\underline{S}_{k,\ell,r}^{n, v, p}(f, g)(x)$.}  $I_k^{p_r}$  which has length $2^{-({am \over 2}-k)}$, based on the above heuristic it becomes natural to partition $I_k^{p_r}$ into $2^{am \over 2}$ consecutive subintervals $\{\underline{I}_{k,r}^{p'}\}$, each of length $ 2^{-(am-k)}$. That is,
\begin{equation}
I_k^{p_r}=\bigsqcup_{p' \sim 2^{am} : \lfloor p' 2^{-\frac{am}{2}}\rfloor=p} \underline{I}_{k,r}^{p'}.
\label{part}
\end{equation}
with
\begin{equation} \underline{I}_{k,r}^{p'}:=\Big[  2^k(r-1) + \frac{p'}{2^{am-k}},  2^k(r-1) + \frac{p'+1}{2^{am -k}} \Big], \label{eq:Ip'}
\end{equation}
where here $p'$ is an integer parameter in $\big[ 2^{am}, 2^{am+1} \big]$.

With these, as a consequence of the above informal claim, the expression in \eqref{lmp} is morally the same as
\begingroup
\allowdisplaybreaks
\begin{equation}
\begin{aligned}
\underline{\L}_{m,k,\ell, r}(f, g, h) \approx \\
\sum_{n,v \sim 2^{am \over 2}}\,\sum_{p' \sim 2^{am}} & 2^{-\frac{am-k}{2}}
 \langle h \, \rho_{am-ak}(\lambda) w_{k, n, v}^e(\l), \check{\varphi}_{am-k,2\ell-1}^{0,2^{am}(r-1)+p'}  \rangle \\
& \cdot  \frac{\int_{\underline{I}_{k,r}^{p'}} \underline{S}_{k,\ell,r}^{n, v, \big \lfloor \frac{p'}{2^{am \over 2}} \big\rfloor}(f_{\ell +1}, g_{\ell -1})(x)  w_{k,n,v}(x)dx}{ \int_{\underline{I}_{k,r}^{p'}}  w_{k,n,v}(x)dx} 
\end{aligned}\label{eq:time-freq-localizationheuris}
\end{equation}
\endgroup
where here, for later use, we let $w_{k,n,v}\in L^1_{loc}(\R)$ be an arbitrary positive weight.

Assume for the sake of the argument in what follows that the masses of $f$ and $g$ are uniformly distributed within $\{I_{k}^{p_r}\}_{p \sim 2^{am \over 2}}$ and
that we restrict our summation over the set of indices $\mathfrak{u}$ that verify\footnote{Incidentally, the above definition of $\mathfrak{u}$ is a loose form of the one in \eqref{eq:def:unif} that rigorously introduces the so-called set of heavy mass uniform $(f,g)$-distribution indices.}
\begin{equation}
\begin{aligned}
\mathfrak{u}& :=\bigg\{ {{(p', n, v)}\atop{p'\sim 2^{am},\;n,v \sim 2^{am \over 2}}} \; : \; \frac{1}{|\underline{I}_{k, r}^{p'}|} \int_{\underline{I}_{k, r}^{p'}} w_{k, n, v}(\lambda)(x) \,  dx\approx 1\bigg\}.
\end{aligned}
\end{equation}

For now, we may pretend that
$$\underline{\L}_{m,k,\ell, r}(f, g, h) \approx\underline{\L}_{m,k,\ell, r}^{\mathfrak{u}}(f, g, h):=\sum_{(p', n, v) \in \mathfrak{u}} \ldots \;.$$

With these, taking $w_{k,n,v}(x):=w_{k,n,v}(\l)(x)$ in \eqref{eq:time-freq-localizationheuris} and applying a Cauchy-Schwarz argument we have that
\begin{equation}
\begin{aligned}
& \underline{\L}_{m, k,\ell,r}^{\mathfrak{u}}(f, g, h) \\
& \lesssim \Big( \sum_{(p', n, v) \in \mathfrak{u}}  \int_{\underline{I}_{k, r}^{p'}} \big|\underline{S}_{k,\ell,r}^{n, v, p'}(f, g)(x)\big|^2\, w_{k, n, v}(x)\,dx  \Big)^\frac{1}{2} \\
&\hspace{1em} \cdot \Big( \sum_{(p', n, v) \in \mathfrak{u}}  \big| \langle h \, \rho_{am-ak}(\lambda(\cdot)) w_{k, n, v}^e(\l), \varphi_{am-k,2\ell-1}^{2^{am}(r-1)+p'}  \rangle \big|^2 \Big)^\frac{1}{2}.
\end{aligned}\label{eq:one:tile:unifheur}
\end{equation}
Finally, an application of \eqref{eq:def:S_n,v,pFinal} and Corollary \ref{cor:uniform} provides us with the exponential $m$-decay under the desired time-frequency localization of all of the three input functions:
\begin{equation}
\begin{aligned}
& \underline{\L}_{m, k,\ell,r}^{\mathfrak{u}}(f, g, h) \lesssim 2^{-\frac{ \delta_1 a m}{2}}  2^{- \frac{am}{4}} 2^{-\frac{k}{2}}  \|f_{\ell+1,r}\|_{L^2} \, \|g_{\ell-1,r}\|_{L^2} \\
& \hspace{4em} \cdot \Big( \sum_{(p', n, v) \in \mathfrak{u}}  \big| \langle h \, \rho_{am-ak}(\lambda(\cdot)) w_{k, n, v}^e(\l), \varphi_{am-k,2\ell-1}^{ 2^{am}(r-1)+p'}  \rangle \big|^2 \Big)^{1/2}.
\end{aligned}\label{eq:one:tile:unif:bisheur}
\end{equation}

\subsection{Sharp time-frequency localization for the input $h$: a rigorous approach to \eqref{eq:time-freq-localizationheuris}.}

In this section we provide a detailed justification for the heuristic presented in the previous section.

We first notice based on \eqref{eq:def:S_n,v,p-bis} and \eqref{eq:def:underline:S} that
\begin{equation}\label{eq:def:underline:SrelS}
S_{k, \ell,r}^{n, v, p}(f, g)(x) = \underline{S}_{k,\ell,r}^{n, v, p}(f, g)(x)\,e^{i 2^{\frac{am}{2}}(2^{\frac{am}{2}-k}x-p_r)(2\ell-2)}\,  \check{\varphi}(2^{{am \over 2}- k}x-p_r)\:.
\end{equation}

With this \eqref{lmp} becomes
\begin{equation}
\begin{aligned}
 & \underline{\L}_{m,k, \ell, r}(f, g, h) = \sum_{n, v, p \sim 2^{am \over 2}} \int \underline{S}_{k,\ell,r}^{n, v, p}(f, g)(x) \, \check{\varphi}(2^{{am \over 2}- k}x-p_r) \\
 & \hspace{6em} \cdot e^{i 2^{am \over 2} (2\ell-2) (2^{{am\over 2}-k} x-p_r)}\,\rho_{am-ak}( \lambda(x))\, w_{k, n, v}^e(\l)(x) h(x) \,dx.
\end{aligned}\label{eq:trilinear:form:withS-bis}
\end{equation}
Notice now that $\underline{S}_{k,\ell,r}^{n, v, p}(f, g)(x)$ is a trigonometric polynomial of period  $2^{k-{am\over 2}}$.

Defining
$$H_{k,\ell,r}^{n, v, p}(x):=\check{\varphi}(2^{{am \over 2}- k}x-p_r)\,
 e^{i 2^{am \over 2} (2\ell-2) (2^{{am\over 2}-k} x-p_r)}\,\rho_{am-ak}( \lambda(x))\, w_{k, n, v}^e(\l)(x) h(x)\,,$$
we use the same periodization argument as in Section \ref{HR2} in order to deduce that
\begin{equation}
\begin{aligned}
 & \int_{\R} \underline{S}_{k,\ell,r}^{n, v, p}(f, g)(x) \, H_{k,\ell,r}^{n, v, p}(x)\,dx \\
 &  \qquad =\int_{I_k^{p_r}} \underline{S}_{k,\ell,r}^{n, v, p}(f, g)(x)\,\Big(\sum_{j\in\Z} H_{k,\ell,r}^{n, v, p}(x+j2^{k-{am\over 2}})\Big)\,dx \,.
\end{aligned}\label{periodarg}
\end{equation}

The main reason for the formulation of \eqref{periodarg} is to make transparent and later separate the contribution of the dominant and error terms. Indeed, since $H_{k,\ell,r}^{n,v,p}$ is morally localized around $I_k^{p_r}$, we immediately see that
\begin{itemize}
\item the main term  corresponds to the $j=0$ case;

\item the error term represents the sum over $|j|>0$ and accounts for the tail behavior of $\check{\varphi}$ in the expression that defines $H_{k,\ell,r}^{n, v, p}$.
\end{itemize}

Take now a smooth, compactly supported partition of unity adapted to \eqref{part}--\eqref{eq:Ip'} such that\footnote{In this section we say that $p' \sim_{\lfloor\rfloor_p} 2^{am}$ iff $p' \sim 2^{am}$ and $\lfloor p' 2^{-\frac{am}{2}}\rfloor=p$.}
$$ \int_{\R} \sum_{ p' \sim_{\lfloor\rfloor_p} 2^{am}} \check \psi (2^{am-k}(x- \underline{t}_{k, r}^{p'}))\,a(x)\,dx=\int_{I_k^{p_r}} a(x)\,dx$$
for any $2^{-({am \over 2}-k)}$-periodic $L^1$ integrable function $a$ where here $\underline{t}_{k, r}^{p'}:=2^k(r-1) + \frac{p'}{2^{am-k}}$ is the left end-point of the interval $\underline{I}_{k,r}^{p'}$. By considering $O(1)$ similar terms, we simply assume that 
\[
x \mapsto \check \psi (2^{am-k}(x- \underline{t}_{k, r}^{p'}))
\]
is supported on $\underline{I}_{k,r}^{p'}$, for any $p' \sim_{\lfloor\rfloor_p} 2^{am}$.

Then the expression in \eqref{periodarg} equals
\begin{equation}
\begin{aligned}
 & \sum_{ p' \sim_{\lfloor\rfloor_p} 2^{am}} \int_{\R}\check \psi (2^{am-k}(x- \underline{t}_{k, r}^{p'}))\,\underline{S}_{k,\ell,r}^{n, v, p}(f, g)(x) \\
 & \hspace{8em} \cdot \Big(\sum_{j\in\Z} H_{k,\ell,r}^{n, v, p}(x+j2^{k-{am\over 2}-k})\Big)\,dx .
\end{aligned}\label{periodarg1}
\end{equation}

Let
\begin{equation}
\underline{H}_{k,r}^{n, v, p}(x):=\check{\varphi}(2^{{am \over 2}- k}x-p_r)
\,\rho_{am-ak}( \lambda(x))\, w_{k, n, v}^e(\l)(x) h(x)\,.
\label{eq:H_k}
\end{equation}

From \eqref{periodarg} and \eqref{periodarg1} we deduce that
\begin{equation}
\begin{aligned}
 &\underline{\L}_{m,k, \ell, r}(f, g, h) = \sum_{n, v, p \sim 2^{am \over 2}} \sum_{ p' \sim_{\lfloor\rfloor_p} 2^{am}} \sum_{j\in\Z} \int \underline{S}_{k,\ell,r}^{n, v, p}(f, g)(x) \\
& \hspace{2em} \cdot \check \psi (2^{am-k}(x- \underline{t}_{k, r}^{p'}))\, e^{i 2^{\frac{am}{2}}(2^{\frac{am}{2}-k}x-p_r)(2\ell-2)} \,\underline{H}_{k,r}^{n, v, p}\Big(x+\frac{j}{2^{{am\over 2}-k}} \Big) \,dx.
\end{aligned}\label{concl}
\end{equation}

Let now $\tilde{\varphi}_{am-k,2 \ell-1}^{2^{am}(r-1)+p'}$ denote the $L^2$-normalized wave-packet
\[\tilde{\varphi}_{am-k,2 \ell-1}^{2^{am}(r-1)+p'}(x) := 2^{\frac{am-k}{2}} e^{i 2^{\frac{am}{2}}(2^{\frac{am}{2}-k}x-p_r)(2\ell-2)} \check \psi (2^{am-k}(x- \underline{t}_{k, r}^{p'})) \]
adapted to scale $2^{-(am-k)}$ in space and hence $2^{am-k}$ in frequency. Let us just emphasize that in contrast with the localization of $\varphi_{am-k,2 \ell-1}^{2^{am}(r-1)+p'}$ which is adapted to $\underline{I}_{k,r}^{p'}$ and compactly supported in frequency, $\tilde{\varphi}_{am-k,2 \ell-1}^{2^{am}(r-1)+p'}$ is compactly supported in space inside $\underline{I}_{k,r}^{p'}$ and so has no compact support in frequency, but is adapted to the frequency support of $\varphi_{am-k,2 \ell-1}^{2^{am}(r-1)+p'}$.

Based on the informal claim made in Section \ref{heurhighres} we further claim that, essentially,\footnote{Strictly speaking, the left-hand side in \eqref{eq:time-freq-localization} is in fact a ($\ell^1$-summable) superposition of terms having a similar shape with that of the right-hand side in \eqref{eq:time-freq-localization}; this comment will be made transparent during the current proof of our above claim.} the following equality holds:
\begingroup
\allowdisplaybreaks
\begin{equation}
\begin{aligned}
 &\int \underline{S}_{k,\ell,r}^{n, v, p}(f, g)(x) e^{i 2^{\frac{am}{2}}(2^{\frac{am}{2}-k}x-p_r)(2\ell-2)} \\
 & \hspace{2em}\cdot \check \psi (2^{am-k}(x- \underline{t}_{k, r}^{p'}))\,\underline{H}_{k,r}^{n, v, p}\Big(x+\frac{j}{2^{{am\over 2}-k}} \Big) \,dx = \\
 &2^{-\frac{am-k}{2}} \Big\langle \underline{H}_{k,r}^{n, v, p}\Big(\cdot+\frac{j}{2^{{am\over 2}-k}} \Big), \tilde{\varphi}_{am-k,2 \ell-1}^{2^{am}(r-1)+p'}  \Big\rangle  \frac{\displaystyle \int_{\underline{I}_{k, r}^{p'}} \underline{S}_{k,\ell,r}^{n, v, p}(f,g)(x)  w(x)\,dx}{ \int_{\underline{I}_{k, r}^{p'}}  w(x)\,dx},
\end{aligned}\label{eq:time-freq-localization}
\end{equation}
\endgroup
where $w$ is any suitable locally integrable function and $n, v, p'$ are assumed fixed.

This claim is a consequence of the following lemma whose easy proof -- that is omitted -- relies on Taylor's formula applied to the real analytic function $e^{i a x}$, $a\in\R$ restricted to a compact interval.
\begin{lemma}
\label{lemma:small-variation}
Consider the interval $\ds I = \Big[ {s \over 2^{am-k}}, {s+1 \over 2^{am-k}}    \Big]$, $s\in\Z$, and let $t_s:=  {s \over 2^{am-k}}$ be its left endpoint. Let $\check \psi$ be a compactly supported bump function adapted to $[0, 1]$. Then for any sequence of complex numbers $\{c_u\}_{u \sim 2^{am/2}}$ and any locally integrable function $w_1$, $w_2$ with $\int_{I}  w_2(x)\, dx\not=0$ we have
\begin{equation*}
\begin{aligned}
&\int \Big(  \sum_{u \sim 2^{am \over 2}} c_u \check \psi (2^{am-k}(x- t_s)) \, e^{i 2^{{am \over 2}-k} 2u (x-t_s)}  \Big) w_1(x)dx \\
&= \sum_{\nu \geq 0}  \frac{1}{\nu !} \Big( \sum_{ u \sim 2^{am \over 2}}  \Big({ 2u \over {2^{am \over 2}}}  \Big)^\nu c_u \Big) \int (i 2^{am-k}(x- t_s))^\nu \check \psi (2^{am-k}(x- t_s)) w_1(x)dx,
 \end{aligned} \label{eq:almost-constant}
 \end{equation*}
and
\begin{align*}
 \sum_{ u \sim 2^{am \over 2}} c_u
 = \sum_{\nu \geq 0} \frac{1}{\nu !} \frac{\displaystyle \int_{I} \Big( \sum\limits_{ u \sim 2^{am \over 2}} \Big({ 2u \over {2^{am \over 2}}}\Big)^\nu   c_u e^{i 2^{{am \over 2}-k} 2u (x-t_s)}  \Big) (-i2^{am-k}(x-t_s))^\nu  w_2(x)\,dx}{ \int_{I}  w_2(x)\, dx}.
\end{align*}
 \end{lemma}

In what follows, based on the above, we show how to make sense of \eqref{eq:time-freq-localization}.

\begin{proof}[Proof of \eqref{eq:time-freq-localization}]
To start with, we observe that $\underline{S}_{k,\ell,r}^{n, v, p}(f, g)$ is given by
\[ 2^{-k} \sum_{u \sim 2^{am \over 2}} e^{i(2^{{am \over 2}-k}\underline{t}_{k,r}^{p'}-p_r)2u} \langle f_{\ell+1} , \check{\varphi}_{{am \over 2}- k, \ell}^{u-v, p_r-n}  \rangle   \langle g_{\ell-1} , \check{\varphi}_{{am \over 2}- k, \ell}^{u+v, p_r+n} \rangle \, e^{i 2^{{am \over 2}-k}(x-\underline{t}_{k,r}^{p'})2u}, \]
so that we can apply the first part of Lemma \ref{lemma:small-variation} to the left-hand side of \eqref{eq:time-freq-localization}
for $I=\underline{I}_{k, r}^{p'}$ and $w_1(x):=e^{i 2^{\frac{am}{2}}(2^{\frac{am}{2}-k}x-p_r)(2\ell-2)} \underline{H}_{k,r}^{n, v, p}(x+\frac{j}{2^{{am\over 2}-k}})$ in order to get
\begingroup
\allowdisplaybreaks
\begin{equation}
\begin{aligned}
&\int \underline{S}_{k,\ell,r}^{n, v, p}(f, g)(x) e^{i 2^{\frac{am}{2}}(2^{\frac{am}{2}-k}x-p_r)(2\ell-2)} \\
& \hspace{5em} \cdot \check \psi (2^{am-k}(x- \underline{t}_{k, r}^{p'})) \underline{H}_{k,r}^{n, v, p} \Big(x+\frac{j}{2^{{am\over 2}-k}}\Big)\,dx \\
&= \sum_{\nu \geq 0}  \frac{1}{\nu !}  \underline{S}_{k,\ell,r}^{n, v, p,\nu}(f, g)(\underline{t}_{k,r}^{p'})   \int  \check{\psi}_\nu (2^{am-k}(x- \underline{t}_{k,r}^{p'})) \\
& \hspace{5em} \cdot e^{i 2^{\frac{am}{2}}(2^{\frac{am}{2}-k}x-p_r)(2\ell-2)}\,
\underline{H}_{k,r}^{n, v, p}\Big(x+\frac{j}{2^{{am\over 2}-k}} \Big) \,dx,
\end{aligned}\label{eq:almost-constant:intermediate-step}
\end{equation}
\endgroup
where $\psi_\nu $ is a slight modification\footnote{Indeed it is simply $\check{\psi}_{\nu}(x) = (ix)^{\nu} \check{\psi}(x)$.} of $\psi$ with the same support properties and $\underline{S}_{k,\ell,r}^{n, v, p,\nu}$ is the minor modification of $\underline{S}_{k,\ell,r}^{n, v, p}$ given by
\begin{equation*}
\begin{aligned}
&\underline{S}_{k,\ell,r}^{n, v, p,\nu}(f, g)(x) \\
&:=  2^{-k} \sum_{u \sim 2^{am \over 2}} \Big(\frac{2u}{2^{\frac{am}{2}}}\Big)^{\nu}\langle f_{\ell+1}, \check{\varphi}_{{am \over 2}- k, \ell}^{u-v, p_r-n}  \rangle   \langle g_{\ell-1}, \check{\varphi}_{{am \over 2}- k, \ell}^{u+v, p_r+n} \rangle \, e^{i (2^{{am \over 2}-k}x-p_r) 2u}.
\end{aligned}
\end{equation*}

Next, we use the second part of Lemma \ref{lemma:small-variation} for $ \sum_{ u \sim 2^{am \over 2}} c_u=\underline{S}_{k,\ell,r}^{n, v, p,\nu}(f, g)(\underline{t}_{k,r}^{p'})$, $I=\underline{I}_{k, r}^{p'}$ and $w_2= w$ in order to deduce

\begin{align*}
\underline{S}_{k,\ell,r}^{n, v, p,\nu}(f, g)(\underline{t}_{k,r}^{p'})=  \sum\limits_{\nu ' \geq 0} \frac{1}{\nu' !} \; \frac{ \int_{\underline{I}_{k, r}^{p'}}  \underline{S}_{k,\ell,r}^{n, v, p,\nu,\nu'}(f, g)(x)  w(x)dx}{ \int_{\underline{I}_{k, r}^{p'}}  w(x)dx},
\end{align*}
where $\underline{S}_{k,\ell,r}^{n, v, p,\nu,\nu'}$ is a further slight modification of $\underline{S}_{k,\ell,r}^{n, v, p,\nu}$ given by\footnote{Notice that $\underline{S}_{k,\ell,r}^{n, v, p,\nu,0} = \underline{S}_{k,\ell,r}^{n, v, p,\nu}$.}
\begin{equation}
\begin{aligned}
& \underline{S}_{k,\ell,r}^{n, v, p,\nu,\nu'}(f,g)(x):= \frac{1}{2^k} \sum_{u \sim 2^{am \over 2}}  \Big(  \frac{2u}{2^{am \over 2}} \Big)^{\nu+\nu'} \langle f_{\ell+1}, \check{\varphi}_{{am \over 2}- k, \ell}^{u-v, p_r-n}  \rangle \\
& \hspace{4em} \cdot   \langle g_{\ell-1}, \check{\varphi}_{{am \over 2}- k, \ell}^{u+v, p_r+n}  \rangle  (-i2^{am-k}(x- \underline{t}_{k, r}^{p'}))^{\nu'} e^{i (2^{{am \over 2}-k}x-p_r) 2u}\,.
\end{aligned}\label{eq:pick:up:powers}
\end{equation}

Inserting the above in \eqref{eq:almost-constant:intermediate-step} we obtain a representation of the left-hand side of \eqref{eq:time-freq-localization} as an $\ell^1$-summable superposition of terms having a similar shape as the right-hand side term in \eqref{eq:time-freq-localization}. Notice that if we only keep the main term in the above superposition -- corresponding to $\nu = \nu' = 0$, then, we obtain the precise form of \eqref{eq:time-freq-localization} thus clarifying our claim.
\end{proof}

\begin{observation}\label{observation:extra:terms:reassurance}
Before proceeding to analyze the consequences of \eqref{eq:time-freq-localization}, we digress briefly to reassure the reader that it will indeed be sufficient to consider only the $\nu=\nu'=0$ case of the operators above.
Observe that since $2u=(u-v)+(u+v)$, the expression $\underline{S}_{k,\ell,r}^{n, v, p',\nu, \nu'}$ in the proof of Lemma \ref{lemma:small-variation} becomes a linear combination of expressions similar to the initial $\underline{S}_{k,\ell,r}^{n, v, p}$: indeed, this can be expressed as
\begin{equation*}
\begin{aligned}
&\underline{S}_{k,\ell,r}^{n, v, p,\nu,\nu'}(f,g)(x)= \sum_{\vartheta =0} ^{\nu + \nu'}  {\nu + \nu' \choose \vartheta} \frac{1}{2^k} \sum_{u \sim 2^{am \over 2}}  \big\langle f_{\ell+1}, \Big(  \frac{u-v}{2^{am \over 2}} \Big)^{\vartheta} \check{\varphi}_{{am \over 2}- k, \ell}^{u-v, p_r-n}  \big\rangle \\
& \hspace{2em}\cdot \big\langle g_{\ell-1}, \Big(  \frac{u+v}{2^{am \over 2}} \Big)^{\nu+\nu' - \vartheta}\check{\varphi}_{{am \over 2}- k, \ell}^{u+v, p_r+n}  \big\rangle\, (i2^{am-k}(x- \underline{t}_{k, r}^{p'}))^{\nu'} e^{i (2^{{am \over 2}-k}x-p_r) 2u }.
\end{aligned}\label{eq:pick:up:powers:trick}
\end{equation*}

In this context it is important to stress that we are only going to need $L^2\times L^2 \to L^2$ bounds for these quantities which therefore are completely determined by the $\ell^2(\Z)$-sum of the Gabor coefficients of the input functions. As in Subsection \ref{passage} then, we can replace each expression
\begin{align*}
 \sum_{u \sim 2^{am \over 2}} \big\langle f_{\ell+1}, \Big(  \frac{u-v}{2^{am \over 2}} \Big)^{\vartheta} \check{\varphi}_{{am \over 2}- k, \ell}^{u-v, p_r-n}  \big\rangle\, \big\langle g_{\ell-1}, \Big(  \frac{u+v}{2^{am \over 2}} \Big)^{\nu+\nu' - \vartheta}\check{\varphi}_{{am \over 2}- k, \ell}^{u+v, p_r+n}  \big\rangle  e^{i (2^{{am \over 2}-k}x-p_r) 2u }
\end{align*}
with the corresponding continuous one as in \eqref{eq:def:S_n,v,pFinal}, where now the frequency projected functions $f_{\ell+1,r}, g_{\ell-1,r}$ are replaced by the slight variants
\begin{align}
\label{eq:f{r, ell}:nu}
f_{\ell+1,r}^{\vartheta}(x) &:= \sum_{\substack{2\cdot 2^{am\over 2} < u_1 \leq 3\cdot 2^{am \over 2}, \\ n_1 - 2^{am \over 2}r\sim 2^{am \over 2}}} \big\langle f_{\ell+1}, \Big(  \frac{u_1}{2^{am \over 2}} \Big)^{\vartheta} \check \varphi_{{am \over 2}-k, \ell}^{u_1, n_1}   \big\rangle  \check \varphi_{{am \over 2}-k, \ell}^{u_1, n_1}(x), \\
\label{eq:g:r, ell:nu}
g_{\ell-1,r}^{\nu+\nu'-\vartheta}(x) &:= \sum_{\substack{1\leq u_2 \leq 2^{am \over 2}, \\ n_2 - 2^{am \over 2}r \sim 2^{am \over 2}}}  \big\langle g_{\ell-1}, \Big(  \frac{u_2}{2^{am \over 2}} \Big)^{\nu+\nu'-\vartheta} \check \varphi_{{am \over 2}-k, \ell}^{u_2, n_2} \big\rangle  \check \varphi_{{am \over 2}-k, \ell}^{u_2, n_2}(x).
\end{align}

Clearly these are minor modifications of the projections. We point out that the expressions above contain another abuse of notation: the wave-packets outside the inner products in the last display are themselves slight modifications of the standard ones, obtained by multiplication on the spatial side by $(-i2^{am-k}(x- \underline{t}_{k, r}^{p'}))^{\nu'}$. This factor is always bounded by $2^{\nu'}$ on $\underline{I}_{k, r}^{p'}$. As $|u_1|$ and $|u_2|$ are both bounded by $3 \cdot 2^{am/2}$, neither the frequency properties of the wave-packets nor their fast decay away from $I_k^{n_1}$ are going to be affected by the modification: the whole analysis -- and in particular the energy estimate that will be proven in Proposition \ref{prop:energies} -- remains valid for these modified operators.\\
As
\[
\sum_{\nu, \nu' \geq 0} \frac{1}{\nu!} \frac{1}{\nu'!} \sum_{\vartheta =0} ^{\nu + \nu'}  {\nu + \nu' \choose \vartheta} 2^{\nu'} 3^{\nu+\nu'}
\]
is absolutely summable, it will indeed be enough to consider the case $\nu=\nu'=\vartheta=0$, that is, \eqref{eq:time-freq-localization}.
\end{observation}

Once at this point it is worth noticing the following feature that later will be used extensively: in \eqref{eq:time-freq-localization} (and thus also later in \eqref{eq:est:w:aver}) we have the flexibility of allowing $w$ to depend on the interval $\underline{I}_{k, r}^{p'}$ and on any of the parameters $p',v,n$, as well as on the shift parameter $j$.

Using this and putting now together \eqref{eq:trilinear:form:withS-bis}--\eqref{eq:time-freq-localization}, we deduce that for any sequence of positive $L^{\infty}$-weights\footnote{Here for notational simplicity we only keep the $j$ dependence although based on the earlier comment one may allow for each $w_j=w_j(m,k,\ell,r,p')$.} $\{w_j\}_j$  not identically zero on any $\underline{I}_{k, r}^{p'}$,  we have that -- up to superposing in the right-hand side similar terms with fast decaying amplitudes -- the following holds:
\begin{equation}
\begin{aligned}
 & \underline{\L}_{m,k, \ell, r}(f, g, h)=\sum_{j\in\Z} \underline{\L}^{j}_{m,k, \ell, r}(f, g, h) \approx  \sum_{j\in\Z} \sum_{n, v, p \sim 2^{am \over 2}}  \sum_{ p' \sim_{\lfloor\rfloor_p} 2^{am}} |\underline{I}_{k, r}^{p'}|^{\frac{1}{2}}\\
 &  \hspace{1em} \cdot\langle \underline{H}_{k,r}^{n, v, p}\big(x+\frac{j}{2^{{am\over 2}-k}} \big), \tilde{\varphi}_{am-k,2 \ell-1}^{2^{am}(r-1)+p'}  \rangle \, \frac{\displaystyle \int_{\underline{I}_{k, r}^{p'}} \underline{S}_{k,\ell,r}^{n, v, p}(f,g)(x)  w_j(x)\,dx}{ \int_{\underline{I}_{k, r}^{p'}}  w_j(x)\,dx}\:.
 \end{aligned}\label{puttogeth}
\end{equation}

It now remains to notice that by applying Cauchy-Schwarz we obtain via \eqref{eq:H_k} that for fixed parameters $k,\ell,r$,
\begingroup
\allowdisplaybreaks
\begin{equation}
\begin{aligned}
& \big |\underline{\L}_{m,k, \ell, r}(f, g, h) \big| \leq  \\
& \sum_{j\in{\mathbb Z}} \sum_{n, v \sim 2^{am \over 2}} \sum_{p' \sim 2^{am}}  \frac{ \big| \langle h \, \rho_{am-ak}(\lambda(\cdot)) w_{k, n, v}^e(\l), \Phi_{am-k,2\ell-1,j}^{2^{am}(r-1)+p'}  \rangle \big|}{\big( \aver{\underline{I}_{k, r}^{p'}}  w_j(x)  \,dx \big)^\frac{1}{2} } \\
& \hspace{7em} \cdot \Big( \int_{\underline{I}_{k, r}^{p'}} \big|\underline{S}_{k,\ell,r}^{n, v, \big \lfloor \frac{p'}{2^{am \over 2}} \big\rfloor}(f, g)(x)\big|^2  w_j(x)\,dx \Big)^\frac{1}{2}\,.
\end{aligned}\label{eq:est:w:aver}
\end{equation}
\endgroup
Above $\Phi_{am-k,2 \ell-1,j}^{2^{am}(r-1)+p'}$ is a modified version of $\tilde{\varphi}_{am-k,2 \ell-1}^{2^{am}(r-1)+p'}(\cdot-j2^{k-{am \over 2}})$ as given by
\begin{equation}\label{suppvstail}
\begin{aligned}
\Phi_{am-k,2 \ell-1,j}^{2^{am}(r-1)+p'}(x) &:= 2^{\frac{am-k}{2}} e^{i 2^{\frac{am}{2}}(2^{\frac{am}{2}-k}x-p_r)(2\ell-2)}  \\
& \hspace{1em} \cdot \check \psi \big(2^{am-k}(x- \underline{t}_{k, r}^{p'} - \frac{j}{2^{{am \over 2}-k}}) \big)\,\check{\varphi}(2^{{am \over 2}- k}x-p_r)\,,
\end{aligned}
\end{equation}
where we recall that $p= \lfloor p'2^{-{am \over 2}}\rfloor$.

\smallskip

Due to the well-localized spatial properties of $\check \psi \big(2^{am-k}(\cdot- \underline{t}_{k, r}^{p'} - \frac{j}{2^{{am \over 2}-k}}) \big)$ (supported in $\underline{I}_{k, r}^{p'}+j 2^{am \over 2} |\underline{I}_{k, r}^{p'}|$) and $ \check{\varphi}(2^{{am \over 2}- k}\cdot-p_r)$ (adapted to the interval $I_k^{p_r}$), the functions $(\Phi_{am-k,2 \ell-1,j}^{2^{am}(r-1)+p'})_j$ satisfy the same properties (in both space and frequency) with an overall amplitude decaying faster than any negative power of $1+|j|$. Hence it is enough to focus on the term $j=0$. We invite the interested reader to consult Remarks \ref{remark:shift} and \ref{remark:shifts} for a more detailed account on the modifications required in order to treat the terms corresponding to the $j \neq 0$ case. In this context, for notational simplicity, we will continue to refer to $\underline{\L}^{0}_{m,k, \ell, r}(f, g, h)$ as simply $\underline{\L}_{m,k, \ell, r}(f, g, h)$.

We conclude this section by noticing that via \eqref{eq:est:w:aver} we achieved our declared aim -- that of providing a good time-frequency localization for $h$ (relative to a proper weight) while bringing $\underline{S}_{k,\ell,r}^{n, v, p}(f, g)$ to a form that is amenable to the one-scale analysis of Section \ref{sec:onescale} with special stress on Corollary \ref{cor:uniform}.
\subsection{Splitting the trilinear form according to mass and input distribution}
We start this section by recalling that for the proof of the single scale estimate in Section \ref{sec:onescale} it was important to split the analysis into cases, according to whether the size\footnote{This was quantified in terms of a parameter $\mu\in (0,1)$ -- see the proofs of Propositions \ref{prop:uniform} and \ref{sgscalelambda}.} distribution of the inputs was uniform or clustered. It should thus come as no surprise that these concepts will continue to be relevant here. We will further refine that analysis by introducing a new concept, that of \emph{mass} -- see Remark \ref{remark:boundedness:cluster:number} below.

%Since, we aim to keep track of the parameter $j$ (dealing with tails), let us temporarily make use of the following notation: for an interval $I$ and an integer $j$, then $I^{(j)}$ denotes the shifted interval
%$$  I^{(j)}:=I + j 2^{am \over 2}|I|.$$

\begin{definition}
\label{def:indices}
Let $\delta_1,\mu\in(0,1)$ be as in Section \ref{sec:onescale} and assume $k,\ell,r\in\Z$ fixed. Let $f$ and $g$ be Schwartz functions, with $f_{\ell+1,r}, g_{\ell-1,r}$ as given in \eqref{eq:def:S_n,v,pFinalprep}. We then partition the set of all possible $(p',n,v)$ indices into:
\begin{itemize}
\item the set of \underline{\emph{light mass}} indices
\begin{equation}
\begin{aligned}
\mathfrak{l}:=\bigg\{ {{(p', n, v)}\atop{p'\sim 2^{am},\;n,v \sim 2^{am \over 2}}} \; : \; \frac{1}{|\underline{I}_{k, r}^{p'}|} \int_{\underline{I}_{k, r}^{p'}} w_{k, n, v}(\lambda)(x) \,  dx < 2^{-\delta_1 a m}  \bigg\};
\end{aligned}\label{eq:def:light}
\end{equation}
\item the set of \underline{\emph{heavy mass - uniform $(f,g)$-distribution}} indices
\begin{equation}
\begin{aligned}
&\mathfrak{u} :=\bigg\{   {{(p', n, v)}\atop{p'\sim 2^{am},\;n,v \sim 2^{am \over 2}}} \; : \; \frac{1}{|\underline{I}_{k, r}^{p'}|} \int_{\underline{I}_{k, r}^{p'}} w_{k, n, v}(\lambda)(x) \, dx \geq 2^{- \delta_1 a m}, \\
& \hspace{9em}\aver{I_k^{\lfloor p' 2^{-\frac{am}{2}}\rfloor_r-n}}|f_{\ell+1,r}|^2 \leq 2^{\mu \frac{am}{2}-k}\|f_{\ell+1,r}\|_{L^2}^2
\\
& \hspace{9em} \text{ and  } \; \aver{I_k^{\lfloor p' 2^{-\frac{am}{2}}\rfloor_r+n}}|g_{\ell-1,r}|^2 \leq 2^{\mu \frac{am}{2}-k}\|g_{\ell-1,r}\|_{L^2}^2 \bigg\};
\end{aligned}\label{eq:def:unif}
\end{equation}
\item the set of \underline{\emph{heavy mass - clustered $(f,g)$-distribution}} indices
\begin{equation}
\begin{aligned}
& \mathfrak{c}:= \bigg\{   {{(p', n, v)}\atop{p'\sim 2^{am},\; n,v \sim 2^{am \over 2}}} \; : \; \frac{1}{|\underline{I}_{k, r}^{p'}|} \int_{\underline{I}_{k, r}^{p'}} w_{k, n, v}(\lambda)(x) \, dx \geq 2^{- \delta_1 a m}, \\
& \hspace{9em} \aver{I_k^{\lfloor p' 2^{-\frac{am}{2}}\rfloor_r-n}}|f_{\ell+1,r}|^2 > 2^{\mu \frac{am}{2}-k}\|f_{\ell+1,r}\|_{L^2}^2 \\
& \hspace{9em} \text{ or  } \; \aver{I_k^{\lfloor p' 2^{-\frac{am}{2}}\rfloor_r+n}}|g_{\ell-1,r}|^2 > 2^{\mu \frac{am}{2}-k}\|g_{\ell-1,r}\|_{L^2}^2  \bigg\}.
\end{aligned}\label{eq:def:cluster}
\end{equation}
\end{itemize}

Of course, the sets of indices $\mathfrak{l}, \mathfrak{u},\mathfrak{c}$ depend implicitly on $f,g,k, \ell,r$ but, for simplicity, we leave this dependency out of the notation.
\end{definition}

\begin{remark}\label{remark:boundedness:cluster:number}

i) A careful inspection of the above definition reveals that, even if not introduced formally, the concept of mass associated with the indices $(p',n,v)$ and defined by
\begin{equation}\label{massintr}
\frac{1}{|\underline{I}_{k, r}^{p'}|} \int_{\underline{I}_{k, r}^{p'}} w_{k, n, v}(\lambda)(x)\, dx,
\end{equation}
plays an important role in our analysis. This should not be surprising since the concept of mass captures the amount of information relative to $\underline{I}_{k, r}^{p'}$ carried by the linearizing stopping time function $\l(x)$ that appears in the maximally modulated kernel of $BC^a$. This can be seen as an artifact inherited from the Carleson type behavior of our operator that ressembles the mass analysis of the classical Carleson operator introduced by Fefferman in \cite{f}.

ii) It will be useful later on to use a larger set $\bar{\mathfrak{c}}$ with $\mathfrak{c} \subset \bar{\mathfrak{c}}:=\bar{\mathfrak{c}}_1 \times \{ v \sim 2^{am \over 2} \}$, which does not contain joint information in $n$ and $v$:
\begin{equation}
\begin{aligned}
& \bar{\mathfrak{c}}_1:= \bigg\{   {{(p', n)}\atop{p'\sim 2^{am}, n \sim 2^{am \over 2}}} \; : \; \aver{I_k^{\lfloor p' 2^{-\frac{am}{2}}\rfloor_r-n}}|f_{\ell+1,r}|^2 > 2^{\mu \frac{am}{2}-k}\|f_{\ell+1,r}\|_{L^2}^2 \\
& \hspace{10em} \text{ or  } \; \aver{I_k^{\lfloor p' 2^{-\frac{am}{2}}\rfloor_r+n}}|g_{\ell-1,r}|^2 > 2^{\mu \frac{am}{2}-k}\|g_{\ell-1,r}\|_{L^2}^2  \bigg\}.
\end{aligned}\label{eq:def:cluster:enlarged}
\end{equation}
Now notice that for any $p' \sim 2^{am}$ fixed, there are at most  $O(2^{(1-\mu) \frac{am}{2}})$ values of $n$ such that $(p', n) \in \bar{\mathfrak{c}}_1$.
\end{remark}

With these steps completed, the trilinear form \eqref{eq:trilinear:form:withS-bis} splits into three terms according to the partition introduced in Definition \ref{def:indices}:
\begin{equation}
\begin{aligned}
\big |\underline{\L}_{m,k, \ell, r}(f, g, h) \big| &\lesssim \sum_{(p', n, v) \in \mathfrak{l}} \ldots+ \sum_{(p', n, v) \in \mathfrak{u}} \ldots + \sum_{(p', n, v) \in \mathfrak{c}} \ldots \\
&=:\underline{\L}_{m,k, \ell, r}^{\mathfrak{l}}(f, g, h) + \underline{\L}_{m,k,\ell, r}^{\mathfrak{u}}(f, g, h) + \underline{\L}_{m,k, \ell,r}^{\mathfrak{c}}(f, g, h).
\end{aligned}\label{eq:splitL}
\end{equation}
\subsection{Estimates for the localized trilinear forms $\underline{\L}_{m,k,\ell,r}^{\mathfrak{l}}, \underline{\L}_{m,k,\ell,r}^{\mathfrak{u}},\underline{\L}_{m,k,\ell,r}^{\mathfrak{c}}$} In this section we will further reshape the information carried within each of the three forms in decomposition \eqref{eq:splitL}. This will streamline the modulation invariant multiscale analysis performed in the next section.\footnote{Recall that we restrict here our attention to the main term $j=0$ in the trilinear form and that for notational simplicity we refer to $\underline{\L}^{0}_{m,k, \ell, r}$ as simply $\underline{\L}_{m,k, \ell, r}$. For more on the general case, \emph{i.e.} arbitrary values for $j$, please see Remarks \ref{remark:shift} and \ref{remark:shifts}.}

\subsubsection{Preliminary bounds} We begin by applying \eqref{eq:est:w:aver} to each of the terms in \eqref{eq:splitL}. For each case we will make use of the above mentioned freedom in choosing our weight $w$.

\begin{enumerate}[label=\arabic*., leftmargin=8pt]

\item \underline{The $\mathfrak{l}$-component}
\smallskip

In order to estimate $\underline{\L}_{m,k,\ell,r}^{\mathfrak{l}}(f, g, h)$, we apply \eqref{eq:est:w:aver} with the constant weights $w = 1$; this choice is hinted by the smallness of the mass parameter paired with the trivial bound guaranteed by Proposition \ref{cor:unweighted:est} below. With these, an application of Cauchy-Schwarz yields
\begin{equation}
\begin{aligned}
& \underline{\L}_{m, k,\ell,r}^{\mathfrak{l}}(f, g, h) \lesssim \Big( \sum_{(p', n, v) \in \mathfrak{l}}  \int_{\underline{I}_{k, r}^{p'}} \big|\underline{S}_{k,\ell,r}^{n, v, \big \lfloor \frac{p'}{2^{am \over 2}} \big\rfloor}(f, g)(x)\big|^2 dx  \Big)^\frac{1}{2} \\
& \hspace{2em} \cdot \Big( \sum_{(p', n, v) \in \mathfrak{l}}  \big| \langle h \, \rho_{am-ak}(\lambda(\cdot)) w_{k, n, v}^e(\l), \Phi_{am-k,2\ell-1}^{2^{am}(r-1)+p'}  \rangle \big|^2 \Big)^\frac{1}{2}.
\end{aligned}\label{eq:one:tile:light}
\end{equation}

\item \underline{The $\mathfrak{u}$-component}
\smallskip

For $\underline{\L}_{m, k,\ell,r}^{\mathfrak{u}}(f, g, h)$, we use the weights $w = w_{k, n, v}$. This choice is dictated by the setting provided by  Corollary \ref{cor:uniform}, that will be required in order to bound with exponential $m$-decay the expression $\underline{\L}_{m, k,\ell,r}^{\mathfrak{u}}$.

Now, using definition \eqref{eq:def:unif} of $\mathfrak{u}$, we have
\[ \aver{\underline{I}_{k, r}^{p'}}  w(x) \, dx = \frac{1}{|\underline{I}_{k, r}^{p'}|}\int_{\underline{I}_{k, r}^{p'}}  w_{k,n,v}(\lambda)(x) \, dx \gtrsim 2^{-\delta_1 a m}, \]
and hence, applying again Cauchy-Schwarz, we deduce
\begin{equation}
\begin{aligned}
& \underline{\L}_{m, k,\ell,r}^{\mathfrak{u}}(f, g, h) \\
& \lesssim 2^{\frac{\delta_1 a m}{2}} \Big( \sum_{(p', n, v) \in \mathfrak{u}}  \int_{\underline{I}_{k, r}^{p'}} \big|\underline{S}_{k,\ell,r}^{n, v, \big \lfloor \frac{p'}{2^{am \over 2}} \big\rfloor}(f, g)(x)\big|^2 w_{k, n, v}(x)dx  \Big)^\frac{1}{2} \\
&\hspace{5em} \cdot \Big( \sum_{(p', n, v) \in \mathfrak{u}}  \big| \langle h \, \rho_{am-ak}(\lambda(\cdot)) w_{k, n, v}^e(\l), \Phi_{am-k,2\ell-1}^{2^{am}(r-1)+p'}  \rangle \big|^2 \Big)^\frac{1}{2}.
\end{aligned}\label{eq:one:tile:unif}
\end{equation}

\item \underline{The $\mathfrak{c}$-component}
\smallskip

Lastly, motivated by the clustered component's low cardinality -- see ii) in Remark \ref{remark:boundedness:cluster:number}, we estimate $\underline{\L}_{m, k,\ell,r}^{\mathfrak{c}}(f, g, h)$ in the same manner as $\underline{\L}_{m, k,\ell,r}^{\mathfrak{l}}$, that is by taking $w = 1$ and applying Cauchy-Schwarz:
\begin{equation}
 \begin{aligned}
& \underline{\L}_{m, k,\ell,r}^{\mathfrak{c}}(f, g, h) \lesssim \Big( \sum_{(p', n, v) \in \mathfrak{c}}  \int_{\underline{I}_{k, r}^{p'}} \big|\underline{S}_{k,\ell,r}^{n, v, \big \lfloor \frac{p'}{2^{am \over 2}} \big\rfloor}(f, g)(x)\big|^2 dx  \Big)^\frac{1}{2} \\
& \hspace{2em} \cdot \Big( \sum_{(p', n, v) \in \mathfrak{c}}  \big| \langle h \, \rho_{am-ak}(\lambda(\cdot)) w_{k, n, v}^e(\l), \Phi_{am-k,2\ell-1}^{2^{am}(r-1)+p'}  \rangle \big|^2 \Big)^\frac{1}{2}.
\end{aligned}\label{eq:one:tile:cluster}
\end{equation}
\end{enumerate}

\subsubsection{Bounds of $L^2 \times L^2 \to L^2$ type for $\underline{S}_{k,\ell,r}^{n, v, p'}(f, g)$}
It is now necessary to estimate the $\ell^2$ quantities involving $\underline{S}_{k,\ell,r}^{n, v, p'}(f, g)$. For the $\mathfrak{l},\mathfrak{c}$ cases we use the trivial bound provided by Plancherel that is nevertheless sufficient for our purposes.
\begin{proposition}
\label{cor:unweighted:est}
For fixed parameters $k, \ell, r$, one has
\begin{equation}
\begin{aligned}
\Big( & \sum_{p' \sim 2^{am}} \sum_{n, v \sim 2^{am \over 2}}  \int_{\underline{I}_{k, r}^{p'}}   |\underline{S}_{k,\ell,r}^{n, v, \big \lfloor \frac{p'}{2^{am \over 2}} \big\rfloor}(f, g)(x)|^2 \,dx \Big)^{\frac{1}{2}} \\
& \lesssim 2^{- \frac{am}{4}} 2^{-\frac{k}{2}} \Big(  \sum_{u_1 \sim 2^{am \over 2}} \sum_{n_1 - 2^{am \over 2} r \sim 2^{am \over 2}} | \langle f_{\ell+1}, \check \varphi_{{am \over 2}-k, \ell}^{u_1, n_1}   \rangle|^2 \Big)^{\frac{1}{2}} \\
& \hspace{6em} \cdot  \Big(  \sum_{u_2 \sim 2^{am \over 2}} \sum_{n_2 - 2^{am \over 2} r \sim 2^{am \over 2}} |\langle g_{\ell-1}, \check \varphi_{{am \over 2}-k, \ell}^{u_2, n_2}   \rangle|^2 \Big)^{\frac{1}{2}} \\
& \lesssim 2^{- \frac{am}{4}} 2^{-\frac{k}{2}} \|f_{\ell+1,r}\|_{L^2} \|g_{\ell-1,r}\|_{L^2}.
\end{aligned}\label{eq:estimate:we:need:no:weight}
\end{equation}
\end{proposition}

\begin{proof}
The estimate above is a direct consequence of the orthogonality of the wave-packets involved in the definition of $\underline{S}_{k,\ell,r}^{n,v,\big \lfloor \frac{p'}{2^{am \over 2}} \big\rfloor}$ -- see \eqref{eq:def:underline:S}. Indeed, squaring the left-hand side in the first inequality of \eqref{eq:estimate:we:need:no:weight} and rearranging the summation, we have
\begin{align*}
&\sum_{n, v \sim 2^{am \over 2}}\sum_{p' \sim 2^{am}} \int_{\underline{I}_{k, r}^{p'}} |\underline{S}_{k,\ell,r}^{n,v,\big \lfloor \frac{p'}{2^{am \over 2}} \big\rfloor}(f,g)(x)|^2\,dx\\
&=\sum_{n, v \sim 2^{am \over 2}} \sum_{p \sim 2^{am \over 2}} \int_{I_k^{p_r}} |\underline{S}_{k,\ell,r}^{n,v,p}(f,g)(x)|^2\,dx.
\end{align*}
Since $\underline{S}_{k,\ell,r}^{n,v,p}(f,g)$ is a periodic function of period $|I_k^{p_r}|$, Plancherel equality implies that the quantity above is equal to
\begin{align*}
 & 2^{-2k}\, 2^{-({am \over 2}-k)}\,\sum_{n, v \sim 2^{am \over 2}} \sum_{p \sim 2^{am \over 2}} \sum_{u \sim 2^{am \over 2}} |\langle f_{\ell+1}, \check{\varphi}_{{am \over 2}- k, \ell}^{u-v, p_r-n}  \rangle|^2\,   |\langle g_{\ell-1}, \check{\varphi}_{{am \over 2}- k, \ell}^{u+v, p_r+n} \rangle|^2\\
&= 2^{- \frac{am}{2}} 2^{-k} \Big(
\sum_{u_1 \sim 2^{am \over 2}} \sum_{n_1 - 2^{am \over 2} r \sim 2^{am \over 2}} \Big \vert \langle f_{\ell+1}, \check \varphi_{{am \over 2}-k, \ell}^{u_1, n_1}   \rangle \Big \vert^2 \Big) \\
 & \hspace{5em} \cdot  \Big(  \sum_{u_2 \sim 2^{am \over 2}} \sum_{n_2 - 2^{am \over 2} r \sim 2^{am \over 2}} \Big \vert \langle g_{\ell-1}, \check \varphi_{{am \over 2}-k, \ell}^{u_2, n_2}   \rangle \Big \vert^2 \Big)\\
&\approx 2^{- \frac{am}{2}} 2^{-k}\, \|f_{\ell+1,r}\|_{L^2}^2 \|g_{\ell-1,r}\|_{L^2}^2.
\end{align*}
\end{proof}

For the $\mathfrak{u}$ case we argue that the treatment of the corresponding uniform component in Proposition \ref{prop:uniform} -- and more precisely in Corollary \ref{cor:uniform} -- carries forward to our situation in \eqref{eq:one:tile:unif}. It yields the following:

\begin{proposition}\label{prop:before:cor:Prop4} For $k,\ell,r$ fixed and $\tilde S_k^{n, v}$ as defined in \eqref{eq:tildeS}, we have
\begin{align*}
 \Big(\sum_{(p', n, v) \in \mathfrak{u}} \int_{\underline{I}_{k, r}^{p'}}   |\tilde S_k^{n, v}(F, G)(x)|^2 & w_{k, n, v}(\lambda)(x) dx \Big)^{1/2}  \\
& \lesssim 2^{-\delta_1\, a\, m} 2^{- \frac{am}{4}} 2^{-\frac{k}{2}} \|F\|_{L^2} \|G\|_{L^2}.
\end{align*}
\end{proposition}
%
%
%\begin{proof}
%We notice that we are exactly in the situation described by Corollary \ref{cor:uniform} with the phase function $\lambda$.
%the only difference that the original phase function $\lambda(\cdot)$ is now replaced by its shift $\lambda(\cdot+j2^{k-{am \over 2}})$ (indeed shifting $w_{k,n,v}(\lambda)$ is equivalent to shifting $\lambda$). However, since the estimates in the proof of Corollary \ref{cor:uniform} do not depend on the stopping time $\lambda$, we can safely apply them.\footnote{Indeed, we stress that the parameter $\delta_1$ and the definition of ``uniform'' distribution for $f,g$ involving the parameter $\mu$ do not depend on the stopping time $\lambda$.}
%\end{proof}
%
From the above Proposition \ref{prop:before:cor:Prop4} and \eqref{eq:def:S_n,v,pFinal} we immediately deduce

\begin{corollary}\label{cor:Prop4}
For $k,\ell,r$ fixed, we have
\begin{equation}
\begin{aligned}
& \Big( \sum_{(p', n, v)\in \mathfrak{u}}  \int_{\underline{I}_{k, r}^{p'}}   |\underline{S}_{k,\ell,r}^{n, v, \big \lfloor \frac{p'}{2^{am \over 2}} \big\rfloor} (f, g)(x)|^2 w_{k, n, v}(\lambda)(x) dx \Big)^{1/2} \\
& \hspace{6em} \lesssim 2^{-\delta_1\,a\, m} 2^{- \frac{am}{4}} 2^{-\frac{k}{2}}  \|f_{\ell+1,r}\|_{L^2} \|g_{\ell-1,r}\|_{L^2}.
\end{aligned}\label{eq:estimate:we:need}
\end{equation}
\end{corollary}

As a final remark, notice that the estimate \eqref{eq:estimate:we:need} -- corresponding to the heavy mass uniform case $\mathfrak{u}$ -- is an improvement over the rougher bound in \eqref{eq:estimate:we:need:no:weight} due to the presence of the extra decaying factor $2^{-\delta_1\, a m}$.

\subsubsection{Final form of the estimates} \label{paragraph:6.4.3}

Putting together all the information collected so far (using \eqref{eq:estimate:we:need:no:weight} for $\mathfrak{l},\mathfrak{c}$ and \eqref{eq:estimate:we:need} for $\mathfrak{u}$), we obtain that:
\begingroup
\allowdisplaybreaks
\begin{align}
& \begin{aligned} \label{eq:one:tile:light:bis0}
&\underline{\L}_{m, k,\ell,r}^{\mathfrak{l}}(f, g, h) \lesssim 2^{- \frac{am}{4}} 2^{-\frac{k}{2}}  \|f_{\ell+1,r}\|_{L^2} \, \|g_{\ell-1,r}\|_{L^2} \\
& \hspace{5em} \cdot \Big( \sum_{(p', n, v) \in \mathfrak{l}}  \big| \langle h \, \rho_{am-ak}(\lambda(\cdot)) w_{k, n, v}^e(\l), \Phi_{am-k,2 \ell-1}^{2^{am}(r-1)+p'}  \rangle \big|^2 \Big)^{1/2},
\end{aligned} \\
&{} \nonumber \\
&\begin{aligned} \label{eq:one:tile:unif:bis0}
&\underline{\L}_{m, k,\ell,r}^{\mathfrak{u}}(f, g, h) \lesssim 2^{-\frac{ \delta_1 a m}{2}}  2^{- \frac{am}{4}} 2^{-\frac{k}{2}}  \|f_{\ell+1,r}\|_{L^2} \, \|g_{\ell-1,r}\|_{L^2} \\
& \hspace{5em} \cdot \Big( \sum_{(p', n, v) \in \mathfrak{u}}  \big| \langle h \, \rho_{am-ak}(\lambda(\cdot)) w_{k, n, v}^e(\l), \Phi_{am-k,2\ell-1}^{ 2^{am}(r-1)+p'}  \rangle \big|^2 \Big)^{1/2},
\end{aligned}\\
&{} \nonumber \\
& \begin{aligned} \label{eq:one:tile:cluster:bis0}
&\underline{\L}_{m, k,\ell,r}^{\mathfrak{c}}(f, g, h) \lesssim 2^{- \frac{am}{2}} 2^{-\frac{k}{2}}  \|f_{ \ell+1,r}\|_{L^2} \, \|g_{\ell-1,r}\|_{L^2} \\
& \hspace{5em} \cdot \Big( \sum_{(p', n, v) \in \mathfrak{c}}  \big| \langle h \, \rho_{am-ak}(\lambda(\cdot)) w_{k, n, v}^e(\l), \Phi_{am-k,2 \ell-1}^{ 2^{am}(r-1)+p'}  \rangle \big|^2 \Big)^{1/2}.
\end{aligned}
\end{align}
\endgroup

Estimates \eqref{eq:one:tile:light:bis0}, \eqref{eq:one:tile:unif:bis0} and \eqref{eq:one:tile:cluster:bis0} are single-scale estimates with a sharp time-frequency localization for all three functions $f,g,h$, as desired. It will remain to incorporate them into the BHT time-frequency analysis (the low resolution model), which will be the subject of the next section. With that in mind, it will be convenient to switch to the tri-tiles notation introduced in Section \ref{LR}. In particular, we refer the reader to \eqref{eq:def:time:frequency:regions} for a dictionary that relates the indices $(k,\ell,r)$ to the time-frequency regions $P\in \BHT$.

We end this section with the following remarks. 

\begin{remark}
\label{remark:shift}
As we pointed out previously, in the two previous sections we have focused only on the main term corresponding to $j=0$. Let us briefly explain the modifications that need to be performed in order to deal with an arbitrary value of the shift $j\in\mathbb Z$. This adaptation essentially requires a ``translation'' of the previous analysis from the interval ${\underline{I}}_{k,r}^{p'}$ to the shifted interval
\[ {\underline{I}}_{k,r}^{p',(j)}:= {\underline{I}}_{k,r}^{p'} + j 2^{am \over 2}|{\underline{I}}_{k,r}^{p'}|. \]

\begin{itemize}
\item The mass is then defined by
$$  \frac{1}{|\underline{I}_{k, r}^{p'}|} \int_{\underline{I}_{k, r}^{p',(j)}} w_{k, n, v}(\lambda)(x) \, dx.$$
\item Proposition \ref{prop:before:cor:Prop4} and Corollary \ref{cor:Prop4} are still valid with the only modification of replacing the weight $w_{k,n,v}(\lambda)$ by its shifted version $w_{k,n,v}(\lambda)(\cdot+j2^{k-{am \over 2}})$. Such an operation has no effect on the different proofs. In particular, shifting the weight is equivalent to shifting the phase function $\lambda$ and since the estimates in the proof of Corollary \ref{cor:uniform} do not depend on the stopping time $\lambda$, we can safely apply them.\footnote{Indeed, we stress that the parameter $\delta_1$ and the definition of ``uniform'' distribution for $f,g$ involving the parameter $\mu$ do not depend on the stopping time $\lambda$.}
\item Everything will bring us to similar estimates as \eqref{eq:one:tile:light:bis0}-\eqref{eq:one:tile:cluster:bis0}, with the only modification of replacing the wave-packet function $\Phi_{am-k,2 \ell-1}^{ 2^{am}(r-1)+p'}$ by its shifted version $\Phi_{am-k,2 \ell-1,j}^{ 2^{am}(r-1)+p'}$, as defined in \eqref{suppvstail}.
\end{itemize}
\end{remark}

As a conclusion of the previous remark, we emphasize that for every $j$, there is a notion of mass and so of subsets $\mathfrak{l},\mathfrak{u},\mathfrak{c}$ and putting together all these contributions lead us to the following estimates\footnote{Of course in the statement below $\underline{\L}_{m, k,\ell,r}$ refers to the original form in \eqref{puttogeth} while $\underline{\L}_{m, k,\ell,r}^{*}$, $\{*\}\in\{\mathfrak{l}, \mathfrak{u}, \mathfrak{c}\}$ stand for the obvious global, modified quantities.}:

\begingroup
\allowdisplaybreaks
\begin{align}
& \begin{aligned} \label{eq:one:tile:light:bis}
&\underline{\L}_{m, k,\ell,r}^{\mathfrak{l}}(f, g, h) \lesssim 2^{- \frac{am}{4}} 2^{-\frac{k}{2}}  \|f_{\ell+1,r}\|_{L^2} \, \|g_{\ell-1,r}\|_{L^2} \\
& \hspace{3em} \cdot \sum_{j \in \Z} \Big( \sum_{(p', n, v) \in \mathfrak{l}}  \big| \langle h \, \rho_{am-ak}(\lambda(\cdot)) w_{k, n, v}^e(\l), \Phi_{am-k,2 \ell-1,j}^{2^{am}(r-1)+p'}  \rangle \big|^2 \Big)^{1/2},
\end{aligned}\\
&{} \nonumber \\
& \begin{aligned} \label{eq:one:tile:unif:bis}
& \underline{\L}_{m, k,\ell,r}^{\mathfrak{u}}(f, g, h) \lesssim 2^{-\frac{ \delta_1 a m}{2}}  2^{- \frac{am}{4}} 2^{-\frac{k}{2}}  \|f_{\ell+1,r}\|_{L^2} \, \|g_{\ell-1,r}\|_{L^2} \\
& \hspace{3em} \cdot \sum_{j \in \Z} \Big( \sum_{(p', n, v) \in \mathfrak{u}}  \big| \langle h \, \rho_{am-ak}(\lambda(\cdot)) w_{k, n, v}^e(\l), \Phi_{am-k,2\ell-1,j}^{ 2^{am}(r-1)+p'}  \rangle \big|^2 \Big)^{1/2},
\end{aligned} \\
&{} \nonumber \\
& \begin{aligned} \label{eq:one:tile:cluster:bis}
& \underline{\L}_{m, k,\ell,r}^{\mathfrak{c}}(f, g, h) \lesssim 2^{- \frac{am}{2}} 2^{-\frac{k}{2}}  \|f_{ \ell+1,r}\|_{L^2} \, \|g_{\ell-1,r}\|_{L^2} \\
& \hspace{3em} \cdot \sum_{j \in \Z} \Big( \sum_{(p', n, v) \in \mathfrak{c}}  \big| \langle h \, \rho_{am-ak}(\lambda(\cdot)) w_{k, n, v}^e(\l), \Phi_{am-k,2 \ell-1,j}^{ 2^{am}(r-1)+p'}  \rangle \big|^2 \Big)^{1/2}.
\end{aligned} 
\end{align}
\endgroup

\begin{remark}
\label{remark:shifts}
Due to the extra decay in $1+|j|$ enjoyed by the wave-packets $ \Phi_{am-k,2 \ell-1,j}^{ 2^{am}(r-1)+p'}$, the right hand side of \eqref{eq:one:tile:light:bis}-\eqref{eq:one:tile:cluster:bis} should be regarded as a superposition of similar terms with fast decaying amplitudes in $j$, thus making it enough to carefully consider the case $j=0$. Indeed, here are the main properties of the functions $\Phi_{am-k,2 \ell-1,j}^{ 2^{am}(r-1)+p'}$ given by \eqref{suppvstail}: for an arbitrarily large exponent $M>0$,
\begin{enumerate}[label=(W\arabic*), leftmargin=*]
\item \label{W1}  the function $\Phi_{am-k,2 \ell-1,j}^{ 2^{am}(r-1)+p'}$ is $L^2$-normalized, supported in space on $\underline{I}_{k, r}^{p',(j)}$, adapted in frequency at the scale $2^{am-k}$ around $[2^{am-k}(2\ell), 2^{am-k}(2\ell+1)]$ and it satisfies the pointwise inequality
\[
\big| \Phi_{am-k,2 \ell-1,j}^{ 2^{am}(r-1)+p'}(x) \big|  \lesssim_M (1+|j|)^{-M} |\underline{I}_{k, r}^{p',(j)}|^{-{1 \over 2}} \one_{\underline{I}_{k, r}^{p',(j)}}(x);
\]
\item \label{W2} for every $j,k,\ell,r$, the functions $\Phi_{am-k,2 \ell-1,j}^{ 2^{am}(r-1)+p'}$ are compactly supported\footnote{Notice that, thanks to the decay in $j$, we trade the characteristic function of $I_{k}^{p_r}+j|I_{k}^{p_r}|$ for a bump function adapted to $I_{k}^{p_r}$.} in $I_{k}^{p_r}+j|I_{k}^{p_r}|$ for $p'\sim_{\lfloor\rfloor_p} 2^{am}$ and uniformly in $j$, we have
$$ \sum_{p'\sim_{\lfloor\rfloor_p} 2^{am}} |\Phi_{am-k,2 \ell-1,j}^{ 2^{am}(r-1)+p'}(x)| \lesssim 2^{\frac{am-k}{2}} (1+|j|)^{-{M\over 2}} \ci_{I_{k}^{p_r}}(x);$$
\item \label{W3} the light mass case is given by the condition
$$ \frac{1}{|\underline{I}_{k, r}^{p'}|} \int_{\underline{I}_{k, r}^{p',(j)}} w_{k, n, v}(\lambda)(x) \,  dx < 2^{-\delta_1 a m};$$
in particular, one deduces that for such indices
$$ \frac{1}{|\underline{I}_{k, r}^{p'}|} \| w_{k,n,v} \Phi_{am-k,2 \ell-1,j}^{ 2^{am}(r-1)+p'}\|_{L^1} \lesssim_M 2^{\frac{am-k}{2}} (1+|j|)^{-M} 2^{-\delta_1 am};$$
\item \label{W4}
for $k,r$ fixed, the functions $(\Phi_{am-k,2 \ell-1,j}^{ 2^{am}(r-1)+p'})_{\ell,p'}$ are almost orthogonal and compactly supported within $j I_{k,r}$ with $I_{k,r}:=[r2^k,(r+1)2^k]$: for every function $F \in L^2$ then
$$ \sum_{\ell\in{\mathbb Z}} \sum_{p'\sim 2^{am}} |\langle F, \Phi_{am-k,2 \ell-1,j}^{ 2^{am}(r-1)+p'}\rangle|^2 \lesssim_M (1+|j|)^{-M} \|F \ci_{I_{k,r}}\|_{L^2}^2.$$
\end{enumerate}

Besides the pointwise decay illustrated in \ref{W1}, the light mass distribution condition \ref{W3} and almost orthogonality of the wave-packets $(\Phi_{am-k,2 \ell-1,j}^{ 2^{am}(r-1)+p'})_{\ell,p'}$ (for $k$ and $r$ fixed) play an important role in the time-frequency analysis that will be performed in the next section. Since we cannot use the spatial information before making use of the frequency information, \footnote{After all, this section's purpose is to obtain a time-frequency localization for the dualizing function $h$ as well.} the analysis described in the next section should be implemented for any fixed value of $j$. Once the phase of estimating the sizes and energies reached, we can finally take advantage of the extra spatial decay $(1+|j|)^{-{M \over 2}}$ -- see Remark \ref{remark:decay}. A careful and rather standard inspection will reveal that indeed it is enough to consider the case $j=0$, as the other cases follow similarly (with the exact same localizations, since they are encoded into bump functions adapted to the right scale, allowing to incorporate the $j$-shift -- see Remark \ref{remark:decay}).
\end{remark}
\section{The low resolution, multi-scale analysis: modulation invariance analysis}
\label{sec:bilinear:analysis}
With the single scale estimate simultaneously featuring bounds with exponential decay in $m$ and securing well-localized time-frequency information for all the three functions $f, g$, and $h$ of the associated trilinear form, we are finally set to perform the multi-scale time-frequency analysis characteristic to multilinear operators having the frequency portrait of BHT. This will allow us to prove Theorem \ref{mdec}, from which Theorems \ref{main_theorem_pxq_to_r} and \ref{main} follow.

With $a\in (0,\infty)\setminus\{1,2\}$ and $m\in\N$ fixed for the reminder of the section, we will proceed as follows:
\begin{itemize}
\item the single-scale estimates of Section \ref{paragraph:6.4.3} will be reformulated in terms of tri-tiles, leading to the study of certain trilinear forms $\Lambda_{\BHT}^*(f, g, h)$ associated to the rank-1 family $\BHT$;
\item to treat these trilinear forms, we adapt classical time-frequency methods to our context. In particular, in our situation both the rank-1 collection $\BHT$ of expanded tri-tiles of area $2^{am}$ and the finer collection of tri-tiles $\{ s: \hat s \in \BHT \}$ of area $1$ will be involved in this analysis;
\item as a consequence of the above we will deduce the boundedness of $T_{m}$ in the strict local-$L^2$ range ($2<p, q, r'<\infty$);
\item finally, using a suitable interpolation procedure -- which consists in localizing the trilinear form information in space and interpolating between maximal sizes -- we extend the boundedness of $T_{m}$ to the range of H\"older exponents with $2<p, q \leq \infty$ and $1<r<\infty$ -- thus proving Theorem \ref{mdec}.
\end{itemize}

\subsection{Reformulation of the estimates in Section \ref{paragraph:6.4.3} in terms of $\BHT$ tri-tiles }
As we move towards the time-frequency analysis part of our study, we need to restate the estimates \eqref{eq:one:tile:light:bis}-\eqref{eq:one:tile:cluster:bis} in the alternative tri-tile notation. This is enabled by the unique correspondence between the triple of parameters $(k,\ell,r)$ and the triple of time-frequency regions $P=(P_1, P_2,P_3)$ introduced in Section \ref{subsubsec:lowmodel}. However, we keep the parameters $p' \sim 2^{am}, n, v \sim 2^{am \over 2}$, as they allow to define the light, uniform and clustered components associated to each $\BHT$ tri-tile. The reader is invited to revisit the parametric and tile formulation of $\underline{\L}_{m}$, and the other definitions in Section \ref{LR}.

Let $(k, \ell, r)$ be a triple of integers and let $P=(P_1, P_2,P_3)$ be the corresponding tri-tile, as defined in \eqref{eq:tri:tiles:def:S4}. Then we write $\rho_{|P|}(\lambda(x))$ for $\rho_{am-ak}(\lambda(x))$, where $|P|$ is used here in order to stress that $\rho$ only depends on the scale of the tri-tile $P$ and not on its time-frequency location; similarly, we set
\begin{equation}
\label{def:eq:weight-oscill:P}
w_{k, n, v}^{e}(x)=w_{|P|, n, v}^e (\lambda)(x):=  \int_{\R} \psi( \xi + 2v ) e^{i \, n \, \xi} e^{i c_a 2^{ ({am  \over 2}) a'}{2} \xi^{a'} |I_P|^{-a'} \lambda(x)^{- {1 \over {a-1}}} } \psi \big( \frac{\xi}{2^{am \over 2}} \big)d \xi
\end{equation}
and 
\[
w_{k, n, v}(\lambda)(x):=w_{|P|, n, v}(\lambda)(x)
\]
defined in \eqref{eq:decay-weight} is the positive weight measuring the decay in $\rho_{|P|}(\lambda(x)) \cdot w_{|P|, n, v}^e (\lambda)(x)$. Finally, we let $\underline{I}_{P}^{p'}$ stand for $\underline{I}_{k,r}^{p'}$, one of the $2^{am}$ intervals of length $2^{-am}|I_P|$ partitioning $I_P$.

The single scale trilinear form $\underline{\L}_{m,k,\ell,r}$ associated to the scale $k$, frequency parameter $\ell$ and spatial parameter $r$ is identified with $\underline{\L}_{m,P}$ and further expressed as
\[
 \underline{\L}_{m,P} =  \underline{\L}_{m,P}^{\mathfrak{l}} + \underline{\L}_{m,P}^{\mathfrak{u}} + \underline{\L}_{m,P}^{\mathfrak{c}},
\]
with the right-hand side defined in the obvious way through Definition \ref{def:indices}.

The function $f_{\ell+1, r}$ becomes
\begin{equation}
\label{eq:def:f_P}
f_{P}:=\sum_{s : \hat s=P} \langle f, \phi_{s_1}  \rangle  \phi_{s_1}
\end{equation}
while $g_{\ell-1, r}$ is expressed as
\begin{equation}
\label{eq:def:g_P}
g_{P}:=\sum_{s : \hat s=P} \langle g, \phi_{s_2}  \rangle  \phi_{s_2}.
\end{equation}

We have
\begin{equation}
\label{eq:def:f(P):tiles}
\|f_{P}\|_{L^2} \lesssim f(P_1):= \Big(  \sum_{s_1: \hat s_1 =P_1} |\langle f, \phi_{s_1}  \rangle|^2 \Big)^\frac{1}{2},
\end{equation}
and
\begin{equation}
\label{eq:def:g(P):tiles}
\|g_{P}\|_{L^2} \lesssim g(P_2):= \Big(  \sum_{s_2: \hat s_2 =P_2} |\langle g, \phi_{s_2}  \rangle|^2 \Big)^\frac{1}{2}.
\end{equation}

If we set $\Psi_{P_3}^{p'}:= \Psi_{am-k,2 \ell-1}^{2^{am}(r-1)+p'}$, the quantity
\[
\Big(  \frac{1}{2^{\frac{am}{2}}} \sum_{(p', n, v) \in *}  \big| \langle h \, \rho_{am-ak}(\lambda(\cdot)) w_{k, n, v}^e(\l), \Psi_{am-k,2 \ell-1}^{2^{am}(r-1)+p'}  \rangle \big|^2  \Big)^{1/2}
\]
appearing in \eqref{eq:one:tile:light:bis}-\eqref{eq:one:tile:cluster:bis}, reads as
\begin{equation}
\label{eq:def:h*(P):tile}
h_\ast(P_3):= \Big( \frac{1}{2^{am \over 2}} \sum_{(n, v, p') \in \ast}  \big| \langle h \, \rho_{|P|}(\lambda(\cdot)) w_{|P|, n, v}^e(\l), \Psi_{P_3}^{p'}  \rangle \big|^2  \Big)^\frac{1}{2},
\end{equation}
where $\ast \in \{  \mathfrak{l}, \mathfrak{u}, \mathfrak{c}\}$.

Using the above dictionary we can now rephrase the estimates obtained in Section \ref{paragraph:6.4.3} as
\smallskip
\begin{proposition} \label{refinedsgscale}
For any $P \in \BHT$ we have
\begin{alignat}{3}
\label{eq:esti:L^P:light}
\underline{\L}_{m,P}^{\mathfrak{l}}(f, g, h) &\lesssim &&|I_P|^{-1/2} f(P_1)g(P_2)h_{\mathfrak{l}}(P_3),\\
\label{eq:esti:L^P:uniform}
\underline{\L}_{m,P}^{\mathfrak{u}}(f, g, h) &\lesssim 2^{- \frac{\delta_1 a m}{2}} &&|I_P|^{-1/2} f(P_1)g(P_2)h_{\mathfrak{u}}(P_3),\\
\label{eq:esti:L^P:cluster}
\underline{\L}_{m,P}^{\mathfrak{c}}(f, g, h) &\lesssim && |I_P|^{-1/2} f(P_1)g(P_2)h_{\mathfrak{c}}(P_3).
\end{alignat}
\end{proposition}

Before moving on to formulate the multi-scale problem, we recall the properties of the wave-packets capturing the information associated to the tri-tile $P=(P_1, P_2, P_3)$:
\begin{enumerate}[leftmargin=*]
\item[$\bullet$] for $j=1, 2$ the wave-packets $\phi_{s_j}$, defined also in \eqref{def:wave:p:f:g}, are $L^2$-normalized, compactly supported in frequency on $\omega_{s_j}$, and adapted in space to $I_{s_j}$; consequently
\[
|\phi_{s_j}(x)| \lesssim \frac{1}{|I_{s_j}|^{1 \over 2}} \ci_{I_{s_j}(x)},
\]
and whenever $s_j, s'_j$ are so that  $s_j \cap s'_j \neq \emptyset$ with $|I_{s'_j}|\leq |I_{s_j}|$,
\[
 \big|\langle \phi_{s_j}, \phi_{s'_j}  \rangle \big| \lesssim \frac{|I_{s'_j}|^{1 \over 2}}{|I_{s_j}|^{1 \over 2}}  \Big( 1+ \frac{\dist(I_{s_j}, I_{s'_j})}{|I_{s_j}|}    \Big)^{-50}.
\]

\item[$\bullet$] the information for $h$ is captured through the wave-packets $\Psi_{P_3}^{p'}$; for convenience, we reformulate in terms of tiles the properties of these wave-packets which were initially\footnote{Due to Remark \ref{remark:shifts}, we only focus on the case $j=0$.} stated via \ref{W1}--\ref{W4} in Section \ref{paragraph:6.4.3}:
\begin{enumerate}[label=(WP\arabic*)]
\item \label{WP1}
 the function $\Psi_{P_3}^{p'}$ is $L^2$-normalized, adapted in frequency to $\pmb{\omega}_{P_3}$ and supported in space on $\underline{I}_{P}^{p'}$;

\item \label{WP2} for every $P$, the functions $\big\{ \Psi_{P_3}^{p'} \big\}_{p'}$ obey the pointwise estimates
\[
|\Psi_{P_3}^{p'}(x)| \lesssim |\underline{I}_{P}^{p'}|^{-{1 \over 2}} \one_{\underline{I}_{P}^{p'}}(x) \quad \text{and} \quad  \sum_{p'\sim 2^{am}}|\Psi_{P_3}^{p'}(x)| \lesssim  |I_{P}|^{-\frac{1}{2}} 2^{am \over 2} \one_{I_P}(x);
\]
\item \label{WP3} for a tri-tile $P$ with $|I_P|=2^k$ and $(p',n,v)$ satisfying the light mass condition, we have
$$ \frac{1}{|\underline{I}_{P}^{p'}|} \| w_{k,n,v} \Psi_{P_3}^{p'}\|_{L^1} < |\underline{I}_{P}^{p'}|^{-\frac{1}{2}}  2^{-\delta_1 am}.$$
\item \label{WP4} for any fixed interval $I$, the functions $(\Psi_{I \times \pmb{\omega}}^{p'})_{\pmb{\omega},p'}$ with $|I||\pmb{\omega}|=2^{am}$ are almost orthogonal and adapted to $I$: for every function $F \in L^2$,
$$ \sum_{\pmb{\omega}: |\pmb{\omega}|=2^{am}|I|^{-1}} \ \sum_{p'\sim 2^{am}} |\langle F, \Psi_{I\times \pmb{\omega}}^{p'}\rangle|^2 \lesssim \|F \cdot \one_{I}\|_{L^2}^2.$$
\end{enumerate}
\end{enumerate}

We refer the reader to Figure \ref{figure:5} in Section \ref{sec:HR2} (as well as to the figures in Section \ref{sec:discretization}), for an illustration of the phase portraits associated to these wave-packets, in relation to the expanded tri-tiles $P=(P_1, P_2, P_3)$; in particular, notice that $\Psi_{P_3}^{p'}$ is not adapted to $s_3$ -- as defined in \eqref{eq:def:time:frequency:regions} -- but to $\underline{I}_{P}^{p'} \times \pmb{\omega}_{P_3}$.

As a consequence of Proposition \ref{refinedsgscale}, for $\ast \in \{  \mathfrak{l}, \mathfrak{u}, \mathfrak{c}\}$, we define\footnote{Of course, the trilinear forms $\Lambda_{\BHT}^{\ast}(f,g,h)$ depend implicitly on $m$ since in particular the very construction of the collection $\BHT$ depends on $m$.} the positive trilinear forms
\begin{equation}
\label{eq:def:Lambda+}
\begin{aligned}
\Lambda_{\BHT}^{\ast}(f,g,h):= \sum_{P \in \BHT} |I_P|^{-\frac{1}{2}} f(P_1) g(P_2) h_{\ast}(P_3).
\end{aligned}
\end{equation}
Since \eqref{eq:esti:L^P:light}-\eqref{eq:esti:L^P:cluster} imply

\begin{equation}
\label{eq:decomposition:Lm}
|\underline{\L}_{m}(f, g, h)| \lesssim 2^{-\frac{\delta_1}{2}am}\Lambda_{\BHT}^{\mathfrak{u}}(f,g,h) + \Lambda_{\BHT}^{\mathfrak{l}}(f,g,h)+\Lambda_{\BHT}^{\mathfrak{c}}(f,g,h),
\end{equation}
the boundedness of $T_m$ in the local-$L^2$ range follows once we prove

\begin{proposition} \label{prop:aim} For every $(p,q,r')$ such that $2<p,q , r' <\infty$ and $\frac{1}{r}=\frac{1}{p}+\frac{1}{q}$, there exists some $\delta_2:=\delta_2(p,q)>0$ such that
\begin{alignat*}{3}
\Lambda_{\BHT}^{\mathfrak{l}}(f, g, h) & \lesssim_{a,p,q} 2^{-\delta_2 a m} &&\|f\|_{L^p} \|g\|_{L^q} \|h\|_{L^{r'}}, \\
\Lambda_{\BHT}^{\mathfrak{u}}(f, g, h) & \lesssim_{a,p,q} &&\|f\|_{L^p} \|g\|_{L^q} \|h\|_{L^{r'}}, \\
\Lambda_{\BHT}^{\mathfrak{c}}(f, g, h) &\lesssim_{a,p,q} 2^{-\delta_2 a m} &&\|f\|_{L^p} \|g\|_{L^q} \|h\|_{L^{r'}}.
\end{alignat*}
\end{proposition}

This will produce the desired decaying estimate for the full trilinear form $\underline{\L}_m$, in the local-$L^2$ range (according to \eqref{eq:decomposition:Lm}) with $\delta:=\min\{\frac{\delta_1}{2},\delta_2\}$:
\begin{corollary}
\label{cor:conclusion:trilinear:form:L}
For every $(p,q,r')$ in the local-$L^2$ range ($2< p, q, r' <\infty$), there exists some $\delta:=\delta(p,q)>0$ such that
\[
\underline{\L}_{m}(f, g, h) \lesssim_{a,p,q} 2^{-\delta a m} \|f\|_{L^p} \|g\|_{L^q} \|h\|_{L^{r'}}.
\]
\end{corollary}

In the next subsections we will focus on proving Proposition \ref{prop:aim}. Before this though, we continue with some notions of bilinear time-frequency analysis.

\smallskip

\subsection{Preliminaries on time-frequency tools}
\label{sec:timeFreq:prelim}
In our approach we will use as a black box the results in \cite{MTTBiest2}. This will require nevertheless the introduction of several adjusted notions as well as some non-trivial adaptations to the present context.

The trilinear forms $\Lambda_{\BHT}^{\ast}(f,g,h)$ are defined with respect to the rank-1 family of tri-tiles $\BHT$. The key geometric observation in this context is that -- loosely speaking -- (pairwise) overlapping of a collection of tiles $\{ P_1 \}_{P \in \P}$ implies (pairwise) disjointness for each of the associated collections $\{ P_2 \}_{P \in \P}$ and $\{ P_3 \}_{P \in \P}$. This insight naturally leads to the notion of tree -- an object of lower complexity (visualized as a multiplier singular at a point) that can be easily controlled. The strategy introduced in \cite{lt1}, \cite{lt2} consists in a greedy algorithm that selects ``orthogonal'' trees according to how much information they store (the size of the tree) and to certain geometric conditions. Next, one needs to put together this information, for which estimating the spatial supports of the trees associated to a fixed size becomes important (this leads to the notion of energy associated to a collection of tri-tiles).

In order to take advantage of the orthogonality among the tree structures, it becomes convenient to assume a suitable sparseness within the collection $\BHT$: a separation of the scales (whenever $P, P' \in \BHT$ are so that  $|I_{P'}|< |I_P|$, we should have $|I_{P'}|\ll|I_P|$) and a separation of the Whitney cubes from \eqref{eq:Whitney} at a given scale. This sparseness, which will be implicit in the present section, is achieved by splitting $\BHT$ into $O(1)$ similar collections. Moreover, standard limiting arguments allow us to assume that $\BHT$ is finite, and hence also that the collection of trees is finite.

\begin{definition} \label{def:tree}
For $1 \leq j \leq 3$, we call a collection of tri-tiles $\mb T \subset \BHT$ a $j$-lacunary \emph{tree} with top data $(I_{\mb T}, (\xi_{\mb T_1}, \xi_{\mb T_2}, \xi_{\mb T_3}))$ if
\[
\forall \,  P \in \mb T, \qquad \xi_{\mb T_j} \in 3\,
 \pmb{\omega}_{P_j} \setminus \pmb{\omega}_{P_j} \quad \text{and} \quad I_{P} \subseteq I_{\mb T_j}.
\]
In contrast, a collection of tri-tiles $\mb T \subset \BHT$ is called a $j$-overlapping \emph{tree} with top data $(I_{\mb T}, (\xi_{\mb T_1}, \xi_{\mb T_2}, \xi_{\mb T_3}))$ if
\[
\forall \,  P \in \mb T, \qquad \xi_{\mb T_j} \in  \pmb{\omega}_{P_j} \quad \text{and} \quad I_{P} \subseteq I_{\mb T_j}.
\]

We call \emph{tree} with top data $(I_{\mb T}, (\xi_{\mb T_1}, \xi_{\mb T_2}, \xi_{\mb T_3}))$ any collection $\mb T \subset \BHT$ which is a $j$-lacunary or $j$-overlapping tree for some $1 \leq j \leq 3$.
\end{definition}

The Whitney property of the collection $\{ \pmb{\omega}_{P_1} \times  \pmb{\omega}_{P_2} \}_{P \in \BHT}$ (see \eqref{eq:Whitney}) ensures that a $j$-overlapping tree is lacunary in the remaining two directions. So we can assume that the tree top information $(\xi_{\mb T_1}, \xi_{\mb T_2}, \xi_{\mb T_3})$ is located in the set $\{ (\xi, \xi, 2 \xi) : \xi \in \R \}$.

\begin{definition}
Two $j$-lacunary trees $\mb T$ and $\mb T'$ are said to be \emph{strongly disjoint} if either $\ds 3 \pmb{\omega}_{P_j} \cap 3 \pmb{\omega}_{P'_j} =\emptyset$ for all $P \in \mb T, P' \in \mb T'$ or
\begin{equation}
\label{def:strong:disj}
 \Big[ 3 \pmb{\omega}_{P_j} \cap 3 \pmb{\omega}_{P'_j} \neq \emptyset \textrm{  and  } |\pmb{\omega}_P|\ll |\pmb{\omega}_{P'}| \Big] \text{      implies      }  I_{P'_j} \cap 3\, I_{\mb T}=\emptyset.
\end{equation}
\end{definition}

\begin{remark}
We emphasize that our notion of strong disjointness is not the standard one from \cite{MTTBiest2} that reads as
\begin{equation}
\label{def:strong:disj:standard}
\Big[ 3 \pmb{\omega}_{P_j} \cap 3 \pmb{\omega}_{P'_j} \neq \emptyset \textrm{  and  } |\pmb{\omega}_P|\ll |\pmb{\omega}_{P'}| \Big] \text{      implies      }  I_{P'_j} \cap \, I_{\mb T}=\emptyset.
\end{equation}
However, \eqref{def:strong:disj} and \eqref{def:strong:disj:standard} are achieved in a similar fashion -- see also Remark \ref{remark:strong:disjointness}.
\end{remark}

The extra separation of the trees required in our definition will be helpful in the proof of the energy estimate \eqref{eq:est:energy-f-g}.

Now, following \cite{MTTBiest2}, we can introduce the concepts of \emph{size} and \emph{energy}:

\begin{definition}[Size] For $\P \subseteq \BHT$ and $(a_{P_j})_{P \in \P}$ a sequence of complex numbers, we define for $j=1,2,3$
\begin{equation}
\label{def:size}
 \size^{\P}_j((a_{P_j})_{P \in \P}):= \sup_{\substack{\mb T \subset \P \\ \mb T \, j-\text{lacunary tree}}} \Big(\frac{1}{|I_{T}|} \sum_{P\in T} |a_{P_j}|^2\Big)^{\frac{1}{2}}.
\end{equation}
\end{definition}

\begin{definition}[Energy] For $\P$ a subcollection of $\BHT$ and $(a_{P_j})_{P \in \P}$ a sequence of complex numbers, we define for $j=1,2,3$
\begin{equation}
\label{def:energy}
 \energy^{\P}_j((a_{P_j})_{P \in \P}):= \sup_{d\in \mathbb{Z}} \ \sup_{\ii F} 2^{d} \Big( \sum_{\bf T \in \ii F} |I_{\mb T}| \Big)^{1/2}
\end{equation}
where $\ii F$ ranges over all collections of strongly disjoint $j$-lacunary trees  $\mb T \subset \P$ with the property that
$$ \Big( \frac{1}{|I_{\mb T}|} \sum_{P\in\mb T} |a_{P_j}|^2 \Big)^{1/2} > 2^{d}$$
and for any subtree $\mb T'\subset \mb T$
$$ \Big( \frac{1}{|I_{\mb T'}|} \sum_{P\in\mb T'} |a_{P_j}|^2 \Big)^{1/2} \leq 2^{d+1}.$$
\end{definition}

With the above notations and definitions we have the following general estimate:

\begin{proposition}[Similar to Proposition 6.5 of \cite{MTTBiest2}] \label{prop:size/energy}
Let $\P \subset \BHT$  be a finite collection of tri-tiles and $(a_{P_j})_{P \in \P}$ a sequence of complex numbers. Then for any $\theta_1,\theta_2,\theta_3\in[0,1)$ with $\theta_1+\theta_2+\theta_3=1$
\begin{align}
\label{eq:size:en:org}
\Big|\sum_{P \in \P} \frac{1}{|I_P|^{\frac{1}{2}}} a_{P_1} a_{P_2} a_{P_3}  \Big| \lesssim  \prod_{j=1}^3 \big(   \size^{\P}_j((a_{P_j})_{P \in \P}) \big)^{\theta_j} \big(   \energy^{\P}_j((a_{P_j})_{P \in \P}) \big)^{1-\theta_j}.
\end{align}
\end{proposition}

\begin{remark}
\label{remark:strong:disjointness}
Only small modifications are needed in order for the proof of \cite[Proposition 6.5]{MTTBiest2} to extend to the present context which requires the more rigid strong disjointness property \eqref{def:strong:disj}.

The reader familiar with \cite{MTTBiest2} or \cite{multilinear_harmonicII} will recognize that the strong disjointness property \eqref{def:strong:disj:standard} is intimately related to the greedy algorithm mentioned in the beginning of the section. Indeed, given $\P \subset \BHT$ a collection of tri-tiles such that  $\size^{\P}_j((a_{P_j})_{P \in \P}) \leq 2^d$, one starts by selecting maximal trees $\mb T \subset \P$ with $\size^{\mb T}_j((a_{P_j})_{P \in \P}) > 2^{d-1}$ satisfying also a certain geometric condition; in order to ensure that \eqref{def:strong:disj:standard} holds for $j$-lacunary trees (interacting with $\mb T$) that will be selected later, one discards also ${\mb T}^\circ$, the $j$-overlapping tree having the same top data $(I_{\mb T}, \xi_{\mb T})$ as $\mb T$.

In our situation, once the $j$-lacunary tree $\mb T$ is selected by the algorithm, we discard not only ${\mb T}^\circ$ as above, but also the $j$-overlapping trees ${\mb T}^+$ having top data $(I_{\mb T}+|I_{\mb T}|, \xi_{\mb T})$ and ${\mb T}^-$ having top data $(I_{\mb T}-|I_{\mb T}|, \xi_{\mb T})$ (i.e. the spatial intervals of ${\mb T}^+$ and ${\mb T}^-$ consist in the right/left neighbour of $I_{\mb T}$). This, in short, guarantees that \eqref{def:strong:disj} is satisfied within the greedy algorithm used for proving \eqref{eq:size:en:org}.
\end{remark}

Proposition \ref{prop:size/energy} above will be applied directly to the study of $\Lambda^{*}_{\BHT}(f, g, h)$ for $*\in\{\mathfrak{l},\mathfrak{c},\mathfrak{u}\}$ with
\[
a_{P_1}:= f(P_1), \qquad a_{P_2}:=g(P_2) \qquad \textrm{and} \qquad a_{P_3}:= h_{*}(P_3),
\]
thus providing the control of these trilinear forms in terms of sizes and energies of $f,g,h$.

In the next subsections, we will obtain bounds for such sizes and energies, which will then allow us to conclude the desired estimates in Section \ref{subsec:proof:localL^2}.

\subsection{Estimates for sizes}
\label{subsec:energy:size}

First we address the size estimates, which will be naturally\footnote{For standard tiles of area 1, the sizes emerge as $L^1$ quantities, i.e. the maximal $L^2$ averages in the right-hand side of \eqref{eq:est:size-f-g} can be replaced by maximal $L^1$ averages.} bounded by $L^2$ maximal averages, associated to trees. Since the tree top data $I_{\mb T}$ is not intrinsically associated to tiles in $\BHT$ whilst the information associated to $\mb T$ is spatially adapted to $I_{\mb T}$, we define for any $\P$ finite subcollection of $\BHT$ the collection $\mathcal{I}_{\bar{\P}}$ of spatial intervals
\begin{equation}
\label{eq:def:spatial:tiles}
\mathcal{I}_{\bar{\P}}:=\{ I \text{  dyadic  }: I=I_P \text{  for some  }P \in \P \text{  or  } I=I_{\mb T} \text{ for some tree  } \mb T \subset \P  \}.
\end{equation}

\begin{proposition}[Size estimates]
\label{prop:size:est}
Let $\delta_1, \mu>0$ be as obtained in Section \ref{sec:onescale} and as used in Definition \ref{def:indices} when introducing the indices sets $\mathfrak{l},\mathfrak{u},\mathfrak{c}$ corresponding to each tile $P \in \BHT$. \\
For any locally $L^2$-integrable functions $f, g,h$ and any subcollection $\P\subset \BHT$, we have
\begin{equation}
\begin{aligned}
 \size^{\P}_1((f(P_1))_{P\in\P}) &\lesssim \sup_{I  \in \mathcal{I}_{\bar{\P}}} \frac{1}{|I|^{\frac{1}{2}}} \|f \cdot \ci_{I}\|_{L^2}, \\
\size^{\P}_2((g(P_2))_{P\in\P}) &\lesssim \sup_{I \in \mathcal{I}_{\bar{\P}}} \frac{1}{|I|^{\frac{1}{2}}} \|g \cdot \ci_{I}\|_{L^2},
\end{aligned} \label{eq:est:size-f-g}
\end{equation}
and
\begingroup
\allowdisplaybreaks
\begin{alignat}{3}
\size^{\P}_3((h_\mathfrak{l}(P_3))_{P\in\P}) &\lesssim 2^{- \frac{\delta_1 a m}{2}} && \sup_{I \in \mathcal{I}_{\bar{\P}}} \frac{1}{|I|^{\frac{1}{2}}} \|h \cdot \ci_{I}\|_{L^2}, \label{eq:est:size-h-} \\
\size^{\P}_3((h_\mathfrak{u}(P_3))_{P\in\P}) &\lesssim &&\sup_{P \in \mathcal{I}_{\bar{\P}}} \frac{1}{|I|^{\frac{1}{2}}} \|h \cdot \ci_{I}\|_{L^2}, \label{eq:est:size-h+} \\
\size^{\P}_3((h_\mathfrak{c}(P_3))_{P\in\P}) &\lesssim 2^{- \frac{\mu a m}{4}} &&\sup_{P \in \mathcal{I}_{\bar{\P}}} \frac{1}{|I|^{\frac{1}{2}}} \|h \cdot \ci_{I}\|_{L^2}. \label{eq:est:size-h:cluster}
\end{alignat}
\endgroup
\end{proposition}

\bigskip
In the above proposition, only estimates \eqref{eq:est:size-h-}-\eqref{eq:est:size-h:cluster} require some work; the $L^2$-based \eqref{eq:est:size-f-g} is standard and follows from the almost orthogonality of the wave-packets. For completeness, we treat all of them.

\begin{proof}[Proof of estimate \eqref{eq:est:size-f-g}]
We let $\mb T$ be a $1$-lacunary tree and we aim to prove that
\begin{align*}
\sum_{P \in \mb T} \sum_{\hat s =P} \big| \langle f, \phi_{s_1}  \rangle   \big|^2 \lesssim \|f \cdot \ci_{I_{\mb T}}\|_{L^2}^2.
\end{align*}
When $f$ is supported close to $I_{\mb T}$, this is a simple consequence of orthogonality; on the contrary, for $f$ supported away from $I_{\mb T}$ we rely on the fast decay of the wave-packets.

Assume that $f$ is supported on $3I_{\mb T}$; then it suffices to prove
\begin{align*}
\sum_{P \in \mb T} \sum_{\hat s=P} \big| \langle f, \phi_{s_1}  \rangle   \big|^2 \lesssim \|f\|_{L^2}^2.
\end{align*}

This follows from the disjointness of the collection of tiles $\ds \{ s_1 \}_{\hat s \in \mb T}$ (a weak consequence of the fact that $\mb T$ is $1$-lacunary) and from the almost orthogonality of the associated collection of wave-packets $\ds \{ \phi_{s_1} \}_{\hat s \in \mb T}$.

Next we assume that $f$ is supported in $I_{\mb T}^\nu:=I_{\mb T}+\nu|I_{\mb T}|$, for $|\nu| \geq 2$ an integer, and we want to prove
\begin{align}
\label{eq:decay:size}
\sum_{P \in \mb T} \sum_{\hat s=P} \big| \langle f, \phi_{s_1}  \rangle   \big|^2 \lesssim  |\nu|^{-10} \|f\|_{L^2}^2.
\end{align}

First, we notice that the left-hand side of the estimate above writes as
\begin{align*}
\sum_{\kappa \geq 0} \sum_{ \substack{\hat s \in \mb T \\ |I_{s_1}|= 2^{-{am \over 2}} 2^{-\kappa} |I_{\mb T}|}} \big| \langle f, \phi_{s_1}  \rangle   \big|^2,
\end{align*}
and also
\begin{align*}
\big| \langle f, \phi_{s_1}  \rangle   \big|^2 & \lesssim \|f\|_{L^2}^2 \, \int_{I_{\mb T}^\nu} \frac{1}{|I_{s_1}|}  \Big( 1+ \frac{\dist(I_{s_1}, x)}{|I_{s_1}|} \Big)^{-20} dx \\
&\lesssim  \|f\|_{L^2}^2  \int_{I_{\mb T}^\nu} \frac{1}{|I_{s_1}|}  \Big( \frac{\dist(I_{s_1}, x)}{ 2^{-{am \over 2}} 2^{-\kappa} |I_{\mb T}|} \Big)^{-20} dx \\
&\lesssim \|f\|_{L^2}^2 \:  \frac{2^{am \over 2} 2^{\kappa}}{|I_{\mb T}|}  \int_{I_{\mb T}^\nu}  2^{-10 am } 2^{- 20 \kappa}   \big( \frac{\dist(I_{\mb T}, x)}{|I_{\mb T}|} \big)^{-20} dx\\
& \lesssim \|f\|_{L^2}^2 \, 2^{-9 am } 2^{- 10 \kappa} |\nu|^{-10} .
\end{align*}

For a fixed scale $\kappa \geq 0$, there are at most $2^{\kappa} \cdot 2^{am}$ tiles $s$ so that $\hat s \in \mb T$, with $|I_{s_1}|= 2^{-{am \over 2}} 2^{-\kappa} |I_{\mb T}|$. This rough estimate is enough for concluding \eqref{eq:decay:size}; by standard arguments, \eqref{eq:est:size-f-g} follows.
\end{proof}

Notice that the estimates for the light-mass and the heavy-mass/clustered sizes $\size^{\P}_3(h_\mathfrak{l}(P_3))$ and $\size^{\P}_3(h_\mathfrak{c}(P_3))$ feature additional decay in the parameter $m$. First we will prove estimate \eqref{eq:est:size-h+}, for which we will not make use of the fact that we are summing over triples $(p', n, v) \in \mathfrak{u}$. Next, we prove \eqref{eq:est:size-h:cluster}, in which the additional decay is a result of the fact that for any fixed $p' \sim 2^{am}$, only a small fraction of the allowed values of $n$ is such that $(p', n, v) \in \mathfrak{c}$ (see Remark \ref{remark:boundedness:cluster:number}). Lastly, for \eqref{eq:est:size-h-}, the decay comes from the assumption that the indices $(p', n, v) \in \mathfrak{l}$ are associated to light averages of the weight $w_{|P|, n, v}(\lambda)$.

\bigskip

\begin{proof}[Proof of estimate \eqref{eq:est:size-h+}]
Let $\mb T$ be a $3$-lacunary tree; it will be enough to prove that
\begin{align*}
\sum_{P \in \mb T} h_\mathfrak{u}(P_3)^2 \lesssim \| h \cdot \ci_{I _{\mb T}}   \|_{L^2}^2.
\end{align*}
By definition \eqref{eq:def:h*(P):tile}, this amounts to proving
\begin{align*}
\frac{1}{2^{am \over 2}} \sum_{P \in \mb T} \sum_{(p', n, v) \in \mathfrak{u}}  \big| \langle h \, \rho_{|P|}(\lambda(\cdot)) w_{|P|, n, v}^e(\l) , \Psi_{P_3}^{p'}  \rangle \big|^2  \lesssim \| h \cdot \ci_{I _{\mb T}}   \|_{L^2}^2.
\end{align*}
We will see that orthogonality presents itself in disjointness across different scales, because of the presence of the stopping time $\rho_{|P|}(\lambda(\cdot))$. First, by using Cauchy-Schwarz and properties \ref{WP1}-\ref{WP2}, one has
\begin{equation}
\begin{aligned}
& \big| \langle h \, \rho_{|P|}(\lambda(\cdot)) w_{|P|, n, v}^e(\l), \Psi_{P_3}^{p'}  \rangle \big|^2 \\
& \lesssim \Big(\int |h(x)|^2 \, \rho_{|P|}(\lambda(x))^2 |w_{|P|, n, v}(\l)(x)|^2 |\Psi_{P_3}^{p'}(x)| \, dx\Big) \cdot \Big( \int |\Psi_{P_3}^{p'}(x)|\, dx \Big) \\
& \lesssim \Big(\int |h(x)|^2 \, \rho_{|P|}(\lambda(x))^2 |w_{|P|, n, v}(\l)(x)|^2 |\underline{I}_{P}^{p'}|^{-{1 \over 2}} \ci_{\underline{I}_{P}^{p'}}(x) \, dx\Big) |\underline{I}_{P}^{p'}|^{1 \over 2}
\end{aligned}\label{eq:est:wave:h:simple}
\end{equation}
and our desired estimate reduces to
\begin{equation}
\label{eq:h:size:square}
\sum_{P \in \mb T} \sum_{(p', n, v) \in \mathfrak{u}} \frac{1}{2^{am \over 2}} \|h \, \rho_{|P|}(\lambda) w_{|P|, n, v}(\l) \|_{L^2( \ci_{\underline{I}_{P}^{p'}})}^2  \lesssim \| h \cdot \ci_{I _{\mb T}}   \|_{L^2}^2.
\end{equation}
Using \eqref{eq:decay-weight} to bound $w_{|P|, n, v}(\l)$ and extending the summation to all admissible triples $(p',n,v)$ we have
\begin{align*}
& \qquad \sum_{P \in \mb T} h_\mathfrak{u}(P_3)^2 \\
& \lesssim \sum_{k \in \Z} \sum_{\substack{P \in \mb T : \\ |I_P|=2^k}} \sum_{(n, v, p') \in \mathfrak{u}} \frac{1}{2^{am \over 2}}  \|  h \, \rho_{am-ak}(\lambda(\cdot)) w_{k, n, v}(\l) \|_{L^2(\ci_{\underline{I}_{P}^{p'}})}^2 \\
&\lesssim \sum_{k \in \Z} \sum_{n \sim 2^{am \over 2}} \frac{1}{2^{am \over 2}} \int_{\R} |h(x)|^2 \Big( \sum_{\substack{ v \sim 2^{am \over 2} \\ p' \sim 2^{am}}}   \sum_{\substack{P \in \mb T \\ |I_P|=2^k}}   \frac{\rho_{am-ak}(\lambda(x))}{ 1+  \Big|\bar c_a\frac{n^{a-1}  \l(x)}{2^{a (\frac{am}{2}-k)}} +2v \Big|^4} \cdot \ci_{\underline{I}_{P}^{p'}}(x)   \Big) dx.
\end{align*}
We claim that the following pointwise bound holds:
\begin{equation}
\label{eq:summation:h}
 \sum_{k \in \Z}  \sum_{\substack{P \in \mb T : \\ |I_P|=2^k}} \sum_{\substack{ v \sim 2^{am \over 2} \\ p' \sim 2^{am}}}  \frac{\rho_{am-ak}(\lambda(x))}{ 1+  \Big|\bar c_a\frac{n^{a-1}  \l(x)}{2^{a (\frac{am}{2}-k)}} +2v \Big|^4} \cdot \ci_{\underline{I}_{P}^{p'}}(x)   \lesssim \ci_{I_{\mb T}}(x).
\end{equation}
If the above is true then this will be enough to conclude \eqref{eq:h:size:square}, since the left/right-hand side in \eqref{eq:summation:h} is independent of $n$ (notice that the sum over $n \sim 2^{am/2}$ is compensated by the factor of $2^{-am/2}$).

To see the claim, first notice that after summing in $v$ the left-hand side of \eqref{eq:summation:h} is bounded above by
\begin{align*}
 & \sum_{k \in \Z}  \sum_{\substack{P \in \mb T : \\ |I_P|=2^k}}  \sum_{ p' \sim 2^{am}}  \rho_{am-ak}(\lambda(x)) \ci_{\underline{I}_{P}^{p'}}(x) =  \sum_{k \in \Z} \rho_{am-ak}(\lambda(x))\sum_{\substack{P \in \mb T \\ |I_P|=2^k}}  \ci_{I_P}(x).
\end{align*}

For every $x \in \R$, there are only $O(1)$ many $k \in \Z$ for which $ \rho_{am-ak}(\lambda(x)) \neq 0$, and for every such $k$
\[
\sum_{\substack{P \in \mb T : \\ |I_P|=2^k}}  \ci_{I_P}(x) \lesssim \ci_{I_{ \mb T}}(x),
\]
which allows us to conclude \eqref{eq:summation:h}.
\end{proof}

\bigskip

\begin{proof}[Proof of estimate \eqref{eq:est:size-h:cluster}]
The overall argument will be similar to the previous one. We start again with $\mb T$, a $3$-lacunary tree, and observe that it will be enough to prove that
\begin{align*}
\sum_{P \in \mb T} h_\mathfrak{c}(P_3)^2 \lesssim 2^{- \mu {am \over 2}} \| h \cdot \ci_{I _{\mb T}}   \|_{L^2}^2.
\end{align*}
As before, we have
\begin{align*}
& \qquad \sum_{P \in \mb T} h_\mathfrak{c}(P_3)^2 \\
& \lesssim  \sum_{k \in \Z} \sum_{\substack{P \in \mb T : \\ |I_P|=2^k}} \frac{1}{2^{am \over 2}}  \int_{\R} |h(x)|^2 \Big( \sum_{(p', n, v) \in \mathfrak{c}}    \frac{\rho_{am-ak}(\lambda(x))}{ 1+  \Big|\bar c_a\frac{n^{a-1}  \l(x)}{2^{a (\frac{am}{2}-k)}} +2v \Big|^4} \cdot \ci_{\underline{I}_{P}^{p'}}(x)  \Big) dx.
\end{align*}

The desired estimate will be a consequence of the pointwise bound
\begin{equation}
\label{eq:aim:cluster:size}
\begin{aligned}
 \sum_{k \in \Z} \sum_{\substack{P \in \mb T : \\ |I_P|=2^k}} \frac{1}{2^{am \over 2}} \sum_{(p', n, v) \in \mathfrak{c}}    \frac{\rho_{am-ak}(\lambda(x))}{ 1+  \Big|\bar c_a\frac{n^{a-1}  \l(x)}{2^{a (\frac{am}{2}-k)}} +2v \Big|^4} & \ci_{\underline{I}_{P}^{p'}}(x)  \\
 &\lesssim 2^{- \mu \frac{am }{2}} \cdot \ci_{I _{\mb T}}(x),
\end{aligned}
\end{equation}
which in turn follows from
\begin{equation}
\label{eq:summation:cluster}
\frac{1}{2^{am \over 2}}  \sum_{(p', n, v) \in \mathfrak{c}} \ci_{\underline{I}_{P}^{p'}}(x)  \frac{1}{ 1+  \Big|\bar c_a \frac{n^{a-1}  \l(x)}{2^{a (\frac{am}{2}-k)}} +2v \Big|^4}  \lesssim 2^{- \mu \frac{am }{2}}  \ci_{I_P}(x),
\end{equation}
for any fixed $k \in \Z$ and $P \in \mb T$.

In Remark \ref{remark:boundedness:cluster:number}, we introduced a larger set $\bar{\mathfrak{c}}=\bar{\mathfrak{c}}_1 \times \{ v \sim 2^{am \over 2} \}$ containing  $\mathfrak{c}$ which allowed us to separate the $n$ and $v$ parameters; thus the left-hand side in the expression above is bounded by
\begin{equation}
\label{eq:summation:cluster3}
\frac{1}{2^{am \over 2}}  \sum_{(p', n) \in \bar{\mathfrak{c}}_1} \sum_{v \sim 2^{am \over 2}} \ci_{\underline{I}_{P}^{p'}}(x)  \frac{1}{ 1+  \Big|\bar c_a \frac{n^{a-1}  \l(x)}{2^{a (\frac{am}{2}-k)}} +2v \Big|^4}.
\end{equation}
With $p'$ and $n$ fixed, we first sum in $v$ to get 
\begin{equation}
\label{eq:summ:v}
\sum_{v \sim 2^{am \over 2}} \frac{1}{ 1+  \Big|\bar c_a \frac{n^{a-1}  \l(x)}{2^{a (\frac{am}{2}-k)}} +2v \Big|^4} \lesssim 1.
\end{equation}
Now we can bound \eqref{eq:summation:cluster3} by
\[
\frac{1}{2^{am \over 2}}  \sum_{(p', n) \in \bar{\mathfrak{c}}_1} \ci_{\underline{I}_{P}^{p'}}(x) \leq  \frac{1}{2^{am \over 2}}  \sum_{p' \sim 2^{am}} \sum_{n: (p', n) \in \bar{\mathfrak{c}}_1}  \ci_{\underline{I}_{P}^{p'}}(x).
\]
We recall the observation in Remark \ref{remark:boundedness:cluster:number}: once $p' \sim 2^{am}$ is fixed, there are at most $O(2^{(1-\mu) \frac{am}{2}})$ values of $n$ so that $(p', n) \in \bar{\mathfrak{c}}_1$. That produces as upper bound for \eqref{eq:summation:cluster3}
\[
 \frac{1}{2^{am \over 2}}  \sum_{p' \sim 2^{am}}  2^{(1-\mu) \frac{am}{2}} \ci_{\underline{I}_{P}^{p'}}(x),
\]
which is of course bounded by the right-hand side of \eqref{eq:summation:cluster}.

Estimate \eqref{eq:aim:cluster:size} follows immediately from \eqref{eq:summation:cluster} and the straightforward pointwise estimate
\begin{equation}
\label{eq:sum:rho:k}
\sum_{k \in \Z} \sum_{\substack{P \in \mb T : \\ |I_P|=2^k}} \rho_{am-ak}(\lambda(x)) \ci_{I_P}(x) \lesssim \ci_{I _{\mb T}}(x).
\end{equation}
\end{proof}

\bigskip

\begin{proof}[Proof of estimate \eqref{eq:est:size-h-}] In this situation, taking $\mb T$ a $3$-lacunary tree, our task is to show that
\begin{equation}
\label{eq:aim:light:size}
\sum_{P \in \mb T} h_\mathfrak{l}(P_3)^2 \lesssim 2^{- \delta_1 a m} \| h \cdot \ci_{I _{\mb T}}   \|_{L^2}^2.
\end{equation}
We will use the fact that for a fixed tri-tile $P$, we are summing over the indices $(p', n, v) \in \mathfrak{l}$ in the ``light-mass regime''. Indeed, by Cauchy-Schwarz and \ref{WP3} we can bound
 \begin{align*}
\big| \langle h \, \rho_{|P|}(\lambda) w_{|P|, n, v}^e(\l), \Psi_{P_3}^{p'}  \rangle \big|^2 &\lesssim \big\|h^2 \rho_{|P|}(\lambda)^2 w_{|P|, n, v} |\Psi_{P_3}^{p'}| \big\|_{L^1} \; \big\|w_{k,n,v} |\Psi_{P_3}^{p'}| \big\|_{L^1} \\
&  \lesssim   2^{- \delta_1 a m} \big\|h^2 \rho_{|P|}(\lambda)^2 w_{|P|, n, v} \ci_{\underline{I}_{P}^{p'}} \big\|_{L^1}
 \end{align*}
(this is the natural ``weighted'' replacement of \eqref{eq:est:wave:h:simple}). For any $P \in \mb T$ we then have
\begin{align*}
 & \qquad  h_{\mathfrak{l}}(P_3)^2 \\
 & \lesssim 2^{-\delta_1 a m} \frac{1}{2^{\frac{a}{2}m}} \sum_{(p',n,v)\in \mathfrak{l}}   \int_{\R} |h(x)|^2 \rho_{|P|}(\lambda(x))^2 w_{|P|,n,v} (\lambda) (x) \ci_{\underline{I}_{P}^{p'}}(x) dx\\
 & \lesssim 2^{-\delta_1 a m}  \int_{\R} |h(x)|^2 \rho_{|P|}(\lambda(x))^2  \ci_{I_P}(x) dx .
\end{align*}
Estimate \eqref{eq:aim:light:size} follows now by essentially the same reasoning used for \eqref{eq:h:size:square} since the information $(p',n,v)\in \mathfrak{l}$ is no longer needed.
\end{proof}

\subsection{Estimates for energies}
\label{subsec:energy:only}

Now we turn towards the energies, which are naturally $L^2$-based quantities as they capture the almost orthogonality of strongly disjoint trees associated to a fixed size level.

\begin{proposition}[Energy estimates]
\label{prop:energies} For any $L^2$ functions $f,g,h$ and any subcollection $\P\subset \BHT$, we have
\begin{equation}
\label{eq:est:energy-f-g}
\begin{aligned} \energy^{\P}_1((f(P_1))_{P\in\P}) &\lesssim \|f\|_{L^2}, \\
 \energy^{\P}_2((g(P_2))_{P\in\P}) &\lesssim \|g\|_{L^2}
 \end{aligned}
\end{equation}
and for $\ast \in \{ \mathfrak{l}, \mathfrak{u}, \mathfrak{c} \}$
\begin{equation}
\label{eq:est:energy-h:pm}
\energy^{\P}_3((h_{\ast}(P_3))_{P\in\P}) \lesssim \|h\|_{L^2}.
\end{equation}
\end{proposition}

\begin{remark}
We point out that in establishing the energy estimates \eqref{eq:est:energy-h:pm} we will not make use of the partitioning of the triples $(p', n, v)$ into subsets $\mathfrak{l}, \mathfrak{u}, \mathfrak{c}$. This partition was only important for the size estimates of Proposition \ref{prop:size:est}.
\end{remark}

\begin{proof}[Proof of estimate \eqref{eq:est:energy-h:pm}]  Let $\ast$ stand for any of the $\mathfrak{l}, \mathfrak{u}$ or $\mathfrak{c}$ and let $\ii F$ be a collection of mutually disjoint $3$-lacunary trees satisfying for some integer $d$
\[
2^{2d} < \frac{1}{|I_{\mb T}|} \sum_{P \in \mb T} h_{*}(P_3)^2 \qquad \text{for all    } \mb T \in \ii F.
\]
It is then enough to show that
\begin{equation}
\label{eq:Bessel}
\sum_{\mb T \in \ii F} \sum_{P \in \mb T} h_{*}(P_3)^2 \lesssim \| h \|_{L^2}^2,
\end{equation}
which will in turn be a direct consequence\footnote{Recall that the factor $2^{-am/2}$ in the definition of $h_{\ast}(P_3)$ allows one to safely sum in $n \sim 2^{am/2}$.} of the estimate, uniform in the parameter $n \sim 2^{am \over 2}$,
\begin{align*}
\sum_{\mb T \in \ii F} \sum_{P \in \mb T}  \sum_{\substack{ v \sim 2^{am \over 2} \\ p' \sim 2^{am}}} \big| \langle h \, \rho_{|P|}(\lambda(\cdot)) w_{|P|, n, v}^e(\l), \Psi_{P_3}^{p'}  \rangle \big|^2  \lesssim \|h\|_{L^2}^2.
\end{align*}

There are two key observations to make:
\begin{enumerate}[label=(\roman*)]
\item  Given $I$ a dyadic interval, we denote by $\pmb {\omega}_{\ii F} (I)$ the collection of frequency intervals $\pmb{\omega}$ for which $I \times \pmb{\omega}$ corresponds to a $P_3$ tile:
\[
\pmb {\omega}_{\ii F} (I):= \{  \pmb{\omega} : |I||\pmb{\omega}|=2^{am},  I \times \pmb{\omega} \in \bigcup_{\mb T \in \ii F} \bigcup_{P \in \mb T} \{P_3\} \}.
\]
Then notice that $\pmb {\omega}_{\ii F} (I)$ consists of mutually disjoint intervals of the same length.
\item $ w_{|P|,n,v}^e$ and $\underline{I}_{P}^{p'}$ do not depend on the frequency interval of the tile $P$ but only depend on the spatial interval $I_P$ and on the triple of parameters $(p', n, v)$.
\end{enumerate}

Hence, for any fixed interval $I$ with $|I|=2^k$ we can use property \ref{WP4} for the collection $\big \{\Psi_{I \times \pmb{\omega}}^{p'} \big\}_{\pmb{\omega}\in \pmb \omega_{\ii F}(I),p'\sim 2^{am}}$ and function $F:=h \, \rho_{I}(\lambda(\cdot)) w_{I, n, v}^e(\l)$  to deduce
\begin{align*}
    \sum_{\pmb{\omega} \in \pmb \omega_{\ii F}(I)} \ \sum_{\substack{ v \sim 2^{\frac{a}{2}m}, \\ p' \sim 2^{am}}} \big| \langle h \, \rho_{I}(\lambda(\cdot)) w_{I, n, v}^e(\l), \Psi_{I \times \pmb{\omega}}^{p'}  \rangle \big|^2  & \lesssim  \sum_{v \sim 2^{\frac{a}{2}m}} \| h \rho_{am-ak}(\lambda) w_{k,n,v} \ci_{I} \|_{L^2}^2.
\end{align*}
Estimate \eqref{eq:Bessel} follows once we notice that for every fixed $n$ (and uniformly so)
\[
\sum_k \ \sum_{\substack{I \in \ii D : \\ |I|=2^k}} \ \sum_{v \sim 2^{\frac{a}{2}m}}  \rho_{am-ak}(\lambda(x)) w_{k,n,v}(x)^2 \ci_{I}(x) \lesssim 1.
\]
This follows from \eqref{eq:summ:v} and an unlocalized version of \eqref{eq:sum:rho:k}.
\end{proof}

It remains to prove energy estimate \eqref{eq:est:energy-f-g}, which relies on more subtle orthogonality properties. Although similar to the classical argument for BHT (see \cite[Lemma 6.6]{MTTBiest2}), we will
\begin{itemize}
\item rely on the extra separation condition in \eqref{def:strong:disj};
\item use the smaller tri-tiles $s=(s_1, s_2, s_3)$ of area 1 (to which wave-packets are naturally adapted), in order to capture the orthogonality between the strongly disjoint trees through a $TT^*$ argument.
\end{itemize}

\begin{proof}[Proof of estimate \eqref{eq:est:energy-f-g}]
By symmetry, we consider only the energy for $f$. Fix an integer $d\in\mathbb{Z}$ and let $\ii F$ be a collection of strongly disjoint $1$-lacunary trees having the property that for any $\mb T \in \ii F$
\begin{equation}
\label{eq:Bess-lowBd}
2^d< \Big(\frac{1}{|I_{\mb T}|} \sum_{ \hat s \in \mb T} |\langle f, \phi_{s_1} \rangle|^2    \Big)^\frac{1}{2}
\end{equation}
and
\begin{equation}
\label{eq:Bess-upBd}
\Big(\frac{1}{|I_{\mb T'}|} \sum_{\hat s \in \mb T'} |\langle f, \phi_{s_1} \rangle|^2    \Big)^\frac{1}{2} \leq 2^{d+1} \quad \text{for any subtree   } \mb T'\subseteq \mb T.
\end{equation}
It will suffice to prove
\begin{equation}
\label{eq:Bess-conc}
 2^{2d} \sum_{\mb T \in \ii F} |I_{\mb T}| \lesssim \|f\|_{L^2}^2,
\end{equation}
which, based on \eqref{eq:Bess-lowBd} and a standard $TT^*$ argument, follows from
\begin{equation}
\label{eq:Bessel-orth}
\Big\|   \sum_{\mb T \in \ii F} \sum_{ \hat s \in \mb T} \langle f, \phi_{s_1}  \rangle   \phi_{s_1}   \Big\|_{L^2}^2 \lesssim 2^{2d}  \sum_{\mb T \in \ii F} |I_{\mb T}|.
\end{equation}
Expanding the square, it suffices to prove
\begin{equation}
\label{eq:en:suff}
\sum_{\mb T, \mb T' \in \ii F} \sum_{\hat s \in \mb T} \sum_{\hat s' \in \mb T'} \big| \langle f, \phi_{s_1}  \rangle \langle f, \phi_{s'_1}  \rangle \langle \phi_{s_1}, \phi_{s'_1}  \rangle \big| \lesssim 2^{2d}  \sum_{\mb T \in \ii F} |I_{\mb T}|.
\end{equation}

Now notice that in order for $\ds \langle \phi_{s_1}, \phi_{s'_1} \rangle $ to be non-zero it must be necessarily that $\omega_{s_1} \cap \omega_{s'_1} \neq \emptyset$ and hence one of them is included in the other.\footnote{This property is inherited from the parent collection $\{ \pmb{\omega}_{P_1}\}_{P \in \BHT}$, which forms a lattice.}

It is thus enough to prove the following two claims:
\begin{itemize}
\item Claim $1$:
\begin{equation}
\label{eq:claim1}
\sum_{\mb T \in \ii F} \sum_{s : \hat s \in \mb T} \sum_{\mb T' \in \ii F} \sum_{\substack{ \hat s' \in \mb T',  \\ \omega_{s_1}=\omega_{s'_1}  }} \big| \langle f, \phi_{s_1}  \rangle \langle f, \phi_{s'_1}  \rangle \langle \phi_{s_1}, \phi_{s'_1}  \rangle \big| \lesssim 2^{2d}  \sum_{\mb T \in \ii F} |I_{\mb T}|;
\end{equation}
\item Claim $2$:
\begin{equation}
\label{eq:claim2}
\sum_{\mb T \in \ii F} \sum_{\hat s \in \mb T} \sum_{\mb T' \in \ii F} \sum_{\substack{ \hat s' \in \mb T',  \\ \omega_{s_1} \subset \omega_{s'_1} \\  |\omega_{s_1}| \ll |\omega_{s'_1}|}} \big| \langle f, \phi_{s_1}  \rangle \langle f, \phi_{s'_1}  \rangle \langle \phi_{s_1}, \phi_{s'_1}  \rangle \big| \lesssim 2^{2d}  \sum_{\mb T \in \ii F} |I_{\mb T}|.
\end{equation}
\end{itemize}
We start with the proof of Claim 1. First, by symmetry,
\begin{align*}
&\sum_{\mb T \in \ii F} \sum_{\hat s \in \mb T} \sum_{\mb T' \in \ii F} \sum_{\substack{ \hat s' \in \mb T',  \\ \omega_{s_1}=\omega_{s'_1}  }} \big| \langle f, \phi_{s_1}  \rangle \langle f, \phi_{s'_1}  \rangle \langle \phi_{s_1}, \phi_{s'_1}  \rangle \big| \\
& \lesssim \sum_{\mb T \in \ii F} \sum_{\hat s \in \mb T} \sum_{\mb T' \in \ii F} \sum_{\substack{\hat  s' \in \mb T',  \\ \omega_{s_1}=\omega_{s'_1}  }} \big| \langle f, \phi_{s_1} \big \rangle|^2 \big|\langle \phi_{s_1}, \phi_{s'_1}  \rangle \big| \\
&\lesssim \sum_{\mb T \in \ii F} \sum_{\hat s \in \mb T}  \big| \langle f, \phi_{s_1} \big \rangle|^2  \sum_{\mb T' \in \ii F} \sum_{\substack{ \hat s' \in \mb T',  \\ \omega_{s_1}=\omega_{s'_1}  }} \Big( 1+ \frac{\dist(I_{s_1}, I_{s'_1})}{|I_{s_1}|}    \Big)^{-10}.
% \lesssim 2^{2d}  \sum_{\mb T \in \ii F} |I_{\mb T}|.
\end{align*}

Now, for $\omega_{s_1}$ fixed, the tiles $s'$ such that $\hat{s}' \in \cup_{\mb T'} \mb T'$ having the property that $\omega_{s_1}=\omega_{s'_1}$ are mutually disjoint and satisfy $|I_{s'_1}|=|I_{s_1}|$; hence
\begin{align*}
\sum_{\mb T' \in \ii F} \sum_{\substack{ \hat s' \in \mb T',  \\ \omega_{s_1}=\omega_{s'_1}  }} \Big( 1+ \frac{\dist(I_{s_1}, I_{s'_1})}{|I_{s_1}|}    \Big)^{-10} &\lesssim \sum_{\mb T' \in \ii F} \sum_{\substack{\hat s' \in \mb T',  \\ \omega_{s_1}=\omega_{s'_1}  }} \frac{1}{|I_{s_1}|} \int_{I_{s'_1}}  \Big( 1+ \frac{\dist(I_{s_1}, x)}{|I_{s_1}|}\Big)^{-10} dx \\
& \lesssim  \frac{1}{|I_{s_1}|} \int_{\rr R}  \Big( 1+ \frac{\dist(I_{s_1}, x)}{|I_{s_1}|}\Big)^{-10} dx \lesssim 1.
\end{align*}
In combination with \eqref{eq:Bess-upBd} this implies Claim 1, since we have
\begin{align*}
\sum_{\mb T \in \ii F} \sum_{\hat s \in \mb T} \sum_{\mb T' \in \ii F} \sum_{\substack{ \hat s' \in \mb T'  \\ \omega_{s_1}=\omega_{s'_1}  }} \big| \langle f, \phi_{s_1}  \rangle \langle f, \phi_{s'_1}  \rangle \langle \phi_{s_1}, \phi_{s'_1}  \rangle \big| & \lesssim \sum_{\mb T \in \ii F} \sum_{\hat s \in \mb T}  \big| \langle f, \phi_{s_1} \big \rangle|^2  \\
 &\lesssim 2^{2d}  \sum_{\mb T \in \ii F} |I_{\mb T}|.
\end{align*}

Let us then focus on the second claim, which will be more technical as interactions between tiles of different frequency scales will need to be counted in. First, note that the left-hand side of \eqref{eq:claim2} can be written as
\begin{align}
\label{eq:en:1}
\sum_{\mb T \in \ii F} \sum_{P \in \mb T} \sum_{\mb T' \in \ii F} \sum_{P' \in \mb T'} \sum_{\substack{\hat s=P}}  \sum_{\substack{ \hat s'= P' ,\\ \omega_{s_1} \subsetneq \omega_{s'_1}}}  \, \big| \langle f, \phi_{s_1}  \rangle \langle f, \phi_{s'_1}  \rangle \langle \phi_{s_1}, \phi_{s'_1}  \rangle \big|
\end{align}
and by Cauchy-Schwarz, we have
\begin{align}
\label{eq:en:01}
\sum_{\hat s =P}  \, \sum_{\sub{ \hat s'=P'}{\omega_{s_1} \subsetneq \omega_{s'_1}}} \big| \langle f, \phi_{s_1}  \rangle \langle f, \phi_{s'_1}  \rangle \langle \phi_{s_1}, \phi_{s'_1} \rangle|  \leq & \big( \sum_{\hat s =P}  | \langle f, \phi_{s_1}  \rangle |^2   \big)^\frac{1}{2} \, \big( \sum_{ \hat s' =P'}  | \langle f, \phi_{s'_1}  \rangle |^2   \big)^\frac{1}{2} \\
& \cdot \big(\sum_{\hat s =P}  \, \sum_{\sub{ \hat s'=P'}{\omega_{s_1} \subsetneq \omega_{s'_1}}}  |\langle \phi_{s_1}, \phi_{s'_1}  \rangle|^2   \big)^\frac{1}{2}.\nonumber
\end{align}

The subtree condition \eqref{eq:Bess-upBd} applied to the singleton subtrees $\{ P \} \subset \mb T$ and $\{ P' \} \subset \mb T'$ provides the following upper bound for  \eqref{eq:en:1}:
\begin{align*}
%\label{eq:en:1}
\sum_{\mb T \in \ii F} \sum_{P \in \mb T} \sum_{\mb T' \in \ii F} \sum_{P' \in \mb T'} 2^{2d+2} |I_P|^\frac{1}{2} |I_{P'}|^{\frac{1}{2}}  \Big(\sum_{\hat s =P}  \, \sum_{\sub{ \hat s'=P'}{\omega_{s_1} \subsetneq \omega_{s'_1}}}  |\langle \phi_{s_1}, \phi_{s'_1}  \rangle|^2   \Big)^\frac{1}{2}.
\end{align*}
The anticipated estimate \eqref{eq:en:suff} follows if we prove, for any tree $\mb T \in \ii F$, the inequality
\begin{equation}
\label{eq:en:suff:2}
\sum_{P \in \mb T} \sum_{\mb T' \in \ii F} \sum_{P' \in \mb T'} |I_P|^\frac{1}{2} |I_{P'}|^{\frac{1}{2}}  \Big(\sum_{\hat s =P}  \, \sum_{\sub{ \hat s'=P'}{\omega_{s_1} \subsetneq \omega_{s'_1}}}  |\langle \phi_{s_1}, \phi_{s'_1}  \rangle|^2   \Big)^\frac{1}{2} \lesssim |I_{\mb T}|.
\end{equation}

Before moving on, we take a closer look at the geometry of the tiles in $\BHT$.
\begin{enumerate}[label=(\roman*), leftmargin=*]
\item \label{geom1} if $P \in \mb T$ and $P' \in \mb T'$ with $\mb T, \mb T' \in \ii F$, and $s, s'$ are so that $\hat s =P, \hat s'=P'$ with $\omega_{s_1} \subsetneq \omega_{s'_1}$ then
\begin{itemize}
\item $\pmb{\omega}_{P_1} \subsetneq \pmb{\omega}_{P'_1}$ and  the sparseness of the collection $\BHT$ implies that  $|\pmb{\omega}_{P_1}| \ll  |\pmb{\omega}_{P'_1}|$ and $|I_P| \gg |I_{P'}|$;
\item the trees $\mb T$ and $\mb T'$ cannot coincide: due to their $1$-lacunarity, tiles of different scales belonging to the same tree cannot overlap in frequency;
\item since $\ii F$ consists of strongly disjoint $1$-lacunary trees, we must have
\begin{equation}
\label{eq:capit:disj}
I_{P'} \cap 3 I_{\mb T} = \emptyset.
\end{equation}
\end{itemize}
\item \label{geom2} notice that once $\mb T \in \ii F$ and $P \in \P$ are fixed, the spatial supports of the tiles $P'$ (contained in any of the trees $\mb T' \in \ii F$) having the property that $\pmb{\omega}_{P_1} \subsetneq \pmb{\omega}_{P'_1}$ and $|\pmb{\omega}_{P_1}| \ll |\pmb{\omega}_{P'_1}|$, are mutually disjoint -- otherwise such tiles would intersect both in space and in frequency and the strong disjointness property would be disobeyed.
\end{enumerate}

We return to understanding
\begin{equation}
\label{eq:error:corr}
 \Big(\sum_{\hat s =P}  \, \sum_{\sub{ \hat s'=P'}{\omega_{s_1} \subsetneq \omega_{s'_1}}}  |\langle \phi_{s_1}, \phi_{s'_1}  \rangle|^2   \Big)^\frac{1}{2},
\end{equation}
where $P \in \mb T, P' \in \mb T'$ are such that $\pmb{\omega}_{P_1} \subsetneq \pmb{\omega}_{P'_1}, |I_P| \gg |I_{P'}|$ and $I_{P'} \cap 3 I_P \neq \emptyset$. This quantity measures the correlation between the tiles $P$ and $P'$, as perceived\footnote{If the wave-packets $\phi_{s_1}$ were supported both in space and frequency on the tile $s_1$, then \eqref{eq:error:corr} would amount to $0$.} through the wave-packets $\ds \{ \phi_{s_1} \}_{\hat s=P}, \{ \phi_{s'_1}\}_{\hat s'=P'}$.

Our claim is that
\begin{equation}
\label{eq:en:alm:orth}
\Big(\sum_{\hat s =P}  \, \sum_{\sub{\hat s'=P'}{\omega_{s_1} \subsetneq \omega_{s'_1}}}  |\langle \phi_{s_1}, \phi_{s'_1}  \rangle|^2   \Big)^\frac{1}{2} \lesssim  \frac{|I_{P'}|^\frac{1}{2}}{|I_P|^{\frac{1}{2}}} \Big( 1+\frac{\dist(I_{P}, I_{P'})}{|I_{P}|}   \Big)^{-5}.
\end{equation}

Before proving this estimate, we show how it allows to conclude \eqref{eq:en:suff:2} and consequently \eqref{eq:claim2}. Using \eqref{eq:en:alm:orth}, the left-hand side of \eqref{eq:en:suff:2} is bounded by
\[
\sum_{P \in \mb T} \sum_{\mb T' \in \ii F} \sum_{\substack{P' \in \mb T' \\ \pmb{\omega}_{P_1} \subsetneq \pmb{\omega}_{P'_1} }} |I_{P'}| \Big( 1+\frac{\dist(I_{P}, I_{P'})}{|I_{P}|}   \Big)^{-5}.
\]
For $P \in \mb T$ fixed, we invoke \ref{geom2} above, to deduce
\begin{align*}
\sum_{P \in \mb T} \sum_{\mb T' \in \ii F} \sum_{\substack{P' \in \mb T' \\ \pmb{\omega}_{P_1} \subsetneq \pmb{\omega}_{P'_1} }} |I_{P'}| \Big( 1+\frac{\dist(I_{P}, I_{P'})}{|I_{P}|}   \Big)^{-5} &\lesssim \sum_{P \in \mb T} \int_{(3I_{\mb T})^C} \Big( 1+ \frac{\dist(I_P,x)}{|I_P|}  \Big)^{-5} dx  \\
&\lesssim  \sum_{P \in \mb T} |I_P|  \Big( 1+ \frac{\dist(I_P,(3I_{\mb T})^C)}{|I_P|} \Big)^{ -3} \\
&\lesssim |I_{\mb T}|.
\end{align*}

Now we return to proving \eqref{eq:en:alm:orth} for $P$ and $P'$ fixed, as above. This inequality will be itself a consequence of
\begin{align*}
\sum_{\hat I=I_P} \sum_{ \hat {I'}=I_{P'}} \sum_{\hat \omega=\pmb{\omega}_{P_1}} \sum_{\hat \omega'=\pmb{\omega}_{P'_1}} |\langle \phi_{I \times \omega}, \phi_{I'\times \omega'}  \rangle|^2  \lesssim  \frac{|I_{P'}|}{|I_P|} \Big( 1+\frac{\dist(I_{P}, I_{P'})}{|I_{P}|}   \Big)^{-10}.
\end{align*}
Here we will mostly rely on the spatial decay: given that $|I_P| \gg |I_{P'}|$, $|I|=2^{- {am \over 2}}|I_P|$ and $|I'|=2^{- {am \over 2}}|I_{P'}|$, we have
\[
|\langle \phi_{I \times \omega}, \phi_{I' \times \omega'}  \rangle|^2 \lesssim \frac{|I'|}{|I|} \Big( 1+\frac{\dist(I, I')}{|I|}   \Big)^{-50}= \frac{|I_{P'}|}{|I_P|} \Big( 1+\frac{\dist(I, I')}{|I|}   \Big)^{-50}.
\]
For fixed intervals $I$ and $I'$ with $\hat I=I_P$ and $\hat {I'}=I_{P'}$, there are precisely $2^{am \over 2}$ intervals $\omega$ and $2^{am \over 2}$ intervals $\omega'$ so that $\hat I \times \hat \omega=P$ and $\hat{I'} \times \hat{\omega'}=P'$ respectively. This implies
\begin{align*}
\sum_{\hat I=I_P} \sum_{ \hat {I'}=I_{P'}} \sum_{\hat \omega=\pmb{\omega}_{P_1}} \sum_{\hat \omega'=\pmb{\omega}_{P'_1}} |\langle \phi_{I \times \omega}, \phi_{I'\times \omega'}  \rangle|^2  \lesssim 2^{am} \frac{|I_{P'}|}{|I_P|} \sum_{\hat I=I_P} \sum_{ \hat {I'}=I_{P'}}  \Big( 1+\frac{\dist(I, I')}{|I|}   \Big)^{-50}.
\end{align*}

We focus on the last part. When $I$ and $I'$ are so that $\hat I=I_P$ and $\hat {I'}=I_{P'}$, with $\pmb {\omega}_{P_1} \subsetneq \pmb {\omega}_{P'_1}$, the separation condition \eqref{eq:capit:disj} has as a consequence
\begin{align*}
\sum_{\hat I=I_P} \Big( 1+\frac{\dist(I, I')}{|I|} \Big)^{-50} & \lesssim \Big( \frac{\dist(I_P, I')}{ 2^{- {am \over 2}}|I_P|} \Big)^{-40} \\
&\lesssim 2^{-20 am} \frac{1}{|I'|} \int_{I'} \Big( \frac{\dist(I_P, x)}{ |I_P|} \Big)^{-40} dx.
\end{align*}

Now we sum in $I'$ to get
\begin{align*}
&\sum_{\hat I=I_P} \sum_{ \hat {I'}=I_{P'}} \sum_{\hat \omega=\pmb{\omega}_{P_1}} \sum_{\hat \omega'=\pmb{\omega}_{P'_1}} |\langle \phi_{I \times \omega}, \phi_{I'\times \omega'}  \rangle|^2\\   &\lesssim 2^{-10 am} \frac{|I_{P'}|}{|I_P|} \frac{1}{|I_{P'}|} \int_{I_{P'}} \Big( \frac{\dist(I_P, x)}{ |I_P|} \Big)^{-40} dx\\
&\lesssim 2^{-10 am} \frac{|I_{P'}|}{|I_P|}  \Big( \frac{\dist(I_P, I_{P'})}{ |I_P|} \Big)^{-40} .
\end{align*}
This proves \eqref{eq:en:alm:orth}, finishing thus our proof.
\end{proof}

\begin{remark}\label{remark:decay}
The size estimates \eqref{eq:est:size-h-}--\eqref{eq:est:size-h:cluster} continue to hold, with an extra $(1+|j|)^{-{M \over 2}}$ decay, when we factor in the $j$-shifts discussed in Remark \ref{remark:shifts}. Indeed, although the wave-packets $\Phi_{am-k,2 \ell-1,j}^{ 2^{am}(r-1)+p'}$ take into account the information in $\underline{I}_{P}^{p'}+ j 2^{am \over 2}|\underline{I}_{P}^{p'}|$ -- thus supported within $I_P+ j|I_P|$ -- thanks to the decaying factor $(1+|j|)^{-{M}}$ we can regard it as adapted to $I_P$ through the bump function $\ci_{I_P}$.

The same reasoning applies to the energy estimate \eqref{eq:est:energy-h:pm}, and further, allows to deduce the same localized estimate \eqref{eq:local:bht:osl} when taking the shifts into account.
\end{remark}

The size and energy estimates together with Proposition \ref{prop:size/energy} allow us to prove Proposition \ref{prop:aim}.

\subsection{Establishing the boundedness range for $T$}
\label{sec:recovering:range}

Thanks to the estimates worked in subsections \ref{sec:timeFreq:prelim}-\ref{subsec:energy:only} we first establish the $m$-decay boundedness of $T_m$ within the strictly local-$L^2$ case ($2<p,q,r'<\infty$). This is the content of Corollary \ref{cor:conclusion:trilinear:form:L} which follows from Proposition \ref{prop:aim} via decomposition \eqref{eq:decomposition:Lm}. Notice that, due to the summability in $m$, the same boundedness range is implied for $T$.

Next, with Proposition \ref{prop:aim} settled, we present a refinement of the above time-frequency analysis, which, based on an intermediate interpolation step executed at a localized level, implies ($m$-exponential decaying) estimates for $T_m$ in the wider range $\{2<p,q \leq \infty\ ; \ 1<r<\infty \}$ thus proving Theorem \ref{mdec}. Finally, as it was explained in Section \ref{sec:reduction:local:L2}, the full range of Theorem \ref{main_theorem_pxq_to_r} follows subsequently from standard interpolation techniques.

\subsubsection{Boundedness of $T$ in the local-$L^2$ range}
\label{subsec:proof:localL^2}

As mentioned above, this is an immediate consequence of Corollary \ref{cor:conclusion:trilinear:form:L}, which, in turn, follows from Proposition \ref{prop:aim}. Thus our focus will be on proving the latter.
\begin{proof}[Proof of Proposition \ref{prop:aim}] We first establish restricted-type estimates: for any finite measure sets $E_1, E_2, E_3 \subset \rr R$ and any triple of functions $f, g, h$ satisfying
\[
|f(x)| \leq \one_{E_1}(x) \text{    a.e.}, \quad |g(x)| \leq \one_{E_2}(x) \text{    a.e.}, \quad |h(x)| \leq \one_{E_3}(x) \text{    a.e.}, \quad
\]
we claim to have
\begin{alignat}{3}
|\Lambda_{\BHT}^{\mathfrak{u}}(f, g, h)| &\lesssim  && |E_1|^\frac{1}{p} \, |E_2|^\frac{1}{q} \, |E_3|^\frac{1}{r'}, \label{eq:aimBHT+} \\
|\Lambda_{\BHT}^{\mathfrak{l}}(f, g, h)| &\lesssim 2^{-\delta_2' a m} && |E_1|^\frac{1}{p} \, |E_2|^\frac{1}{q} \, |E_3|^\frac{1}{r'}, \label{eq:aimBHT-} \\
|\Lambda_{\BHT}^{\mathfrak{c}}(f, g, h)| &\lesssim 2^{-\delta_2' am} && |E_1|^\frac{1}{p} \, |E_2|^\frac{1}{q} \, |E_3|^\frac{1}{r'}, \label{eq:aimBHT:cluster}
\end{alignat}
for any $p, q, r' \in (2, \infty)$ satisfying $\frac{1}{r}=\frac{1}{p}+\frac{1}{q}$ and some exponent $\delta_2':=\delta_2'(p,q)>0$.

We start with \eqref{eq:aimBHT+}. Estimates \eqref{eq:est:size-f-g} and \eqref{eq:est:size-h+} imply for $\P=\BHT$
\[
 \size^{\P}_1(f(P_1)) \lesssim 1, \qquad  \size^{\P}_2(g(P_2)) \lesssim 1, \qquad  \size^{\P}_3(h_\mathfrak{u}(P_3)) \lesssim 1;
\]
similarly, due to \eqref{eq:est:energy-f-g} and \eqref{eq:est:energy-h:pm},
\[
\energy^{\P}_1(f(P_1)) \lesssim |E_1|^\frac{1}{2}, \quad  \energy^{\P}_2(g(P_2)) \lesssim |E_2|^\frac{1}{2}, \quad  \energy^{\P}_3(h_\mathfrak{u}(P_3)) \lesssim |E_3|^\frac{1}{2}.
\]

Applying Proposition \ref{prop:size/energy} with the size/energy estimates and the choice $\theta_1= 1- \frac{2}{p} \in (0,1),  \theta_2=1-\frac{2}{q} \in (0, 1)$ and $\theta_3=-1+\frac{2}{r} \in (0,1)$, one has
\begin{align*}
\Lambda_{\BHT}^{\mathfrak{u}}(f,g,h) & \lesssim |E_1|^{{1-\theta_1} \over 2} \cdot |E_2|^{{1-\theta_2} \over 2} \cdot |E_3|^{{1-\theta_3} \over 2} \\
& \lesssim  |E_1|^{1 \over p} \cdot |E_2|^{1 \over q} \cdot |E_3|^{1 \over {r'}},
\end{align*}
which is \eqref{eq:aimBHT+}.

For \eqref{eq:aimBHT-}, the corresponding size and energy estimates are
\[
 \size^{\P}_1(f(P_1)) \lesssim 1, \qquad  \size^{\P}_2(g(P_2)) \lesssim 1, \qquad  \size^{\P}_3(h_\mathfrak{l}(P_3)) \lesssim 2^{- \frac{\delta_1 a m}{2}};
\]
and
\[
\energy^{\P}_1(f(P_1)) \lesssim |E_1|^\frac{1}{2}, \quad  \energy^{\P}_2(g(P_2)) \lesssim |E_2|^\frac{1}{2}, \quad  \energy^{\P}_3(h_\mathfrak{l}(P_3)) \lesssim |E_3|^\frac{1}{2}.
\]
We apply again Proposition \ref{prop:size/energy} with the same choice of exponents $\theta_1, \theta_2, \theta_3$, obtaining
\begin{align*}
\Lambda_{\BHT}^\mathfrak{l}(f,g,h) & \lesssim |E_1|^{{1-\theta_1} \over 2} \cdot |E_2|^{{1-\theta_2} \over 2} \cdot 2^{-\frac{\delta_1}{2} am \theta_3} |E_3|^{{1-\theta_3} \over 2}  \\
& \lesssim  2^{-\delta_1 am(-\frac{1}{2}+\frac{1}{r})} |E_1|^{1 \over p} \cdot |E_2|^{1 \over q} \cdot |E_3|^{1 \over {r'}};
\end{align*}
this is estimate \eqref{eq:aimBHT-}, where the decaying factor is $2^{-\delta_1 am(-\frac{1}{2}+\frac{1}{r})}$. The proof of \eqref{eq:aimBHT:cluster} follows analogously (with the only difference consisting in the estimate $\size^{\P}_3(h_\mathfrak{c}(P_3)) \lesssim 2^{- \frac{\mu a m}{4}}$), which concludes the proof of the three restricted-type estimates with $\delta_2':= \min\{\delta_1, \frac{\mu}{2}\}(-\frac{1}{2}+\frac{1}{r})>0$.

Finally, by invoking restricted-type interpolation (see for example \cite{Mlinop} or \cite{multilinear-Marcink-constant} for a reference), we obtain the strong-type estimates of Proposition \ref{prop:aim} with, say,
$\delta_2(p,q):=\frac{1}{2} \min\{\delta_1,\frac{\mu}{2}\}(-\frac{1}{2}+\frac{1}{r})>0$.
\end{proof}

\subsubsection{Boundedness of $T$ in the wider range $\{(p, q, r): 2<p,q \leq \infty\ ; \ 1<r<\infty \}$}
\label{sec:size:interpolation}

To deal with the extended range and, in particular, to be able to prove some $L^p \times L^\infty \to L^p$ endpoint estimates, we need to move outside of the local-$L^2$ range. As it turns out, it will be sufficient to remove the local-$L^2$ constraints on the third function $h$. To this aim, we will use the helicoidal method of \cite{BM-sparse}, as follows:
\begin{enumerate}[label=(\alph*)]
\item \label{s:hel:1} first we transform the time-frequency information obtained in Sections  \ref{sec:timeFreq:prelim}-\ref{subsec:energy:only} into exclusively spatial information (in the form of maximal averages over arbitrary intervals);
\item \label{s:hel:2} next, we interpolate between this localized information and the localized, exclusively spatial information for an operator resembling a linearized version of the bilinear maximal function from \eqref{eq:bilinearmaxfunction};
\item\label{s:hel:3} lastly, we re-organise this information in the form of a Fefferman-Stein inequality from which the extended range of boundedness will follow.
\end{enumerate}

\begin{proof}[Proof of Theorem \ref{mdec}]
To implement the strategy outlined above, we rely on the initial discretization performed in Section \ref{subsec:final:form:disc} which allows us to represent $\underline{\L}_{m}(f, g, h)$ as $\sum_{P \in \BHT} \underline{\L}_{m,P}(f, g, h)$ -- see \eqref{eq:model}.

For any $\P \subset \BHT$ and $I_0$ a dyadic interval
\[
\P(I_0):= \{ P \in \P : I_P \subseteq I_0 \}
\]
denotes the subcollection of $\P$ with spatial intervals contained in $I_0$. Recalling the definition \eqref{eq:def:spatial:tiles}, we also define
\begin{equation}
\mathcal{I}_{\bar{\P}(I_0)}:=\{ I \subseteq I_0 \text{  dyadic  }: I \in \mathcal{I}_{\bar{\P}} \text{  or  } I=I_0  \}.
\end{equation}
In this way, we only keep track of the spatial intervals included in $I_0$ corresponding to tiles or trees in $\P$, together with the spatial interval $I_0$ itself.

%\[
%\P(I_0):= \{ P \in \P : I_P \subseteq I_0 \} \qquad \text{and} \qquad \P(I_0)^+:= \P(I_0) \cup \{ P_{I_0}\},
%\]
%where $P_{I_0}$ is some\footnote{For example $P_{I_0}:=I_0 \times [0,2^{am} |I_0|^{-1}]$.} tri-tile with spatial interval $I_0$, not necessarily contained in $\mathbb{P}$.

Item \ref{s:hel:1} above indicates that $|\sum_{P \in \mathbb{P}(I_0)} \underline{\L}_{m,P}(f, g, h)|$ will be estimated using the strategy highlighted in Sections \ref{sec:timeFreq:prelim}-\ref{subsec:energy:only}, which also produces an exponential decay in $m$; an upper bound for the same expression $|\sum_{P \in \mathbb{P}(I_0)} \underline{\L}_{m,P}(f, g, h)|$ can be obtained directly from \eqref{eq:model} and its equivalent \eqref{eq:model:with:tiles}, although this time the exponential decay in $m$ is not available. Interpolating between the two -- as item \ref{s:hel:2} suggests, will produce estimates for $|\sum_{P \in \mathbb{P}(I_0)} \underline{\L}_{m,P}(f, g, h)|$ that fall outside the local-$L^2$ range and have exponential decay in $m$. It will remain to reconstruct the global information for $|\sum_{P \in \mathbb{P}} \underline{\L}_{m,P}(f, g, h)|$ from the local one, and this is achieved via Fefferman-Stein inequalities -- this is the last step \ref{s:hel:3}.

Propositions \ref{prop:size/energy}, \ref{prop:size:est}, and \ref{prop:energies}, together with the observation that the energies localize well,\footnote{Using the decay of the wave-packets it is not difficult to see that\[
 \energy^{\P(I_0)}_1((f(P_1))_{P\in\P(I_0)}) \lesssim \|f \cdot \ci_{I_0} \|_{L^2},
\]
and similarly for $g,h$. This resembles somehow the proof of the size estimate \eqref{eq:est:size-f-g}; see also \cite[Lemma 39]{vv_BHT}.
} imply that for some $\tilde \delta > 0$
\begin{equation}
\begin{aligned}
& \Big|\sum_{P \in \mathbb{P}(I_0)} \underline{\L}_{m,P}(f, g, h)\Big| \\
 & \lesssim 2^{- \tilde \delta a m} \Big( \sup_{I \in \PI} \frac{1}{|I|^{\frac{1}{2}}} \|f \cdot \ci_{I}\|_{L^2} \Big) \, \Big( \sup_{I \in \PI} \frac{1}{|I|^{\frac{1}{2}}} \|g \cdot \ci_{I}\|_{L^2}\Big) \\
& \hspace{12em} \cdot  \Big(\sup_{I \in \PI} \frac{1}{|I_P|^{\frac{1}{2}}} \|h \cdot \ci_{I}\|_{L^2}\Big) \, \cdot  |I_0|.
\end{aligned}\label{eq:local:bht:osl}
\end{equation}

Notice that the above estimate represents a localized multi-scale version of the $L^2 \times L^2 \times L^2 \to \C$ estimate \eqref{TTstargum} presented in Section \ref{sec:onescale}. In order to pass from the latter to the former, all the time-frequency discussion performed in Subsections \ref{sec:timeFreq:prelim} and \ref{subsec:energy:size}, as well as the refinement in Section \ref{sec:HR2} are necessary.

If we were now to invoke Theorem 12 of \cite{BM-sparse}, we could deduce from \eqref{eq:local:bht:osl} that
\begin{equation}
\label{eq:FS:2:2:2}
|\underline{\L}_{m}(f, g, h)|\lesssim 2^{- \tilde \delta a m } \int_{\R} M_{2}f(x) \cdot M_{2}g(x)  \cdot M_{2}h(x)  \, dx
\end{equation}
(where we recall that $M_2 f= (M |f|^2)^{1/2}$). This inequality is clearly sufficient to deduce the summability of the operator norms of $\underline{\L}_m$ (and hence of $T_m$) in the strict local-$L^2$ range -- in this way we would recover the result in Corollary \ref{cor:conclusion:trilinear:form:L}. However, $L^p \times L^\infty \to L^p$ bounds cannot be concluded from the estimate above.

Identity \eqref{eq:local:bht:osl} illustrates step \ref{s:hel:1} of the strategy described at the beginning of the present subsection. We continue now with step \ref{s:hel:2}, which consists of obtaining a different type of local estimate for $\sum_{P \in \mathbb P (I_0)} \underline{\L}_{m,P}(f, g, h)$.

For this, we rely on the initial definition \eqref{def:trilinear:formtr} of $\underline{\L}_{m,P}(f, g, h)$, but will momentarily return to the parametric formulation based on indices $k, \ell, r, n, p, u, v$. With this, $|\underline{\L}_{m,P}(f, g, h)|$ becomes
\begin{align}
\label{eq:return:1}
\frac{1}{2^{\frac{am+2k}{4}}} \Big| \sum_{n,p \sim 2^\frac{am}{2}} & \sum_{u, v \sim 2^\frac{am}{2}} \langle f_{\ell+1}, \check{\varphi}_{\frac{am}{2}- k, \ell}^{u-v,p_r-n}  \rangle   \langle g_{\ell-1}, \check{\varphi}_{\frac{am}{2}- k, \ell}^{u+v,p_r+n} \rangle \\
& \cdot \int_{\R} \rho_{am-ak}(\l(x))\, \,  \check{\varphi}_{\frac{am}{2}- k, 2\ell-1}^{2u, p_r}(x) w_{k, n,v}^e(\l)(x)h(x) dx \Big|. \nonumber
\end{align}

Recalling the definition of $w_{k, n,v}^e(\l)$, the decaying estimate \eqref{eq:decay-weight}, and the pointwise inequality
\[
 \big|\check{\varphi}_{\frac{am}{2}- k, 2\ell-1}^{2u, p_r}(x)\big| \lesssim 2^{\frac{am-2k}{4}} \ci_{I_k^{p_r}}(x),
\]
we bound \eqref{eq:return:1} by
\begin{align}
\label{eq:return:2}
\frac{1}{2^{k}}\Big| \sum_{n,p \sim 2^\frac{am}{2}} & \sum_{u, v \sim 2^\frac{am}{2}} \langle f_{\ell+1}, \check{\varphi}_{\frac{am}{2}- k, \ell}^{u-v,p_r-n}  \rangle   \langle g_{\ell-1}, \check{\varphi}_{\frac{am}{2}- k, \ell}^{u+v,p_r+n} \rangle \Big| \\
& \cdot \int_{\R}  \rho_{am-ak}(\l(x))\, \,  \ci_{I_k^{p_r}}(x)  \frac{1}{ 1+  \Big|\bar c_a \frac{n^{a-1}  \l(x)}{2^{a(\frac{am}{2}-k)}} +2v \Big|^2}  |h(x)| dx. \nonumber
\end{align}
Notice that the second line does not depend on the $u$ parameter anymore since we disregarded the frequency information contained in the wave-packet $\check{\varphi}_{\frac{am}{2}- k, 2\ell-1}^{2u, p_r}$. We can use Cauchy-Schwarz in $u$ to obtain, for fixed parameters $v, n, p$,
\begin{align*}
& \sum_{u \sim 2^{am \over 2}} \big| \langle f_{\ell+1}, \check{\varphi}_{\frac{am}{2}- k, \ell}^{u-v,p_r-n}  \rangle   \langle g_{\ell-1}, \check{\varphi}_{\frac{am}{2}- k, \ell}^{u+v,p_r+n} \rangle \big| \\
& \lesssim  \big(  \sum_{u_1 \sim 2^\frac{am}{2}} \big| \langle f, \check{\varphi}_{\frac{am}{2}- k, \ell}^{u_1,p_r-n}  \rangle \big|^2  \big)^\frac{1}{2}  \big(  \sum_{u_2 \sim 2^\frac{am}{2}} \big|  \langle g, \check{\varphi}_{\frac{am}{2}- k, \ell}^{u_2,p_r+n} \rangle \big|^2  \big)^\frac{1}{2}.
\end{align*}
So that $|\underline{\L}_{m,P}(f, g, h)|$ is further bounded by
\begin{align}
\label{eq:return:3}
\frac{1}{2^{k}}\sum_{n,p \sim 2^\frac{am}{2}} & \sum_{ v \sim 2^\frac{am}{2}}  \big(  \sum_{u_1 \sim 2^\frac{am}{2}} \big| \langle f, \check{\varphi}_{\frac{am}{2}- k, \ell}^{u_1,p_r-n}  \rangle \big|^2  \big)^\frac{1}{2}  \big(  \sum_{u_2 \sim 2^\frac{am}{2}} \big|  \langle g, \check{\varphi}_{\frac{am}{2}- k, \ell}^{u_2,p_r+n} \rangle \big|^2  \big)^\frac{1}{2} \\
& \cdot \int_{\R}  \rho_{am-ak}(\l(x))\, \,  \ci_{I_k^{p_r}}(x)  \frac{1}{ 1+  \Big|\bar c_a \frac{n^{a-1}  \l(x)}{2^{a(\frac{am}{2}-k)}} +2v \Big|^2}  |h(x)| dx. \nonumber
\end{align}

Now the expression in the first line does not depend on $v$ anymore, and thus, summing in $v$ (see \eqref{eq:summ:v}), we get
\begin{align}
\label{eq:return:4}
|\underline{\L}_{m,P}(f, g, h)| \lesssim \frac{1}{2^{k}}\sum_{n,p \sim 2^\frac{am}{2}} &  \big(  \sum_{u_1 \sim 2^\frac{am}{2}} \big| \langle f, \check{\varphi}_{\frac{am}{2}- k, \ell}^{u_1,p_r-n}  \rangle \big|^2  \big)^\frac{1}{2}  \big(  \sum_{u_2 \sim 2^\frac{am}{2}} \big|  \langle g, \check{\varphi}_{\frac{am}{2}- k, \ell}^{u_2,p_r+n} \rangle \big|^2  \big)^\frac{1}{2} \\
& \cdot \int_{\R}  \rho_{am-ak}(\l(x))\, \,  \ci_{I_k^{p_r}}(x) |h(x)| dx. \nonumber
\end{align}
We apply Cauchy-Schwarz with respect to $n$ -- which only affects the first line -- and thus, modulo considering $O(1)$ similar terms, we have
\begin{align}
\label{eq:return:5}
|\underline{\L}_{m,P}(f, g, h)| \lesssim \frac{1}{2^{k}} &  \big(  \sum_{n_1, u_1 \sim 2^\frac{am}{2}} \big| \langle f, \check{\varphi}_{\frac{am}{2}- k, \ell}^{u_1,n_1}  \rangle \big|^2  \big)^\frac{1}{2}  \big(  \sum_{u_2, n_2 \sim 2^\frac{am}{2}} \big|  \langle g, \check{\varphi}_{\frac{am}{2}- k, \ell}^{u_2,n_2} \rangle \big|^2  \big)^\frac{1}{2} \\
& \cdot \int_{\R}  \rho_{am-ak}(\l(x))\, \,  \big( \sum_{p \sim 2^\frac{am}{2}} \ci_{I_k^{p_r}}(x) \big) |h(x)| dx. \nonumber
\end{align}

At this point, we can return to the notation involving only the tile formulation: for any $P \in \BHT$, we have proved that
\begin{align*}
|\underline{\L}_{m,P}(f, g, h)| \lesssim \frac{1}{|I_P|} f(P_1) g(P_2)  \int_{\R}  \rho_{|P|}(\l(x))\, \ci_{I_P}(x) |h(x)| dx,
\end{align*}
where we recall the notations \eqref{eq:def:f(P):tiles} and \eqref{eq:def:g(P):tiles}.

Putting these together we obtain the local estimate:
\begin{align*}
\big|\sum_{P \in \mathbb P (I_0)} \underline{\L}_{m,P}(f, g, h)\big|& \lesssim \sum_{\kappa : 2^\kappa \leq |I_0|} \sum_{\substack{I \subseteq I_0 \\ |I|=2^\kappa}} \frac{1}{|I|} \sum_{\substack{P \in \mathbb{P} \\ I_P=I}}  \|f_P\|_{L^2} \|g_P\|_{L^2}  \int_{\R}  \rho_{am- a \kappa}(\l(x))\, \ci_{I}(x) |h(x)| dx
\end{align*}

The tiles in $\mathbb{P}(I_0)$ sharing the same spatial interval correspond to mutually disjoint frequency intervals and due to orthogonality and Cauchy-Schwartz we have
\[
\sum_{\substack{P \in \mathbb{P}(I_0) \\ I_P=I}}  \|f_P\|_{L^2} \|g_P\|_{L^2}  \lesssim \|f \cdot \ci_I\|_{L^2}  \|g \cdot \ci_I\|_{L^2}.
\]
This observation produces
\begin{align*}
\big|\sum_{P \in \mathbb P (I_0)} \underline{\L}_{m,P}(f, g, h)\big| & \lesssim  \Big( \sup_{I \in \PI} \frac{1}{|I|^{\frac{1}{2}}} \|f \cdot \ci_{I}\|_{L^2} \Big) \, \Big( \sup_{I \in \PI} \frac{1}{|I|^{\frac{1}{2}}} \|g \cdot \ci_{I}\|_{L^2}\Big)   \\
& \sum_{\kappa : 2^\kappa \leq |I_0|} \sum_{\substack{J \subseteq I_0 \\ |J|=2^\kappa}} \frac{1}{|J|}\int_{\R}  \rho_{am- a \kappa}(\l(x))\, \ci_{J}(x) |h(x)| dx.
\end{align*}
Since the intervals have the same scale,
\[
\sum_{\substack{J \subseteq I_0 \\ |J|=2^\kappa}}  \ci_{J}(x) \lesssim \ci_{I_0}(x);
\] in addition, thanks to the presence of the stopping-time $\lambda(x)$, for every $x$ there are only $O(1)$ scales $\kappa$ for which $\rho_{am-a\kappa}(\lambda(x)) \neq 0$; and therefore
\[
 \sum_{\kappa : 2^\kappa \leq |I_0|} \sum_{\substack{J \subseteq I_0 \\ |J|=2^\kappa}} \int_{\R}  \rho_{am- a \kappa}(\l(x))\, \ci_{J}(x) |h(x)| dx \lesssim  \int_{\R} \ci_{I_0}(x) |h(x)| dx.
\]

We have therefore proven the following bound:
\begin{equation}
\begin{aligned}
\big|\sum_{P \in \mathbb{P} (I_0)} \underline{\L}_{m,P}(f, g, h)\big|& \lesssim \big( \sup_{I \in \PI} \frac{1}{|I|^{\frac{1}{2}}} \|f \cdot \ci_{I}\|_{L^2} \big) \cdot \big( \sup_{I \in \PI} \frac{1}{|I|^{\frac{1}{2}}} \|g \cdot \ci_{I}\|_{L^2}\big) \\
& \hspace{1em} \cdot \big( \sup_{I \in \PI} \frac{1}{|I|} \int_{\R}|h(x)|  \ci_{I}(x)  dx \big) \, |I_0|.
\end{aligned}\label{eq:localization:from:sparse}
\end{equation}
Although the expression involving $h$ is a genuine average, for symmetry reasons we express it as a maximal average. Similarly to \eqref{eq:FS:2:2:2}, \cite[Theorem 12]{BM-sparse} implies the global Fefferman-Stein-type inequality
\begin{equation}
\label{eq:FS:2:2:1}
|\underline{\L}_{m}(f, g, h)|\lesssim \int_{\R} M_{2}f(x) \cdot M_{2}g(x)  \cdot M h(x)  \, dx.
\end{equation}

At this point, we would like to apply a suitable interpolation procedure between \eqref{eq:FS:2:2:2} and \eqref{eq:FS:2:2:1}, in a way that avoids canonical interpolation for the operator $T_m$.\footnote{Estimates of the type $L^p \times L^\infty \to L^p$ cannot be obtained through classical interpolation since they correspond to boundary points.}

This is achieved by interpolating between the two local estimates \eqref{eq:local:bht:osl} and \eqref{eq:localization:from:sparse}, in the particular situation when $h$ is a restricted-type function: i.e. $h$ is such that $|h(x)| \leq \one_H(x)$ for some measurable set $H \subset \R$ of finite measure. The maximal averages of $h$ in \eqref{eq:local:bht:osl} are then bounded by
\[
 \Big( \sup_{I \in \PI} \frac{1}{|I|} \int_{\R}\one_H(x)  \ci_{I}(x) dx \Big)^{{1\over 2}}
\]
and in \eqref{eq:localization:from:sparse} by
\[
 \Big( \sup_{I \in \PI} \frac{1}{|I|} \int_{\R} \one_H(x)  \ci_{I}(x)dx \Big).
\]
Using trivial interpolation (geometric average), we deduce that
\begin{equation}
\begin{aligned}
& \Big|\sum_{P \in \mathbb{P}(I_0)} \underline{\L}_{m,P}(f, g, h)\Big| \\
 \lesssim  & 2^{- \tilde \delta a m \theta} \Big( \sup_{I \in \PI} \frac{1}{|I_P|^{\frac{1}{2}}} \|f \cdot \ci_{I}\|_{L^2} \Big) \Big( \sup_{I \in \PI} \frac{1}{|I|^{\frac{1}{2}}} \|g \cdot \ci_{I}\|_{L^2}\Big) \\
& \hspace{1em} \cdot \Big( \sup_{I \in \PI} \frac{1}{|I|} \int_{\R} \ci_{I}(x) \one_H(x) dx \Big)^{{\theta\over 2}+(1-\theta)} \, |I_0|,
\end{aligned} \label{eq:localization:from:sparse:restricted-type}
\end{equation}
for any $0<\theta<1$. Then a ``mock interpolation'' result -- see Proposition 14 of \cite{BM-sparse} -- implies that
\begin{equation}
\begin{aligned}
& \Big|\sum_{P \in \BHT(I_0)} \underline{\L}_{m,P}(f, g, h)\Big| \\
 \lesssim  &  2^{- \tilde \delta a m \theta} \Big( \sup_{I \in \PI} \frac{1}{|I|^{\frac{1}{2}}} \|f \cdot \ci_{I}\|_{L^2} \Big) \Big( \sup_{I \in \PI} \frac{1}{|I|^{\frac{1}{2}}} \|g \cdot \ci_{I}\|_{L^2}\Big)
\\ & \hspace{1em} \cdot \Big( \sup_{I \in \PI} \frac{1}{|I|} \int_{\R} \ci_{I}(x) |h(x)|^\tau dx \Big)^{1 \over \tau} \, |I_0|,
\end{aligned}\label{eq:localization:from:sparse:rinterp}
\end{equation}
for any $\tau>0$ such that $\frac{1}{\tau}<1-\frac{\theta}{2}$. We note that the implicit constant blows up as $\tau$ nears $\frac{2}{2-\theta}$.

For the last part \ref{s:hel:3} of the strategy, we appeal again to \cite[Theorem 12]{BM-sparse}, which allows to put back together the information obtained at the local level and deduce -- via a sparse domination argument -- that
\begin{equation}
\label{eq:FS:L}
|\underline{\L}_{m}(f, g, h)|\lesssim 2^{- \tilde \delta a m \theta} \int_{\R} M_{2}f(x) \cdot M_{2}g(x)  \cdot M_{\tau}h(x)  \, dx.
\end{equation}
This is indeed an interpolation of \eqref{eq:FS:2:2:2} and \eqref{eq:FS:2:2:1}, but it is achieved through local control of $\underline{\L}_{m}$.

Finally, this implies by duality that with $\bar \delta=\tilde \delta \theta$
\[
\| \underline{\L}_{m}\|_{L^p \times L^q \times L^{r'} \to {\mathbb C}} = \| T_m\|_{L^p \times L^q \to L^r} \lesssim 2^{-\bar \delta a m}
\]
for any $p, q, r$ satisfying
\[
\frac{1}{p}<\frac{1}{2}, \quad \frac{1}{q} <\frac{1}{2}, \quad \frac{1}{r'} < 1-{\theta \over 2},
\]
additionally to the usual H\"older scaling condition $1/p  +1/q = 1/r$. Since $\theta \in (0, 1)$ was arbitrary, these conditions reduce to $2<p, q\leq \infty$ and $r'>1$.
\end{proof}

\begin{remark}
\begin{enumerate}
\item The inequality
\begin{align*}
|\underline{\L}_{m,P}(f, g, h)| \lesssim \frac{1}{|I_P|} f(P_1) g(P_2) \int_{\R}  \rho_{|P|}(\l(x))\, \ci_{I_P}(x) |h(x)| dx,
\end{align*}
is a discretized, dualized version of the estimate
\begin{align*}
&\Big|\rho_a(2^{-a(m-k)}\lambda(x)) \int f(x-t)g(x+t) e^{i\lambda(x) t^a} \rho(2^{-k}t) \frac{dt}{t} \Big| \\
 & \lesssim \rho_a(2^{-a(m-k)}\lambda(x)) M(f,g)(x)
\end{align*}
used in Section \ref{sec:initial:reductions} to bound pointwise $T_m$ by the bilinear maximal function $M$ of \eqref{eq:bilinearmaxfunction}. The Gabor frame decomposition used for the higher resolution model naturally induces $L^2$ norms for the functions $f$ and $g$.

\item If in \eqref{eq:est:size-h-}-\eqref{eq:est:size-h:cluster} the size of $h$ was an $L^\tau$ quantity for some $1<\tau<2$, some $L^p \times L^\infty \to L^p$ estimates could have been obtained by applying \ref{s:hel:1} directly.
The difference between \eqref{eq:local:bht:osl} and \eqref{eq:localization:from:sparse} consists in the fact that \eqref{eq:local:bht:osl} presents a $2^{- \tilde \delta am}$ decay, but we have to use $L^2$ maximal averages for $h$, while \eqref{eq:localization:from:sparse} presents no decay in $m$, but the (maximal) averages of $h$ are measured relative to the better (smaller) $L^1$ norm. This dichotomy is explained as follows: The compatibility step in Section \ref{sec:HR2}, which allows to extract spatial and frequency information for the third function $h$ as well, requires a careful distribution of the information on the level set of the oscillatory factor carried by $ \rho_{|P|}(\lambda(x)) w_{|P|, n, v}^e(\lambda)(x)$. Using $\ell^2$ estimates in the parameters $n, v$ and $p'$, this information is split between the bilinear term corresponding to $f$ and $g$, and the term corresponding to the dualizing function $h$, respectively. Because of this, the sizes of $f, g$ and $h$ used above in the previous Subsection \ref{subsec:energy:size} are naturally $L^2$-adapted. In obtaining \eqref{eq:localization:from:sparse} however, one may disregard $w_{|P|, n, v}^e(\lambda)(x)$ altogether exchanging the lack of $m$-decay with the use of $L^1$ averages for $h$.
\end{enumerate}
\end{remark}

\appendix

\section{Lack of $m$-decaying absolute summability for the discrete phase-linearized wave-packet model}
\label{sec:counterexample}

As stated in the Introduction, in this first appendix we prove the lack of the absolute single scale $m$-decay for the discrete phase-linearized wave-packet model defined by \eqref{def:trilinear:form0}.

We start by recalling the formulations in  \eqref{eq:model:with:tiles} and   \eqref{def:trilinear:formtr} and focus our attention on the expression of this model localized within a fixed tri-tile $P=P(k,\ell,r)\in\BHT$ given by some prescribed parameters $k,\ell,r\in \N$ and a fixed $m\in\N$:
\begin{align*}
\underline{\L}_{m,P}(f,g,h) & = \underline{\L}_{m,k,  \ell,r}(f, g, h) \\
 & = \sum_{n, p \sim 2^{am \over 2}} \sum_{u,v \sim 2^{ am \over 2}} \frac{1}{2^{am +2 k \over 4}} \langle f_{\ell+1}, \check{\varphi}_{{am \over 2}- k, \ell}^{u-v, p_r-n}  \rangle   \langle g_{\ell-1}, \check{\varphi}_{{am \over 2}- k, \ell}^{u+v, p_r+n}  \rangle  \\
& \int_{\R}  \check{\varphi}_{{am \over 2}- k, 2\ell-1}^{2u, p_r}(x)  h(x) \rho_{am-ak}( \lambda(x)) w_{k, n, v}^e(\l)(x) dx.
\end{align*}

For convenience we take $P=P_0:=P(0,1,1)$ so that $k=0$, $\ell=1$, $r=1$, $I_{P^0}=[1,2]$, $\pmb{\omega}_{P^0_1}=2^{am}\,[1,2]$ and $\pmb{\omega}_{P^0_2}=2^{am}\,[-1,0]$. In what follows we show that if we take absolute values inside the summation signs appearing in the definition of $\underline{\L}_{m,P}\big|_{P=P_0}$ above then no analogue of Proposition \ref{sgscale} can hold. Moreover, one can not obtain a single scale $m$-decay estimate even within the H\"older scaling $L^2\times L^2\times L^{\infty}$:

\begin{proposition} Fix\footnote{The case $a\in\{1,2\}$ is superfluous.} $a>0$. There are no absolute constants $C,\delta>0$ for which the following holds\footnote{Here we recall that since $r=1$ we must have that $p_r=p$.} uniformly in the choices of $m\in\N$ (large enough), phase function $\lambda$ and input functions $f,g\in L^2$, $h\in L^\infty$:
\begin{equation}
\begin{aligned}
\underline{\L}_{m,P^0}^{abs}(f,g,h)
  & := \sum_{n, p \sim 2^{am \over 2}} \sum_{u,v \sim 2^{ am \over 2}} \frac{1}{2^{am \over 4}} |\langle f_{2}, \check{\varphi}_{{am \over 2},1}^{u-v, p-n}  \rangle| \cdot |\langle g_{0}, \check{\varphi}_{{am \over 2},1}^{u+v, p+n}  \rangle|   \\
& \hspace{5em} \cdot \Big|\int_{\R}  \check{\varphi}_{{am \over 2},1}^{2u,p}(x)  h(x) \rho_{am}( \lambda(x)) w_{0, n, v}^e(\l)(x) dx\Big|\\
& \leq C 2^{-\delta am} \|f\|_{L^2}\|g\|_{L^2} \|h\|_{L^\infty}. \label{eq:disprove}
\end{aligned}
\end{equation}
\end{proposition}

\begin{remark}
(a) The above proposition remains true if one modifies the norm in which $h$ is measured: one can substitute the $L^{\infty}$ norm with $L^2$ or even with Sobolev type norms. The existence of a counterexample at \eqref{eq:disprove} is not specific to the control/norm that we put on $h$ but rather depends on the interaction between, on the one hand, the information carried by  $f$ and $g$ and, on the other hand, the stricture of the linearizing function $\l$; \\
(b) In fact, our proof reveals the existence of an absolute constant $K >0$ so that for any $m \in \N$ (large enough), there exist nontrivial $f,g \in L^2$, $h \in L^\infty$ and a real measurable phase function $\lambda$ for which
\[
\underline{\L}_{m,P^0}^{abs}(f,g,h) > K  \|f\|_{L^2}\|g\|_{L^2} \|h\|_{L^\infty}.
\]
\end{remark}

\begin{proof}
Our proof will proceed via {\it reductio ad absurdum}. Thus we assume that for any $f,g\in L^2$ and $h\in L^{\infty}$ there exist  $C,\delta>0$ absolute constants such that \eqref{eq:disprove} holds uniformly in $m\in {\mathbb N}$ large enough.

Take now
$$ f:= \sum_{\substack{2\cdot 2^{am\over 2} \leq |n_1|\leq 3\cdot 2^{am \over 2} \\ |p| \leq 3\cdot 2^{am \over 2}}}  \check{\varphi}_{{am \over 2}, 1}^{n_1,p} \qquad \textrm{and}  \qquad g:=\sum_{\substack{1\leq |n_2|\leq 2^{am \over 2} \\ |p| \leq 3\cdot 2^{am \over 2}}}  \check{\varphi}_{{am \over 2}, 1}^{n_2, p}\,.$$
Thus $f=f_{2}$, $g=g_0$ and for $n, p,u,v \sim 2^{am \over 2}$ (suitably understood) we have
$$ \langle f, \check{\varphi}_{{am \over 2}, 1}^{u-v,p-n} \rangle  \sim 1 \sim \langle g,\check{\varphi}_{{am \over 2},1}^{u+v, p+n}\rangle$$
with
$$ \|f\|_{L^2} \simeq 2^{am \over 2} \simeq \|g\|_{L^2}.$$
For this specific choice of $f$ and $g$ we have that \eqref{eq:disprove} is equivalent with
\begin{equation}
 \sum_{n, p \sim 2^{am \over 2}} \sum_{u,v \sim 2^{ am \over 2}}
\Big|\int_{I_{P_0}}  \check{\varphi}_{{am \over 2}, 1}^{2u,p}(x)  h(x) \rho_{am}( \lambda(x)) w_{0, n, v}^e(\l)(x) dx\Big| \leq C 2^{-\delta am} 2^{\frac{5}{4} am }  \|h\|_{L^\infty}.
\label{eq:disprove2}
\end{equation}

Next, we are going to construct the essential item in what will become our counterexample to \eqref{eq:disprove} -- the phase function $\lambda$. For this we decompose
\begin{equation}\label{declam}
\lambda(x)=\lambda_u(x)\,+\lambda_s(x)\,,
\end{equation}
where here
\begin{itemize}
\item the uniform component is defined as $\lambda_u(x):= 2^{am}$ for every $x\in I_{P^0}$;

\item the structured component has both the input and the output correlated with some suitable (generalized) arithmetic progressions.\footnote{The fact that such a generalized arithmetic progressions structure plays the key r\^{o}le in the construction of our present counterexample suggests a natural connection between the (limiting) behavior of $BC^a$ and that of the Carleson operator $C$. Indeed, same type of structures are essential in providing counterexamples for the behavior of the pointwise convergence of Fourier Series near $L^1$ -- for the latter see \cite{lv9} and the bibliography therein.} Indeed, for some $L\in\N$ large enough, we take
 \begin{equation}\label{ls1}
\textrm{supp}\,\lambda_s=A=\bigcup_{{j\simeq 2^{am\over 2}}\atop{j\in L\N}} A_j\,,
\end{equation}
with
 \begin{equation}\label{ls2}
A_j := \bigsqcup_{{q\simeq 2^{am \over 2}}\atop{q\in L\N}} ( 2^{-am}[j,j+1]+ q2^{-\frac{am}{2}})=: \bigsqcup_{{q\simeq 2^{am \over 2}}\atop{q\in L\N}} I_{j,q}\,,
\end{equation}
and let
\begin{equation}\label{ls3}
\lambda_s(x) := \sum_{{j\sim 2^{am \over 2}}\atop{j\in L \N}} (2^{am\over 2}\,(j-2^{am\over 2})) \one_{A_j}(x)\,.
\end{equation}
\end{itemize}

With these, let us record the following facts:

\begin{itemize}
\item by properly adjusting the support of $\rho$ we have from the above that $\rho_{am}(\lambda(x)))=1$ for every $x\in I_{P^0}$;

\item the phase function $\lambda$ is constant on each of the subsets $A_j$ as well as on $B:=I_{P_0}\setminus A$;

\item since $w_{0, n, v}^e(\l_u)$  is constant on the whole interval $I_{P^0}$ (it does not depend on $x$) it is not hard to see that \eqref{eq:disprove} holds for $\l=\l_u$.\footnote{This should become transparent once the reader mimics the reasoning presented in \eqref{stph-bis}--\eqref{absv} for the case in which $x\in B$.}
\end{itemize}

From the above we deduce that it is enough to disprove \eqref{eq:disprove2} for $\rho_{am}(\lambda(x)))=1$ and
$I_{P_0}$ replaced by $A$ where here $\l$ and $A$ are as constructed above.

We now fix $x\in A$ and compute the contribution of the weight defined in \eqref{eq:w_k}
\begin{align*}
\one_{A}(x)\, w_{0, n, v}^e (\lambda)(x) & =  \int_{\R} \psi( \xi + 2v ) e^{i \, n \, \xi} e^{i c_a 2^{ ({am  \over 2}) a'} \xi^{a'} \lambda(x)^{- {a' \over a}}}\, d \xi \\
& =  \sum_{{j\sim 2^{am \over 2}}\atop{j\in L \N}} \one_{A_j}(x) \int_{\R} \psi( \xi + 2v ) e^{i \, n \, \xi} e^{i c_a 2^{ ({am  \over 2}) a'} \xi^{a'} (2^{am\over 2}j)^{- {a' \over a}}}\, d \xi\,.
\end{align*}
Using now the principle of (non)stationary phase (recall that we are in the regime in which the phase function $\nu$ in the integral above satisfies $|\nu''(\xi)| \sim 1$ on the domain of integration), we have the following approximation:
\begin{equation}
\begin{aligned}
 w_{0, n, v}^e (\lambda)(x)  & \approx  \sum_{{j\sim 2^{am \over 2}}\atop{j\in L \N}} \one_{A_j}(x) e^{i \bar c_a n^a  2^{am\over 2} j 2^{- a\left( \frac{am}{2} \right)}} \psi \big( \bar c_a\frac{n^{a \over {a'}}2^{am\over 2}j}{2^{a\left( \frac{am}{2}\right)}} +2v \big) \\
& =  \sum_{{j\sim 2^{am \over 2}}\atop{j\in L \N}} \one_{A_j}(x) e^{i \bar c_a n^a  j 2^{- \frac{a^2m}{2a'} }} \psi \big( \bar c_a\frac{n^{a \over {a'}}j}{2^{\frac{a^2m}{2a'}}} + 2v \big).
\end{aligned} \label{stph-bis}
\end{equation}
Indeed, the above holds up to an error term which involves a decaying factor of the form $2^{-\d am}$ for some $\d>0$ thus not affecting the construction of our counterexample below.

Consequently, in order to disprove \eqref{eq:disprove2} it is enough to show that there are no $C,\,\d>0$ such that the following holds uniformly in $m\in\N$ and $h\in L^{\infty}$:
\begin{align}
 \sum_{n, p \sim 2^{am \over 2}} \sum_{u, v \sim 2^{ am \over 2}}
\Big| \sum_{{j\sim 2^{am \over 2}}\atop{j\in L \N}}\big( \int_{\R}  \check{\varphi}_{{am \over 2}, 1}^{2u,p}(x)  h(x) \one_{A_j}(x) dx \big) e^{i \bar c_a n^a  j 2^{- \frac{a^2m}{2a'} }} \psi \big( \bar c_a \frac{n^{a \over {a'}}j}{2^{\frac{a^2m}{2a'}}} + 2v \big)  \Big| \nonumber \\
\leq C 2^{- \delta a m}\,2^{\frac{5}{4} am }\,\|h\|_{L^\infty}.
\label{eq:disprove3}
\end{align}
Recalling that $\psi\in C_0^{\infty}(\R)$ and choosing appropriately its support, we deduce that in order for the inner summation term in \eqref{eq:disprove3} to be nonzero we must have
\begin{equation}\label{oto}
 \Big| \bar c_a \frac{n^{a \over {a'}}j}{2^{\frac{a^2m}{2a'}}} +2v \Big| \leq 1
\end{equation}
which, since $n\sim 2^{am \over 2}$, is equivalent to
$$ \Big| j + \frac{2^{\frac{a^2m}{2a'}+1}}{\bar c_a n^{a \over {a'}}}v \Big| \leq  \frac{2^{\frac{a^2m}{2a'}}}{\bar c_an^{a \over {a'}}} \sim 1\:.$$

Hence for $n$ fixed, $j$ is determined by $v$ since $j\in L \N$ and $L\in\N$ may be chosen as large as we want (but still independently of $m$). Thus, without loss of generality we may assume that \eqref{oto} describes a $1$-to-$1$ map $v \leftrightarrow j$.

Fixing $n\sim 2^{am \over 2}$, we have from the above that
\begingroup
\allowdisplaybreaks
\begin{align}\label{absv}
 & \sum_{v \sim 2^{ am \over 2}}
\Big| \sum_{{j\sim 2^{am \over 2}}\atop{j\in L \N}} \Big( \int_{\R}  \check{\varphi}_{{am \over 2}, 1}^{2u,p}(x)  h(x) \one_{A_j}(x) dx \Big) e^{i \bar c_a n^a  j 2^{- \frac{a^2m}{2a'} }} \psi \big( \bar c_a\frac{n^{a \over {a'}}j}{2^{\frac{a^2m}{2a'}}} +2v \big)  \Big| \\
& \qquad =  \sum_{v \sim 2^{ am \over 2}}
\sum_{{j\sim 2^{am \over 2}}\atop{j\in L \N}} \Big| \Big( \int_{\R}  \check{\varphi}_{{am \over 2}, 1}^{2u,p}(x)  h(x) \one_{A_j}(x) dx \Big) e^{i \bar c_a n^a  j 2^{- \frac{a^2m}{2a'} }} \psi \big( \bar c_a \frac{n^{a \over {a'}}j}{2^{\frac{a^2m}{2a'}}} +2v \big)  \Big|\nonumber \\
& \qquad =  \sum_{v \sim 2^{ am \over 2}}
\sum_{{j\sim 2^{am \over 2}}\atop{j\in L \N}}  \Big| \Big( \int_{\R}  \check{\varphi}_{{am \over 2}, 1}^{2u,p}(x)  h(x) \one_{A_j}(x) dx \Big) \Big|  \Big| \psi \big( \bar c_a\frac{n^{a \over {a'}}j}{2^{\frac{a^2m}{2a'}}} +2v \big)  \Big| \nonumber\\
& \qquad \gtrsim
\sum_{{j\sim 2^{am \over 2}}\atop{j\in L \N}}  \big| \langle h \one_{A_j}, \check{\varphi}_{{am \over 2}, 1}^{2u,p}\rangle \big|\nonumber.
\end{align}
\endgroup
Notice that the last lower bound does not depend on the parameter $n$.

So in order for  \eqref{eq:disprove3} to be true, we must have
\begin{equation}
 \sum_{u, p\sim 2^{am \over 2}}\sum_{{j\sim 2^{am \over 2}}\atop{j\in L \N}} \big| \langle h \one_{A_j}, \check{\varphi}_{{am \over 2}, 1}^{2u,p}\rangle \big|   \leq C 2^{-\delta a m} 2^{\frac{3}{4} am }  \|h\|_{L^\infty}.
\label{eq:disprove4}
\end{equation}

\begin{remark}
Observe that the trivial case $\delta=0$ in \eqref{eq:disprove4} is an immediate consequence of the Cauchy-Schwarz inequality via the pairwise disjointness of $\{A_j\}_j$ and the almost orthogonality of $\big\{\check{\varphi}_{{am \over 2}, 1}^{2u, p} \big\}_{u,p}$. Thus, in order to maximize the left hand-side of \eqref{eq:disprove4} it is natural to pursue the equality case in the above Cauchy-Schwarz argument thus constructing $\{A_j\}_j$ and $h\in L^{\infty}$ such that the local Fourier coefficients $\big| \langle h \one_{A_j}, \check{\varphi}_{{am \over 2}, 1}^{2u, p}\rangle \big|$ are (essentially) pairwise equal and hence independent of the specific choice of $u,p,j$. In particular, this explains why we design the structure of  $(A_j)_j$ as essentially a $2^{-am}$-neighborhood of a generalized arithmetic progression of rank 2 (hence its ``$2^{-\frac{am}{2}}$--translation invariance behavior'' in \eqref{ls2}).
\end{remark}

Take now\footnote{Or some smooth variant of it which is still equal to $1$ on $[1,2]$.} $h=\one_{[1,2]}$ and recall the definition of  $A_j$ in \eqref{ls2}. By properly massaging the properties of $\check {\varphi}$ we can assume without loss of generality that for every $j\sim 2^{am \over 2}$ with $j\in L \N$ we have that
\begin{equation}\label{ijq}
\check {\varphi}(2^{\frac{a m}{2}} x - q) \sim 1 \qquad \textrm{for any} \qquad x\in I_{j,q}\,,
\end{equation}
 and
\begin{equation}\label{ijqpr}
 |\check {\varphi}(2^{\frac{a m}{2}} x - q)| \leq |q-q'|^{-10}\qquad \textrm{for any}\:\:x\in I_{j,q'}\:\: \textrm{with}\:\:\:q,q'\in L \Z,\:q\not=q'\,.
\end{equation}

In what follows we fix $u, p\sim 2^{am \over 2}$. From the above we then deduce that
\begingroup
\allowdisplaybreaks
\begin{align}\label{dech}
\big|\langle h \one_{A_j}, \check{\varphi}_{{am \over 2}, 1}^{2u, p}\rangle \big| & = 2^{\frac{a m}{4}} \Big| \int_{A_j}   \check {\varphi}(2^{\frac{a m}{2}} x - p) \, e^{i \,2^{\frac{a m}{2}} x \left( 2u  + 2^{\frac{am}{2}}\right)} dx \Big|\nonumber \\
& = 2^{\frac{a m}{4}} \Big| \sum_{{{q\simeq 2^{am \over 2}}\atop{q\in L\N}}} \int_{I_{j,q}}    \check {\varphi}(2^{\frac{a m}{2}} x - p) \, e^{i \,2^{\frac{a m}{2}} x \left( 2u  + 2^{\frac{am}{2}}\right)} dx \Big|\nonumber \\
& \geq 2^{\frac{a m}{4}} \Big| \int_{I_{j,p}} \check {\varphi}(2^{\frac{a m}{2}} x - p) e^{i \,2^{\frac{a m}{2}} x \left( 2u  + 2^{\frac{am}{2}}\right)} dx \Big| \nonumber\\
& \qquad \qquad - 2^{\frac{a m}{4}}\,\sum_{{q\neq p}\atop{q\in L \Z}}  \Big| \int_{I_{j,q}} \check {\varphi}(2^{\frac{a m}{2}} x - p) e^{i \,2^{\frac{a m}{2}} x \left( 2u  + 2^{\frac{am}{2}}\right)} dx \Big|\nonumber\\
& =: (I)-(II).
\end{align}
\endgroup

For the first term, applying \eqref{ijq} and  a change of variable, we claim that\footnote{Recall our earlier convention $e^{iy}=e^{2 \pi i y}$.}
\begin{equation}
\begin{aligned}
(I) & = 2^{-\frac{3a m}{4}} \Big| \int_{[0,1]} \check {\varphi}(2^{-\frac{a m}{2}} (j+y)) e^{i y \left( 2^{-\frac{am}{2}+1}u  + 1\right)} dy \Big|  \\
& \simeq 2^{-\frac{3a m}{4}} \Big| \check {\varphi}(2^{-\frac{a m}{2}} j ) \int_{[0,1]}  e^{i y \left( 2^{-\frac{am}{2}+1}u  + 1\right)} dy \Big| \simeq 2^{-\frac{3a m}{4}}
\end{aligned}\label{eq:app}
\end{equation}
holds for a large fraction of the possible values of $u\sim 2^{am \over 2}$.

Indeed, for this we first notice that while $u$ runs through the set ${\mathbb Z} \cap [2^{am \over 2}, 2\cdot 2^{am\over 2}]$,  the expression $2^{-\frac{am}{2}+1}u  + 1$ covers the set $2^{-{am \over 2}+1}{\mathbb Z}+1 \cap [3,5]$. Thus, for $m\in\N$ large enough, (say $am\geq 10$), there exists a value of $u=u_0$ in the range $u_0\sim 2^{am \over 2}$ such that
$$ \dist(2^{-\frac{am}{2}+1}u_0  + 1, {\mathbb Z}) \geq \frac{1}{10}.$$
Then an easy argument shows that for $u\in {\mathbb Z} \cap [2^{am \over 2}, 2\cdot 2^{am\over 2}] \cap [u_0-\frac{1}{100}2^{am \over 2}, u_0+\frac{1}{100}2^{am \over 2}]$ one still has
$$ \dist(2^{-\frac{am}{2}+1}u  + 1, {\mathbb Z}) \geq \frac{1}{4}$$
which implies the lower bound in \eqref{eq:app}
$$ |\int_{[0,1]}  e^{i y \left( 2^{-\frac{am}{2}+1}u  + 1\right)} dy| \geq \frac{1}{100}.$$

Hence, for $u$ belonging to ${\mathbb Z} \cap [2^{am \over 2}, 2\cdot 2^{am\over 2}] \cap [u_0-\frac{1}{100}2^{am \over 2}, u_0+\frac{1}{100}2^{am \over 2}]$ one can now verify the validity of the statement
\begin{equation}
(I)\geq 10^{-10} \,2^{-\frac{3a m}{4}}.
\label{eq:u} \end{equation}

For the second part $(II)$ we are using \eqref{ijqpr} for, say, a choice of $L\geq 100$. With this we have
\begin{equation}
\begin{aligned}
(II) &\lesssim 2^{-\frac{3a m}{4}} \sum_{{q\neq p}\atop{q\in L \Z}} \int_{[0,1]} |\check {\varphi}(2^{-\frac{a m}{2}}(j+y) +(p-q))|dy \\
& \leq 2^{-\frac{3a m}{4}} \sum_{d\in \Z\setminus\{0\}} (1+|d|\,L)^{-10} \leq 2^{-\frac{3a m}{4}}\,100\,L^{-10}.
\end{aligned}\label{sndtr}
\end{equation}

Putting now together \eqref{dech}--\eqref{sndtr} we obtain a refutation of \eqref{eq:disprove4} thus contradicting the assumption that \eqref{eq:disprove} holds.
\end{proof}

\section{Analysis of the multiplier} \label{appendix:multiplier}
In this appendix we perform the last part of the initial reductions, the result of which was anticipated in Subsection \ref{subsection_stationary_phase_analysis}: namely that the initial symbol \eqref{bigM} is equal to a main term of the form \eqref{smallM} plus an error term which is suitably bounded. This is a rather standard non-stationary phase argument\footnote{A classical reference for this type of stationary phase analysis is provided by \cite{BigStein}.}, applied to our specific setting.

We start by recalling our symbol
\[ \mathfrak{M}_{m,k}(\zeta;\lambda) = \int e^{i \lambda t^a} \frac{\rho(2^{-k}t)}{t} e^{-i \zeta t} \,dt. \]

The analysis of the above expression proceeds as follows: first one reduces the matters to the subset of parameters $(\zeta, \lambda)$ for which the phase contains a stationary point (this produces controlled errors of non-stationary nature); secondly, one specializes the amplitude of the symbol to a sufficiently small neighbourhood of the stationary point (this produces further controlled errors of non-stationary nature); thirdly, one applies the stationary phase method to the remaining part, obtaining the main contribution to the oscillatory integral $\mathfrak{M}_{m,k}$; finally, the resulting symbol is put into a tensorized form by a standard Fourier series decomposition argument. The errors that have arisen along the way are all controlled by the bilinear maximal function -- more precisely, the error contribution to $T_m(f,g)(x)$ will be controlled pointwise by $\lesssim 2^{-\delta am} M(f,g)(x)$ for some absolute $\delta > 0$ -- any such $\d$ is good enough for our purposes.
\subsection{Non-stationary contributions to $\mathfrak{M}_{m,k}$ when there is no critical point}\label{subsec:appendix:B:1}
First of all we tackle the non-stationary parts of the symbol. In what follows we will always assume that $\lambda\equiv\lambda(x) \sim 2^{a(m-k)}$ because of the presence of the factor $\tilde{\rho}(2^{-a(m-k)}\lambda(x))$ in the expression for $T_m(f,g)$.

If $\varphi(t) = \varphi_{\zeta,\lambda}(t) := \lambda t^a - \zeta t$ denotes the phase of $\mathfrak{M}_{m,k}(\zeta;\lambda)$, a simple dimensional analysis shows that in the domain of integration we have $|\lambda t^{a-1}| \sim 2^{am-k}$, and therefore $\varphi'$ will certainly be non-zero if $|\zeta| \not\sim 2^{am-k}$; the non-stationary phase principle then suggests that the symbol will be suitably small for such $\zeta$. To make this intuition rigorous, we begin by observing that by integration by parts
\begin{equation}
\begin{aligned}
\int e^{i \lambda t^a} & \frac{{\rho}(2^{-k}t)}{t} e^{-i \zeta t} \,dt = -\int e^{i \varphi(t)} \frac{d}{dt} \Big( \frac{{\rho}(2^{-k}t)}{i t \varphi'(t)} \Big) \,dt \\
&= -\int e^{i \varphi(t)} \Big[ \frac{2^{-k}{\rho}'(2^{-k}t)}{i t \varphi'(t)} -  \frac{{\rho}(2^{-k}t)}{i t^2 \varphi'(t)} -  \frac{{\rho}(2^{-k}t) \varphi''(t)}{i t (\varphi'(t))^2}\Big]\,dt.
\end{aligned}\label{eq:IBP_nonstationary_phase}
\end{equation}
Recall that $\rho$ as given in \eqref{eq:smooth_dyadic_decomposition} is smooth, compactly supported on $\Big[\frac{1}{2},2\Big]$ and such that $\sum_{j \in \mathbb{Z}} \rho(2^{-j}\zeta) = 1$ for $\zeta \neq 0$. Let us then decompose smoothly $1 = \chi_0 + \chi_1 + \sum_{j \geq 2}\chi_j$, where
\begin{enumerate}[label=\roman*)]
\item $\chi_1(\zeta) := \sum_{\ell=-2\lceil a \rceil}^{2\lceil a \rceil} \rho(a^{-1}2^{-\ell}\zeta)$;
\item $\chi_j(\zeta) := \rho(a^{-1} 2^{-2\lceil a \rceil+1} 2^{-j}\zeta)$ for any $j\geq 2$;
\item $\chi_0 = 1 - \chi_1 - \sum_{j \geq 2} \chi_j$.
\end{enumerate}
Observe that $\chi_0$ is a bump function supported in a neighbourhood of the origin, and more precisely it is supported in $\{\zeta : |\zeta| < a 2^{-2\lceil a \rceil}\}$.

We can then split
\[ \mathfrak{M}_{m,k}(\zeta;\lambda) = \chi_1\Big(\frac{\zeta}{2^{am-k}}\Big)\mathfrak{M}_{m,k}(\zeta;\lambda) + \Big(1 - \chi_1\Big(\frac{\zeta}{2^{am-k}}\Big)\Big)\mathfrak{M}_{m,k}(\zeta;\lambda) \]
and further split the last term (which will be shown to contribute an error term) according to \eqref{eq:IBP_nonstationary_phase} into
\[ \Big(1 - \chi_1\Big(\frac{\zeta}{2^{am-k}}\Big)\Big)\mathfrak{M}_{m,k}(\zeta;\lambda) = -\mathfrak{E}^{1}_{m,k}(\zeta;\lambda) +\mathfrak{E}^{2}_{m,k}(\zeta;\lambda) + \mathfrak{E}^{3}_{m,k}(\zeta;\lambda), \]
where for $i \in \{1,2,3\}$ the term $\mathfrak{E}^{i}_{m,k}$ is given by the sum
\[ \mathfrak{E}^{i}_{m,k}(\zeta;\lambda) := \mathfrak{E}^{i,0}_{m,k}(\zeta;\lambda) + \sum_{j\geq 2} \mathfrak{E}^{i,j}_{m,k}(\zeta;\lambda) \]
and for $j = 0$ or $j\geq 2$ we have set
\begin{align*}
\mathfrak{E}^{1,j}_{m,k}(\zeta;\lambda) &:= \chi_j \Big(\frac{\zeta}{2^{am-k}}\Big) \int e^{i \varphi(t)} \frac{{\rho}'(2^{-k}t)}{i t} \frac{2^{-k}}{\varphi'(t)} \,dt, \\
\mathfrak{E}^{2,j}_{m,k}(\zeta;\lambda) &:= \chi_j \Big(\frac{\zeta}{2^{am-k}}\Big) \int e^{i \varphi(t)} \frac{{\rho}(2^{-k}t)}{i t} \frac{1}{t\varphi'(t)} \,dt, \\
\mathfrak{E}^{3,j}_{m,k}(\zeta;\lambda) &:=  \chi_j \Big(\frac{\zeta}{2^{am-k}}\Big) \int e^{i \varphi(t)} \frac{{\rho}(2^{-k}t)}{i t} \frac{\varphi''(t)}{(\varphi'(t))^2} \,dt.
\end{align*}
Let us begin with showing that
\begin{equation}
\begin{aligned}
\Big|\sum_{k \in \mathbb{Z}} \tilde{\rho}\Big(\frac{\lambda(x)}{2^{a(m-k)}}\Big) \iint \widehat{f}(\xi) \widehat{g}(\eta) & e^{i(\xi+\eta)x} \mathfrak{E}^{1,0}_{m,k}(\xi - \eta;\lambda(x)) \,d\xi\,d\eta \Big|,\\
&\lesssim 2^{-am} M(f,g)(x),
\end{aligned}\label{eq:error_term_nonstationary_phase}
\end{equation}
so that this contribution to $T_m (f,g)$ can be safely ignored (thanks to the fact that the bilinear maximal function is bounded in the range given in Theorem \ref{main_theorem_pxq_to_r}). Considering $\mathfrak{E}^{1,0}_{m,k}$, we see that for the factor $2^{-k}/ \varphi'(t)$ we can write
\[ \frac{2^{-k}}{\varphi'(t)} = 2^{-am} \frac{1}{a \lambda t^{a-1} 2^{-(am-k)} - \zeta 2^{-(am-k)}}, \]
and if we let
\[ \Phi_{\lambda,t} (\zeta') = \Phi_{m,k,\lambda,t}(\zeta'):= \frac{\chi_0(\zeta')}{a \lambda t^{a-1} 2^{-(am-k)} - \zeta'} \]
we can rewrite
\[ \mathfrak{E}^{1,0}_{m,k}(\zeta;\lambda) = 2^{-am} \int \Phi_{\lambda,t}\Big(\frac{\zeta}{2^{am-k}}\Big)e^{-i\zeta t} \frac{{\rho}'(2^{-k}t)}{it} e^{i\lambda t^a} \,dt. \]

Notice that $|a \lambda t^{a-1} 2^{-(am-k)}|\sim_a 1$ and more precisely that
\begin{alignat*}{2}
a 2^{-2a+1} & \leq |a \lambda t^{a-1} 2^{-(am-k)}| \leq a 2^{2a-1}\qquad   &&  \textrm{when $a>1$,} \\
\frac{a}{2} & \leq |a \lambda t^{a-1} 2^{-(am-k)}| \leq 2a  && \textrm{when $a<1$.}
\end{alignat*}
For $\zeta'$ in the support of $\chi_0$, we then have in both cases $a>1$ and $a<1$ that
\[ |a \lambda t^{a-1} 2^{-(am-k)} - \zeta'|\gtrsim_a 1. \]
From this, one can see that $\Phi_{\lambda,t}$ is a smooth bump function adapted to an interval of length $\sim 1$ uniformly in all the parameters $m,k, \lambda \sim 2^{a(m-k)}, |t| \sim 2^k$, and in particular we have uniformly
\begin{equation}\label{eq:uniform_bound_nonstationary_analysis}
|\widehat{\Phi_{\lambda,t}}(s)| \lesssim_{a} (1 + |s|)^{-100}
\end{equation}
for all $|t| \sim 2^{k}$. All implicit constants here are allowed to depend on $a$, so from this point on we do not track the dependence systematically. Fixing $m,k, \lambda(x) \sim 2^{a(m-k)}$ and using Fubini twice, we see by going back to the spatial representation of the operator that we have
\begin{align*}
\Big| & \iint \widehat{f}(\xi) \widehat{g}(\eta) e^{i(\xi+\eta)x} \mathfrak{E}^{1,0}_{m,k}(\xi - \eta;\lambda(x)) \,d\xi\,d\eta \Big| \\
&= 2^{-am} \Big|\int f(x-s) g(x+s) \Big[\int 2^{am-k}\widehat{\Phi_{\lambda,t}}(2^{am-k}(s-t))\frac{{\rho}'(2^{-k}t)}{it} e^{i\lambda t^a} \,dt \Big] \,ds \Big| \\
&\leq 2^{-am} \int |f(x-s)| |g(x+s)| \int 2^{am-k} \big|\widehat{\Phi_{\lambda,t}}(2^{am-k}(s-t))\big|\frac{|{\rho}'(2^{-k}t)|}{|t|}\,dt \,ds.
\end{align*}
Using \eqref{eq:uniform_bound_nonstationary_analysis} we obtain for the resulting kernel
\begin{align*}
\int & 2^{am-k}\big|\widehat{\Phi_{\lambda,t}}(2^{am-k}(s-t))\big|\frac{|{\rho}'(2^{-k}t)|}{|t|}\,dt \\
&\lesssim \int 2^{am-k} \Big(1 + 2^{am-k}|s-t|\Big)^{-100}\frac{|{\rho}'(2^{-k}t)|}{|t|} \,dt \lesssim 2^{-k} (1 + 2^{-k}|s|)^{-50},
\end{align*}
where we have used the fact that ${\rho}'(2^{-k}t)/t = 2^{-k}{\rho}'(2^{-k}t)/(2^{-k}t)$ is an $L^1$-normalised bump function adapted to $|t| \sim 2^k$. Thus we have shown
\begin{align*}
\Big| & \tilde{\rho}\Big(\frac{\lambda(x)}{2^{a(m-k)}}\Big) \iint \widehat{f}(\xi) \widehat{g}(\eta) e^{i(\xi+\eta)x} \mathfrak{E}^{1,0}_{m,k}(\xi - \eta;\lambda(x)) \,d\xi\,d\eta \Big| \\
&\lesssim 2^{-am} \tilde{\rho}\Big(\frac{\lambda(x)}{2^{a(m-k)}}\Big) \int |f(x-s)||g(x+s)| 2^{-k} (1 + 2^{-k}|s|)^{-50} \,ds \\
&\lesssim 2^{-am} \tilde{\rho}\Big(\frac{\lambda(x)}{2^{a(m-k)}}\Big) M(f,g)(x);
\end{align*}
summing over $k \in \mathbb{Z}$ shows the desired pointwise inequality \eqref{eq:error_term_nonstationary_phase}.
%%
%\begin{remark}
%We will be using the argument just presented above over and over in the rest of the section to deal with the error terms. Ultimately this is just a workaround to compensate for the lack of a direct Plancherel-type theorem in the bilinear setting - however, other arguments are possible. In particular, one could use the fact that the support of the kernel of $T_{m,k}$ is in $|t|\lesssim 2^k$ to write
%\[ T_{m,k}(f,g)(x) = T_{m,k}(f \mathbf{1}_{[x-C 2^k, x + C 2^k]},g \mathbf{1}_{[x-C 2^k, x + C 2^k]})(x);  \]
%using good decay bounds on $|\mathfrak{M}_{m,k}|$ and Cauchy-Schwarz, followed by an appeal to Plancherel, would allow one to bound error terms by $2^{-\delta am} M_2 f(x)\, M_2 g(x)$ for some $\delta > 0$. The disadvantage is that the use of Cauchy-Schwarz is inefficient here, producing a factor of $2^{am}$ that has to be compensated for by applying integration by parts to $\mathfrak{M}_{m,k}$ two or more times instead of just once like we did above. It should also be noted that the error term that results this way is only bounded in the local $L^2$ range.
%\end{remark}
%

For $\mathfrak{E}^{1,j}_{m,k}$ and $j\geq 2$ the situation is entirely similar but we have summability in $j$ as well: one can write
\[ \chi_{j}\Big(\frac{\zeta}{2^{am-k}}\Big) \frac{2^{-k}}{\varphi'(t)} = 2^{-am} 2^{-j} \Phi_{\lambda,j,t} \Big(\frac{\zeta}{2^{am-k}}\Big) \]
for $\Phi_{\lambda,j,t}$ a smooth bump function adapted to an interval of length $\sim 2^j$ uniformly in $j$ and $|t| \sim 2^{k}$ (and the rest of the parameters); in particular we have
\[  |\widehat{\Phi_{\lambda,j,t}}(s)| \lesssim 2^j(1+2^j|s|)^{-100} \]
uniformly in $j$ and in $|t| \sim 2^k$, which replaces \eqref{eq:uniform_bound_nonstationary_analysis}. The same argument given above then shows that for any $j \geq 2$
\begin{align*}
\Big| \sum_{k\in \mathbb{Z}} \tilde{\rho}\Big(\frac{\lambda(x)}{2^{a(m-k)}}\Big) \iint \widehat{f}(\xi) \widehat{g}(\eta)e^{i(\xi+\eta)x} & \mathfrak{E}^{1,j}_{m,k}(\xi-\eta;\lambda(x)) \,d\xi\,d\eta\Big|\\
&\lesssim 2^{-am} 2^{-j} M(f,g)(x),
\end{align*}
which is summable in $j$ and therefore we have shown that
\begin{align*}
\Big| \sum_{k\in \mathbb{Z}} \tilde{\rho}\Big(\frac{\lambda(x)}{2^{a(m-k)}}\Big) \iint \widehat{f}(\xi) \widehat{g}(\eta)e^{i(\xi+\eta)x} & \mathfrak{E}^{1}_{m,k}(\xi-\eta;\lambda(x)) \,d\xi\,d\eta\Big| \\
&\lesssim 2^{-am} M(f,g)(x),
\end{align*}
and this whole contribution to $T_{m}(f,g)$ can be safely ignored.

The error terms with symbols $\mathfrak{E}^{2}_{m,k}, \mathfrak{E}^{3}_{m,k}$ can be treated analogously: one just repeats the arguments given above with $2^{-k}/\varphi'(t)$ replaced by $1/(t\varphi'(t))$ and $\varphi''(t) / (\varphi'(t))^2$ -- which can be seen to behave in exactly the same way for the purposes of the above argument (we stress that the presence of the factor of ${\rho}(2^{-k}t)$ forces $|t| \sim 2^k$). Thus we can dispense with these error terms as well and have effectively reduced ourselves to the study of the bilinear operator with multiplier symbol $\chi_1(\zeta 2^{-(am-k)}) \mathfrak{M}_{m,k}(\zeta;\lambda)$.
\subsection{Non-stationary contributions to $\mathfrak{M}_{m,k}$ close to the critical point}
At this point the phase $\varphi(t) = \lambda t^a - \zeta t$ can have one or two critical points\footnote{In the latter case they agree in absolute value.}: we denote either one of these critical points by $t_0$ and observe that
\[ |t_0| = \Big(\frac{|\zeta|}{a |\lambda|}\Big)^{1/(a-1)} \sim 2^{-k}, \]
since in what follows we require $|\zeta| \sim 2^{am-k}$. We will further resolve the symbol by decomposing the domain of integration in $\mathfrak{M}_{m,k}$ according to a Whitney decomposition relative to $t_0$: when we are not too close, $|\varphi'|$ is still large enough to produce decay while when we get too close, we split the phase into its quadratic part and the remainder Taylor tail with the former giving rise to the main term while the latter representing an error term. \par

We proceed now by splitting the function ${\rho}$ into two functions, one supported in the positive reals and one supported in the negative reals (with a slight abuse of notation we will continue to denote either one of them ${\rho}$). Next, we assume without loss of generality that there is a critical point $\varphi'(t_0)=0$ and that $t,t_0, \zeta > 0$. In this case we write
\[ t_0 = t_0(\zeta;\lambda) = \Big(\frac{\zeta}{a \lambda}\Big)^{1/(a-1)} \]
and notice that $t_0$ is smooth as a function of $\zeta$. We have $|t_0| \sim 2^{k}$ as remarked before and therefore $|t - t_0| \lesssim 2^{k}$ for all $t$ in the support of ${\rho}(2^{-k}t)$, so if we use ${\rho}_0$ to denote the same function as in \eqref{eq:smooth_dyadic_decomposition} (a smooth bump function supported in an annulus) we can decompose in $j$ as
\[ \int e^{i\varphi(t)} \frac{{\rho}(2^{-k}t)}{t} \,dt = \sum_{j : 2^j \lesssim 2^k} \int e^{i\varphi(t)} {\rho}_0 \Big( \frac{t - t_0(\zeta)}{2^j} \Big) \frac{{\rho}(2^{-k}t)}{t} \,dt. \]
We first restrict ourselves to those $j$ for which $2^{-am/3 + k} \lesssim 2^j \lesssim 2^k$ (the reason for this particular choice will be clear a posteriori) and claim that all these terms actually give rise to error terms - by the same arguments used in \ref{subsec:appendix:B:1}. Indeed, using integration by parts on
\begin{equation}
\chi_1\Big(\frac{\zeta}{2^{am-k}}\Big) \int e^{i \varphi(t)} {\rho}_0 \Big( \frac{t - t_0(\zeta)}{2^j} \Big)\frac{{\rho}(2^{-k}t)}{t} \,dt \label{eq:term_j_annulus_decomposition_oscillatory_integral}
\end{equation}
we obtain a number of terms of similar nature; we consider here for shortness only the term
\[ \mathcal{E}^{j}_{m,k}(\zeta;\lambda) := \chi_1 \Big(\frac{\zeta}{2^{am-k}}\Big) \int e^{i \varphi(t)} \frac{{\rho}(2^{-k}t)}{i t} {{\rho}_0}' \Big( \frac{t - t_0(\zeta)}{2^j} \Big) \frac{2^{-j}}{\varphi'(t)} \,dt \]
which provides the largest contribution to the error term - the other terms are all treated similarly. Applying the argument used in \ref{subsec:appendix:B:1} together with the fact that
\begin{equation}
\begin{aligned}
|\varphi'(t)| = a \lambda \Big| t^{a-1} - t_0(\zeta)^{a-1}\Big| &= a \lambda \Big| \int_{t_0(\zeta)}^{t} (a-1) s^{a-2} \,ds\Big| \\
&\sim_a 2^{a(m-k)} |t - t_0(\zeta)| 2^{(a-2)k} \sim_a 2^{am-2k+j},
\end{aligned}\label{eq:magnitude_phase_derivative}
\end{equation}
we obtain that
\[ \Big|\iint \widehat{f}(\xi) \widehat{g}(\eta) e^{i(\xi+\eta)x} \mathcal{E}^{j}_{m,k}(\xi-\eta;\lambda(x)) \,d\xi \, d\eta \Big| \lesssim 2^{-am} 2^{2k-2j} M(f,g)(x). \]
Summing over the $2^{-am/3 + k} \lesssim 2^j \lesssim 2^k$ range we obtain then an error contribution bounded by $\lesssim 2^{-am/3} M(f,g)(x)$, which is acceptable. The other contributions arisen from the integration by parts can be shown to be controlled by $\lesssim 2^{-2am/3} M(f,g)(x)$, and thus we have completed this intermediate reduction.

Finally, we deal with the terms of the form \eqref{eq:term_j_annulus_decomposition_oscillatory_integral} for those $j$ such that $2^j \lesssim 2^{-am/3 + k}$. These terms contain the main contribution to $\mathfrak{M}_{m,k}$, but to properly isolate this we need to remove some further error terms. Expanding the phase in Taylor series around the critical point, we let
\[ \varphi_0(t) := \varphi(t_0) + \frac{\varphi''(t_0)}{2}(t-t_0)^2 \]
denote the quadratic part of the phase ($\varphi'(t_0)=0$ by definition) and $\widetilde{\varphi}(t) := \varphi(t) - \varphi_0(t)$ denote the tail. We decompose for any $j$
\begin{align}
\chi_1 \Big(\frac{\zeta}{2^{am-k}}\Big) &\int e^{i \varphi(t)} {\rho}_0 \Big(\frac{t-t_0(\zeta)}{2^j}\Big) \frac{{\rho}(2^{-k}t)}{t} \,dt \nonumber \\
& =  \chi_1 \Big(\frac{\zeta}{2^{am-k}}\Big)\int e^{i \varphi_0(t)} {\rho}_0 \Big(\frac{t-t_0(\zeta)}{2^j}\Big) \frac{{\rho}(2^{-k}t)}{t} \,dt  \label{eq:quadratic_contribution_near_critical_pt} \\
& \hspace{1em} + \underbrace{\chi_1 \Big(\frac{\zeta}{2^{am-k}}\Big)\int e^{i \varphi_0(t)} (e^{i\widetilde{\varphi}(t)} - 1) \tilde{\rho}_0 \Big(\frac{t-t_0(\zeta)}{2^j}\Big) \frac{\tilde{\rho}(2^{-k}t)}{t} \,dt}_{ \textstyle =: \mathscr{E}^{j}_{m,k}(\zeta;\lambda)}\label{eq:tail_contribution_near_critical_pt}
\end{align}
and claim that \eqref{eq:tail_contribution_near_critical_pt} represents an error term. Indeed, for $t$ in the support of ${\rho}_0(2^{-j}(t-t_0))$, we observe that\footnote{Here we are using the real analyticity of $\widetilde{\varphi}$ on any given compact interval that does not contain the origin.}
\begin{align*}
|\widetilde{\varphi}(t)| &\leq \sum_{\ell \geq 3} \frac{|\widetilde{\varphi}^{(\ell)}(t_0)|}{\ell !}|t - t_0|^{\ell} \lesssim \sum_{\ell \geq 3} \frac{2^{a(m-k)} 2^{(a-\ell)k}}{\ell!} 2^{j\ell} \\
&\lesssim 2^{am} \sum_{\ell \geq 3} 2^{-(k-j)\ell} \lesssim 2^{am - 3k + 3j},
\end{align*}
and this is $\ll 1$ when $2^{j} \lesssim 2^{-am/3 + k}$ (this is the reason for our choice of cutoff in the previous argument). A similar estimate shows that we have
\[ |{\widetilde{\varphi}}'(t)| \lesssim  2^{am - 4k + 3j} \]
as well; if we let $\nu(t) := e^{i\widetilde{\varphi}(t)}-1$, the trivial inequality $|e^{is}-1| \leq |s|$ shows that we have
\begin{equation}
|\nu(t)| \lesssim 2^{am - 3k + 3j} \hspace{4em} \text{ and } \hspace{4em} |\nu'(t)| \lesssim 2^{am - 3k + 3j} \cdot 2^{-k}  \label{eq:bounds_for_nu}
\end{equation}
and analogous bounds for higher derivatives. We also observe that a simple computation shows
\begin{equation}
|{\varphi_0}'(t)| \sim 2^{am-2k+j} \label{eq:magnitude_phase_derivative_quadratic}
\end{equation}
(the same bound one has for $\varphi'$; see \eqref{eq:magnitude_phase_derivative}). After integrating \eqref{eq:tail_contribution_near_critical_pt} by parts, one can feed these estimates to the -- by now usual -- argument of \ref{subsec:appendix:B:1} to show that
\begin{align*}
\Big| \sum_{k\in \mathbb{Z}} \tilde{\rho}\Big(\frac{\lambda(x)}{2^{a(m-k)}}\Big) \iint \widehat{f}(\xi) \widehat{g}(\eta)e^{i(\xi+\eta)x}  \Big(\sum_{j : 2^j \lesssim 2^{-am/3 + k}} \mathscr{E}^{j}_{m,k}(\xi-\eta;\lambda(x))\Big) \,d\xi\,d\eta\Big| & \\
\lesssim 2^{-am/3} M(f,g)(x), &
\end{align*}
which is another acceptable error term.

We see that we have effectively reduced our task to that of treating the bilinear operator with multiplier symbol given by terms \eqref{eq:quadratic_contribution_near_critical_pt}, or more precisely
\[ \chi_1\Big(\frac{\zeta}{2^{am-k}}\Big) \int e^{i\varphi_0(t)} \Big[\sum_{j : 2^j \lesssim 2^{-am/3+k}} {\rho}_0\Big(\frac{t-t_0(\zeta)}{2^{j}}\Big)\Big] \frac{{\rho}(2^{-k}t)}{t}\,dt. \]
It is a little inconvenient to have the cut-off factor in square brackets in the expression, and thus we remove it by writing the above as
\begin{align}
& \chi_1\Big(\frac{\zeta}{2^{am-k}}\Big) \int e^{i\varphi_0(t)} \frac{{\rho}(2^{-k}t)}{t}\,dt \label{eq:main_contribution_to_multiplier} \\
& - \sum_{j : 2^{-am/3+k} \lesssim 2^j \lesssim 2^k} \chi_1\Big(\frac{\zeta}{2^{am-k}}\Big) \int e^{i\varphi_0(t)} {\rho}_0\Big(\frac{t-t_0(\zeta)}{2^{j}}\Big) \frac{{\rho}(2^{-k}t)}{t}\,dt; \label{eq:final_error_term_multiplier}
\end{align}
a repetition of the arguments given above\footnote{Notice that phases $\varphi$ and $\varphi_0$ satisfy the same estimates.} shows that the contribution to $T_m(f,g)(x)$ from \eqref{eq:final_error_term_multiplier} amounts to another error term dominated by $\lesssim 2^{-am/3} M(f,g)(x)$. What is left is \eqref{eq:main_contribution_to_multiplier}, which is the main contribution that will be analysed below.

\subsection{Stationary contributions to $\mathfrak{M}_{m,k}$}
We have managed to reduce the multiplier symbol from $\mathfrak{M}_{m,k}$ to \eqref{eq:main_contribution_to_multiplier}, where the difference between the two amounts to the fact that the phase is now purely quadratic and we have localised the multiplier symbol to $|\zeta| \sim 2^{am-k}$. The advantage is that we can now appeal to the identity\footnote{The formula is well-known; it is an easy consequence of the formula for the Fourier transform of a Gaussian and of an analytic continuation argument. Recall also our convention in footnote \ref{footnote:2pi}.}
\[ \int e^{i \mu t^2} \phi(t) \,dt = \Big(\frac{i}{2\mu}\Big)^{1/2} \int e^{-i\alpha^2 / (4\mu)} \, \widehat{\phi}(\alpha)\,d\alpha \]
for (say) Schwartz functions $\phi$ in order to produce explicit formulas for the multiplier symbol. Indeed, if we let $\theta(s):= {\rho}(s)/s$ and apply the above identity to the integral expression in \eqref{eq:main_contribution_to_multiplier} we obtain (with a change of variables) the expression
\begin{equation}
e^{i \varphi(t_0)} \Big(\frac{i}{2^{2k}\varphi''(t_0)}\Big)^{1/2} \int e^{- i \alpha^2 /(2^{2k+1} \varphi''(t_0))} e^{i (t_0 / 2^k) \alpha} \, \widehat{\theta}( \alpha) \,d\alpha.  \label{eq:appendix_B_complex_gaussian_expression}
\end{equation}
Easy computations--with $c_a := a^{-a/(a-1)}-a^{-1/(a-1)}$--yield that
\begin{align*}
\varphi(t_0) & = c_a \zeta^{\frac{a}{a-1}}\lambda^{-\frac{1}{a-1}}, \\
 \frac{t_0}{2^k} &= \Big(\frac{\zeta}{2^{am-k}}\Big)^{\frac{1}{a-1}} \Big(\frac{a \lambda}{2^{a(m-k)}}\Big)^{-\frac{1}{a-1}}, \\
 \frac{1}{2^{2k} \varphi''(t_0)} & = \frac{2^{-am}}{a-1} \Big(\frac{\zeta}{2^{am-k}}\Big)^{-\frac{a-2}{a-1}}\Big(\frac{a\lambda}{2^{a(m-k)}}\Big)^{-\frac{1}{a-1}}.
\end{align*}
The latter in particular shows that, in the integration region $|\alpha| \lesssim 1$ where most of the mass of $\widehat{\theta}(\alpha)$ is concentrated, the quadratic phase $\alpha^2 /(2^{2k+1} \varphi''(t_0))$ is extremely small (in particular it is $\lesssim 2^{-am}$). To take advantage of this we split \eqref{eq:appendix_B_complex_gaussian_expression} as
\begin{align}
& e^{i \varphi(t_0)} \Big(\frac{i}{2^{2k}\varphi''(t_0)}\Big)^{1/2} \int e^{i (t_0 / 2^k) \alpha} \, \widehat{\theta}( \alpha) \,d\alpha \label{eq:appendix_B_removed_complex_gaussian} \\
& + e^{i \varphi(t_0)} \Big(\frac{i}{2^{2k}\varphi''(t_0)}\Big)^{1/2}  \int \big(e^{-i \alpha^2 /(2^{2k+1} \varphi''(t_0))} - 1 \big) e^{i (t_0 / 2^k) \alpha} \, \widehat{\theta}( \alpha) \,d\alpha. \label{eq:appendix_B_difference_complex_gaussian}
\end{align}
Expression \eqref{eq:appendix_B_removed_complex_gaussian} is clearly equal to
\[ e^{i \varphi(t_0)} \Big(\frac{i}{2^{2k}\varphi''(t_0)}\Big)^{1/2} \theta\Big(\frac{t_0}{2^k}\Big), \]
which will give our main term, while \eqref{eq:appendix_B_difference_complex_gaussian} (we claim) contributes another error term. If we define the smooth function
\[ \Theta(\zeta,\lambda) := \frac{\tilde{\rho}(a^{-1} \lambda) \chi_1(\zeta)}{\big(\zeta^{\frac{a-2}{a-1}} \lambda^{\frac{1}{a-1}}\big)^{\frac{1}{2}}} \theta\Big(\Big(\frac{\zeta}{\lambda}\Big)^{\frac{1}{a-1}}\Big) \]
(compactly supported in $|\zeta| \sim |\lambda| \sim 1$), then the main term in $\tilde{\rho}(\lambda 2^{-a(m-k)})\mathfrak{M}_{m,k}(\zeta;\lambda)$ is (a constant multiple of)
\begin{equation}
\mathfrak{m}_{m,k}(\zeta;\lambda) := 2^{-\frac{am}{2}} \; e^{i c_a \zeta^{\frac{a}{a-1}} \lambda^{-\frac{1}{a-1}}} \; \Theta\Big( \frac{\zeta}{2^{am-k}}, \frac{a \lambda}{2^{a(m-k)}} \Big). \label{eq:appendix_B_main_term_non_tensorised}
\end{equation}

It remains to prove our claim about the error term nature of \eqref{eq:appendix_B_difference_complex_gaussian}. Observe that with a little more algebra this remaining contribution to the symbol can be written as
\[ \mathfrak{e}_{m,k}(\zeta; \lambda):= 2^{-\frac{am}{2}} \; e^{i c_a \zeta^{\frac{a}{a-1}} \lambda^{-\frac{1}{a-1}}} \; \tilde{\Theta}\Big( \frac{\zeta}{2^{am-k}}, \frac{a \lambda}{2^{a(m-k)}} \Big) \Psi_m\Big( \frac{\zeta}{2^{am-k}}, \frac{a \lambda}{2^{a(m-k)}} \Big), \]
where we have defined
\begin{align*}
\tilde{\Theta}(\zeta,\lambda) & := \frac{\tilde{\rho}(a^{-1} \lambda) \chi_1(\zeta)}{\big(\zeta^{\frac{a-2}{a-1}} \lambda^{\frac{1}{a-1}}\big)^{\frac{1}{2}}}, \\
\Psi_m(\zeta,\lambda) &:= \int \big(e^{i c'_a  2^{-am} \zeta^{-\frac{a-2}{a-1}} \lambda^{-\frac{1}{a-1}} \alpha^2 } - 1 \big) e^{i \zeta^{\frac{1}{a-1}} \lambda^{-\frac{1}{a-1}} \alpha} \;\widehat{\theta}(\alpha) \,d\alpha
\end{align*}
(and $c'_a$ is another constant that depends on $a$). We have that
\begin{align*}
& \iint \widehat{f}(\xi) \widehat{g}(\eta) e^{i (\xi + \eta) x} \mathfrak{e}_{m,k}(\xi - \eta ; \lambda) \,d\xi \, d\eta \\
& = \int f(x - s) g(x + s) \Big[ \int \mathfrak{e}_{m,k}(\zeta; \lambda) e^{-i \zeta s} \,d\zeta \Big] \, ds;
\end{align*}
the oscillatory integral expression in square brackets has phase $c_a \zeta^{\frac{a}{a-1}} \lambda^{-\frac{1}{a-1}} - \zeta s$ and, thanks to the support properties of $\tilde{\Theta}$, there will be a critical point only if $|s| \sim 2^k$. Splitting the domain of integration into $|s| \sim 2^k$ and its complement, it is easy to show (using the trivial inequality $|e^{i\alpha} - 1| \leq |\alpha|$ and the rapid decay of $|\widehat{\theta}(\alpha)|$) that, say,
\[ |\tilde{\Theta}(\zeta, \lambda)\Psi_{m}(\zeta,\lambda)| \lesssim 2^{-(3/4)am}, \]
and therefore one can bound in the former case
\begin{align*}
&\Big|\int_{|s| \sim 2^k} f(x - s) g(x + s) \Big[ \int \mathfrak{e}_{m,k}(\zeta; \lambda) e^{-i \zeta s} \,d\zeta \Big] \, ds \Big| \\
&\lesssim \int_{|s| \sim 2^k} |f(x - s)| |g(x + s)|  2^{-\frac{am}{2}} 2^{am-k} 2^{-(3/4)am} \,ds \\
&= 2^{-am/4} \frac{1}{2^k} \int_{|s|\sim 2^k} |f(x-s)||g(x+s)| \,ds \lesssim 2^{-\frac{am}{4}} M(f,g)(x),
\end{align*}
an acceptable contribution. In the complementary case ($|s| \not\sim 2^k$) we find ourselves with a non-stationary phase situation, which can be treated by a dyadic decomposition in $|s|$ and by means of integration by parts. The resulting elementary argument is quite analogous to that in \ref{subsec:appendix:B:1}, albeit longer and possibly even more tedious - hence we omit the unnecessary details. Overall we have shown that
\[ \Big| \sum_{k \in \mathbb{Z}} \iint \widehat{f}(\xi) \widehat{g}(\eta) e^{i(\xi + \eta)x} {\mathfrak{e}}_{m,k}(\xi-\eta;\lambda(x))\,d\xi\,d\eta \Big| \lesssim 2^{-\delta am} M(f,g)(x), \]
for some absolute $\d>0$ thus proving our claim that $\mathfrak{m}_{m,k}(\xi - \eta; \lambda)$ is the main contribution to the multiplier symbol.
\subsection{Tensorization of the main contribution}
The function $\Theta$ in \eqref{eq:appendix_B_main_term_non_tensorised} is very well-behaved: it is smooth and compactly supported in $\{(\zeta,\lambda)\,:\,|\zeta|\sim 1\:\:\textrm{and}\:\:|\lambda|\sim 1\}$. However it is not in the most convenient form, since currently it is defined as a function that depends on $\zeta$ and $\lambda$ simultaneously. It will be more convenient to separate the two variables by means of a tensorization argument. Indeed, by a standard double Fourier series decomposition one can rewrite the function $\Theta$ as
\[ \Theta(\zeta, \lambda) = \sum_{j,\ell \in \mathbb{Z}} c_{j, \ell} \rho_j(\lambda) \psi_\ell(\zeta), \]
where $\rho_j(\lambda), \psi_\ell(\zeta)$ are smooth bump functions supported in $|\lambda| \sim 1$ and $|\zeta| \sim 1$ respectively  which satisfy
\[ \|\rho_j^{(n)}\|_{\infty} \lesssim_n |j|^{n}, \quad \|\psi_{\ell}^{(n)}\|_{\infty} \lesssim_n |\ell|^{n} \]
for sufficiently many derivatives and the coefficients $c_{j,\ell}$ are fastly decaying. By triangle/H\"older inequality it is thus sufficient\footnote{Provided that the $L^p \times L^q \to L^r$ bounds thus obtained depend at most polynomially on the $C^1$ norms of $\rho, \psi$.} to prove Theorem \ref{main_theorem_local_L2} for the operators
\[ \sum_{k \in \mathbb{Z}} {\rho}\Big(\frac{\lambda(x)}{2^{-a(m-k)}}\Big)\iint \widehat{f}(\xi) \widehat{g}(\eta) e^{i(\xi + \eta)x}\mathfrak{m}_{m,k}(\xi-\eta;\lambda(x))\,d\xi\,d\eta, \]
where we have redefined $\mathfrak{m}_{m,k}$ to be
\[ \mathfrak{m}_{m,k}(\zeta;\lambda) := 2^{-am/2} \; e^{i c_a \zeta^{\frac{a}{a-1}} \lambda^{-\frac{1}{a-1}}} \; \psi\Big(\frac{\zeta}{2^{am-k}}\Big) \]
and $\rho, \psi$ are now \emph{generic} smooth bump functions supported in $|\lambda|\sim 1, |\zeta|\sim 1$ respectively. This is precisely the form of the multiplier symbol given in \eqref{smallM}, as desired.
\begin{remark}
This last tensorization step is where the smoothness of the function $\tilde{\rho}$ has finally come in handy. All the previous steps in the reduction could have equally been achieved with a characteristic function in place of $\tilde{\rho}$, but the resulting function $\Theta$ would not have been smooth in $\lambda$.
\end{remark} 
%
%
%\nocite{*}

\bibliographystyle{plain}
\bibliography{BCa_bibliography}

\begin{thebibliography}{100}

\bibitem{Bat}
Michael Bateman.
\newblock Single annulus {$L^p$} estimates for {H}ilbert transforms along
  vector fields.
\newblock {\em Rev. Mat. Iberoam.}, 29(3):1021--1069, 2013.

\bibitem{BT}
Michael Bateman and Christoph Thiele.
\newblock {$L^p$} estimates for the {H}ilbert transforms along a one-variable
  vector field.
\newblock {\em Anal. PDE}, 6(7):1577--1600, 2013.

\bibitem{vv_BHT}
Cristina Benea and Camil Muscalu.
\newblock Multiple vector-valued inequalities via the helicoidal method.
\newblock {\em Anal. PDE}, 9(8):1931--1988, 2016.

\bibitem{BM-sparse}
Cristina Benea and Camil Muscalu.
\newblock Sparse domination via the helicoidal method.
\newblock to appear in \textit{Rev. Mat. Iberoamericana},
  https://arxiv.org/abs/1707.05484, 2017.

\bibitem{B02}
Jonathan~M. Bennett.
\newblock Hilbert transforms and maximal functions along variable flat curves.
\newblock {\em Trans. Amer. Math. Soc.}, 354(12):4871--4892, 2002.

\bibitem{Bernicot}
Fr\'{e}d\'{e}ric Bernicot.
\newblock {$L^p$} estimates for non-smooth bilinear {L}ittlewood-{P}aley square
  functions on {$\Bbb R$}.
\newblock {\em Math. Ann.}, 351(1):1--49, 2011.

\bibitem{Be12}
Fr\'{e}d\'{e}ric Bernicot.
\newblock Fiber-wise {C}alder\'{o}n-{Z}ygmund decomposition and application to
  a bi-dimensional paraproduct.
\newblock {\em Illinois J. Math.}, 56(2):415--422, 2012.

\bibitem{Bo86}
Jean Bourgain.
\newblock A {S}zemer\'{e}di type theorem for sets of positive density in {${\bf
  R}^k$}.
\newblock {\em Israel J. Math.}, 54(3):307--316, 1986.

\bibitem{Bo88}
Jean Bourgain.
\newblock A nonlinear version of {R}oth's theorem for sets of positive density
  in the real line.
\newblock {\em J. Analyse Math.}, 50:169--181, 1988.

\bibitem{Bou87-erg}
Jean Bourgain.
\newblock Pointwise ergodic theorems for arithmetic sets.
\newblock {\em Inst. Hautes \'{E}tudes Sci. Publ. Math.}, (69):5--45, 1989.
\newblock With an appendix by the author, Harry Furstenberg, Yitzhak Katznelson
  and Donald S. Ornstein.

\bibitem{Bolip}
Jean Bourgain.
\newblock A remark on the maximal function associated to an analytic vector
  field.
\newblock In {\em Analysis at {U}rbana, {V}ol. {I} ({U}rbana, {IL},
  1986--1987)}, volume 137 of {\em London Math. Soc. Lecture Note Ser.}, pages
  111--132. Cambridge Univ. Press, Cambridge, 1989.

\bibitem{Cal}
Alberto-P. Calder\'{o}n.
\newblock Cauchy integrals on {L}ipschitz curves and related operators.
\newblock {\em Proc. Nat. Acad. Sci. U.S.A.}, 74(4):1324--1327, 1977.

\bibitem{CZ1}
Alberto-P. Calder{\'o}n and Antoni Zygmund.
\newblock On the existence of certain singular integrals.
\newblock {\em Acta Math.}, 88:85--139, 1952.

\bibitem{CZ2}
Alberto-P. Calder\'{o}n and Antoni Zygmund.
\newblock On singular integrals.
\newblock {\em Amer. J. Math.}, 78:289--309, 1956.

\bibitem{CSWW}
Anthony Carbery, Andreas Seeger, Stephen Wainger, and James Wright.
\newblock Classes of singular integral operators along variable lines.
\newblock {\em J. Geom. Anal.}, 9(4):583--605, 1999.

\bibitem{CWW93}
Anthony Carbery, Stephen Wainger, and James Wright.
\newblock Hilbert transforms and maximal functions along variable flat plane
  curves.
\newblock In {\em Proceedings of the {C}onference in {H}onor of {J}ean-{P}ierre
  {K}ahane ({O}rsay, 1993)}, number Special Issue, pages 119--139, 1995.

\bibitem{c1}
Lennart Carleson.
\newblock On convergence and growth of partial sums of {F}ourier series.
\newblock {\em Acta Math.}, 116:135--157, 1966.

\bibitem{Chhilb}
Michael Christ.
\newblock Hilbert transforms along curves. {I}. {N}ilpotent groups.
\newblock {\em Ann. of Math. (2)}, 122(3):575--596, 1985.

\bibitem{C20}
Michael Christ.
\newblock On trilinear oscillatory integral inequalities and related topics.
\newblock https://arxiv.org/abs/2007.12753, 2020.

\bibitem{CDR}
Michael Christ, Polona Durcik, and Joris Roos.
\newblock Trilinear smoothing inequalities and a variant of the triangular
  {H}ilbert transform.
\newblock https://arxiv.org/abs/2008.10140, 2020.

\bibitem{cltt}
Michael Christ, Xiaochun Li, Terence Tao, and Christoph Thiele.
\newblock On multilinear oscillatory integrals, nonsingular and singular.
\newblock {\em Duke Math. J.}, 130(2):321--351, 2005.

\bibitem{CNSW}
Michael Christ, Alexander Nagel, Elias~M. Stein, and Stephen Wainger.
\newblock Singular and maximal {R}adon transforms: analysis and geometry.
\newblock {\em Ann. of Math. (2)}, 150(2):489--577, 1999.

\bibitem{CMM}
Ronald~R. Coifman, Alan McIntosh, and Yves Meyer.
\newblock L'int\'egrale de {C}auchy d\'efinit un op\'erateur born\'e sur
  {$L^{2}$}pour les courbes lipschitziennes.
\newblock {\em Ann. of Math. (2)}, 116(2):361--387, 1982.

\bibitem{DemeterTaoThiele}
Ciprian Demeter, Terence Tao, and Christoph Thiele.
\newblock {Maximal multilinear operators.}
\newblock {\em {Trans. Am. Math. Soc.}}, 360(9):4989--5042, 2008.

\bibitem{DT}
Ciprian Demeter and Christoph Thiele.
\newblock On the two-dimensional bilinear {H}ilbert transform.
\newblock {\em Amer. J. Math.}, 132(1):201--256, 2010.

\bibitem{DGTZ}
Francesco Di~Plinio, Shaoming Guo, Christoph Thiele, and Pavel Zorin-Kranich.
\newblock Square functions for bi-{L}ipschitz maps and directional operators.
\newblock {\em J. Funct. Anal.}, 275(8):2015--2058, 2018.

\bibitem{DiestelGrafakos}
Geoff Diestel and Loukas Grafakos.
\newblock Unboundedness of the ball bilinear multiplier operator.
\newblock {\em Nagoya Math. J.}, 185:151--159, 2007.

\bibitem{var_paraproducts}
Yen Do, Camil Muscalu, and Christoph Thiele.
\newblock Variational estimates for paraproducts.
\newblock {\em Rev. Mat. Iberoam.}, 28(3):857--878, 2012.

\bibitem{DoMuscaluThiele:VarIteratedFourier}
Yen Do, Camil Muscalu, and Christoph Thiele.
\newblock Variational estimates for the bilinear iterated {F}ourier integral.
\newblock {\em J. Funct. Anal.}, 272(5):2176--2233, 2017.

\bibitem{DGR}
Polona Durcik, Shaoming Guo, and Joris Roos.
\newblock A polynomial {R}oth theorem on the real line.
\newblock {\em Trans. Amer. Math. Soc.}, 371(10):6973--6993, 2019.

\bibitem{DKST}
Polona Durcik, Vjekoslav Kova{\v{c}}, Kristina~Ana Skreb, and Christoph Thiele.
\newblock Norm variation of ergodic averages with respect to two commuting
  transformations.
\newblock {\em Ergodic Theory Dynam. Systems}, 39(3):658--688, 2019.

\bibitem{Fab}
Eugene~B. Fabes.
\newblock Singular integrals and partial differential equations of parabolic
  type.
\newblock {\em Studia Math.}, 28:81--131, 1966/67.

\bibitem{fr}
Eugene~B. Fabes and Nestor~M. Rivi{\`e}re.
\newblock Singular integrals with mixed homogeneity.
\newblock {\em Studia Math.}, 27:19--38, 1966.

\bibitem{Fef71}
Charles Fefferman.
\newblock The multiplier problem for the ball.
\newblock {\em Ann. of Math. (2)}, 94:330--336, 1971.

\bibitem{Fefferman-ConvergencePolygonalFourierSeries}
Charles Fefferman.
\newblock On the convergence of multiple {F}ourier series.
\newblock {\em Bull. Amer. Math. Soc.}, 77:744--745, 1971.

\bibitem{Fefferman-DivergenceMultipleFourierSeries}
Charles Fefferman.
\newblock On the divergence of multiple {F}ourier series.
\newblock {\em Bull. Amer. Math. Soc.}, 77:191--195, 1971.

\bibitem{f}
Charles Fefferman.
\newblock Pointwise convergence of {F}ourier series.
\newblock {\em Ann. of Math. (2)}, 98:551--571, 1973.

\bibitem{Fou}
Joseph Fourier.
\newblock {\em Th\'eorie analytique de la chaleur}.
\newblock \'Editions Jacques Gabay, Paris, 1988.
\newblock Reprint of the 1822 original.

\bibitem{Fu}
Hillel Furstenberg.
\newblock Nonconventional ergodic averages.
\newblock In {\em The legacy of {J}ohn von {N}eumann ({H}empstead, {NY},
  1988)}, volume~50 of {\em Proc. Sympos. Pure Math.}, pages 43--56. Amer.
  Math. Soc., Providence, RI, 1990.

\bibitem{GL}
Alejandra Gaitan and Victor Lie.
\newblock The boundedness of the (sub)bilinear maximal function along
  ``non-flat'' smooth curves.
\newblock {\em J. Fourier Anal. Appl.}, 26(4):Paper No. 69, 33, 2020.

\bibitem{gowers}
William~T. Gowers.
\newblock A new proof of {S}zemer{\'e}di's theorem for arithmetic progressions
  of length four.
\newblock {\em Geom. Funct. Anal.}, 8(3):529--551, 1998.

\bibitem{GrafHonzik:MaxTransference}
Loukas Grafakos and Petr Honz\'{\i}k.
\newblock Maximal transference and summability of multilinear {F}ourier series.
\newblock {\em J. Aust. Math. Soc.}, 80(1):65--80, 2006.

\bibitem{GrafakosLi}
Loukas Grafakos and Xiaochun Li.
\newblock The disc as a bilinear multiplier.
\newblock {\em Amer. J. Math.}, 128(1):91--119, 2006.

\bibitem{multilinear-Marcink-constant}
Loukas Grafakos, Liguang Liu, Shanzhen Lu, and Fayou Zhao.
\newblock The multilinear {M}arcinkiewicz interpolation theorem revisited: the
  behavior of the constant.
\newblock {\em J. Funct. Anal.}, 262(5):2289--2313, 2012.

\bibitem{Graham_Kolesnik}
S.~W. Graham and G.~Kolesnik.
\newblock {\em van der {C}orput's method of exponential sums}, volume 126 of
  {\em London Mathematical Society Lecture Note Series}.
\newblock Cambridge University Press, Cambridge, 1991.

\bibitem{Gu1}
Shaoming Guo.
\newblock Hilbert transform along measurable vector fields constant on
  {L}ipschitz curves: {$L^2$} boundedness.
\newblock {\em Anal. PDE}, 8(5):1263--1288, 2015.

\bibitem{Gu2}
Shaoming Guo.
\newblock Hilbert transform along measurable vector fields constant on
  {L}ipschitz curves: {$L^p$} boundedness.
\newblock {\em Trans. Amer. Math. Soc.}, 369(4):2493--2519, 2017.

\bibitem{GHLR}
Shaoming Guo, Jonathan Hickman, Victor Lie, and Joris Roos.
\newblock Maximal operators and {H}ilbert transforms along variable non-flat
  homogeneous curves.
\newblock {\em Proc. Lond. Math. Soc. (3)}, 115(1):177--219, 2017.

\bibitem{GPRY}
Shaoming Guo, Lillian~B. Pierce, Joris Roos, and Po-Lam Yung.
\newblock Polynomial {C}arleson operators along monomial curves in the plane.
\newblock {\em J. Geom. Anal.}, 27(4):2977--3012, 2017.

\bibitem{GRSY}
Shaoming Guo, Joris Roos, Andreas Seeger, and Po-Lam Yung.
\newblock A maximal function for families of {H}ilbert transforms along
  homogeneous curves.
\newblock {\em Math. Ann.}, 377(1-2):69--114, 2020.

\bibitem{HKr}
Bernard Host and Bryna Kra.
\newblock Convergence of polynomial ergodic averages.
\newblock {\em Israel J. Math.}, 149:1--19, 2005.
\newblock Probability in mathematics.

\bibitem{hu}
Richard~A. Hunt.
\newblock On the convergence of {F}ourier series.
\newblock In {\em Orthogonal {E}xpansions and their {C}ontinuous {A}nalogues
  ({P}roc. {C}onf., {E}dwardsville, {I}ll., 1967)}, pages 235--255. Southern
  Illinois Univ. Press, Carbondale, Ill., 1968.

\bibitem{Jon}
B.~Frank Jones, Jr.
\newblock A class of singular integrals.
\newblock {\em Amer. J. Math.}, 86:441--462, 1964.

\bibitem{Kol2}
Andrey~N. Kolmogoroff.
\newblock Une serie de {F}ourier-{L}ebesgue divergente partout.
\newblock Number 183, pages 1327--1329. 1926.

\bibitem{Kol1}
Andrey~N. Kolmogorov.
\newblock Une serie de {F}ourier-{L}ebesgue divergente presque partout.
\newblock Number~4, pages 324--328. 1923.

\bibitem{KV}
Vjekoslav Kova{\v{c}}.
\newblock Boundedness of the twisted paraproduct.
\newblock {\em Rev. Mat. Iberoam.}, 28(4):1143--1164, 2012.

\bibitem{KTZ}
Vjekoslav Kova{\v{c}}, Christoph Thiele, and Pavel Zorin-Kranich.
\newblock Dyadic triangular {H}ilbert transform of two general functions and
  one not too general function.
\newblock {\em Forum Math. Sigma}, 3:e25, 27, 2015.

\bibitem{K19}
Ben Krause.
\newblock A non-linear {R}oth theorem for sets of positive density.
\newblock Arxiv: https://arxiv.org/abs/1901.01371, 2019.

\bibitem{la1}
Michael~T. Lacey.
\newblock The bilinear {H}ilbert transform is pointwise finite.
\newblock {\em Rev. Mat. Iberoamericana}, 13(2):411--469, 1997.

\bibitem{Lacey}
Michael~T. Lacey.
\newblock {The bilinear maximal functions map into \(L^p\) for \(2/3 < p \leq
  1\).}
\newblock {\em {Ann. Math. (2)}}, 151(1):35--57, 2000.

\bibitem{LL1}
Michael~T. Lacey and Xiaochun Li.
\newblock Maximal theorems for the directional {H}ilbert transform on the
  plane.
\newblock {\em Trans. Amer. Math. Soc.}, 358(9):4099--4117, 2006.

\bibitem{Lali}
Michael~T. Lacey and Xiaochun Li.
\newblock On a conjecture of {E}. {M}. {S}tein on the {H}ilbert transform on
  vector fields.
\newblock {\em Mem. Amer. Math. Soc.}, 205(965):viii+72, 2010.

\bibitem{lt1}
Michael~T. Lacey and Christoph Thiele.
\newblock {$L^p$} estimates on the bilinear {H}ilbert transform for
  {$2<p<\infty$}.
\newblock {\em Ann. of Math. (2)}, 146(3):693--724, 1997.

\bibitem{lt2}
Michael~T. Lacey and Christoph Thiele.
\newblock On {C}alder\'{o}n's conjecture.
\newblock {\em Ann. of Math. (2)}, 149(2):475--496, 1999.

\bibitem{lt3}
Michael~T. Lacey and Christoph Thiele.
\newblock A proof of boundedness of the {C}arleson operator.
\newblock {\em Math. Res. Lett.}, 7(4):361--370, 2000.

\bibitem{LY2}
Junfeng Li and Haixia Yu.
\newblock ${L}^p$ {B}oundedness of {H}ilbert {T}ransforms and {M}aximal
  {F}unctions along {V}ariable {C}urves.
\newblock Arxiv: https://arxiv.org/abs/1808.01447, 2018.

\bibitem{li}
Xiaochun Li.
\newblock Bilinear {H}ilbert transforms along curves {I}: {T}he monomial case.
\newblock {\em Anal. PDE}, 6(1):197--220, 2013.

\bibitem{LiMuscalu}
Xiaochun Li and Camil Muscalu.
\newblock Generalizations of the {C}arleson-{H}unt theorem. {I}. {T}he
  classical singularity case.
\newblock {\em Amer. J. Math.}, 129(4):983--1018, 2007.

\bibitem{LX}
Xiaochun Li and Lechao Xiao.
\newblock Uniform estimates for bilinear {H}ilbert transforms and bilinear
  maximal functions associated to polynomials.
\newblock {\em Amer. J. Math.}, 138(4):907--962, 2016.

\bibitem{lv3n}
Victor Lie.
\newblock A {N}ote on the {P}olynomial {C}arleson {O}perator in higher
  dimensions.
\newblock Arxiv: https://arxiv.org/abs/1712.03092, 2017.

\bibitem{lvUnif}
Victor Lie.
\newblock A unified approach to three themes in harmonic analysis
  ({I}$\,\&\,${II}), 106 pp.
\newblock Arxiv: https://arxiv.org/abs/1902.03807, 2019.

\bibitem{lv1}
Victor Lie.
\newblock The (weak-{$L^2$}) boundedness of the quadratic {C}arleson operator.
\newblock {\em Geom. Funct. Anal.}, 19(2):457--497, 2009.

\bibitem{L1}
Victor Lie.
\newblock {On the boundedness of the bilinear Hilbert transform along
  ``non-flat'' smooth curves.}
\newblock {\em {Amer. J. Math.}}, 137(2):313--363, 2015.

\bibitem{lv9}
Victor Lie.
\newblock Pointwise convergence of {F}ourier series ({I}). {O}n a conjecture of
  {K}onyagin.
\newblock {\em J. Eur. Math. Soc. (JEMS)}, 19(6):1655--1728, 2017.

\bibitem{lv10}
Victor Lie.
\newblock On the boundedness of the bilinear {H}ilbert transform along
  ``non-flat'' smooth curves. {T}he {B}anach triangle case {$(L^r,\ 1\leq
  r<\infty)$}.
\newblock {\em Rev. Mat. Iberoam.}, 34(1):331--353, 2018.

\bibitem{lv19}
Victor Lie.
\newblock The pointwise convergence of {F}ourier series ({II}). {S}trong
  {$L^1$} case for the lacunary {C}arleson operator.
\newblock {\em Adv. Math.}, 357:106831, 84, 2019.

\bibitem{lv3}
Victor Lie.
\newblock The polynomial {C}arleson operator.
\newblock {\em Ann. of Math. (2)}, 192(1):47--163, 2020.

\bibitem{Luz}
Nikolai~N. Luzin.
\newblock Integral i trigonometriceskii ryad.
\newblock page 550. Gosudarstv. Izdat. Tehn.-Teor. Lit., Moscow-Leningrad,
  1951.
\newblock Editing and commentary by N. K. Bari and D. E. Mensov.

\bibitem{MRi}
Gianfranco Marletta and Fulvio Ricci.
\newblock Two-parameter maximal functions associated with homogeneous surfaces
  in {$\bold R^n$}.
\newblock {\em Studia Math.}, 130(1):53--65, 1998.

\bibitem{multilinear_harmonicII}
Camil Muscalu and Wilhelm Schlag.
\newblock {\em Classical and multilinear harmonic analysis. {V}ol. {II}},
  volume 138 of {\em Cambridge Studies in Advanced Mathematics}.
\newblock Cambridge University Press, Cambridge, 2013.

\bibitem{Mlinop}
Camil Muscalu, Terence Tao, and Christoph Thiele.
\newblock Multi-linear operators given by singular multipliers.
\newblock {\em J. Amer. Math. Soc.}, 15(2):469--496, 2002.

\bibitem{MTTBiest2}
Camil Muscalu, Terence Tao, and Christoph Thiele.
\newblock {\(L^p\) estimates for the biest. II: The Fourier case.}
\newblock {\em {Math. Ann.}}, 329(3):427--461, 2004.

\bibitem{bi-Carleson}
Camil Muscalu, Terence Tao, and Christoph Thiele.
\newblock The {Bi}-{Carleson} operator.
\newblock {\em Geom. Funct. Anal.}, 16(1):230--277, 2006.

\bibitem{NRS74}
Alexander Nagel, Nestor Rivi\`ere, and Stephen Wainger.
\newblock On {H}ilbert transforms along curves.
\newblock {\em Bull. Amer. Math. Soc.}, 80:106--108, 1974.

\bibitem{NRS76}
Alexander Nagel, Nestor~M. Rivi\`ere, and Stephen Wainger.
\newblock On {H}ilbert transforms along curves. {II}.
\newblock {\em Amer. J. Math.}, 98(2):395--403, 1976.

\bibitem{variational_Carleson}
Richard Oberlin, Andreas Seeger, Terence Tao, Christoph Thiele, and James
  Wright.
\newblock A variation norm {C}arleson theorem.
\newblock {\em J. Eur. Math. Soc. (JEMS)}, 14(2):421--464, 2012.

\bibitem{PS94}
Duong~H. Phong and Elias~M. Stein.
\newblock Operator versions of the van der {C}orput lemma and {F}ourier
  integral operators.
\newblock {\em Math. Res. Lett.}, 1(1):27--33, 1994.

\bibitem{PP}
Lillian~B. Pierce and Po-Lam Yung.
\newblock A polynomial {C}arleson operator along the paraboloid.
\newblock {\em Rev. Mat. Iberoam.}, 35(2):339--422, 2019.

\bibitem{Rie}
Marcel Riesz.
\newblock Sur les fonctions conjugu\'{e}es.
\newblock {\em Math. Z.}, 27(1):218--244, 1928.

\bibitem{JR}
Joris Roos.
\newblock Bounds for anisotropic {C}arleson operators.
\newblock {\em J. Fourier Anal. Appl.}, 25(5):2324--2355, 2019.

\bibitem{RF}
Jos{\'e} Rubio~de Francia.
\newblock A {Littlewood}-{Paley} {Inequality} for {Arbitrary} {Intervals}.
\newblock {\em Revista Matematica Iberoamericana}, 1(2):891--921, 1985.

\bibitem{SeWa}
Andreas Seeger and Stephen Wainger.
\newblock Singular {R}adon transforms and maximal functions under convexity
  assumptions.
\newblock {\em Rev. Mat. Iberoamericana}, 19(3):1019--1044, 2003.

\bibitem{sj2}
Per Sj{\"o}lin.
\newblock Convergence almost everywhere of certain singular integrals and
  multiple {F}ourier series.
\newblock {\em Ark. Mat.}, 9(3):65--90, 1971.

\bibitem{s1}
Elias~M. Stein.
\newblock On limits of seqences of operators.
\newblock {\em Ann. of Math. (2)}, 74:140--170, 1961.

\bibitem{BigStein}
Elias~M. Stein.
\newblock {\em Harmonic analysis: real-variable methods, orthogonality, and
  oscillatory integrals}, volume~43 of {\em Princeton Mathematical Series}.
\newblock Princeton University Press, Princeton, NJ, 1993.
\newblock With the assistance of Timothy S. Murphy, Monographs in Harmonic
  Analysis, III.

\bibitem{Stein}
Elias~M. Stein.
\newblock Oscillatory integrals related to {R}adon-like transforms.
\newblock In {\em Proceedings of the {C}onference in {H}onor of {J}ean-{P}ierre
  {K}ahane ({O}rsay, 1993)}, number Special Issue, pages 535--551, 1995.

\bibitem{SS1}
Elias~M. Stein and Brian Street.
\newblock Multi-parameter singular {R}adon transforms {III}: {R}eal analytic
  surfaces.
\newblock {\em Adv. Math.}, 229(4):2210--2238, 2012.

\bibitem{sw70}
Elias~M. Stein and Stephen Wainger.
\newblock The estimation of an integral arising in multiplier transformations.
\newblock {\em Studia Math.}, 35:101--104, 1970.

\bibitem{SteinWainger}
Elias~M. Stein and Stephen Wainger.
\newblock Oscillatory integrals related to {C}arleson's theorem.
\newblock {\em Math. Res. Lett.}, 8(5-6):789--800, 2001.

\bibitem{zk1}
Pavel Zorin-Kranich.
\newblock Maximal polynomial modulations of singular integrals.
\newblock {\em Adv. Math.}, 386:107832, 2021.

\end{thebibliography}

\end{document}